\documentclass[12pt]{article}
\usepackage{setspace}
\usepackage{mathtools}
\usepackage[utf8]{inputenc}

\usepackage[square,numbers]{natbib}
\usepackage[dvipsnames]{xcolor}

\RequirePackage[colorlinks,citecolor=blue,urlcolor=blue]{hyperref}

\usepackage{graphicx,graphics}
\usepackage{subcaption}

\usepackage[flushleft]{threeparttable}
\usepackage{enumerate}
\usepackage{algorithm}
\usepackage{algpseudocode}
\usepackage{setspace}
\usepackage{url}
\newcommand{\circnum}[1]{\text{\textcircled{\raisebox{-0.2ex}{\small{#1}}}}}
\usepackage{supertabular}
\usepackage{amsmath,amscd,amsfonts,amssymb}
\usepackage{fancyhdr}
\usepackage{color}
\usepackage{verbatim}

\usepackage{array,supertabular}
\usepackage{float}
\usepackage{rotating}
\usepackage{lscape}

\usepackage{booktabs,caption,fixltx2e}
\usepackage[flushleft]{threeparttable}
\usepackage{amsthm}

\usepackage{dcolumn}
\newcolumntype{L}{D{.}{.}{2,5}}

\usepackage{chngpage} % allows for temporary adjustment of side margins

\newtheorem{theorem}{Theorem}[section]
\newtheorem{lemma}[theorem]{Lemma}

\newtheorem{proposition}[theorem]{Proposition}

\newtheorem{assumption}{Assumption}
\newtheorem{example}{Example}
\newtheorem{remark}{Remark}
\usepackage[title]{appendix}

\def\mds{\medskip}
\def\Rb{{\mathbb R}}
\def\Pc{{\mathcal P}}

\def\Nc{{\mathcal N}}

\long\def\symbolfootnote[#1]#2{\begingroup%
\def\thefootnote{\fnsymbol{footnote}}\footnotetext[#1]{#2}\footnotemark[#1]\endgroup}

\makeatletter
\@addtoreset{equation}{section}

\makeatother

\newcounter{Fig}[figure]

\newcounter{Tab}[table]

  {\refstepcounter{Tab}
   \addtocounter{table}{1}
   \samepage\vspace{0.2cm}
   \centerline{#1} \nobreak
   \begin{verse}{Table \thesection.\arabic{Tab}:}
  }%
  {\end{verse} \vspace{0.2cm}}

\relax

\newcommand\tr{\mathrm{tr}}
\newcommand\prox{\mathrm{prox}}
\newcommand\acc{\mathrm{acc}}
\DeclareMathOperator*{\argmin}{\arg\min}
\usepackage{footnote}
\usepackage{adjustbox}

\def \Eb{{\mathbb E}}

\def \Gb{{\mathbb G}}

\def \Lb{{\mathbb L}}

\def \Pb{{\mathbb P}}
\def \Rb{{\mathbb R}}

\def \Bc{{\mathcal B}}

\def \Uc{{\mathcal U}}

\def \Nc{{\mathcal N}}
\def \Pc{{\mathcal P}}

\def \Tc{{\mathcal T}}

\def \Xc{{\mathcal X}}

\DeclareMathOperator*{\Sym}{\mathrm{Sym}}

\newcommand{\bqa}{\begin{eqnarray*}}
\newcommand{\eqa}{\end{eqnarray*}}
\newcommand{\bqan}{\begin{eqnarray}}
\newcommand{\eqan}{\end{eqnarray}}
\newcommand{\bqt}{\begin{quote}}
\newcommand{\eqt}{\end{quote}}
\newcommand{\bt}{\begin{tabbing}}
\newcommand{\et}{\end{tabbing}}
\newcommand{\bit}{\begin{itemize}}
\newcommand{\eit}{\end{itemize}}
\newcommand{\ben}{\begin{enumerate}}
\newcommand{\een}{\end{enumerate}}
\newcommand{\beq}{\begin{equation}}
\newcommand{\eeq}{\end{equation}}
\newcommand{\beqw}{\begin{equation*}}
\newcommand{\eeqw}{\end{equation*}}

\newcommand{\eps}{\epsilon}

\newcommand{\revisionmode}{0}

\ifnum\revisionmode=1
  \newcommand{\rev}[1]{{\color{red} #1}}
\else
  \newcommand{\rev}[1]{#1}
\fi

\def\mds{\medskip}
\def\1{{\mathbf 1}}
\def\0{{\mathbf 0}}

\begin{document}

\title{\Large \bf \rev{Change Point Detection} in \rev{Precision Matrices} with D-trace Loss}

\author{Ying Lin \and Benjamin Poignard \and Ting Kei Pong \and Akiko Takeda}

\author{Ying Lin\footnote{Department of Data and Systems Engineering, The University of Hong Kong, People's Republic of China. E-mail address: ylin95@hku.hk}, Benjamin Poignard\footnote{Faculty of Science and Technology, Keio University, Japan; Jointly affiliated at RIKEN and OIST (MLDS unit). E-mail address: bpoignard@math.keio.ac.jp}, Ting Kei Pong\footnote{Department of Applied Mathematics, The Hong Kong Polytechnic University, People's Republic of China. E-mail address: tk.pong@polyu.edu.hk}, Akiko Takeda\footnote{Graduate School of Information Science and Technology, The University of Tokyo, Japan; Center for Advanced Intelligence Project, RIKEN, Japan. E-mail address: takeda@mist.i.u-tokyo.ac.jp}}

\maketitle

\begin{abstract}
We consider the problem of estimating a time-varying sparse precision matrix, which is assumed to evolve in a piecewise constant manner. Building upon the Group Fused LASSO and LASSO penalty functions, we estimate both the \rev{precision matrix} and the change points. We propose an alternative estimator to the commonly employed Gaussian likelihood loss, namely the D-trace loss. We provide the conditions for the consistency of the estimated change points and of the sparse estimators in each block.
  We show that the solutions to the corresponding estimation problem exist when some conditions relating to the tuning parameters of the penalty functions are satisfied. Unfortunately, these conditions are not verifiable in general, posing challenges for tuning the parameters in practice.
  To address this issue, we introduce a modified regularizer and develop a revised problem that always admits  solutions: these solutions can be used for detecting possible unsolvability of the original problem or obtaining a solution of the original problem otherwise.
  An alternating direction method of multipliers (ADMM) is then proposed to solve the revised problem.
  The relevance of the method is illustrated through \rev{numerical experiments}.


\medskip
\noindent

\end{abstract}

\section{Introduction}

Much attention has been devoted to the detection of multiple change points in the underlying data generating process of some realized sample. The model parameters are typically assumed to be constant in each regime but can experience ``jumps'' over time. The extraction of these \rev{change points} is usually performed through the application of some filtering techniques. A popular method is the fused LASSO which screens the existence of \rev{change points in the parameters} over the sample of observations. In particular, \cite{Harchaoui2010} developed a framework to recover multiple change points for one-dimensional piecewise constant signals. \cite{Chan2014} extended this procedure to grouped parameters with the Group Fused LASSO in the context of autoregressive time series models. \cite{Qian2016} applied the Group Fused LASSO to linear regression models. 

This paper considers the detection of \rev{change points in the precision matrix characterizing piecewise-constant, time-evolving dependence in multivariate time series. Many dependent data sets exhibit heterogeneity due to nonstationarity, a situation in which the data-generating process evolves over time and experiences change points. This phenomenon commonly arises in economic data in the presence of regime switches, for example, following a financial crisis. Biological data, such as gene expression measurements, may also exhibit time-evolving dependence related to life cycles or treatment effects. The problem consists in detecting these changes in the dependence structure.} Some works have been devoted to change point detection for precision matrix. \cite{Zhou2010} assumed that the covariance matrix evolves smoothly over time. \rev{Unlike} \cite{Kolar2012}, which focused on the detection of multiple \rev{change points} at a node level, \cite{Roy2017} considered single change point detection which impacts the global \rev{dependence} structure. Our aim is to detect multiple change points that affect the whole \rev{precision matrix}. \rev{From this perspective, our work is related to \cite{Hallac2017} and \cite{Gibberd2017}, who introduced the Group Fused Graphical LASSO (GFGL) for change point detection in time-varying graphical models}. The latter filtering technique is a mixture of the Group Fused LASSO of \cite{Bleakley2011} and the LASSO applied to the Gaussian likelihood, \rev{and aims to detect changes in the underlying graphical network encoded in the precision matrix. Interestingly, the graphical LASSO also works in the context of non-Gaussian data.} We propose an alternative estimator to the penalized Gaussian likelihood, which builds upon the D-trace loss of \cite{Zhang2014}. The formulation of the latter is much simpler than the Gaussian likelihood, thus allowing for a more direct theoretical analysis and implementation. 

The D-trace loss and its extensions have been applied to diverse problems: while \cite{Yuan2017} considered network change analysis between two samples of vector-valued observations \rev{through the estimation of the precision matrix difference of the two samples}, \cite{Ji2021} adapted the latter framework to differential network analysis for fMRI data; in the context of compositional data, \cite{Yuan2019} and \cite{Zhang2024} considered modified versions of the D-trace loss to estimate the high-dimensional compositional precision matrix under sparsity constraint; using the matching of the coefficients of SVAR processes and the entries of the precision matrix of the observed sample, \cite{poignard2023} considered the sparse estimation of the latter based on the D-trace loss. In addition to the detection of \rev{change points} in the \rev{precision matrix}, we allow for sparse dependence in the entries of the precision matrix in each regime. This motivates the use of the LASSO penalization function applied to the coefficients of the piecewise constant precision matrix.

The corresponding estimation problem formulated in \eqref{stat_crit}, Section \ref{framework}, with the D-trace loss, includes two tuning parameters. In practice, the {\em optimal} parameters are unknown a priori and are expected to be selected upon solving the estimation problems with a series of tuning parameters.
However, we show that the problem with general tuning parameters may be unbounded from below (and hence, does not have solutions). Consequently, to identify the optimal parameters, one may encounter estimation problems that are unbounded from below, and the unboundedness can be numerically difficult to detect.
To address this issue, we introduce a new regularizer, thereby constructing a revised problem that consistently has solutions regardless of the choice of the tuning parameters.
In addition to the existence of a solution, this revised problem also enjoys several desirable properties.
On the one hand, if the solutions to the revised problem possess some easy-to-detect patterns, then the original problem may not have solutions, and we can further \textit{update} the tuning parameters towards obtaining reliable estimators.
On the other hand, if the patterns are not detected, then the solutions to the revised problem also solve the original problem.
We adapt the celebrated alternating direction method of multipliers (ADMM) to solve the revised problem, with its convergence guaranteed by \cite{FPST2013}.
ADMM is a widely used algorithm for solving convex optimization problems with separable objectives and linear coupling constraints.
Its classical version was first introduced by \cite{GM1975} and \cite{GM1976}, with the convergence established in \cite{FG1983}.
Then \cite{EB92} showed that ADMM is equivalent to the proximal point algorithm applied to a certain maximal monotone operator.
This insight led to the development of the first proximal ADMM by \cite{E1994}.
Building upon Eckstein's work, \cite{HLHY2002} extended the proximal ADMM to include a broader class of proximal terms.
This approach was further advanced by \cite{FPST2013}, which generalized the method to allow the use of semi-proximal terms, enhancing the algorithm's applicability.
More recently, considerable efforts have been made to further accelerate variants of ADMM: see, e.g.,
the accelerated linearized ADMM in \cite{OCLP2015}, \cite{X2017}, \cite{LL2019};
the accelerated proximal ADMM based on Halpern iteration in \cite{ZYS2022}, \cite{YZLS2024}.

Our contributions can be summarized as follows: we propose a new estimator for \rev{change points} detection in the precision matrix, the Group Fused D-trace LASSO (GFDtL); we derive the conditions for which all the \rev{change points} and the precision matrices can be consistently estimated when the estimated \rev{change points} coincide with the true number of \rev{change points}; we provide a modified regularizer to ensure the existence of solutions to the revised problem, and show that these solutions either exhibit some easily verifiable patterns indicating the \textit{possible} unsolvability of the original problem so that we can further update the tuning parameters towards obtaining reliable estimators, or solve the original problem if the patterns are not detected; an ADMM is adapted for implementation with convergence guarantees. The relevance of our proposed change points estimator compared with standard estimators, such as GFGL, is illustrated through numerical experiments.

The rest of the paper is organized as follows.
Section \ref{framework} details the framework and estimation procedure.
Section \ref{asymptotic_properties} contains the asymptotic properties.
Section \ref{sec:opt} is devoted to the optimization aspect of the estimation procedure.
Section \ref{sec:implementation} provides the implementation details of the estimator.
Section \ref{sec:tuning-para} discusses different methods to select the two tuning parameters.
Section \ref{sec:simulations} provide experimental results based on simulations.
Section \ref{sec:sen-ana} provides a sensitivity analysis, while Section \ref{sec:time} details the computational complexity.
Section \ref{sec:real_data} provides experimental results based on real data.
All the proofs of the main text, the details of the ADMM algorithm, and an illustrative example are deferred to the Appendices.

\textbf{\textit{Notation:}} Throughout this paper, we use $ V_k $ and $ A_{kl} $ to denote the $ k $-th element of a vector $ V \in \Rb^d $ and the $ (k, l) $-th element of a matrix $ A \in \Rb^{m \times n} $, respectively.
We write $A^\top$ (resp. $V^\top$) to denote the transpose of the matrix $A$ (resp. the vector $V$).
For any square matrix $ A \in \Rb^{n \times n} $, we write $ \Sym(A) \coloneqq (A + A^{\top}) / 2 $ to denote the symmetrization of $ A $.
The $ p $-th order identity matrix is denoted by $ I_p $.
We denote by $\textbf{0}_{k \times l} \in \Rb^{k \times l}$ (resp. $\mathbf{1}_{k \times l} \in \Rb^{k \times l}$) the $k \times l$ zero matrix (resp. $ k \times l $ matrix of ones).
We write $\text{vec}(A)$ to denote the vectorization operator that stacks the columns of $A$ on top of one another into a vector.
For two matrices $A$ and $B$, $A \otimes B$ is the Kronecker product.
The set of symmetric matrices is denoted by $ \mathcal{S}^n $.
A symmetric matrix $ S \in \mathcal{S}^n $ is said to be positive semi-definite (resp. positive definite) and written as $ S \succeq 0 $ (resp. $ S \succ 0 $) if all its eigenvalues are non-negative (resp. positive).
The expression $ A \succeq B $ (resp. $ A \succ B $) for $ A $, $ B \in \mathcal{S}^n $ means $ A - B \succeq 0 $ (resp. $ A - B \succ 0 $).
We use $ \mathcal{S}_+^n $ and $ \mathcal{S}_{++}^n $ to denote the sets of positive semi-definite matrices and positive definite matrices, respectively.
For a positive semi-definite matrix $ S $, $ S^{\frac{1}{2}} $ is the unique positive semi-definite matrix such that $ S = S^{\frac{1}{2}} S^{\frac{1}{2}} $.
The $\ell_p$ norm for $V \in \Rb^{d}$ is denoted by $\|V\|_p = \big(\sum^{d}_{k=1} |V_k|^p \big)^{1/p}$, $p \geq 1$.
The Frobenius norm and the off-diagonal $ \ell_1 $ (semi)norm of a matrix $ A \in \Rb^{m \times n} $ are denoted by $ \| A \|_F = \sqrt{\sum_{k=1}^m \sum_{l=1}^n | A_{kl} |^2 } $ and $ \| A \|_{1, \mathrm{off}} = \sum_{k\neq l} | A_{kl} | $, respectively.
The spectral norm, i.e., the maximum singular value of a matrix $ A $ is written as $ \| A \|_s $. For a symmetric matrix A, we use $\lambda_{\max}(A)$ (resp. $\lambda_{\min}(A)$) to denote its largest (resp. smallest) eigenvalue.
The maximum absolute value of all entries of a matrix $ A $ is denoted by $ \| A \|_{\max} $.
We use $ \mathcal{S}_{\mathrm{off}}^n $ to denote the set of $ n \times n $ symmetric matrices whose diagonal elements are $ 0 $.
By $ \{Y_{t, \mathrm{off}}\}_{t=1}^T $ we denote a sequence of symmetric matrices whose diagonal elements are $ 0 $, i.e., $ Y_{t, \mathrm{off}} \in \mathcal{S}_{\mathrm{off}}^n $.
For a sequence of symmetric matrices $ \{\Theta_t\}_{t=1}^T $, $ \Theta_{uv, t} $ refers to the $ (u, v) $-th element of $ \Theta_t $, and $ \Theta_{t, \mathrm{off}} $ is the copy of $ \Theta_t $ with diagonal elements set to $ 0 $; in particular, $ \| \Theta_t \|_{1, \mathrm{off}} = \| \Theta_{t, \mathrm{off}} \|_1 $.
Given $ \epsilon \geq 0 $, we use $ \mathrm{Proj}_{\cdot \succeq \epsilon I_p} $ to denote the projection onto $ \{S \,:\, S \succeq \epsilon I_p\} $.
The proximal operator of a function $ f $ at $ x $ is defined as $\prox_f(x) = \argmin_y \{ f(y) + \frac{1}{2} \| y - x \|^2 \}$; for more details about the proximal operator, we refer the interested readers to Sections 12.4, 14.1, 14.2 in \cite{BC2011}.

\section{Framework}\label{framework}

For a sequence of a $p$-dimensional random vector $(X_t)$ observed at $t=1,\ldots,T$, we consider the estimation of the underlying \rev{precision matrix}. The latter is assumed to evolve over time and the task is to recover the \rev{change points}. More formally, we denote by $\{\Bc_j\}_{1 \leq j \leq m}$ a disjoint partitioning of the set $\{1,\ldots,T\}$ such that $\Bc_j \cap \Bc_{j'}=\emptyset, j \neq j'$, $\cup_j \Bc_j=\{1,\ldots,T\}$ and $\Bc_j = \{T_{j-1},T_{j-1}+1,\ldots,T_j-1\}$. The partition of the \rev{change points} is denoted by $\Tc_m =\{T_1<T_2<\ldots<T_m\}$ with the convention $T_0=1, T_{m+1}=T+1$. Then, we assume $\Eb[X_t]=0$ and Var$(X_t) = \Sigma_j$ for $t\in \Bc_j$, such that the observations indexed by elements in $\Bc_j$ are $p$-dimensional realizations of a centered random variable with variance-covariance $\Sigma_j$. We denote by $\Omega_j = \Sigma^{-1}_j$ the precision matrix with entries $\Omega_{uv,j}, 1 \leq u,v\leq p$. In practice, we consider the sequence $\{\Theta_1,\ldots,\Theta_T\}$ such that the total number of distinct matrices in the set is $m+1$ and $\Theta_t = \Omega_j, \; t \in \Bc_j, \; j = 1,\ldots,m+1$.
We are interested in estimating the unknown true number $m^\ast$ of unknown \rev{change points}, the true partition $\Tc^\ast_{m^\ast}=\{T^\ast_1<T^\ast_2<\ldots<T^\ast_{m^\ast}\}$ and the true unknown precision matrices $\Omega^\ast_j$. As a consequence, the true data generating process is assumed to be \rev{$\Eb[X_t] = 0$, $\text{Var}(X_t) = \Sigma^\ast_j$, $\Theta^\ast_t = \Omega^\ast_j= \Sigma^{\ast-1}_j$ when $t =T^\ast_{j-1},T^\ast_{j-1}+1,\ldots,T^\ast_{j}-1$,}
and $1 \leq j \leq m^\ast+1$, $T^\ast_0=1, T^\ast_{m^\ast+1}=T+1$ with blocks $\Bc^\ast_j = \{T^\ast_{j-1},\ldots,T^\ast_j-1\}$. While $m^\ast$ and the \rev{change points} are unknown, $m^\ast$ is much smaller than $T$ and, assuming the underlying \rev{precision matrix} may exhibit some sparse structures, we consider the sparse estimation of $\Theta_t$'s and the estimation of $\Tc_{m}$ via a mixture of LASSO and Group Fused LASSO, which we will refer as the Group Fused D-trace LASSO (GFDtL), defined as
\begin{equation}\label{stat_crit}
\{\widehat{\Theta}_t\}^T_{t=1} = \underset{\Theta_t\succeq 0, 1 \leq t \leq T}{\arg\;\min} \Big\{\Lb(\{\Theta_t\}^T_{t=1},\mathcal{X}_T) + \lambda_1\overset{T}{\underset{t=1}{\sum}}\|\Theta_t\|_{1,\text{off}}+\lambda_2\overset{T}{\underset{t=2}{\sum}} \|\Theta_t-\Theta_{t-1}\|_F\Big\},
\end{equation}
where $\mathcal{X}_T=(X_1,\ldots,X_T)$, $\lambda_1,\lambda_2$ are the tuning parameters, $\|\Theta_t\|_{1,\text{off}}=\sum_{k\neq l}|\Theta_{kl,t}|$ and the D-trace loss of \cite{Zhang2014} is defined as
\begin{equation*}
\Lb(\{\Theta_t\}^T_{t=1},\mathcal{X}_T) = \frac{1}{T}\overset{T}{\underset{t=1}{\sum}}\Big[\text{tr}(\frac{1}{2}\Theta^2_t X_tX^\top_t) - \text{tr}(\Theta_t)\Big].
\end{equation*}
\rev{For a solution $\widehat{\Theta}_t,t=1,\ldots,T$ of (\ref{stat_crit}), there is a block partition $\{\widehat{\mathcal{B}}_1,\ldots,\widehat{\mathcal{B}}_{\widehat{m}+1}\}$ of $\{1,\ldots,T\}$ such that $\widehat{\Theta}_t = \widehat{\Theta}_{s}$ for any $t,s \in \widehat{\mathcal{B}}_j=[\widehat{T}_{j-1},\widehat{T}_j-1]$ and $\widehat{\Theta}_{\widehat{T}_j} \neq \widehat{\Theta}_{\widehat{T}_{j}-1}$, $j=1,\ldots,\widehat{m}$, where $\widehat{T}_0=1, \widehat{T}_{\widehat{m}+1}=T+1$. So $\widehat{m}$ and $\widehat{\Tc}_{\widehat{m}}\coloneqq \{\widehat{T}_1<\widehat{T}_2<\ldots<\widehat{T}_{\widehat{m}}\}$ are the estimated number of change points and the estimated set of change points. We define $\widehat{\Omega}_{j}=\widehat{\Theta}_{\widehat{T}_{j-1}}$ as the estimator of $\Omega_j$ for $j=1,\ldots,\widehat{m}+1$.}
\rev{In Lemma~\ref{optimality_cond}, we provide the optimality conditions satisfied by GFDtL: we show that if $t=\widehat{T}_j$ with $j\in \{1,\ldots,\widehat{m}\}$, that is, $t$ corresponds to one of the estimated change points, then $\widehat{\Theta}_t-\widehat{\Theta}_{t-1} \neq \mathbf{0}_{p \times p}$, and we obtain the corresponding optimality conditions. Furthermore, sparsity in the estimated precision matrix for a given block is controlled by $\lambda_1$, whereas $\lambda_2$ affects the smoothing and guarantees that the solution is piecewise constant, as highlighted by Lemma~H.3. }

\section{Asymptotic Properties}\label{asymptotic_properties}

First, we define some notations and present the assumptions. Define $\mathcal{I}^\ast_j=T^\ast_j-T^\ast_{j-1}$, $\mathcal{I}_{\min}=\underset{1 \leq j \leq m^\ast+1}{\min}|\mathcal{I}^\ast_j|, \; \eta_{\min}=\underset{1 \leq j \leq m^\ast}{\min}\|\Omega^\ast_{j+1}-\Omega^\ast_j\|_F, \; \eta_{\max}=\underset{1 \leq j \leq m^\ast}{\max}\|\Omega^\ast_{j+1}-\Omega^\ast_j\|_F,  \;s^\ast_{\max} = \underset{1\leq j \leq m^\ast+1}{\max}\|\Omega^\ast_j\|_F$. \rev{Let $\mathcal{F}^0_{-\infty},\mathcal{F}^{\infty}_T$ denote the filtrations generated by $\{(X_t): - \infty \leq t \leq 0\}$ and $\{(X_t): T \leq t \leq \infty\}$. Define the mixing coefficient $\alpha(T) = \sup_{A \in \mathcal{F}^0_{-\infty}, B\in \mathcal{F}^{\infty}_T } |\rev{\Pb(A \cap B)-\Pb(A)\,\Pb(B)}|$.}

\begin{assumption}\label{assumption_dgp}
\begin{itemize}
    \item[(i)] \rev{$\exists 0 < \rho < 1$ such that for all $t \in \mathbb{Z}^+$, $\alpha(t) \leq c_{\alpha}\rho^{t}$, with $c_{\alpha}>0$.}
    \item[(ii)] $\exists \gamma, b>0$ such that $\forall s>0$, $\forall 1 \leq k,l \leq p$, $\sup_{t\geq 1}\,\Pb(|X_{k,t}X_{l,t}|>s)\leq \exp(1-(s/b)^\gamma)$.
    \end{itemize}
\end{assumption}

\begin{assumption}\label{assumption_regularity}
$\exists \underline{\mu}, \overline{\mu}$:
    $0 < \underline{\mu} \leq \underset{1 \leq j \leq m^\ast+1}{\min}\lambda_{\min}\big(\Sigma^\ast_j\big) \;\; \rev{\text{and}} \;\; \underset{1 \leq j \leq m^\ast+1}{\max}\lambda_{\rev{\max}}\big(\Sigma^\ast_j\big)\leq \overline{\mu}<\infty$.
\end{assumption}

\begin{assumption}\label{assumption_rates}
Let $(\delta_T)$ be a non-increasing positive sequence converging to zero.
\begin{itemize}
    \item[(i)] $T \delta_T \geq c_v \log(pT)^{(2+\gamma)/\gamma}$ for some $c_v >0$.
    \item[(ii)] $m^\ast=O(\log(T))$ and $\mathcal{I}_{\min}/(T\delta_T)\rightarrow \infty$ as $T\rightarrow \infty$.
    \item[(iii)] $\eta_{\max}=O(1)$, $p\sqrt{\log(pT)/T\delta_T} \rightarrow 0$ and $(\sqrt{T\delta_T}\eta_{\min})^{-1}p\, s^\ast_{\max}\sqrt{\log(pT)}\rightarrow 0$.
    \item[(iv)] $\lambda_2/(\eta_{\min}\delta_T) \rightarrow 0$ and $\lambda_1Tp/\eta_{\min} \rightarrow 0$ as $T \rightarrow \infty$.
\end{itemize}
\end{assumption}

Assumption \ref{assumption_dgp}-(i) relates to the \rev{strong mixing property} of $(X_t)$. 
Assumption \ref{assumption_dgp}-(ii) is a tail condition and will allow us to apply exponential inequalities for dependent processes. Assumption \ref{assumption_regularity} ensures the identification of the model: it is similar to Assumption A.1 of \cite{Kolar2012} or Assumption A.2 of \cite{Qian2016}.
Assumption \ref{assumption_rates} provides conditions on $\delta_T$, $m^\ast$, $\mathcal{I}_{\min}$, $\eta_{\min}$ and the tuning parameters $\lambda_1,\lambda_2$. Condition (i) concerns the convergence rate of $\delta_T$ to $0$. In condition (ii), the sample size in each regime may diverge with rate $T\delta_T$, but at a slower rate than $T$, and the number of \rev{change points} $m^\ast$ may diverge slowly: this is similar to Assumption A.3-(i) of \cite{Qian2016} or Assumption H3 of \cite{Chan2014}. It also sets the slowest rate at which $\delta_T$ may shrink to zero: $\delta_T=o(\mathcal{I}_{\min}/T)$. Conditions (iii) and (iv) specify the fastest rate at which $\delta_T$ may shrink to zero, which is $\delta_T \gg \max(\lambda_2/\eta_{\min},p^2(s^\ast_{\max})^2\log(pT)/(T\eta_{\min}^2))$. It is worth emphasizing that conditions (ii)-(iii) imply $p^2(s^\ast_{\max})^2\log(pT) =o(\mathcal{I}_{\min}\eta^2_{\min})$ and conditions (ii) and (iv) imply that $\lambda_2T = o(\eta_{\min}\mathcal{I}_{\min})$. Finally, the effect of the LASSO shrinkage through $\lambda_1$ does not relate to $\delta_T$: this is because the Group Fused LASSO penalty only allows to detect change points.
The consistency of $\widehat{T}_j, \widehat{\Omega}_j$, given $\widehat{m}=m^\ast$, is provided in the next Theorem.

\begin{theorem}\label{consistency}
Suppose Assumptions \ref{assumption_dgp}-\ref{assumption_rates} are satisfied. Under $\widehat{m}=m^\ast$, then:
\begin{itemize}
    \item[(i)] $\Pb\Big(\underset{1 \leq j \leq m^\ast}{\max}|\widehat{T}_j-T^\ast_j|\leq T \delta_T\Big) \rightarrow 1$ as $T \rightarrow \infty$.
    \item[(ii)] $\|\widehat{\Omega}_j-\Omega^\ast_j\|_F = O_p(\frac{\lambda_2T}{\mathcal{I}^\ast_j}+\lambda_1Tp(1+\frac{T\delta_T}{\mathcal{I}^\ast_j})+\frac{T\delta_T}{\mathcal{I}^\ast_j} + s^\ast_{\max}\,p\sqrt{\frac{\log(pT)}{\mathcal{I}^\ast_j}})$, for $j = 1,\ldots,m^\ast+1$.
\end{itemize}
\end{theorem}
\begin{remark}
  Result (i) implies $\max_{1 \leq j \leq m^\ast}T^{-1}|\widehat{T}_j-T^\ast_j|=O_p(\delta_T)$. Since $\delta_T=o(1)$, this means $T^{-1}|\widehat{T}_j-T^\ast_j|=o_p(1)$. Here, $\delta_T$ is a key quantity to control for the rate at which $\widehat{T}_j/T$ converges to $T^\ast_j/T$. Note that $\delta_T \gg \max(\frac{\lambda_2}{\eta_{\min}},p^2(s^\ast_{\max})^2\log(pT)/(T\eta_{\min}^2))$, implies that the fastest convergence rate for the \rev{change point} ratio estimator depends on the regularization parameter $\lambda_2$ and $p^2(s^\ast_{\max})^2\log(pT)/(T\eta_{\min}^2)$. Result (ii) relates to the consistency of the precision matrix in each regime.
\end{remark}

The true number of \rev{change points} $m^\ast$ is unknown \rev{in practice}. Following the common practice in the change point literature, we assume that $m^\ast$ is bounded by a known conservative upper bound $m_{\max}$ \rev{with $m_{\max} \leq C\log(T)$, $C>0$ large enough}. Define $h(A,B)\coloneqq \sup_{b \in B}\inf_{a\in A}|a-b|$ for any two sets $A, B$. The next result shows that all true \rev{change points} in $\Tc^\ast_{m^\ast}$ can be consistently estimated by some points in $\widehat{\Tc}_{\widehat{m}}\coloneqq \{\widehat{T}_1<\widehat{T}_2<\ldots<\widehat{T}_{\widehat{m}}\}$.
\begin{theorem}\label{date_recov}
Suppose Assumptions \ref{assumption_dgp}-\ref{assumption_rates} are satisfied. If $m^\ast \leq \widehat{m} \leq m_{\max}$, then
$\Pb(h(\widehat{\Tc}_{\widehat{m}},\Tc^\ast_{m^\ast})\leq T\delta_T) \rightarrow 1 \; \text{as} \; T \rightarrow \infty$.
\end{theorem}
The proof of Theorem \ref{date_recov} is done by contradiction and follows similar arguments as in the proof of Theorem \ref{consistency}. It relies on the optimality conditions from Lemma~\ref{optimality_cond}. Theorem \ref{date_recov} ensures that even if the number of blocks is overestimated, there will be an estimated change point close to each unknown true change point.

\section{Optimization}\label{sec:opt}
\rev{
  The results of Section \ref{asymptotic_properties} establish the consistency of the GFDtL estimator defined by criterion \eqref{stat_crit} under suitable assumptions.
  In this section, we focus on the computation of this estimator.
  As we will show in Section \ref{subsec:dual-existence}, the problem associated with~\eqref{stat_crit} (see \eqref{eq:ori-D-tr-opt-prob} below) may be unbounded from below for general tuning parameters, posing numerical challenges for estimation.
  To overcome this, we introduce in Section \ref{subsec:revised-prob} a revised problem with a modified regularizer.
  This revised problem guarantees solution existence for all tuning parameters and, via Proposition \ref{prop:rel-ori-new-D-tr}, serves as a bridge: its optimal solutions can either recover the optimal solutions to the original problem~\eqref{stat_crit} (thereby inheriting the theoretical guarantees of Section \ref{asymptotic_properties}) or signal the need to adjust tuning parameters by checking condition~\eqref{eq:prop-nosol-ori-new}.
  We then adapt the ADMM to solve this revised problem in Section \ref{subsec:ADMM}.
  In practice, once the precision matrices $\{\widehat{\Theta}_t\}_{t=1}^T$ are estimated, the \rev{change points} are determined numerically by thresholding: a \rev{change point} is detected at time $t$ if $\|\widehat{\Theta}_t-\widehat{\Theta}_{t-1}\|_F\geq \epsilon_{\text{tol}}$ for some threshold $\epsilon_{\text{tol}}>0$: the implementation details and a sensitivity analysis with respect to $\epsilon_{\text{tol}}$ are provided in Sections \ref{sec:implementation} and \ref{sec:sen-ana}.
}

Specifically, given $ \mathcal{X}_T $, $ \epsilon > 0 $, $ \lambda_1 > 0 $ and $ \lambda_2 > 0 $, we consider the following problem
\begin{equation}
  \label{eq:ori-D-tr-opt-prob}
  \min_{\substack{\Theta_t \succeq \epsilon I_p, \\[0.1cm] 1 \leq t \leq T}} \left\{ \sum_{t=1}^T\! \left[ \text{tr}(\frac{1}{2}\Theta_t^2X_tX_t^{\top}) \!-\! \text{tr}(\Theta_t) \right] \!+\! \lambda_1T\sum_{t=1}^T\! \| \Theta_t \|_{1, \text{off}} \!+\! \lambda_2T\sum_{t=1}^{T-1}\!\| \Theta_{t+1} - \Theta_t \|_F \right\},
\end{equation}
where we scale Problem \eqref{stat_crit} by a factor of $T$ for numerical stability.
One can also notice that in \eqref{eq:ori-D-tr-opt-prob} we use $ \Theta_t \succeq \epsilon I_p $ rather than $ \Theta \succ 0 $ as in \eqref{stat_crit}.
This choice is made for practical reasons, as setting $ \epsilon > 0 $ ensures non-singular solutions, and the set $\{S \, :\, S \succeq \epsilon I_p\} $ is \textit{closed} and convex and hence the projection onto it is well defined.

\subsection{Dual Problem and Existence of \rev{Optimal} Solutions} \label{subsec:dual-existence}
We first deduce the dual problem of \eqref{eq:ori-D-tr-opt-prob}, and show that \eqref{eq:ori-D-tr-opt-prob} may be unsolvable.

\begin{proposition}\label{prop:dual-prob-ori-D-tr}
	\begin{enumerate}[(i)]
      \item\label{item:dual-prob-ori-D-tr-prob}
		The dual problem of \eqref{eq:ori-D-tr-opt-prob} is
		\begin{equation}\displayindent0pt\displaywidth\textwidth
			\everymath{\displaystyle}
			\begin{split}
				\max_{\mathbf{Y}} \quad & \left\{ \sum_{t=1}^T-\frac{1}{2} \tr(W_t^{\top}W_t) + \epsilon \sum_{t=1}^T \tr \left( Z_t - Z_{t - 1} - I_p + (X_tX_t^{\top})^{\frac{1}{2}}W_t - Y_{t, \mathrm{off}} \right) \right\} \\
				\text{s.t.} \quad       & Z_0 \!=\! Z_T \!=\! \mathbf{0}_{p \times p}; Z_t \!-\! Z_{t \!-\! 1} \!-\! I_p \!+\! \Sym\!\left((X_tX_t^{\top})^{\frac{1}{2}}W_t\right) \!-\! Y_{t, \mathrm{off}} \!\succeq\! 0 \,\,\,\, \forall t \!=\! 1, \!\dots\! , T;                                                          \\
				                        & \| Z_t \|_F \leq \lambda_2T, \quad | Y_{uv, t} | \leq \lambda_1T \,\,\,\, \forall t = 1, \dots , T, u, v = 1, \dots , p \text{ with } u \neq v,
			\end{split}
			\raisetag{30pt}\label{eq:dual-prob-ori-D-tr}
		\end{equation}
            where \( \mathbf{Y} = \left\{ \{W_t\}_{t=1}^T, \{Y_{t, \mathrm{off}}\}_{t=1}^T, \{Z_t\}_{t=1}^{T-1} \right\} \) is the dual variable with $ W_t \in \Rb^{p\times p} $, $ Y_{t, \mathrm{off}} \in \mathcal{S}_{\mathrm{off}}^p $, $ Z_t \in \mathcal{S}^p $ for all $ t $; $ \Sym $ is the symmetrization operator. Moreover, the optimal values of \eqref{eq:ori-D-tr-opt-prob} and \eqref{eq:dual-prob-ori-D-tr} are the same.

      \item\label{item:strict-fea-ori-dual} If \( \sum_{t=1}^T\! X_tX_t ^{\top}\! \succ\! 0 \), then there exists $ \overline{\lambda}_2 \!>\! 0 $ such that for any $ \lambda_1 \!>\! 0 $ and any $ \lambda_2 \!\geq\! \overline{\lambda}_2 $, the dual problem \eqref{eq:dual-prob-ori-D-tr} has a Slater point\footnote{A Slater point of \eqref{eq:dual-prob-ori-D-tr} is a feasible point that satisfies all the inequality and positive semi-definite constraints strictly, i.e., satisfies all the ``$\leq$" and ``$\succeq$" as ``$<$" and ``$\succ$", respectively.} and the primal problem \eqref{eq:ori-D-tr-opt-prob} has \rev{optimal} solutions.
	\end{enumerate}
\end{proposition}

A simple example to illustrate the possible unsolvability of \eqref{eq:ori-D-tr-opt-prob} is included in Appendix~\ref{sec:opt-examples}.
We also note that the assumption $ \sum_{t=1}^T X_t X_t^{\top} \succ 0 $ in Proposition \ref{prop:dual-prob-ori-D-tr}-(\ref{item:strict-fea-ori-dual}) is reasonable because it can be viewed as a sample-based version of Assumption \ref{assumption_regularity}.

\begin{remark}
    It is worth pointing out that the nonexistence of \rev{optimal} solution does not contradict our findings in Section~\ref{asymptotic_properties} because those results only indicate that when there is a ground truth, under suitable assumptions, the ground truth can be (approximately) recovered from \rev{an optimal} solution of problem~\eqref{stat_crit} with suitably chosen $T$, $\lambda_1$ and $\lambda_2$; in particular, it did not imply solution existence of problem~\eqref{stat_crit} for general $T$, $\lambda_1$ and $\lambda_2$.
\end{remark}

To ensure the existence of \rev{an optimal} solution, although we can obtain some lower bound $ \overline{\lambda}_2 $ of $ \lambda_2 $ (as detailed in the proof of Proposition \ref{prop:dual-prob-ori-D-tr}-(\ref{item:strict-fea-ori-dual})) to ensure the solution existence for \eqref{eq:ori-D-tr-opt-prob}, this $ \overline{\lambda}_2 $ may not be tight.
This can imply practical issues because the optimal $ \lambda^*_2 $ can be strictly smaller than $ \overline{\lambda}_2 $. Then we will need to work with problem~\eqref{eq:ori-D-tr-opt-prob} with some $\lambda_2 < \overline{\lambda}_2$ to locate such $\lambda^*_2$. However, since $\lambda_2 < \overline{\lambda}_2$, \eqref{eq:ori-D-tr-opt-prob} may be unbounded from below and certifying such a scenario is a challenging problem. This motivates us to modify problem~\eqref{eq:ori-D-tr-opt-prob} to obtain a new model which has \rev{optimal} solutions for {\em all choices} of $\lambda_1$ and $\lambda_2$, and (under some mild condition) returns \rev{an optimal} solution of problem~\eqref{eq:ori-D-tr-opt-prob} when the latter problem is solvable.

To this end, we notice that the unsolvability of problem~\eqref{eq:ori-D-tr-opt-prob} when $\lambda_2$ is small may be related to the fact that {\em the relations between different groups} induced by the Group Fused LASSO regularizer is not strong enough to leverage the condition \( \sum_{t=1}^T X_tX_t ^{\top} \succ 0 \) to ensure dual strict feasibility (and hence the solvability of \eqref{eq:ori-D-tr-opt-prob}).
Hence, this motivates the introduction of a (modified) new regularizer to replace the Group Fused LASSO regularizer in problem~\eqref{eq:ori-D-tr-opt-prob}: intuitively, this regularizer should be similar to Group Fused LASSO regularizer when $\|\Theta_{t+1} - \Theta_t\|_F$ is small for inducing the (same) desired \rev{change points}, but penalize more on large $\|\Theta_{t+1} - \Theta_t\|_F$ to ensure the solution existence of the new model.

\subsection{A Revised Problem with a Modified Regularizer}\label{subsec:revised-prob}
Let $ \lambda_3 \geq 0.5 $ and let
\begin{equation}
	\label{eq:def-R}
	\mathcal{R}(x; \lambda_3) \coloneqq
	\begin{cases}
		| x |                         & \text{if } | x | \leq \lambda_3, \\[-.3cm]
		x^2 - \lambda_3^2 + \lambda_3 & \text{otherwise}.
	\end{cases}
\end{equation}
Here, $ \lambda_3 \geq 0.5 $ is necessary and sufficient to ensure the convexity of $ \mathcal{R} $.
The function $ \mathcal{R} $ employs the absolute value in a small region near 0 (determined by $ \lambda_3 $) and switches to a quadratic function outside this region.
In this way, it reduces to the classical $ \ell_1 $ penalty when $ x $ is near 0, while imposing a more substantial penalty as $ x $ goes away from 0.

Replacing $\|\Theta_{t+1} - \Theta_t\|_F$ by $ \mathcal{R}(\|\Theta_{t+1} - \Theta_t\|_F;\lambda_3) $ in problem \eqref{eq:ori-D-tr-opt-prob}, we obtain the following revised optimization problem:
\begin{equation}
	\label{eq:D-tr-opt-prob}
  \hspace{-0.2cm}\min_{\substack{\Theta_t \succeq \epsilon I_p, \\[0.1cm] 1 \leq t \leq T}}\! \left\{ \sum_{t=1}^T\! \left[ \text{tr}(\frac{1}{2}\Theta_t^2X_tX_t^{\top}) \!-\! \text{tr}(\Theta_t) \right] \!\!+\!\! \lambda_1T\!\sum_{t=1}^T\! \| \Theta_t \|_{1, \text{off}} \!+\! \lambda_2T\!\sum_{t=1}^{T-1}\!\mathcal{R}(\| \Theta_{t+1} \!-\! \Theta_t \|_F;\! \lambda_3) \right\}.
\end{equation}

The next proposition shows the dual problem and the existence of \rev{optimal} solutions to \eqref{eq:D-tr-opt-prob}.
\begin{proposition}
  \label{prop:dual-prob-D-tr}
    \begin{enumerate}[(i)]
        \item\label{item:dual-prob-D-tr-prob}
        Let
        $$
            \mathcal{G}(x; \lambda_3) = \min \left\{ -\left( x - \lambda_2T \right)_+ \lambda_3, \lambda_2T \left( \left( \lambda_3 - \frac{x}{2\lambda_2T} \right)_+ ^2 - \frac{x^2}{4\lambda_2^2T^2} - \lambda_3^2 + \lambda_3 \right) \right\},
        $$
        where $ (\cdot)_+ = \max \{\cdot , 0\} $.
        Then the dual problem of \eqref{eq:D-tr-opt-prob} is
        \begin{equation}\displayindent0pt\displaywidth\textwidth
            \everymath{\displaystyle}
            \raisetag{30pt}\label{eq:dual-prob-D-tr}
            \begin{split}
              \max_{\mathbf{Y}} \;\;& \Bigg\{\! \sum_{t=1}^T\!-\frac{1}{2} \!\tr(W_t^{\top}W_t) \!+\! \epsilon\! \sum_{t=1}^T \!\tr\!\left( Z_t \!-\! Z_{t \!-\! 1} \!-\! I_p \!+\! (X_tX_t^{\top})^{\frac{1}{2}}W_t \!-\! Y_{t, \mathrm{off}} \!\right) \!+\! \sum_{t=1}^{T-1}\! \mathcal{G}(\| Z_t \|_F;\! \lambda_3) \!\Bigg\},                                                                                        \\
                \text{s.t.} \;\;       & Z_0 \!=\! Z_T \!=\! \mathbf{0}_{p \times p}; Z_t \!-\! Z_{t \!-\! 1} \!-\! I_p \!+\! \Sym\!\left((X_tX_t^{\top})^{\frac{1}{2}}W_t\right) \!-\! Y_{t, \mathrm{off}} \!\succeq\! 0 \,\, \forall t \!=\! 1, \!\dots\! , T;                                                                 \\
                                        & | Y_{uv, t} | \leq \lambda_1T \,\,\,\, \forall t = 1, \dots, T, u, v = 1, \dots , p \text{ with } u \neq v,
            \end{split}
        \end{equation}
        where \( \mathbf{Y} = \left\{ \{W_t\}_{t=1}^T, \{Y_{t, \mathrm{off}}\}_{t=1}^T, \{Z_t\}_{t=1}^{T-1} \right\} \) is the dual variable with $ W_t \in \Rb^{p\times p} $, $ Y_{t, \mathrm{off}} \in \mathcal{S}_{\mathrm{off}}^p $, $ Z_t \in \mathcal{S}^p $ for all $ t $; $ \Sym $ is the symmetrization operator. Moreover, \eqref{eq:D-tr-opt-prob} and \eqref{eq:dual-prob-D-tr} have the same optimal values.

        \item\label{item:strict-fea-dual} If \( \sum_{t=1}^T X_tX_t ^{\top} \succ 0 \), then the dual problem \eqref{eq:dual-prob-D-tr} has a Slater point and the primal problem \eqref{eq:D-tr-opt-prob} has \rev{optimal} solutions.
    \end{enumerate}
\end{proposition}

The relationship between \eqref{eq:ori-D-tr-opt-prob} and \eqref{eq:D-tr-opt-prob} is summarized as follows.
\begin{proposition}
  \label{prop:rel-ori-new-D-tr}
  Given $ \mathcal{X}_T $ and $ \epsilon > 0 $, the following statements hold:
  \begin{enumerate}[(i)]
    \item\label{item:equivalence-opt-probs} For any positive $ \lambda_1$ and $\lambda_2$ such that \eqref{eq:ori-D-tr-opt-prob} has \rev{optimal} solutions, there exists $ \overline{\lambda}_3 \geq 0.5 $ such that for any \rev{$\lambda_3\ge \overline{\lambda}_3$}, any \rev{optimal} solution of \eqref{eq:D-tr-opt-prob} also solves \eqref{eq:ori-D-tr-opt-prob}.
    \item\label{item:properties-solution} Fix any positive $ \lambda_1$ and $\lambda_2 $ such that \eqref{eq:ori-D-tr-opt-prob} does not have \rev{optimal} solutions. Then for any $ \lambda_3 \geq 0.5 $, any \rev{optimal} solution $ \{\Theta_t^{*}\}_{t=1}^T $ to \eqref{eq:D-tr-opt-prob}, satisfies
          \begin{equation}
            \label{eq:prop-nosol-ori-new}
            \max_{t = 1, \dots, T-1} \| \Theta_{t+1}^{*} - \Theta_t^{*} \|_F \ge \lambda_3.
          \end{equation}
  \end{enumerate}
\end{proposition}

\rev{
  \begin{remark}
    Condition \eqref{eq:prop-nosol-ori-new} provides a numerical certificate for the possible unsolvability of the original problem \eqref{eq:ori-D-tr-opt-prob}.
    Specifically, large jumps (i.e., $\|\Theta_{t+1} - \Theta_t\|_F \geq \lambda_3$) indicate that the original objective can be further decreased by allowing the precision matrices to diverge across time slots because the modified regularizer $\mathcal{R}(\cdot; \lambda_3)$ imposes a stronger quadratic penalty for jumps exceeding $\lambda_3$ compared to the linear penalty in the Group Fused LASSO.
    Therefore, if an optimal solution to the revised problem still exhibits such large jumps, the original linear Group Fused penalty may not be strong enough to keep the objective bounded.
    This behavior is consistent with the discussion at the end of Section \ref{subsec:dual-existence}.
  \end{remark}
}

Equipped with Proposition \ref{prop:rel-ori-new-D-tr}, we can derive the following practical way to search for a suitable $\lambda_2$ such that \eqref{eq:ori-D-tr-opt-prob} is solvable by solving (a sequence of) \eqref{eq:D-tr-opt-prob}.
Specifically, for any positive $ \lambda_1$ and $\lambda_2 $, we solve \eqref{eq:D-tr-opt-prob} with an appropriately large $ \lambda_3 $, and then check if the \rev{optimal} solution satisfies \eqref{eq:prop-nosol-ori-new}: if it does, we increase $ \lambda_2 $ further to pursue a reliable estimator; otherwise we obtain \rev{an optimal} solution to \eqref{eq:ori-D-tr-opt-prob}, and \rev{all results described in Section \ref{asymptotic_properties} hold}.

\subsection{An Alternating Direction Method of Multipliers}\label{subsec:ADMM}
In this subsection, we discuss how to adapt the alternating direction method of multipliers (ADMM) to solve \eqref{eq:D-tr-opt-prob}.
We are particularly solving the following equivalent problem:
\begin{equation*}
    \vspace{-.3cm}
	\begin{split}
		\min_{\mathbf{X}}\! & \left\{ \sum_{t=1}^T\! \left[ \text{tr}(\frac{1}{2}\Theta_t^2X_tX_t^{\top}) \!-\! \text{tr}(\Theta_t) + \delta_{\cdot \succeq \epsilon I_p}(V_t) \right] \!+\! \lambda_1T\!\sum_{t=1}^T\! \| \Upsilon_{t, \mathrm{off}} \|_{1, \text{off}} \!+\! \lambda_2T\!\sum_{t=1}^{T-1}\!\mathcal{R}(\| D_t \|_F;\! \lambda_3) \right\}, \\[0.1cm]
		\text{s.t.}         & \quad  V_t = \Theta_t, \Upsilon_{t, \mathrm{off}} = \Theta_{t, \mathrm{off}} \,\, \forall t = 1, \dots , T; \quad D_t = \Theta_{t+1} - \Theta_t \,\, \forall t = 1, \dots ,T-1,
	\end{split}
\vspace{.2cm}
\end{equation*}
where we denote $ \mathbf{X} = \left\{ \{\Theta_t\}_{t=1}^T, \{V_t\}_{t=1}^T, \{\Upsilon_{t, \mathrm{off}}\}_{t=1}^T, \{D_t\}_{t=1}^{T-1} \right\} $ for the sake of notional simplicity; $ \Theta_{t, \mathrm{off}} $ is the copy of $ \Theta_t $ with the diagonal elements set to 0, and $ \Upsilon_{t, \mathrm{off}} \in \mathcal{S}_{\mathrm{off}}^p $.

Given $ \beta > 0 $, the augmented Lagrangian function is
\begin{align*}
       & L(\{\Theta_t\}_{t=1}^T, \{V_t\}_{t=1}^T, \{\Upsilon_{t, \mathrm{off}}\}_{t=1}^T, \{D_t\}_{t=1}^{T-1}, \{A_t\}_{t=1}^T, \{Y_{t, \mathrm{off}}\}_{t=1}^T, \{Z_t\}_{t=1}^{T-1})                                                                                    \\
  \coloneqq  & \sum_{t=1}^T \left[ \text{tr}(\frac{1}{2}\Theta_t^2X_tX_t^{\top}) - \text{tr}(\Theta_t) + \delta_{\cdot \succeq \epsilon I_p}(V_t) \right] + \lambda_1T\sum_{t=1}^T \| \Upsilon_{t, \mathrm{off}} \|_{1, \text{off}}           \\
      & + \lambda_2T\sum_{t=1}^{T-1}\mathcal{R}(\| D_t \|_F; \lambda_3) + \sum_{t=1}^T \left[ - \left\langle A_t, \Theta_t - V_t  \right\rangle  + \frac{\beta}{2} \| \Theta_t - V_t \|_F^2 \right]                               \\
      &  + \sum_{t=1}^T \left[ - \left\langle Y_{t, \mathrm{off}}, \Theta_{t, \mathrm{off}} - \Upsilon_{t, \mathrm{off}} \right\rangle + \frac{\beta}{2} \| \Theta_{t, \mathrm{off}} - \Upsilon_{t, \mathrm{off}} \|_F^2 \right] \\
      &  + \sum_{t=1}^{T-1} \left[ - \left\langle Z_t, \Theta_{t+1} - \Theta_t - D_t \right\rangle + \frac{\beta}{2} \| \Theta_{t+1} - \Theta_t - D_t \|_F^2 \right],
\end{align*}
where \( \{A_t\}_{t=1}^T \), \( \{Y_{t, \mathrm{off}}\}_{t=1}^T \), \( \{Z_t\}_{t=1}^{T-1} \) are dual variables.
For simplicity, we let \( Z_0^k = Z_T^k = D_0^k = D_T^k = \mathbf{0}_{p\times p} \) for all iteration \( k \).
In iteration $ k+1 $, our ADMM consists of three update steps:
\begin{enumerate}
	\item $ \{\Theta_t\}_{t=1}^T $ update: it is equivalent to solving the following linear system
	      \begin{equation}
		      \label{eq:linear-sys-Theta}
		      \begin{aligned}
			      \vspace{-0.2cm}\frac{1}{2}(X_1X_1^{\top}\Theta_1 \!+\! \Theta_1X_1X_1^{\top}) \!+\! 2\beta \Theta_1 \!+\! \beta \Theta_{1, \mathrm{off}} \!-\! \beta \Theta_2                            & \!=\! \Psi_1^k, \\
			      \!-\! \beta \Theta_1 \!+\! \frac{1}{2}(X_2X_2^{\top}\Theta_2 \!+\! \Theta_2X_2X_2^{\top}) \!+\! 3\beta \Theta_2 \!+\! \beta \Theta_{2, \mathrm{off}} \!-\! \beta \Theta_3 & \!=\! \Psi_2^k, \\[-0.3cm]
			      \vdots                                                                                                                                                                     &                 \\[-0.3cm]
			      \!-\! \beta \Theta_{T\!-\!1} \!+\! \frac{1}{2}(X_TX_T^{\top}\Theta_T \!+\! \Theta_TX_TX_T^{\top}) \!+\! 2\beta \Theta_T + \beta \Theta_{T, \mathrm{off}}                  & \!=\! \Psi_T^k,
		      \end{aligned}
	      \end{equation}
        where \( \Psi_t^k = I_p + A_t^k + Y_{t, \mathrm{off}}^k - Z_t^k + Z_{t-1}^k + \beta V_t^k + \beta \Upsilon_{t, \mathrm{off}}^k - \beta D_t^k + \beta D_{t-1}^k \) for all \( t \).

	\item $ \{V_t\}_{t=1}^T, \{\Upsilon_{t, \mathrm{off}}\}_{t=1}^T, \{D_t\}_{t=1}^{T-1} $ update: for all \( t \), let \( \Xi_t^k = \Theta_{t+1}^{k+1} - \Theta_t^{k+1} - \frac{Z_t^k}{\beta} \), then
        \begin{align}
           & \hspace{-2cm}\label{eq:V-update} V_t^{k+1} = \mathrm{Proj}_{\cdot \succeq \epsilon I_p}\left( \Theta_t^{k+1} - \frac{A_t^k}{\beta} \right);\quad \Upsilon_{t, \mathrm{off}}^{k+1}  = \left( \prox_{\frac{\lambda_1T}{\beta} | \cdot |} \left( \Theta_{uv, t}^{k+1} - \frac{Y_{uv, t}^k}{\beta} \right) \right)_{\substack{u, v = 1, \dots , p \\[0.1cm] \text{ with } u \neq v}}; \\[-0.5cm]
           & \label{eq:D-update}
          D_t^{k + 1} =
          \begin{cases}
            \underbrace{\min \left\{ \left( \| \Xi_t^k \|_F - \frac{\lambda_2T}{\beta} \right)_+, \lambda_3 \right\}}_{\imath_1} \frac{\Xi_t^k}{\| \Xi_t^k \|_F} & \text{ if } \mathrm{val}^{\diamond} \leq \mathrm{val}^{+}, \\[-0.2cm]
            \underbrace{\max \left\{ \frac{\beta \| \Xi_t^k \|_F}{\beta + 2\lambda_2T}, \lambda_3 \right\}}_{\imath_2} \frac{\Xi_t^k}{\| \Xi_t^k \|_F}        & \text{ if } \mathrm{val}^{\diamond} > \mathrm{val}^{+},    \\[-0.2cm]
          \end{cases}
        \end{align}
        where \( \mathrm{val}^{\diamond} = \frac{1}{2}\left( \imath_1 - \| \Xi_t^k \|_F \right)^2 + \frac{\lambda_2T}{\beta}\imath_1\), \( \mathrm{val}^{+} = \frac{1}{2}\left( \imath_2 - \| \Xi_t^k \|_F \right)^2 + \frac{\lambda_2T}{\beta} \left( \imath_2^2 - \lambda_3^2 + \lambda_3 \right) \).

	\item Dual update: for all $ t $,\footnote{The dual stepsize 1.61 in \eqref{eq:dual-update} can be more generally chosen from the interval $(0,(\sqrt{5}+1)/2)$.}
	      \begin{equation}
		      \label{eq:dual-update}
		      \begin{aligned}
			       & A_t ^{k+1} = A_t ^k - 1.61 \beta (\Theta_t^{k+1} - V_t^{k+1}), \quad Y_{t, \mathrm{off}} ^{k+1} = Y_{t, \mathrm{off}} ^k - 1.61 \beta (\Theta_{t, \mathrm{off}}^{k+1} - \Upsilon_{t, \mathrm{off}}^{k+1}), \\[-0.2cm]
			       & Z_t ^{k+1} = Z_t ^k - 1.61 \beta (\Theta_{t+1}^{k+1} - \Theta_t^{k+1} - D_t^{k+1}).
		      \end{aligned}
	      \end{equation}
\end{enumerate}

We next discuss how these update steps are derived.

\underline{$ \{\Theta_t\}_{t=1}^T $ update:}
The $ \{\Theta_t\}_{t=1}^T $ update minimizes $L$ over $\{\Theta_t\}_{t=1}^T$ with other variables fixed at their iteration $k$ values:
$$\displayindent0pt\displaywidth\textwidth
    \resizebox{\textwidth}{!}{$\displaystyle\{\Theta_t^{k+1}\}_{t=1}^T \!=\! \argmin_{\{\Theta_t\}_{t=1}^T} L\Big( \{\Theta_t\}_{t=1}^T,\!\{V_t^k\}_{t=1}^T,\!\{\Upsilon_{t, \mathrm{off}}^k\}_{t=1}^T,\!\{D_t^k\}_{t=1}^{T-1},\! \{A_t^k\}_{t=1}^T,\!\{Y_{t, \mathrm{off}}^k\}_{t=1}^T,\!\{Z_t^k\}_{t=1}^{T-1} \Big).$}
$$
Note that this update is well defined as $L$ is strongly convex in $\{\Theta_t\}_{t=1}^T$ (with other variables fixed).
Setting the derivative of the objective function with respect to $ \Theta_t $ to 0, we have
$$
	\left\{\begin{aligned}
		\mathbf{0}_{p \times p} = & \frac{1}{2}(X_1 X_1 ^{\top} \Theta_1 + \Theta_1 X_1 X_1 ^{\top}) \!-\! I_p \!-\! A_1^k \!-\! Y_{1, \mathrm{off}}^k \!+\! Z_1^k \!+\! \beta (\Theta_1 \!-\! V_1^k) \!+\! \beta ( \Theta_{1, \mathrm{off}} \!-\! \Upsilon_{1, \mathrm{off}}^k )             \\
		                          & \qquad \!-\! \beta ( \Theta_2 \!-\! \Theta_1 \!-\! D_1^k ),                                                                                                                                                                            \\
		\mathbf{0}_{p \times p} = & \frac{1}{2}(X_t X_t ^{\top} \Theta_t \!+\! \Theta_t X_t X_t ^{\top}) \!-\! I_p \!-\! A_t^k \!-\! Y_{t, \mathrm{off}}^k \!+\! Z_t^k \!-\! Z_{t\!-\!1}^k \!+\! \beta (\Theta_t \!-\! V_t^k) \!+\! \beta ( \Theta_{t, \mathrm{off}} \!-\! \Upsilon_{t, \mathrm{off}}^k ) \\
		                          & \qquad \!-\! \beta ( \Theta_{t \!+\! 1} \!-\! \Theta_t \!-\! D_t^k ) \!+\! \beta ( \Theta_t \!-\! \Theta_{t \!-\! 1} \!-\! D_{t \!-\! 1}^k ) \quad \forall\, t = 2, \dots ,T \!-\! 1,                                                                                    \\
		\mathbf{0}_{p \times p} = & \frac{1}{2}(X_T X_T ^{\top} \Theta_T \!+\! \Theta_T X_T X_T ^{\top}) \!-\! I_p \!-\! A_T^k \!-\! Y_{T, \mathrm{off}}^k \!-\! Z_{T\!-\!1}^k \!+\! \beta (\Theta_T \!-\! V_T^k) \!+\! \beta ( \Theta_{T, \mathrm{off}} \!-\! \Upsilon_{T, \mathrm{off}}^k )         \\
		                          & \qquad \!+\! \beta ( \Theta_T \!-\! \Theta_{T \!-\! 1} \!-\! D_{T \!-\! 1}^k ).
	\end{aligned}\right.
$$
By denoting
$$
	\Psi_t^k = I_p + A_t^k + Y_{t, \mathrm{off}}^k - Z_t^k + Z_{t-1}^k + \beta V_t^k + \beta \Upsilon_{t, \mathrm{off}}^k - \beta D_t^k + \beta D_{t-1}^k
$$
for all $ t $ and $ k $, we obtain the linear system \eqref{eq:linear-sys-Theta}.
This system does not have a closed form solution in general.
Here we use \textsf{pcg} in \textsc{Matlab} to solve it.
Specifically, in each iteration, we use the solution from the previous iteration as the initial point and solve the system only up to some tolerance that decreases with iterations.

\underline{$ \{V_t\}_{t=1}^T, \{\Upsilon_{t, \mathrm{off}}\}_{t=1}^T, \{D_t\}_{t=1}^{T-1} $ update:}
The update for these variables is given by
$$\hspace{-1cm}
	\begin{aligned}
		  & \quad \{V_t^{k+1}\}_{t=1}^T, \{\Upsilon_{t, \mathrm{off}}^{k+1}\}_{t=1}^T, \{D_t^{k+1}\}_{t=1}^{T-1}            \\[0.05cm]
		= & \argmin_{\substack{\{V_t\}_{t=1}^T, \{\Upsilon_{t, \mathrm{off}}\}_{t=1}^T,                                     \\[0.1cm] \{D_t\}_{t=1}^{T-1}}} \, L \Big(\{\Theta_t^{k+1}\}_{t=1}^T, \{V_t\}_{t=1}^T, \{\Upsilon_{t, \mathrm{off}}\}_{t=1}^T, \{D_t\}_{t=1}^{T-1}, \\[-0.8cm]
		  & \qquad \qquad \qquad \qquad  \{A_t^k\}_{t=1}^T, \{Y_{t, \mathrm{off}}^k\}_{t=1}^T, \{Z_t^k\}_{t=1}^{T-1} \Big).
	\end{aligned}
$$
Note that this update is well defined as $L$ is strongly convex in the variables $\{V_t\}_{t=1}^T$, $\{\Upsilon_{t, \mathrm{off}}\}_{t=1}^T$, $\{D_t\}_{t=1}^{T-1}$ (with other variables fixed).
It is notable that this subproblem is block separable, allowing us to solve it by addressing three further subproblems with respect to \( \{V_t\}_{t=1}^T \), \( \{\Upsilon_{t, \mathrm{off}}\}_{t=1}^T \), and \( \{D_t\}_{t=1}^{T-1} \), respectively.

For $ V_t $, one has
\begin{eqnarray*}
	\lefteqn{V_t^{k+1} = \argmin_{V_t \succeq \epsilon I_p}\left\{ \left\langle A_t^k, V_t \right\rangle + \frac{\beta}{2} \| \Theta_t^{k+1} - V_t \|_F^2 \right\}}\\
		          & & \:\:\: = \argmin_{V_t \succeq \epsilon I_p}\left\{ \left\| V_t - \left(\Theta_t^{k+1} - \frac{A_t^k}{\beta}\right) \right\|_F^2 \right\}     = \mathrm{Proj}_{\cdot \succeq \epsilon I_p}\left( \Theta_t^{k+1} - \frac{A_t^k}{\beta} \right),
\end{eqnarray*}
which yields the projection formula in \eqref{eq:V-update}.

For $ \Upsilon_{t, \mathrm{off}} $, it holds that
\begin{equation*}
	\begin{aligned}
		\Upsilon_{t, \mathrm{off}}^{k+1} & = \argmin_{\Upsilon_{t, \mathrm{off}}} \left\{ \lambda_1 T \| \Upsilon_{t, \mathrm{off}} \|_{1, \mathrm{off}} + \left\langle Y_{t, \mathrm{off}}^k, \Upsilon_{t, \mathrm{off}} \right\rangle + \frac{\beta}{2} \| \Theta_{t, \mathrm{off}}^{k+1} - \Upsilon_{t, \mathrm{off}} \|_F^2 \right\} \\
		                                 & = \argmin_{\Upsilon_{t, \mathrm{off}}} \left\{ \frac{\lambda_1 T}{\beta} \| \Upsilon_{t, \mathrm{off}} \|_{1, \mathrm{off}} + \left\| \Upsilon_{t, \mathrm{off}} - \left( \Theta_{t, \mathrm{off}}^{k+1} - \frac{Y_{t, \mathrm{off}}^k}{\beta} \right) \right\|_F^2 \right\}                  \\
                                  & = \prox_{\frac{\lambda_1T}{\beta} \| \cdot \|_{1, \mathrm{off}}} \left( \Theta_{t, \mathrm{off}} ^{k+1} - \frac{Y_{t, \mathrm{off}}^k}{\beta} \right) = \left( \prox_{\frac{\lambda_1T}{\beta} | \cdot |} \left( \Theta_{uv, t}^{k+1} - \frac{Y_{uv, t}^k}{\beta} \right) \right)_{\substack{u, v = 1, \dots , p \\[0.1cm] \text{ with } u \neq v}},
	\end{aligned}
\end{equation*}
which is the proximal operator of the $\ell_1$ norm applied element-wise to off-diagonal entries, as given in \eqref{eq:V-update}.

The update scheme for $ D_t $ is slightly more complicated.
For simplicity, we denote
$$
	\Xi_t^k = \Theta_{t+1}^{k+1} - \Theta_t^{k+1} - \frac{Z_t^k}{\beta} \;\;\;\;\; \forall t = 1, \dots , T-1, k = 1, 2, \dots.
$$
For each $ t $, we need to solve the following optimization problem:
$$
	D_t^{k + 1} = \argmin_{D_t} \frac{1}{2} \| D_t - \Xi_t^k \|_F^2 + \frac{\lambda_2T}{\beta} \mathcal{R}(\| D_t \|_F; \lambda_3).
$$
By the definition of $ \mathcal{R} $ in (4), it is equivalent to solving the following two problems
$$
\begin{aligned}
  D_t^{\diamond} & \coloneqq \argmin_{\| D_t \|_F \leq \lambda_3} \frac{1}{2} \| D_t - \Xi_t^k \|_F^2 + \frac{\lambda_2T}{\beta} \| D_t \|_F,                                         & \circnum{1} \\
  D_t^{+} & \coloneqq \argmin_{\| D_t \|_F \geq \lambda_3} \frac{1}{2} \| D_t - \Xi_t^k \|_F^2 + \frac{\lambda_2T}{\beta} \left( \| D_t \|_F^2 - \lambda_3^2 + \lambda_3 \right), & \circnum{2}
\end{aligned}
$$
and let $ D_t^{k+1} = D_t^{\diamond} $ if its objective value is smaller, otherwise $ D_t^{k+1} = D_t^{+} $.
We first note that if $ \Xi_t^k = \mathbf{0}_{p \times p} $, then the optimal value of \circnum{2} satisfies $\mathrm{val}^{+} \ge 0 = \mathrm{val}^{\diamond}$ (the optimal value of \circnum{1}) and we have $ D^{k+1}_t = D_t^{\diamond} = \mathbf{0}_{p \times p} $.
We next assume that $ \Xi_t^k \neq \mathbf{0}_{p \times  p} $.
For \circnum{1}, from the proximal operator of Frobenius norm, we have
\begin{equation*}
		D_t^{\diamond} = \underbrace{\min \left\{ \left( \| \Xi_t^k \|_F - \frac{\lambda_2T}{\beta} \right)_+, \lambda_3 \right\}}_{\imath_1} \frac{\Xi_t^k}{\| \Xi_t^k \|_F}, \;\; \mathrm{val}^{\diamond} = \frac{1}{2}\left( \imath_1 - \| \Xi_t^k \|_F \right)^2 + \frac{\lambda_2T}{\beta}\imath_1.
\end{equation*}
For \circnum{2}, by considering the derivative of its objective function, we get
\begin{equation*}
		D_t^{+} = \underbrace{\max \left\{ \frac{\beta \| \Xi_t^k \|_F}{\beta + 2\lambda_2T}, \lambda_3 \right\}}_{\imath_2} \frac{\Xi_t^k}{\| \Xi_t^k \|_F}, \;\;
		\mathrm{val}^{+} = \frac{1}{2}\left( \imath_2 - \| \Xi_t^k \|_F \right)^2 + \frac{\lambda_2T}{\beta} \left( \imath_2^2 - \lambda_3^2 + \lambda_3 \right).
\end{equation*}
The update scheme for $ D_t $ in \eqref{eq:D-update} is obtained upon combining these two cases.

Overall, the ADMM is summarized in Algorithm \ref{algo:ADMM}, whose convergence directly follows from \cite[Appendix B]{FPST2013}.

\begin{algorithm}[htbp]
	\caption{An ADMM for solving \eqref{eq:D-tr-opt-prob}}\label{algo:ADMM}
	\begin{algorithmic}[1]
		\Statex \hspace{-0.5cm}\textbf{input:} $ \mathcal{X}_T \!=\! (X_1,\! \!\dots\!,\! X_T) $; $ \lambda_1 \!>\! 0 $, $ \lambda_2 \!>\! 0 $, $ \epsilon \!>\! 0 $, $ \lambda_3 \!\geq\! 0.5 $; $ \beta \!> 0 \!$; $ \epsilon_{\mathsf{pcg}} \!\in\! (0,\! 1) $, $ \tau \!\in\! (0,\! 1) $.
    $ \mathbf{X}^0 = \{ \{\Theta_t^0\}_{t=1}^T, \{V_t^0\}_{t=1}^T, \{\Upsilon_{t, \mathrm{off}}^0\}_{t=1}^T, \{D_t^0\}_{t=1}^{T-1} \}$; $ \mathbf{Y}^0 = \{ \{A_t^0\}_{t=1}^T, \{Y_{t, \mathrm{off}}^0\}_{t=1}^T, \{Z_t^0\}_{t=1}^{T-1} \} $.
		\Statex \hspace{-0.5cm}\textbf{output:} $ \mathbf{\widehat{X}} \!=\! \{\! \{\widehat{\Theta}_t\}_{t=1}^T, \!\{\widehat{V}_t\}_{t=1}^T,\! \{\widehat{\Upsilon}_{t, \mathrm{off}}\}_{t=1}^T,\! \{\widehat{D}_t\}_{t=1}^{T-1} \!\}$; $ \mathbf{\widehat{Y}} \!=\! \{\! \{\widehat{A}_t\}_{t=1}^T,\! \{\widehat{Y}_{t, \mathrm{off}}\}_{t=1}^T,\! \{\widehat{Z}_t\}_{t=1}^{T-1}\! \} $.

		\State Set $ k \gets 0 $.
		\While{the termination criterion is not met}
		\State $ \{\Theta_t\}_{t=1}^T $ update: call \textsf{pcg} in \textsc{Matlab} to solve the linear system \eqref{eq:linear-sys-Theta} using $ \{\Theta_t^k\}_{t=1}^T $ as initial point up to tolerance $ \epsilon_{\mathsf{pcg}} $ to obtain $ \{\Theta_t^{k+1}\}_{t=1}^T $.
    \State If $ {\rm mod}(k + 1, 10) = 0 $, then update $ \epsilon_{\mathsf{pcg}} \gets \max \{ \tau \epsilon_{\mathsf{pcg}}, 10^{-12}\} $.
		\State Update $ \{V_t\}_{t=1}^T $, $ \{\Upsilon_{t, \mathrm{off}}\}_{t=1}^T $ and $ \{D_t\}_{t=1}^{T-1} $ according to \eqref{eq:V-update} and \eqref{eq:D-update}.
		\State Update dual variables according to \eqref{eq:dual-update}.
		\State Set $ k \gets k + 1 $.
		\EndWhile
        \State Return $ \widehat{\Theta}_t = \Theta_t^k $, $ \widehat{V}_t = V_t^k $, $ \widehat{\Upsilon}_{t, \mathrm{off}} = \Upsilon_{t, \mathrm{off}}^k $, $ \widehat{D}_t = D_t^k$, $\widehat{A}_t = A_t^k$, $\widehat{Y}_{t, \mathrm{off}} = Y_{t, \mathrm{off}}^k$, $\widehat{Z}_t = Z_t^k$ for all $ t $.
	\end{algorithmic}
\end{algorithm}

\section{Implementation Details}\label{sec:implementation}
We implement Algorithm~1 in \textsc{Matlab} R2023a.
The \textsc{Matlab} codes for the implementation of Algorithm~1 and the experiments are available at \url{https://github.com/linyopt/GFDtL}.
In this section, we provide some implementation details about the algorithm and the numerical experiments; interested readers can check the codes for more technical details that are not covered here.

First, \rev{we briefly detail the procedure} to obtain an estimator and to identify the corresponding \rev{change points} based on a given $ \epsilon > 0 $ and a sample $ \mathcal{X}_T $ with $ \sum_{t=1}^T X_t X_t^{\top} \succ 0 $.
Specifically, for any pair of tuning parameters $ (\lambda_1, \lambda_2) $, Proposition~7 suggests that we can apply Algorithm~1 to solve (5) with a sufficiently large $ \lambda_3 $ to either assert that (2) may not have solutions or obtain a solution to (2).
\rev{GFDtL criterion yields the estimated block partition $\{\widehat{\mathcal{B}}_1,\ldots,\widehat{\mathcal{B}}_{\widehat{m}+1}\}$ of $\{1,\ldots,T\}$, with $\widehat{m}$ the estimated number of change points, and where $\widehat{\Theta}_t$ is piece-wise constant. That is, if $t$ is a change point, $\|\widehat{\Theta}_t-\widehat{\Theta}_{t-1}\|_F>0$. In our implementation procedure, } with $ \{\widehat{\Theta}_t\}_{t=1}^T $ in hand, we identify the \rev{change points} by selecting the $t$'s with $ \| \widehat{\Theta}_{t+1} - \widehat{\Theta}_t \|_F \geq 10^{-6} $.
Since $(\lambda_1,\lambda_2)$ controls the model complexity and smoothing, they must be calibrated accordingly.
We search for the optimal tuning parameters over a user-specified grid based on some criteria, which will be discussed in Section~\ref{sec:tuning-para}.

\rev{In contrast, the Gaussian loss-based GFGL approach of \cite{Gibberd2017} applies the group-fused penalty only to off-diagonal entries of the precision matrix, thereby imposing piece-wise constancy on the off-diagonal structure rather than the full precision matrix.
Therefore, the GFGL estimator does not possess the piece-wise constancy property.
Their algorithm leverages an inherent group lasso subproblem for the off-diagonal jump parameters, with \rev{change points} identified through the grouping structure of nonzero jumps.}
By comparison, we identify \rev{change points} by thresholding successive differences of the full $\widehat{\Theta}_t$.

\subsection{Initialization and Termination Criterion}
Throughout this paper, we initialize the algorithm as follows: for all $ t $,
\[
  \begin{array}{lll}
    \Theta_t^0 = \left( \sum_{t=1}^T X_tX_t^{\top} \right)^{-1}, \:\: & V_t^0 = \Theta_t^0, \:\: & \Upsilon_{t, \mathrm{off}}^0 = \Theta_{t, \mathrm{off}}^0, \\
    D_t^0 = \Theta_{t+1}^0 - \Theta_t^0 = \mathbf{0}_{p \times p}, \:\: & A_t^0 = \mathbf{1}_{p \times p}, \:\: & Y_{t, \mathrm{off}}^0 = Z_t^0 = \mathbf{0}_{p \times p};
  \end{array}
\]
here the inverse is well-defined since we have $ \sum_{t=1}^T X_tX_t^{\top} \succ 0 $.

We next describe the termination criterion, which essentially consists of checking the constraint violations for the dual problem (6), and also the gap between the primal and dual objective values.
Recall that $ \mathbf{X} \!=\! \left\{\!\{\Theta_t\}_{t=1}^T,\!\{V_t\}_{t=1}^T,\!\{\Upsilon_{t, \mathrm{off}}\}_{t=1}^T,\!\{D_t\}_{t=1}^{T-1}\!\right\} $ is the primal variable with $ \Theta_t \!\in\! \mathcal{S}^p $, $ V_t \!\in\! \mathcal{S}^p $, $ \Upsilon_{t, \mathrm{off}} \!\in\! \mathcal{S}_{\mathrm{off}}^p $ and $ D_t \!\in\! \mathcal{S}^p $ for all $ t $; \( \mathbf{Y} \!=\! \left\{\! \{W_t\}_{t=1}^T,\!\{Y_{t, \mathrm{off}}\}_{t=1}^T,\!\{Z_t\}_{t=1}^{T-1}\!\right\} \) is the dual variable with $ W_t \!\in\! \Rb^{p\times p} $, $ Y_{t, \mathrm{off}} \!\in\! \mathcal{S}_{\mathrm{off}}^p $, $ Z_t \!\in\! \mathcal{S}^p $ for all $ t $; and let \( \zeta_t \!=\! Z_t \!-\! Z_{t-1} \!-\!I_p \!+\! \Sym\left((X_t X_t ^{\top})^{\frac{1}{2}}W_t\right) \!-\! Y_{t, \mathrm{off}} \) for all \( t \), where $ Z_0 = Z_T = \mathbf{0}_{p\times p} $.
We define the relative infeasibility for the positive semi-definite constraint as
\[
  \mathsf{dfeas}_1(\mathbf{Y}) \coloneqq \max_{t = 1, \dots , T} \left\{ \frac{| \min \{\lambda_{\min}(\zeta_t), 0\} |}{\| \zeta_t \|_F + 1} \right\}.
\]
For the bound constraint of $ \{Y_{t, \mathrm{off}}\}_{t=1}^T $, we similarly define the relative infeasibility as
\[
  \mathsf{dfeas}_2(\mathbf{Y}) \coloneqq \frac{\left(\max_{t, u \neq v}\{ | Y_{uv, t, \mathrm{off}} | \} - \lambda_1 \right)_+}{1 + \max_{t, u \neq v}\{ | Y_{uv, t, \mathrm{off}} | \}},
\]
where $ Y_{uv, t, \mathrm{off}} $ is the \( (u, v) \)-th element of \( Y_{t, \mathrm{off}} \).
The relative dual infeasibility is defined as
\[
 \mathsf{dfeas}(\mathbf{Y}) \coloneqq \max \{\mathsf{dfeas}_1(\mathbf{Y}), \mathsf{dfeas}_2(\mathbf{Y})\}.
\]
The relative duality gap is defined by
\[
 \mathsf{gap}(\mathbf{X}, \mathbf{Y}) \coloneqq \frac{| v_p(\mathbf{X}) - v_d(\mathbf{Y}) |}{1 + | v_p(\mathbf{X}) | + | v_d(\mathbf{Y}) |},
\]
where \( v_p(\mathbf{X}) \) and \( v_d(\mathbf{Y}) \) are the objective values of primal problem (cf. (5)) at $ \mathbf{X} $ and dual problem (cf. (6)) at $ \mathbf{Y} $, respectively.

Overall, throughout this paper, we terminate Algorithm~1\footnote{We note that Algorithm~1 does not involve $ \{W_t\}_{t=1}^T $ while the dual problem (6) does. Here, we set $ W^k_t = (X_t X_t^{\top})^{\frac{1}{2}} \Theta^k_t $ for all $ t $ and $ k $, which comes from the derivation of the dual problem (6); see \eqref{eq:Lagrangian-min-modified-I}.} when both the relative primal-dual gap and the relative dual infeasibility are sufficiently small:
\[
  \max \left\{ \mathsf{gap}(\mathbf{X}^k, \mathbf{Y}^k), \mathsf{dfeas}(\mathbf{Y}^k) \right\} \leq \epsilon_{\mathrm{tol}},
\]
or, it is detected that Problem (2) may not have solutions:
\begin{equation}
  \label{eq:stop-cri-2}
  \max_{t=1, \dots , T-1} \| \Theta_{t+1}^k - \Theta_t^k \|_F \geq \lambda_3,
\end{equation}
or, the relative successive changes of both primal and dual variables are sufficiently small:
\[
  \begin{aligned}
    \max \Big\{ & \frac{\| \{\Theta_t^{k+1} - \Theta_t^k\}_{t=1}^T \| }{1 + \| \{\Theta_t^{k+1}\}_{t=1}^T \| + \| \{\Theta_t^k\}_{t=1}^T \|}, \frac{\| \{A_t^{k+1} - A_t^k\}_{t=1}^T \| }{1 + \| \{A_t^{k+1}\}_{t=1}^T \| + \| \{A_t^k\}_{t=1}^T \|}, \\
                & \frac{\| \{Y_{t, \mathrm{off}}^{k+1} - Y_{t, \mathrm{off}}^k\}_{t=1}^T \| }{1 + \| \{Y_{t, \mathrm{off}}^{k+1}\}_{t=1}^T \| + \| \{Y_{t, \mathrm{off}}^k\}_{t=1}^T \|}, \frac{\| \{Z_t^{k+1} - Z_t^k\}_{t=1}^{T-1} \| }{1 + \| \{Z_t^{k+1}\}_{t=1}^{T-1} \| + \| \{Z_t^k\}_{t=1}^{T-1} \|} \Big\} \leq \frac{\epsilon_{\mathrm{tol}}}{10^3},
  \end{aligned}
\]
where \( \| \{\Theta_t\}_{t=1}^T \| \coloneqq ( \sum_{t=1}^T \sum_{u=1}^p \sum_{v=1}^p (\Theta_{uv, t})^2 )^{\frac{1}{2}} \).
We set $\epsilon_{\mathrm{tol}} = 10^{-3}$ for experiments on both synthetic and real datasets.

\subsection{Choices of Algorithm Parameters}
The parameter $ \epsilon $ specifies a lower bound of the eigenvalues of matrices in the resulting solution, which ensures the non-singularity of each matrix in the obtained solution if $\epsilon > 0$.
The choice of $ \epsilon $ depends on how ``non-singular'' the solution is expected to be.
In our experiments, we set $ \epsilon = 0.01 $.
In view of Proposition~7 and its proof, $ \lambda_3 $ should be large enough, with its lower bound related to $ (\lambda_1, \lambda_2) $ and the possible solution to the corresponding (2).
After some trial experiments, we set $ \lambda_3 = 10 $ for experiments on synthetic datasets and $ \lambda_3 = 50 $ for experiments on real datasets.
For the tolerance of \textsf{pcg}, we set $\epsilon_{\mathsf{pcg}} = 10^{-2}$ and $\tau = 0.9$.
The parameter $ \beta $ in ADMM has no effect on the results but only affects the convergence speed of the algorithm, so we did not fine-tune it; interested readers can refer to the codes for further details regarding these settings.

\section{Optimal Selection of the Tuning Parameters}\label{sec:tuning-para}

The Bayesian information criterion (BIC) is a common criterion to choose the optimal tuning parameters: see, e.g., \cite{Monti2014}.
However, it is known that there are some scenarios where BIC may fail to select the optimal tuning parameters: see, e.g., Section 4.2 and the Appendix in \cite{Gibberd2017}.
This motivates us to propose the following three methods for the selection of the optimal pair of tuning parameters:
\begin{itemize}
  \item[(a)] Method a: When the true underlying data generating process is known,  the optimal pair can be selected as the pair that minimizes or maximizes suitable performance measures, such as the Hausdorff distance, model recovery, or estimation error. Although this strategy can be employed when the true underlying structure is known only, it is informative about the relative performances of different competing procedures when in a position to minimize/maximize the performance metrics.
  \item[(b)] Method b: In the case of simulated data, we divide the simulated samples of observations into a training set and $B$ test sets, all of them sampled from the same data generating process: the training set is denoted by $\mathcal{X}^{\text{train}}_T$ and the test sets are denoted by $\mathcal{X}^{\text{test}}_{(1),T},\ldots,\mathcal{X}^{\text{test}}_{(B),T}$. Based on $\mathcal{X}^{\text{train}}_T$, we apply Algorithm~1 to solve (5) with an appropriately large $ \lambda_3 $ for different $(\lambda_1,\lambda_2)$ candidates and obtain $\{\widehat{\Theta}(\mathcal{X}^{\text{train}}_T)_{\lambda_1,\lambda_2}\}^T_{t=1}$.
        The optimal $(\lambda^\ast_1,\lambda^\ast_2)$ is selected as the pair which minimizes
        \begin{equation}
          \label{eq:def-lossval}
          \text{lossval}(\lambda_1, \lambda_2) \coloneqq \frac{1}{B}\sum^{B}_{j=1}\Lb_T(\{\widehat{\Theta}(\mathcal{X}^{\text{train}}_T)_{\lambda_1,\lambda_2}\}^T_{t=1},\mathcal{X}^{\text{test}}_{(j),T}).
        \end{equation}
        $B$ is user-specified. Throughout this paper, we set $B=10$.
  \item[(c)] Method c: When the true underlying data generating process is unknown, as in real data, following \cite{Gibberd2017}, we employ a BIC-type criterion given by
\begin{equation}
  \label{eq:def-BIC}
\text{BIC}(\lambda_1,\lambda_2) = \Lb(\{\widehat\Theta_t\}^T_{t=1},\mathcal{X}_T) + K \log(T),
\end{equation}
where $K$ represents the complexity, or degrees of freedom, and is defined as
\begin{equation*}
K = \text{card}\big(\mathbf{1}(\widehat\Theta_{uv,t}\neq \widehat\Theta_{uv,t-1}), \forall 2 \leq t\leq T, \forall u \neq v\big) + \text{card}\big(\mathbf{1}(\widehat\Theta_{uv,1} \neq 0), \forall u \neq v\big).
\end{equation*}
Then we select the optimal values $(\lambda^\ast_1,\lambda^\ast_2)$ according to the criterion
\begin{equation*}\label{bic_crit}
(\lambda^\ast_1,\lambda^\ast_2)=\underset{\lambda_1,\lambda_2}{\arg\,\min}\;\text{BIC}(\lambda_1,\lambda_2).
\end{equation*}
The definition of $K$ varies across the literature: see, e.g., the different definitions provided in \cite{Gibberd2017} and \cite{Monti2014}. In the former work, $K$ is the sum of the number of active edges at $t = 1$ and of the corresponding changes for $t = 1,\ldots,T$. In the latter work, $K$ is the number of non-zero coefficient blocks in $\widehat{\Theta}_{uv,t}, t=1,\ldots,T$ for $1 \leq u \neq v \leq p$.
Based on our preliminary experiments, we found that these two definitions do not result in significant differences in the selection of the optimal tuning parameters. Therefore, since we will compare our algorithm with the GFGL method of \cite{Gibberd2017}, we use their definition for consistency.
\end{itemize}

For these three methods, the minimization problem is solved w.r.t. $(\lambda_1,\lambda_2)$ over a user-specified grid.
Since Algorithm~1 is solving~(5), where the objective function is multiplied by \( T \), for simplicity, we search for the optimal tuning parameters over the grid of \( \lambda_1T \) and \( \lambda_2T \) instead of \( \lambda_1 \) and \( \lambda_2 \).

It is worth noting that from Proposition~4, point~(ii), within the user-specified grid there may be some pairs of $ (\lambda_1, \lambda_2) $ such that (2) does not have solutions, which can be detected by checking (7).
For these pairs of $ (\lambda_1, \lambda_2) $, we set $\text{lossval}(\lambda_1, \lambda_2) = +\infty$ and $\text{BIC}(\lambda_1, \lambda_2) = +\infty$, so that they will never be selected. A sensitivity analysis on a simulated dataset to evaluate and compare Method b and Method c for selecting the optimal pair of tuning parameters is provided in Section \ref{sec:sen-ana}.

\section{Synthetic Experiments}\label{sec:simulations}

In this section, we conduct some simulation experiments to assess the performances of the GFDtL procedure.

To assess the finite-sample relevance of the GFDtL procedure, we consider the following settings, where we denote by $m^\ast$ the true number of \rev{change points} and by $\Omega^\ast_j, j = 1,\ldots,m^\ast+1$ the true precision matrices:

\noindent\textbf{Setting (i):} For each true $\Omega^\ast_j, j =1,\ldots,m^\ast+1$, its structure is uniformly drawn from the set of graphs with vertex size $\text{card}(V_j)=p$, that is, the number of variables, and $\text{card}(E_j)=M_j$ edges, giving the graph $G(V_j,E_j) \sim$ Erd\"{os}-R\'enyi$(\mathcal{P},M_j)$ for the block $\Bc^\ast_j$. The zero entries are generating by matching the pattern of the adjacency matrix $E_j$ and the precision matrix $\Omega^\ast_j$, that is, $(u,v) \in E_j\Leftrightarrow \Omega^\ast_{uv,j}\neq 0$, providing the sparsity pattern in the off-diagonal coefficients of $\Omega^\ast_j$. The proportion of the zero entries is calibrated by the probability $\mathcal{P}$ of connecting nodes. The off-diagonal non-zero entries of $\Omega^\ast_j$ are drawn in $\Uc([-0.8,-0.05]\cup[0.05,0.8])$, where $\Uc([-a,-b]\cup[b,a])$ denotes the uniform distribution in $[-a,-b]\cup[b,a]$. The diagonal elements are drawn in $\Uc([0.5,1])$. To ensure that the resulting matrix is positive-definite, if $\Omega^\ast_j$ satisfies $\lambda_{\min}(\Omega^\ast_j)<0.01$, we apply $\Omega^\ast_j = \Omega^\ast_j + (\zeta+|\lambda_{\min}(\Omega^\ast_j)|)I_p$, where $\zeta$ is the first value in $\{0.005,0.01,0.015,\ldots\}$ such that $\lambda_{\min}(\Omega^\ast_j)>0.01$.

\noindent\textbf{Setting (ii):} For each true $\Omega^\ast_j, j = 1,\ldots,m^\ast+1$, its off-diagonal non-zero entries are generated in $\Uc([-1,1])$ and diagonal elements are drawn in $\Uc([1.1,1.5])$. To ensure that the resulting matrix is positive-definite, we apply the same final step as in \textbf{Setting (i)}.

\noindent\textbf{Setting (iii):} The precision matrix is generated following the same spirit as in Section 5 of \cite{cai2016}. We construct $\Sigma^\ast_j = \Omega^{\ast-1}_j= D^{1/2}\,C\,D^{1/2}$, $j =1,\ldots,m^\ast+1$, where $D = \text{diag}(U_1,\ldots,U_p)$ with $U_k \in \Uc([0.5,2]), 1 \leq k \leq p$, and where $D$ makes the diagonal
entries in $\Sigma^\ast_j$ and $\Theta^\ast_j$ different. We set $C = (c_{uv})_{1 \leq u,v \leq p}$ with $c_{uv}=a^{|v-u|}$. The coefficient $a$ equals $0.4$ with probability $0.5$, and equals $0.1$ otherwise. Then, we set $\Omega^{\ast}_{j,uv}=0$ when $|\Omega^{\ast}_{j,uv}|<0.05$ and each non-zero off-diagonal coefficient is multiplied by $1$ (resp. $-1$) with probability $0.5$ (resp. $0.5$). Finally, to ensure that the resulting matrix is positive-definite, we apply the same final step as in \textbf{Setting (i)}. This creates a banded structure.

For each of these settings, for $t=1,\ldots,T$, we draw $X_t \sim \Nc_{\Rb^p}(0,\Theta^{\ast -1}_t)$ the $p$-dimensional Gaussian distribution, with $\Theta^\ast_t = \Omega^\ast_j$ when $t \in \Bc^\ast_j$. We set $p=10$ and three cases relating to the \rev{change points} are considered: (a) no \rev{change point}; (b) a single \rev{change point}; (c) \rev{multiple change points}. In case (b), $m^\ast=1$ and in case (c), $m^\ast = 4$; the location of the \rev{change points}, i.e., $T^\ast_j,j=1,\ldots,m^\ast$, are randomly set, conditionally on $\mathcal{I}_{\min}$ being at least $\kappa T$, where $\kappa=1/(m^\ast+c)$. We set $c=8$ so that $\kappa = 0.11$ (resp. $\kappa=0.0833$) when $m^\ast=1$ (resp. $m^\ast=4$): the regimes may have different time lengths but satisfy a minimum time length condition. The latter relates to the issue of trimming: see the Introduction and Section 4 in \cite{Qian2016}, who mentioned that $\kappa \in [0.05,0.25]$ is a standard choice. We set the sample size $T=100$ in case (a) and $T=150$ in cases (b), (c). The ``sparsity degree'', that is the proportion of zero entries in the lower triangular part of $\Omega^\ast_j$, is set as $80\%$ and $30\%$ in settings (i) and (ii), which represents approximately $36$ and $14$ zero entries in each regime out of the $45$ lower triangular elements, respectively. Note that in setting (i), we allow the true number of zero entries to slightly vary between each regime around the corresponding sparsity degree, although $\Pc$ remains constant. In setting (ii), we keep the true number of zero entries constant across the regimes, i.e., $36$ and $14$ zero entries exactly. In setting (iii), the sparsity degree can slightly vary depending on the value of $a$.

For each setting, we draw one hundred batches of $T$ independent samples $\Xc_T$.
For each batch, we apply Algorithm \ref{algo:ADMM} to solve \eqref{eq:D-tr-opt-prob} with some large $ \lambda_3 $ over a grid specified as $0.1,0.2,0.3,\ldots,1$ for $\lambda_1T$ and $10,20, 30,\ldots,200$ for $\lambda_2T$, and select the optimal pair $ (\lambda^\ast_1, \lambda^\ast_2) $ using the selection methods described in Section~\ref{sec:tuning-para}, which will be denoted by Method a, Method b and Method c hereafter.
Specifically, Method a refers to the one that reveals the best performance of the estimator by minimizing or maximizing some performance measures, assuming we know the underground truth; Method b is the one minimizing the loss over some test sets; Method c \rev{selects the} tuning parameters by minimizing BIC.
As competing methods, we also solve the Gaussian loss-based GFGL of \cite{Gibberd2017}\footnote{The Matlab code for GFGL is available on the corresponding journal website as Supplemental.} by selecting the optimal tuning parameters using the same strategies, \rev{and apply the Threshold Block Fused Lasso (TBFL) of \cite{bai2024unified} with parameters set as the default parameters from their \texttt{R} package \texttt{LinearDetect}.\footnote{Archived at \url{https://cran.r-project.org/src/contrib/Archive/LinearDetect/}.} TBFL is a three-step procedure based on penalized multivariate regression, which can be employed as neighborhood
selection and change point analysis for the  Gaussian graphical model.}
\rev{Additionally, we compare GFDtL with three \rev{change point} detection-only algorithms: Binary Segmentation through Operator Norm (BSOP) and Wild Binary Segmentation through Independent Projection (WBSIP) of \cite{wang2021}, and Divide and Conquer Dynamic Programming (DCDP) of \cite{li2023dcdp}.
For BSOP and WBSIP, since explicit tuning parameter selection procedures are not provided in~\cite{wang2021}, we use the parameters following their \texttt{R} package \texttt{changepoints};\footnote{Available at \url{https://cran.r-project.org/web/packages/changepoints/refman/changepoints.html}.} for DCDP, we set $\zeta = 0$ (as the precision matrix setting corresponds to $R = 0$ in their framework) and select $\gamma$ via the cross-validation procedure based on testing datasets.}
For \rev{GFDtL, GFGL and TBFL}, we report the following \rev{five} metrics as performance measures:
\rev{\textbf{(i)} the number of \rev{change points} detected by the procedure denoted by $nb$;
\textbf{(ii)} the Hausdorff distance $d_h =100\times \max\big(h(\widehat{T}_{\widehat{m}},\Tc^\ast_{m^\ast}),h(\Tc^\ast_{m^\ast},\widehat{T}_{\widehat{m}})\big)/T$, which serves as a measure of the estimation accuracy of the \rev{change points};
\textbf{(iii)} the $\text{F}_1$ score defined as $\text{F}_1=2\text{TP}/(2\text{TP}+\text{FN}+\text{FP})$, where $\text{TP}$ is the number of correctly estimated non-zero coefficients, $\text{FN}$ is the number of incorrectly estimated zero entries and $\text{FP}$ is the number of incorrectly estimated non-zero entries; 
\textbf{(iv)} the accuracy defined as $ \acc = (\text{TP} + \text{TN}) / T $, where $ \text{TN} $ is the number of correctly estimated zero entries;
\textbf{(v)} the averaged mean squared error (MSE) $\sqrt{(p^2T)^{-1}\sum^{T}_{t=1}\|\widehat{\Theta}_t-\Theta^\ast_t\|^2_F}$.}
\rev{Since BSOP, WBSIP, and DCDP focus on detecting change points, we report the number of \rev{change points} ($nb$) and the Hausdorff distance ($d_h$) for these three methods.}
These metrics are averaged over the \rev{100} independent batches and are reported in \rev{Table \ref{Table}}. \rev{These experiments} have been run on an Intel(R) Xeon(R) Gold 6242R CPU@3.10GHz 3.09 GHz, 128 GB.

\begin{table}[htb]
	\centering
	\caption{Recovery of the \rev{change points}, model selection and precision accuracy based on 100 replications with respect to $(m^\ast,s^\ast,T)$. For $ d_h $ and MSE, smaller numbers are better; for $ \text{F}_1 $ and $ \acc $, larger numbers are better. Bold figures represent the best performing models.\label{Table}}
	\resizebox{\textwidth}{!}{%
		\rev{
            \begin{tabular}{cc|cccccc|cccccc|ccc|ccc|ccc}\hline\hline
                \multicolumn{23}{c}{\textbf{Setting (i)}}                                                                                                                                                                                                                                                                                                                                                                                                                                                                      \\\hline\hline
                                    &                         & \multicolumn{6}{c|}{$nb$} & \multicolumn{6}{c|}{$d_h$} & \multicolumn{3}{c|}{$\text{F}_1$} & \multicolumn{3}{c|}{$\acc$} & \multicolumn{3}{c}{MSE}                                                                                                                                                                                                                                                                                                             \\
                $(m^\ast,s^\ast,T)$ & $(\lambda_1,\lambda_2)$ & GFDtL                     & \multicolumn{1}{c}{GFGL}   & BSOP                              & WBSIP                       & DCDP                    & TBFL & GFDtL          & \multicolumn{1}{c}{GFGL}            & BSOP  & WBSIP & DCDP          & TBFL  & GFDtL           & GFGL                                 & TBFL   & GFDtL           & GFGL                                 & TBFL   & GFDtL           & GFGL                                 & TBFL   \\ \hline

                $(0,0.8,100)$       & Method a                & 0                         & \multicolumn{1}{c|}{0}     &                                   &                             &                         &      & 0              & \multicolumn{1}{c|}{0}              &       &       &               &       & \textbf{0.8472} & \multicolumn{1}{c|}{0.7174}          &        & \textbf{0.9133} & \multicolumn{1}{c|}{0.8084}          &        & \textbf{0.1707} & \multicolumn{1}{c|}{0.2266}          &        \\
                                    & Method b                & 0                         & \multicolumn{1}{c|}{3.68}  & 1.00                              & 2.70                        & 0.00                    & 1.14 & \textbf{0}     & \multicolumn{1}{c|}{46.01}          & 49.90 & 36.83 & \textbf{0.00} & 50.33 & \textbf{0.8072} & \multicolumn{1}{c|}{0.5377}          & 0.6769 & \textbf{0.8919} & \multicolumn{1}{c|}{0.5518}          & 0.7869 & \textbf{0.1662} & \multicolumn{1}{c|}{0.2089}          & 0.1799 \\
                                    & Method c                & 0.14                      & \multicolumn{1}{c|}{2.70}  &                                   &                             &                         &      & \textbf{4.11}  & \multicolumn{1}{c|}{42.22}          &       &       &               &       & \textbf{0.7525} & \multicolumn{1}{c|}{0.6915}          &        & \textbf{0.8425} & \multicolumn{1}{c|}{0.8060}          &        & \textbf{0.1335} & \multicolumn{1}{c|}{0.2331}          &        \\

                \hline

                $(0,0.3,100)$       & Method a                & 0                         & \multicolumn{1}{c|}{0}     &                                   &                             &                         &      & 0              & \multicolumn{1}{c|}{0}              &       &       &               &       & \textbf{0.8246} & \multicolumn{1}{c|}{0.7247}          &        & \textbf{0.8036} & \multicolumn{1}{c|}{0.6429}          &        & \textbf{0.2098} & \multicolumn{1}{c|}{0.4219}          &        \\
                                    & Method b                & 0                         & \multicolumn{1}{c|}{2.50}  & 1.00                              & 2.28                        & 0.00                    & 1.09 & \textbf{0}     & \multicolumn{1}{c|}{36.84}          & 49.94 & 31.17 & \textbf{0.00} & 49.66 & \textbf{0.7178} & \multicolumn{1}{c|}{0.7078}          & 0.6724 & \textbf{0.7251} & \multicolumn{1}{c|}{0.5811}          & 0.5979 & \textbf{0.3452} & \multicolumn{1}{c|}{0.4127}          & 0.3723 \\
                                    & Method c                & 0.07                      & \multicolumn{1}{c|}{5.98}  &                                   &                             &                         &      & \textbf{2.61}  & \multicolumn{1}{c|}{45.02}          &       &       &               &       & \textbf{0.7903} & \multicolumn{1}{c|}{0.6526}          &        & \textbf{0.7541} & \multicolumn{1}{c|}{0.6281}          &        & \textbf{0.1755} & \multicolumn{1}{c|}{0.4491}          &        \\
                \hline

                $(1,0.8,150)$       & Method a                & 1.01                      & \multicolumn{1}{c|}{1.27}  &                                   &                             &                         &      & \textbf{0.49}  & \multicolumn{1}{c|}{1.35}           &       &       &               &       & \textbf{0.7261} & \multicolumn{1}{c|}{0.6592}          &        & \textbf{0.8294} & \multicolumn{1}{c|}{0.7641}          &        & 0.2496          & \multicolumn{1}{c|}{\textbf{0.2382}} &        \\
                                    & Method b                & 2.18                      & \multicolumn{1}{c|}{23.06} & 1.00                              & 7.89                        & 0.99                    & 1.09 & \textbf{20.99} & \multicolumn{1}{c|}{69.96}          & 12.46 & 79.64 & \textbf{2.56} & 22.01 & \textbf{0.6815} & \multicolumn{1}{c|}{0.4712}          & 0.6616 & \textbf{0.7805} & \multicolumn{1}{c|}{0.4209}          & 0.7575 & \textbf{0.2049} & \multicolumn{1}{c|}{0.2092}          & \textbf{0.1812} \\
                                    & Method c                & 1.40                      & \multicolumn{1}{c|}{3.72}  &                                   &                             &                         &      & \textbf{9.51}  & \multicolumn{1}{c|}{21.94}          &       &       &               &       & 0.6480          & \multicolumn{1}{c|}{\textbf{0.6707}} &        & 0.7264          & \multicolumn{1}{c|}{\textbf{0.8016}} &        & \textbf{0.2029} & \multicolumn{1}{c|}{0.2427}          &        \\
                \hline

                $(1,0.3,150)$       & Method a                & 1.00                      & \multicolumn{1}{c|}{1.36}  &                                   &                             &                         &      & \textbf{0.07}  & \multicolumn{1}{c|}{0.43}           &       &       &               &       & \textbf{0.7441} & \multicolumn{1}{c|}{0.7323}          &        & \textbf{0.6882} & \multicolumn{1}{c|}{0.6347}          &        & \textbf{0.3558} & \multicolumn{1}{c|}{0.4237}          &        \\
                                    & Method b                & 2.79                      & \multicolumn{1}{c|}{20.07} & 1.00                              & 8.38                        & 1.00                    & 1.02 & \textbf{27.78} & \multicolumn{1}{c|}{66.72}          & 17.25 & 88.73 & \textbf{1.16} & 18.12 & \textbf{0.7163} & \multicolumn{1}{c|}{0.7140}          & 0.7152 & \textbf{0.6696} & \multicolumn{1}{c|}{0.5784}          & 0.6424 & \textbf{0.3774} & \multicolumn{1}{c|}{0.4068}          & \textbf{0.3231} \\
                                    & Method c                & 1.13                      & \multicolumn{1}{c|}{5.40}  &                                   &                             &                         &      & \textbf{9.73}  & \multicolumn{1}{c|}{23.40}          &       &       &               &       & \textbf{0.7223} & \multicolumn{1}{c|}{0.6670}          &        & \textbf{0.6583} & \multicolumn{1}{c|}{0.6343}          &        & \textbf{0.3335} & \multicolumn{1}{c|}{0.4505}          &        \\

                \hline

                $(4,0.8,150)$       & Method a                & 4.21                      & \multicolumn{1}{c|}{5.76}  &                                   &                             &                         &      & \textbf{2.27}  & \multicolumn{1}{c|}{2.71}           &       &       &               &       & \textbf{0.6389} & \multicolumn{1}{c|}{0.6088}          &        & \textbf{0.7409} & \multicolumn{1}{c|}{0.7093}          &        & 0.2906          & \multicolumn{1}{c|}{\textbf{0.2585}} &        \\
                                    & Method b                & 5.45                      & \multicolumn{1}{c|}{44.11} & 1.00                              & 7.59                        & 2.22                    & 1.22 & \textbf{11.58} & \multicolumn{1}{c|}{27.67}          & 67.97 & 28.43 & 32.34         & 57.88 & \textbf{0.6040} & \multicolumn{1}{c|}{0.4555}          & 0.5339 & \textbf{0.6924} & \multicolumn{1}{c|}{0.3810}          & 0.6440 & 0.2476          & \multicolumn{1}{c|}{\textbf{0.2159}} & 0.2734 \\
                                    & Method c                & 2.92                      & \multicolumn{1}{c|}{7.23}  &                                   &                             &                         &      & 29.31          & \multicolumn{1}{c|}{\textbf{18.01}} &       &       &               &       & 0.5792          & \multicolumn{1}{c|}{\textbf{0.6228}} &        & 0.6579          & \multicolumn{1}{c|}{\textbf{0.7488}} &        & 0.2853          & \multicolumn{1}{c|}{\textbf{0.2599}} &        \\

                \hline

                $(4,0.3,150)$       & Method a                & 4.11                      & \multicolumn{1}{c|}{5.11}  &                                   &                             &                         &      & 1.74           & \multicolumn{1}{c|}{\textbf{1.18}}  &       &       &               &       & 0.7021          & \multicolumn{1}{c|}{\textbf{0.7175}} &        & \textbf{0.6180} & \multicolumn{1}{c|}{0.5920}          &        & 0.4475          & \multicolumn{1}{c|}{\textbf{0.4379}} &        \\
                                    & Method b                & 6.30                      & \multicolumn{1}{c|}{36.81} & 1.00                              & 7.59                        & 2.56                    & 1.23 & \textbf{14.91} & \multicolumn{1}{c|}{27.65}          & 68.57 & 29.32 & 25.15         & 53.96 & 0.6989          & \multicolumn{1}{c|}{\textbf{0.7163}} & 0.6540 & \textbf{0.6209} & \multicolumn{1}{c|}{0.5759}          & 0.5677 & 0.4218          & \multicolumn{1}{c|}{\textbf{0.4114}} & 0.4504 \\
                                    & Method c                & 1.88                      & \multicolumn{1}{c|}{6.08}  &                                   &                             &                         &      & 41.04          & \multicolumn{1}{c|}{\textbf{11.50}} &       &       &               &       & \textbf{0.6616} & \multicolumn{1}{c|}{0.6615}          &        & 0.5700          & \multicolumn{1}{c|}{\textbf{0.6067}} &        & \textbf{0.4555} & \multicolumn{1}{c|}{0.4762}          &        \\

                \hline\hline
                \multicolumn{23}{c}{\textbf{Setting (ii)}}                                                                                                                                                                                                                                                                                                                                                                                                                                                                     \\\hline\hline
                                    &                         & \multicolumn{6}{c|}{$nb$} & \multicolumn{6}{c|}{$d_h$} & \multicolumn{3}{c|}{$\text{F}_1$} & \multicolumn{3}{c|}{$\acc$} & \multicolumn{3}{c}{MSE}                                                                                                                                                                                                                                                                                                             \\
                $(m^\ast,s^\ast,T)$ & $(\lambda_1,\lambda_2)$ & GFDtL                     & \multicolumn{1}{c}{GFGL}   & BSOP                              & WBSIP                       & DCDP                    & TBFL & GFDtL          & \multicolumn{1}{c}{GFGL}            & BSOP  & WBSIP & DCDP          & TBFL  & GFDtL           & GFGL                                 & TBFL   & GFDtL           & GFGL                                 & TBFL   & GFDtL           & GFGL                                 & TBFL   \\ \hline

                $(0,0.8,100)$       & Method a                & 0                         & \multicolumn{1}{c|}{0}     &                                   &                             &                         &      & 0              & \multicolumn{1}{c|}{0}              &       &       &               &       & \textbf{0.8364} & \multicolumn{1}{c|}{0.7131}          &        & \textbf{0.9051} & \multicolumn{1}{c|}{0.8109}          &        & \textbf{0.2087} & \multicolumn{1}{c|}{0.3582}          &        \\
                                    & Method b                & 0.01                      & \multicolumn{1}{c|}{2.76}  & 0.91                              & 1.79                        & 0.00                    & 1.08 & \textbf{0.36}  & \multicolumn{1}{c|}{40.04}          & 45.38 & 26.39 & \textbf{0.00} & 47.69 & \textbf{0.7940} & \multicolumn{1}{c|}{0.5841}          & 0.6735 & \textbf{0.8861} & \multicolumn{1}{c|}{0.6200}          & 0.7815 & \textbf{0.2239} & \multicolumn{1}{c|}{0.3417}          & 0.2351 \\
                                    & Method c                & 0.02                      & \multicolumn{1}{c|}{2.22}  &                                   &                             &                         &      & \textbf{0.80}  & \multicolumn{1}{c|}{44.01}          &       &       &               &       & \textbf{0.7576} & \multicolumn{1}{c|}{0.6655}          &        & \textbf{0.8532} & \multicolumn{1}{c|}{0.8048}          &        & \textbf{0.1560} & \multicolumn{1}{c|}{0.3654}          &        \\

                \hline

                $(0,0.3,100)$       & Method a                & 0                         & \multicolumn{1}{c|}{0}     &                                   &                             &                         &      & 0              & \multicolumn{1}{c|}{0}              &       &       &               &       & \textbf{0.8437} & \multicolumn{1}{c|}{0.8143}          &        & \textbf{0.7784} & \multicolumn{1}{c|}{0.7069}          &        & \textbf{0.2451} & \multicolumn{1}{c|}{0.6669}          &        \\
                                    & Method b                & 0                         & \multicolumn{1}{c|}{1.27}  & 1.00                              & 2.81                        & 0.00                    & 0.98 & \textbf{0}     & \multicolumn{1}{c|}{29.15}          & 50.02 & 38.90 & \textbf{0.00} & 45.34 & 0.6995          & \multicolumn{1}{c|}{\textbf{0.8155}} & 0.7559 & 0.6500          & \multicolumn{1}{c|}{\textbf{0.7027}} & 0.6507 & \textbf{0.5162} & \multicolumn{1}{c|}{0.6652}          & 0.5449 \\
                                    & Method c                & 0.01                      & \multicolumn{1}{c|}{3.26}  &                                   &                             &                         &      & \textbf{0.36}  & \multicolumn{1}{c|}{38.27}          &       &       &               &       & \textbf{0.8318} & \multicolumn{1}{c|}{0.6885}          &        & \textbf{0.7639} & \multicolumn{1}{c|}{0.6071}          &        & \textbf{0.2423} & \multicolumn{1}{c|}{0.6981}          &        \\

                \hline

                $(1,0.8,150)$       & Method a                & 1.02                      & \multicolumn{1}{c|}{1.29}  &                                   &                             &                         &      & \textbf{0.91}  & \multicolumn{1}{c|}{1.38}           &       &       &               &       & \textbf{0.7021} & \multicolumn{1}{c|}{0.6546}          &        & \textbf{0.7968} & \multicolumn{1}{c|}{0.7645}          &        & \textbf{0.2833} & \multicolumn{1}{c|}{0.3470}          &        \\
                                    & Method b                & 1.89                      & \multicolumn{1}{c|}{23.90} & 1.00                              & 6.21                        & 0.95                    & 1.09 & \textbf{15.73} & \multicolumn{1}{c|}{86.60}          & 13.94 & 68.95 & \textbf{4.75} & 19.68 & \textbf{0.6777} & \multicolumn{1}{c|}{0.5139}          & 0.6618 & \textbf{0.7769} & \multicolumn{1}{c|}{0.4931}          & 0.7726 & \textbf{0.2530} & \multicolumn{1}{c|}{0.3206}          & \textbf{0.2274} \\
                                    & Method c                & 1.01                      & \multicolumn{1}{c|}{3.12}  &                                   &                             &                         &      & \textbf{16.38} & \multicolumn{1}{c|}{21.06}          &       &       &               &       & 0.6264          & \multicolumn{1}{c|}{\textbf{0.6398}} &        & 0.7374          & \multicolumn{1}{c|}{\textbf{0.7998}} &        & \textbf{0.2552} & \multicolumn{1}{c|}{0.3506}          &        \\

                \hline

                $(1,0.3,150)$       & Method a                & 0.99                      & \multicolumn{1}{c|}{1.45}  &                                   &                             &                         &      & \textbf{0.30}  & \multicolumn{1}{c|}{0.80}           &       &       &               &       & \textbf{0.7846} & \multicolumn{1}{c|}{0.8271}          &        & \textbf{0.6909} & \multicolumn{1}{c|}{0.7186}          &        & \textbf{0.5085} & \multicolumn{1}{c|}{0.6511}          &        \\
                                    & Method b                & 3.51                      & \multicolumn{1}{c|}{20.56} & 1.00                              & 8.15                        & 1.00                    & 1.03 & \textbf{39.85} & \multicolumn{1}{c|}{91.12}          & 18.43 & 86.99 & \textbf{0.98} & 26.43 & 0.7465          & \multicolumn{1}{c|}{\textbf{0.8209}} & 0.7876 & 0.6560          & \multicolumn{1}{c|}{\textbf{0.7054}} & 0.6834 & \textbf{0.5470} & \multicolumn{1}{c|}{0.6431}          & \textbf{0.4682} \\
                                    & Method c                & 1.03                      & \multicolumn{1}{c|}{4.08}  &                                   &                             &                         &      & \textbf{14.26} & \multicolumn{1}{c|}{25.08}          &       &       &               &       & \textbf{0.7717} & \multicolumn{1}{c|}{0.6935}          &        & \textbf{0.6765} & \multicolumn{1}{c|}{0.6089}          &        & \textbf{0.4779} & \multicolumn{1}{c|}{0.6874}          &        \\

                \hline

                $(4,0.8,150)$       & Method a                & 4.31                      & \multicolumn{1}{c|}{6.15}  &                                   &                             &                         &      & 4.55           & \multicolumn{1}{c|}{\textbf{4.53}}  &       &       &               &       & \textbf{0.5985} & \multicolumn{1}{c|}{0.5901}          &        & 0.6861          & \multicolumn{1}{c|}{\textbf{0.6931}} &        & \textbf{0.3334} & \multicolumn{1}{c|}{0.3466}          &        \\
                                    & Method b                & 4.39                      & \multicolumn{1}{c|}{40.81} & 1.00                              & 6.14                        & 1.97                    & 1.12 & \textbf{16.09} & \multicolumn{1}{c|}{33.64}          & 66.73 & 31.10 & 39.68         & 59.18 & \textbf{0.5788} & \multicolumn{1}{c|}{0.4821}          & 0.5318 & \textbf{0.6624} & \multicolumn{1}{c|}{0.4297}          & 0.6372 & \textbf{0.3099} & \multicolumn{1}{c|}{0.3153}          & 0.3324 \\
                                    & Method c                & 1.93                      & \multicolumn{1}{c|}{5.45}  &                                   &                             &                         &      & 55.61          & \multicolumn{1}{c|}{\textbf{32.40}} &       &       &               &       & 0.5226          & \multicolumn{1}{c|}{\textbf{0.5971}} &        & 0.5968          & \multicolumn{1}{c|}{\textbf{0.7630}} &        & \textbf{0.3480} & \multicolumn{1}{c|}{0.3559}          &        \\

                \hline

                $(4,0.3,150)$       & Method a                & 4.10                      & \multicolumn{1}{c|}{5.13}  &                                   &                             &                         &      & 1.94           & \multicolumn{1}{c|}{\textbf{1.19}}  &       &       &               &       & 0.7772          & \multicolumn{1}{c|}{\textbf{0.8265}} &        & 0.6702          & \multicolumn{1}{c|}{\textbf{0.7150}} &        & \textbf{0.6536} & \multicolumn{1}{c|}{0.6727}          &        \\
                                    & Method b                & 6.88                      & \multicolumn{1}{c|}{33.92} & 1.00                              & 7.60                        & 2.82                    & 1.32 & \textbf{17.04} & \multicolumn{1}{c|}{34.21}          & 67.59 & 30.02 & 20.02         & 51.70 & 0.7691          & \multicolumn{1}{c|}{\textbf{0.8233}} & 0.7512 & 0.6655          & \multicolumn{1}{c|}{\textbf{0.7077}} & 0.6421 & \textbf{0.6031} & \multicolumn{1}{c|}{0.6516}          & 0.6353 \\
                                    & Method c                & 1.50                      & \multicolumn{1}{c|}{6.37}  &                                   &                             &                         &      & 50.40          & \multicolumn{1}{c|}{\textbf{12.09}} &       &       &               &       & \textbf{0.7507} & \multicolumn{1}{c|}{0.7154}          &        & \textbf{0.6546} & \multicolumn{1}{c|}{0.6171}          &        & \textbf{0.6527} & \multicolumn{1}{c|}{0.7177}          &        \\

                \hline\hline
                \multicolumn{23}{c}{\textbf{Setting (iii)}}                                                                                                                                                                                                                                                                                                                                                                                                                                                                    \\\hline\hline
                                    &                         & \multicolumn{6}{c|}{$nb$} & \multicolumn{6}{c|}{$d_h$} & \multicolumn{3}{c|}{$\text{F}_1$} & \multicolumn{3}{c|}{$\acc$} & \multicolumn{3}{c}{MSE}                                                                                                                                                                                                                                                                                                             \\
                $(m^\ast,T)$        & $(\lambda_1,\lambda_2)$ & GFDtL                     & \multicolumn{1}{c}{GFGL}   & BSOP                              & WBSIP                       & DCDP                    & TBFL & GFDtL          & \multicolumn{1}{c}{GFGL}            & BSOP  & WBSIP & DCDP          & TBFL  & GFDtL           & GFGL                                 & TBFL   & GFDtL           & GFGL                                 & TBFL   & GFDtL           & GFGL                                 & TBFL   \\ \hline

                $(0,100)$           & Method a                & 0                         & \multicolumn{1}{c|}{0}     &                                   &                             &                         &      & 0              & \multicolumn{1}{c|}{0}              &       &       &               &       & \textbf{0.8236} & \multicolumn{1}{c|}{0.7848}          &        & \textbf{0.8230} & \multicolumn{1}{c|}{0.7704}          &        & \textbf{0.1089} & \multicolumn{1}{c|}{0.2241}          &        \\
                                    & Method b                & 0.06                      & \multicolumn{1}{c|}{5.66}  & 0.99                              & 1.45                        & 0.00                    & 1.12 & \textbf{1.59}  & \multicolumn{1}{c|}{48.68}          & 49.42 & 26.95 & \textbf{0.00} & 46.51 & \textbf{0.7611} & \multicolumn{1}{c|}{0.7268}          & 0.7084 & \textbf{0.7941} & \multicolumn{1}{c|}{0.6448}          & 0.7340 & \textbf{0.1504} & \multicolumn{1}{c|}{0.2100}          & 0.1438 \\
                                    & Method c                & 0.03                      & \multicolumn{1}{c|}{7.38}  &                                   &                             &                         &      & \textbf{0.89}  & \multicolumn{1}{c|}{53.90}          &       &       &               &       & \textbf{0.5396} & \multicolumn{1}{c|}{0.5322}          &        & \textbf{0.6835} & \multicolumn{1}{c|}{0.6533}          &        & \textbf{0.1709} & \multicolumn{1}{c|}{0.2511}          &        \\

                \hline

                $(1,150)$           & Method a                & 1.00                      & \multicolumn{1}{c|}{1.12}  &                                   &                             &                         &      & \textbf{2.38}  & \multicolumn{1}{c|}{3.19}           &       &       &               &       & \textbf{0.7799} & \multicolumn{1}{c|}{0.7357}          &        & \textbf{0.7670} & \multicolumn{1}{c|}{0.7425}          &        & \textbf{0.1800} & \multicolumn{1}{c|}{0.2344}          &        \\
                                    & Method b                & 2.55                      & \multicolumn{1}{c|}{19.27} & 1.00                              & 5.28                        & 0.82                    & 1.05 & \textbf{27.82} & \multicolumn{1}{c|}{85.50}          & 18.13 & 69.86 & 15.66         & 22.09 & \textbf{0.7564} & \multicolumn{1}{c|}{0.7136}          & 0.7040 & \textbf{0.7600} & \multicolumn{1}{c|}{0.6084}          & 0.7313 & \textbf{0.1825} & \multicolumn{1}{c|}{0.2092}          & \textbf{0.1508} \\
                                    & Method c                & 0.44                      & \multicolumn{1}{c|}{7.51}  &                                   &                             &                         &      & 59.11          & \multicolumn{1}{c|}{\textbf{42.37}} &       &       &               &       & 0.3961          & \multicolumn{1}{c|}{\textbf{0.4830}} &        & 0.6028          & \multicolumn{1}{c|}{0.6325}          &        & \textbf{0.2317} & \multicolumn{1}{c|}{0.2554}          &        \\

                \hline

                $(4,150)$           & Method a                & 4.99                      & \multicolumn{1}{c|}{7.27}  &                                   &                             &                         &      & \textbf{8.25}  & \multicolumn{1}{c|}{9.67}           &       &       &               &       & \textbf{0.6950} & \multicolumn{1}{c|}{0.6816}          &        & 0.6709          & \multicolumn{1}{c|}{\textbf{0.6791}} &        & \textbf{0.2181} & \multicolumn{1}{c|}{0.2274}          &        \\
                                    & Method b                & 4.52                      & \multicolumn{1}{c|}{34.16} & 1.00                              & 4.72                        & 1.16                    & 1.19 & \textbf{24.39} & \multicolumn{1}{c|}{32.71}          & 62.19 & 34.25 & 65.79         & 55.52 & \textbf{0.7104} & \multicolumn{1}{c|}{0.6951}          & 0.6162 & \textbf{0.7029} & \multicolumn{1}{c|}{0.5862}          & 0.6906 & 0.2251          & \multicolumn{1}{c|}{\textbf{0.2095}} & 0.2145 \\
                                    & Method c                & 0.37                      & \multicolumn{1}{c|}{4.33}  &                                   &                             &                         &      & 107.96         & \multicolumn{1}{c|}{\textbf{66.84}} &       &       &               &       & 0.3538          & \multicolumn{1}{c|}{\textbf{0.4279}} &        & 0.5836          & \multicolumn{1}{c|}{\textbf{0.6102}} &        & 0.2589          & \multicolumn{1}{c|}{\textbf{0.2478}} &        \\

                \hline\hline
            \end{tabular}}
	}
\end{table}

\rev{
Compared to GFGL, GFDtL performs better in most settings and across most metrics for Methods a and b, with notably lower MSE indicating more accurate precision matrix estimation.  
TBFL tends to detect approximately one change point across most settings, and achieves moderate $\text{F}_1$ scores and accuracy, and slightly better MSE in some cases, but overall underperforms compared to GFDtL.
The \rev{change point} detection-only methods exhibit distinct behaviors.
BSOP consistently returns a single \rev{change point} across all settings, suggesting an overly conservative approach for multiple change points.
WBSIP tends to overestimate the number of \rev{change points}, e.g., returning around 4--7 when $m^\ast = 1, 4$.
DCDP performs well for $m^\ast \leq 1$ but tends to underestimate the number of \rev{change points} when $m^\ast = 4$.
Overall, BSOP, WBSIP, and TBFL yield poor Hausdorff distances across most scenarios, while DCDP achieves competitive localization accuracy when $m^\ast \leq 1$.
}

\rev{
  In the case of no \rev{change point}, i.e., $m^\ast=0$, $d_h=0$ means that the procedure returns no \rev{change point}.
  In this setting, our GFDtL procedure is comparable with DCDP and performs much better than GFGL, BSOP, WBSIP, and TBFL.
  Compared to GFGL, GFDtL achieves substantially lower MSE and Hausdorff distance, particularly when applying Method b, while GFGL tends to overestimate the number of \rev{change points}.
  BSOP and TBFL consistently detect one \rev{change point}, while WBSIP detects approximately 1--2 \rev{change points}.

  Finally, despite good MSE performance, the BIC-based results (Method c) favor the GFGL method, especially for $d_h$. This occurs because BIC tends to underestimate the number of \rev{change points} when applied to the GFDtL, i.e., it tends to select large $\lambda_2$. Indeed, when $m^\ast \leq 1$, Method c performs well; however, with multiple \rev{change points}, the number of \rev{change points} detected by BIC is much lower than the true number of \rev{change points}. This behavior is further detailed in Section~\ref{sec:sen-ana}.
}

\rev{
  Overall, these results demonstrate that GFDtL outperforms the competing methods in most settings, both in accurately detecting the true number of \rev{change points} and in precisely estimating the precision matrices.
}

\section{Sensitivity Analysis with Respect to the Tuning Parameters}\label{sec:sen-ana}

\subsection{Sensitivity Analysis with Respect to $\lambda_1$ and $\lambda_2$}

We propose a sensitivity analysis of Methods b and c provided in Section \ref{sec:tuning-para} with respect to the calibration of $(\lambda_1,\lambda_2)$. More precisely, we illustrate the ability of the proposed strategies to identify the optimal pair $(\lambda_1,\lambda_2)$ for \rev{change point detection} and sparse estimation.
The experiments are conducted on datasets simulated according to \textbf{Setting (i)} in Section~5 of the main text with $ T = 100 $, $ p = 10 $, the ``sparsity degree'' being $ 80\% $ and $ m^\ast = 0, 1, 3 $.
To better approximate the metric surfaces, we use a denser grid specified as $0.1,0.11,0.12,\ldots,1$ for $\lambda_1T$ and $10,11,12,\ldots,200$ for $\lambda_2T$.
The results are displayed in Figure \ref{fig:sen-ana}, with the three columns corresponding to $ m^\ast = 0, 1, 3 $, and the four rows showcasing the results for BICs (cf. (\ref{eq:def-BIC})), lossvals (cf. (\ref{eq:def-lossval})), Hausdorff distances, and F1 scores, respectively.
In each subfigure, lighter colors represent more favorable tuning parameters, indicating areas where the metric values are minimized or maximized as appropriate.
To facilitate visualization given the wide range of values for BICs and lossvals, we pre-processed these metrics before plotting: specifically, we subtracted the minimum value to ensure non-negativity, then applied the \textsf{log1p} function in \textsc{Matlab} (i.e., $ \log (1 + \cdot) $) to compress the scale of the values, making the patterns more discernible and enhancing the interpretability of the results.

\begin{figure}[p]
    \centering
    \caption{\rev{Sensitivity analysis of tuning parameters \( \lambda_1 \) and \( \lambda_2 \).}}
    \subfloat{\includegraphics[width=0.33\textwidth]{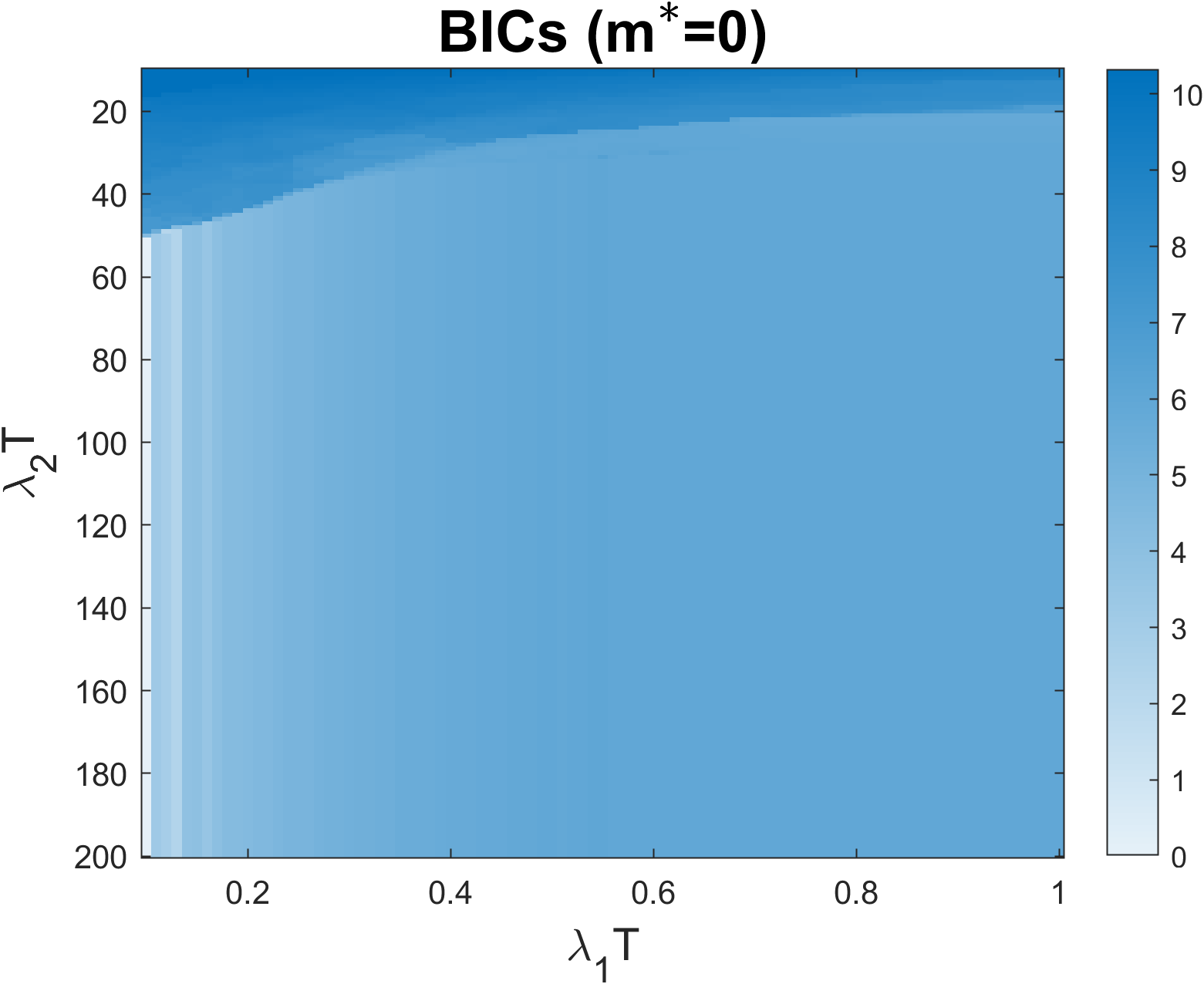}}\hfill
    \subfloat{\includegraphics[width=0.33\textwidth]{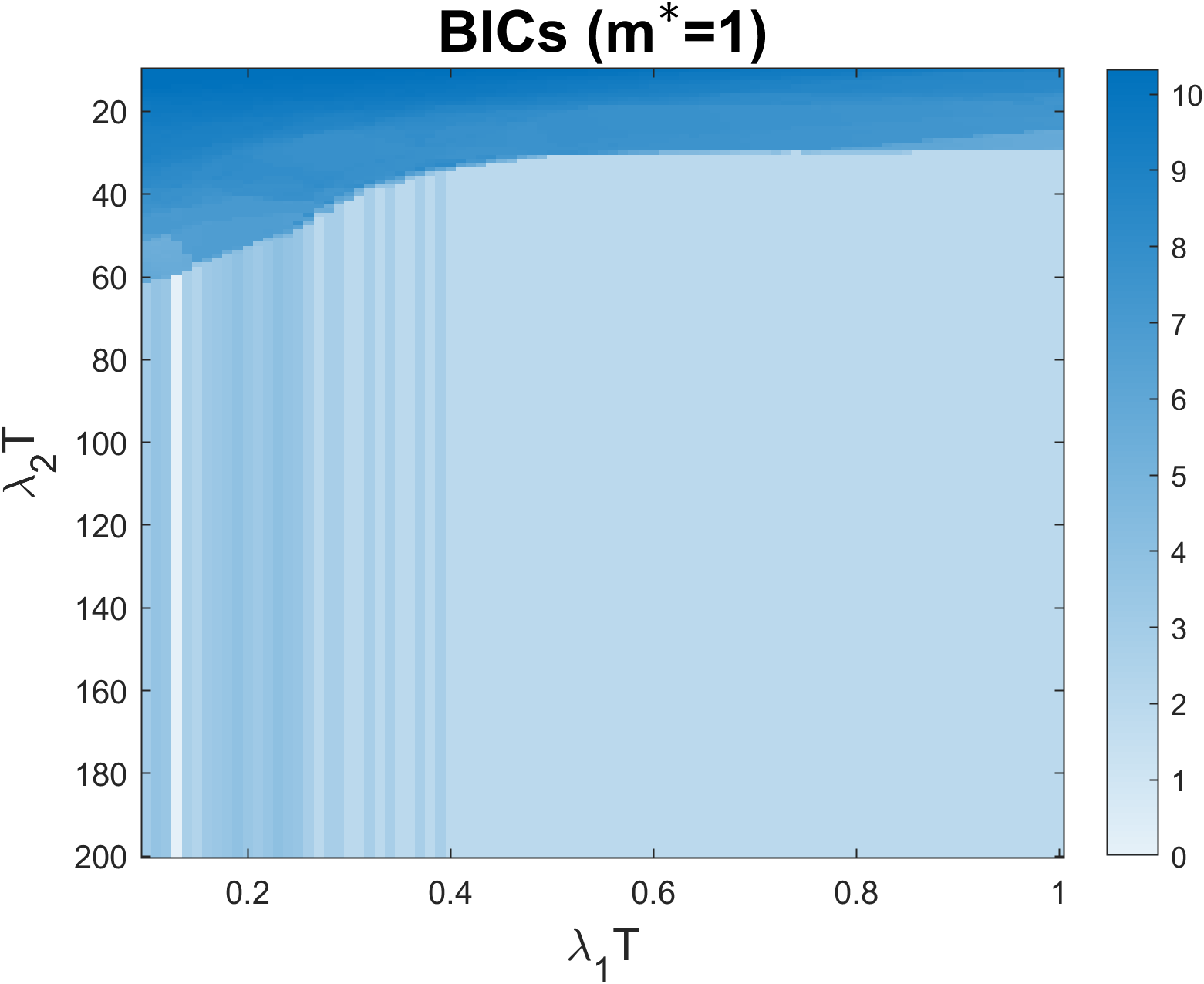}}\hfill
    \subfloat{\includegraphics[width=0.33\textwidth]{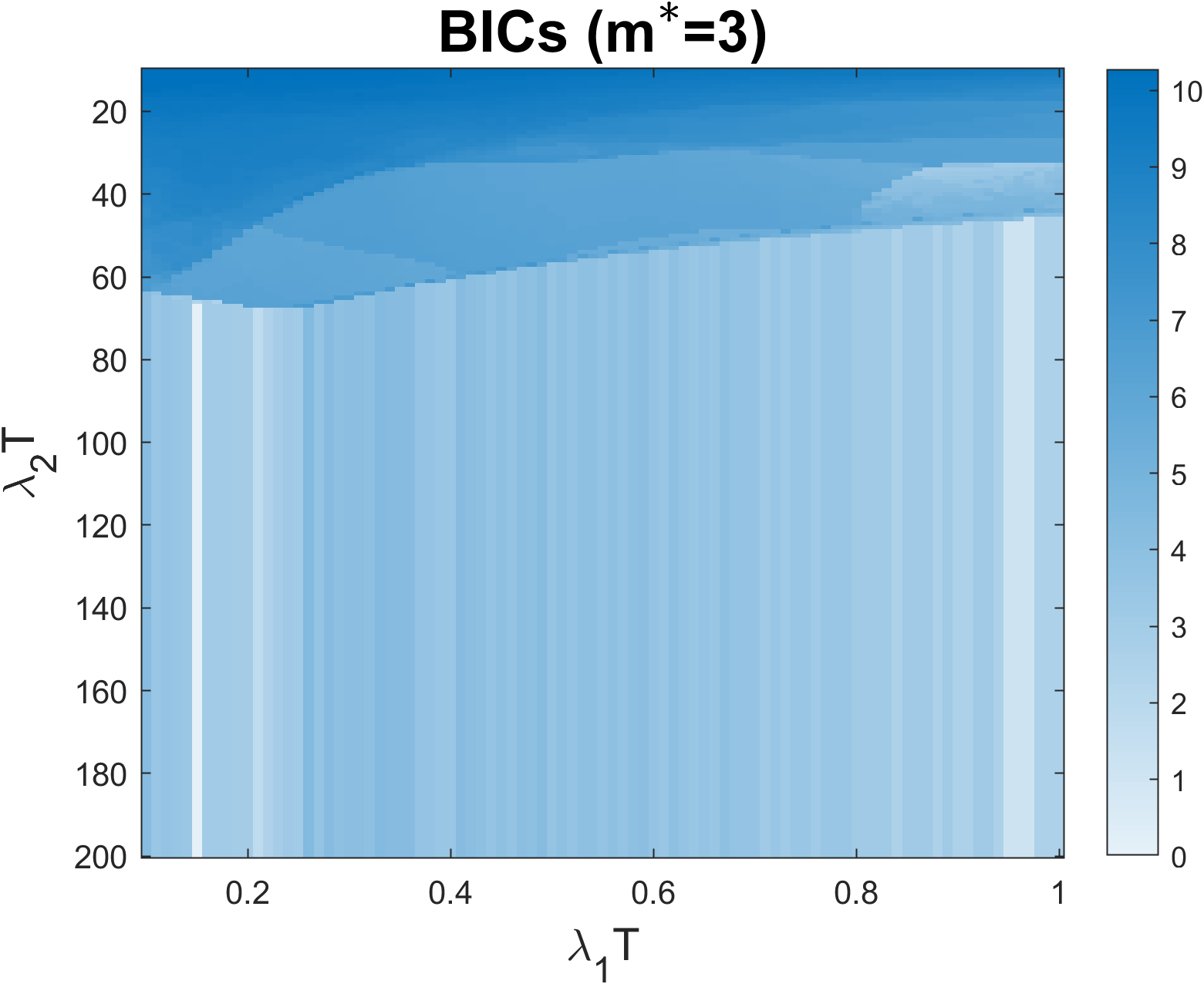}}\\

    \subfloat{\includegraphics[width=0.33\textwidth]{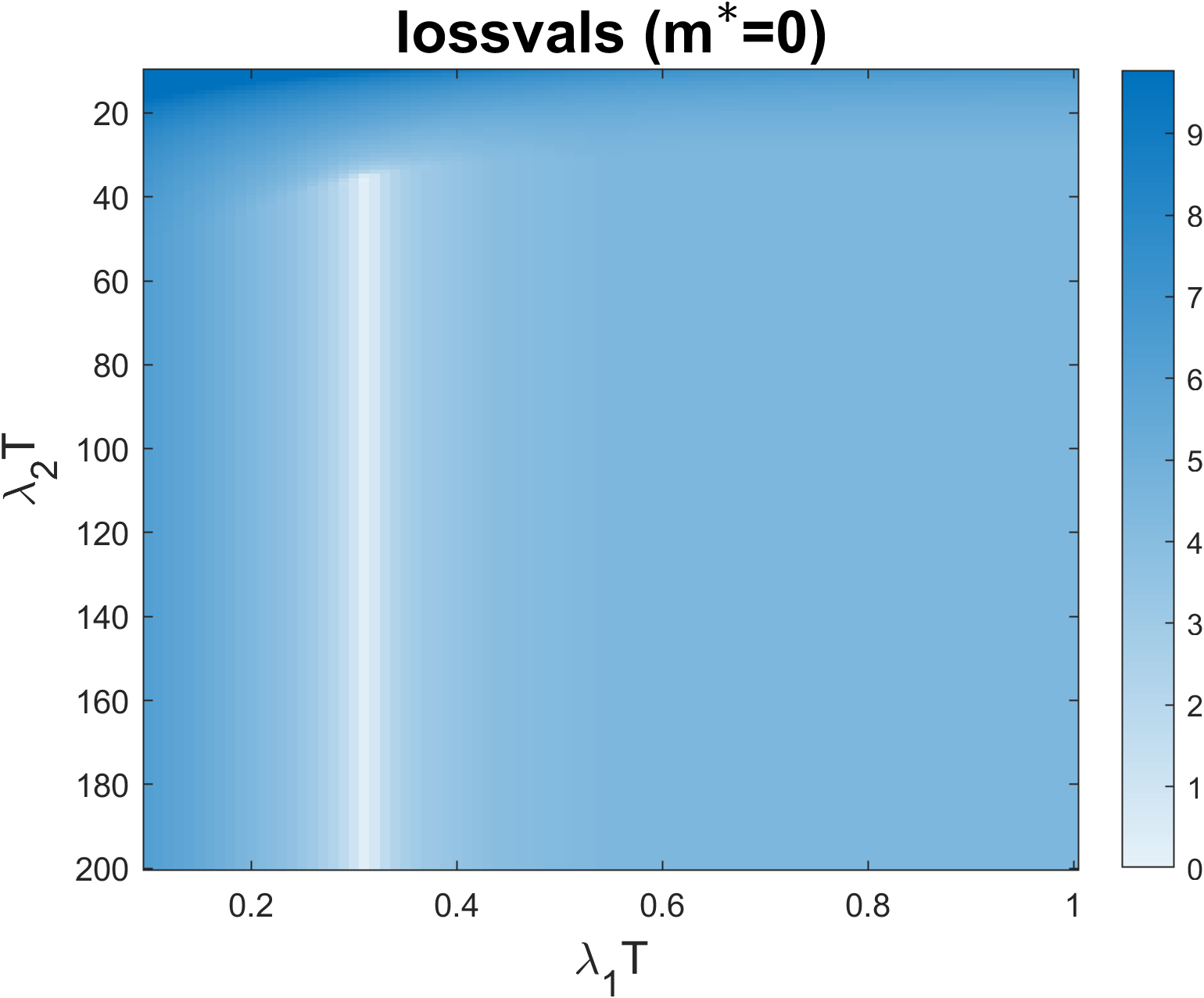}}\hfill
    \subfloat{\includegraphics[width=0.33\textwidth]{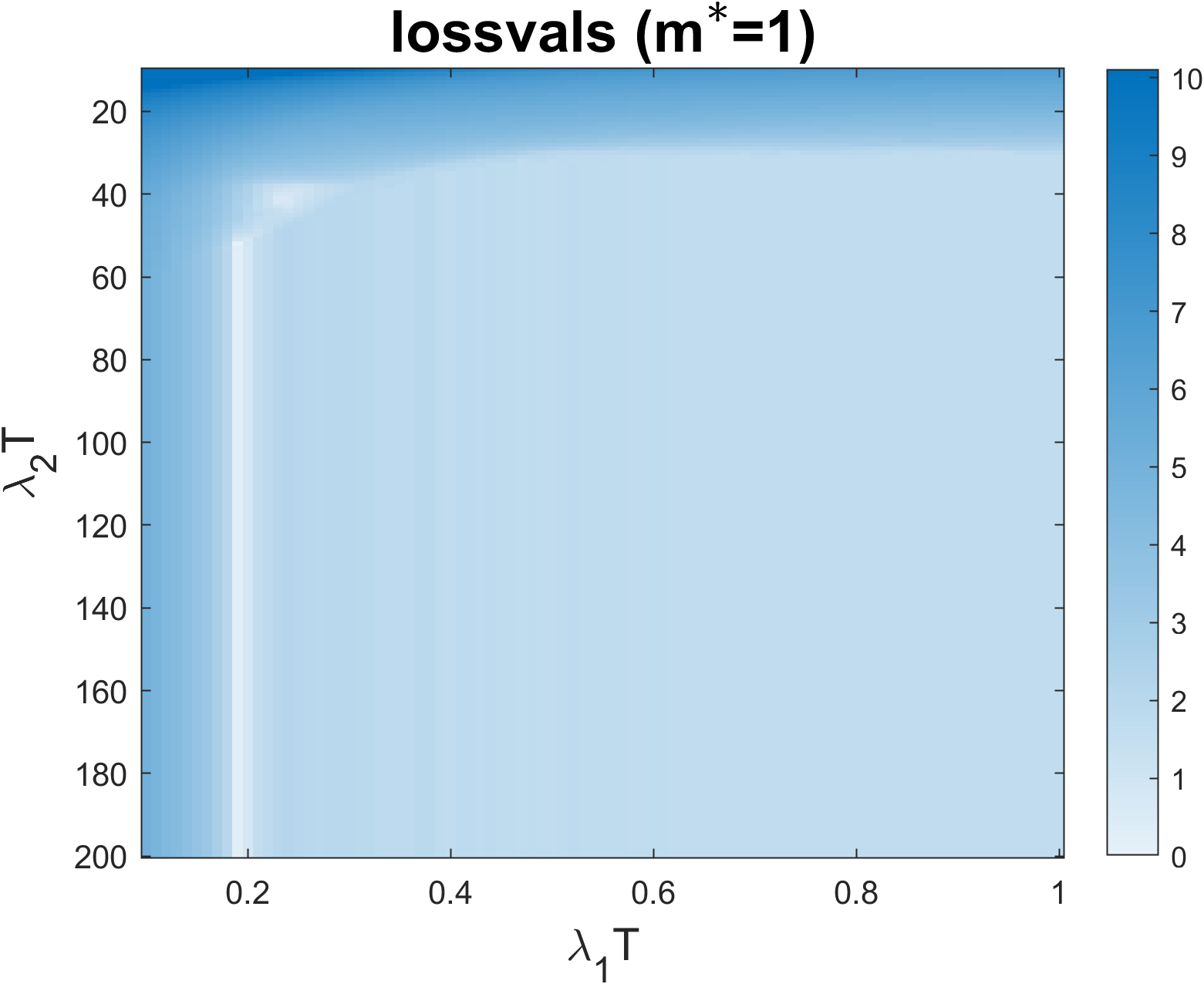}}\hfill
    \subfloat{\includegraphics[width=0.33\textwidth]{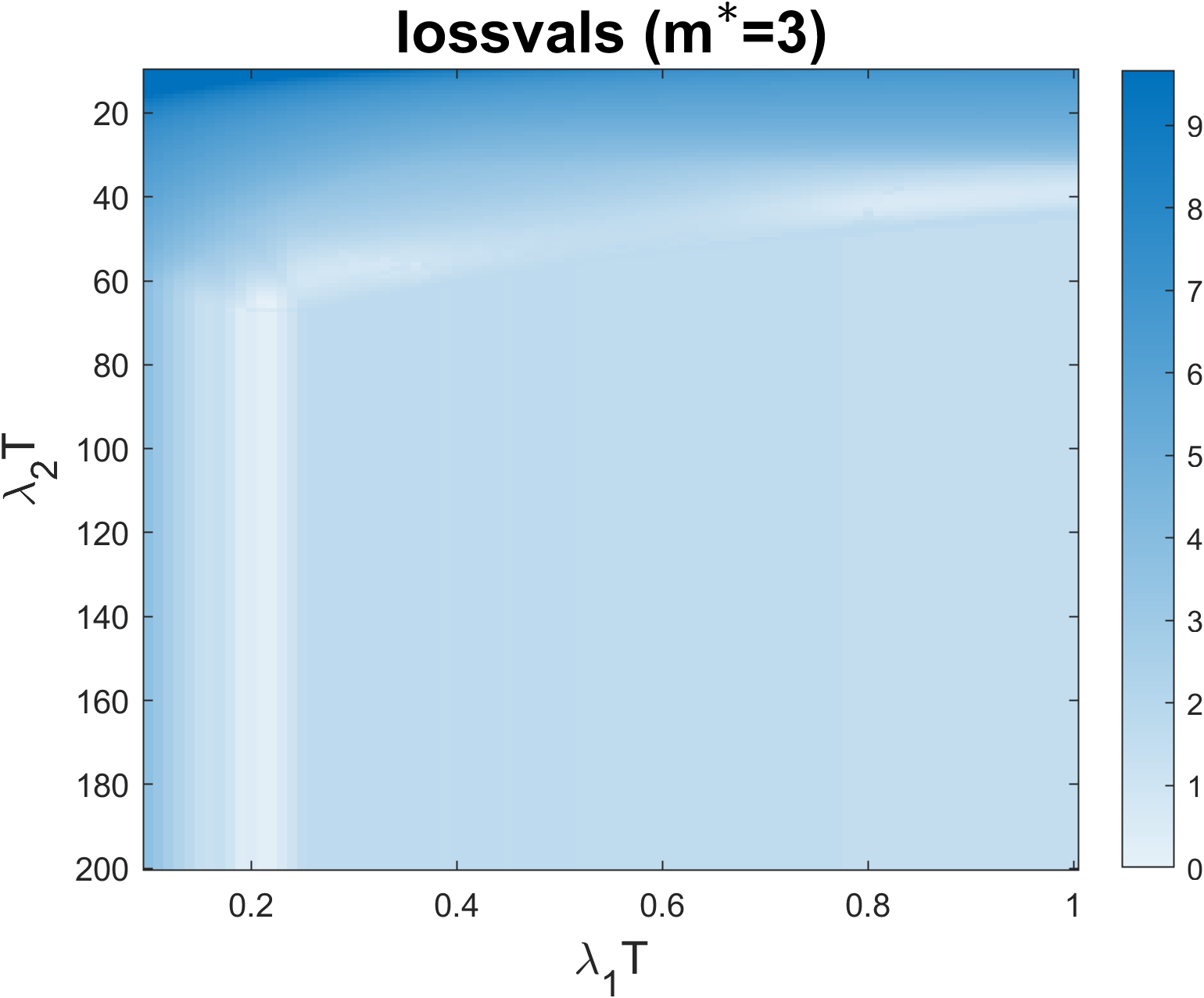}}\\

    \subfloat{\includegraphics[width=0.33\textwidth]{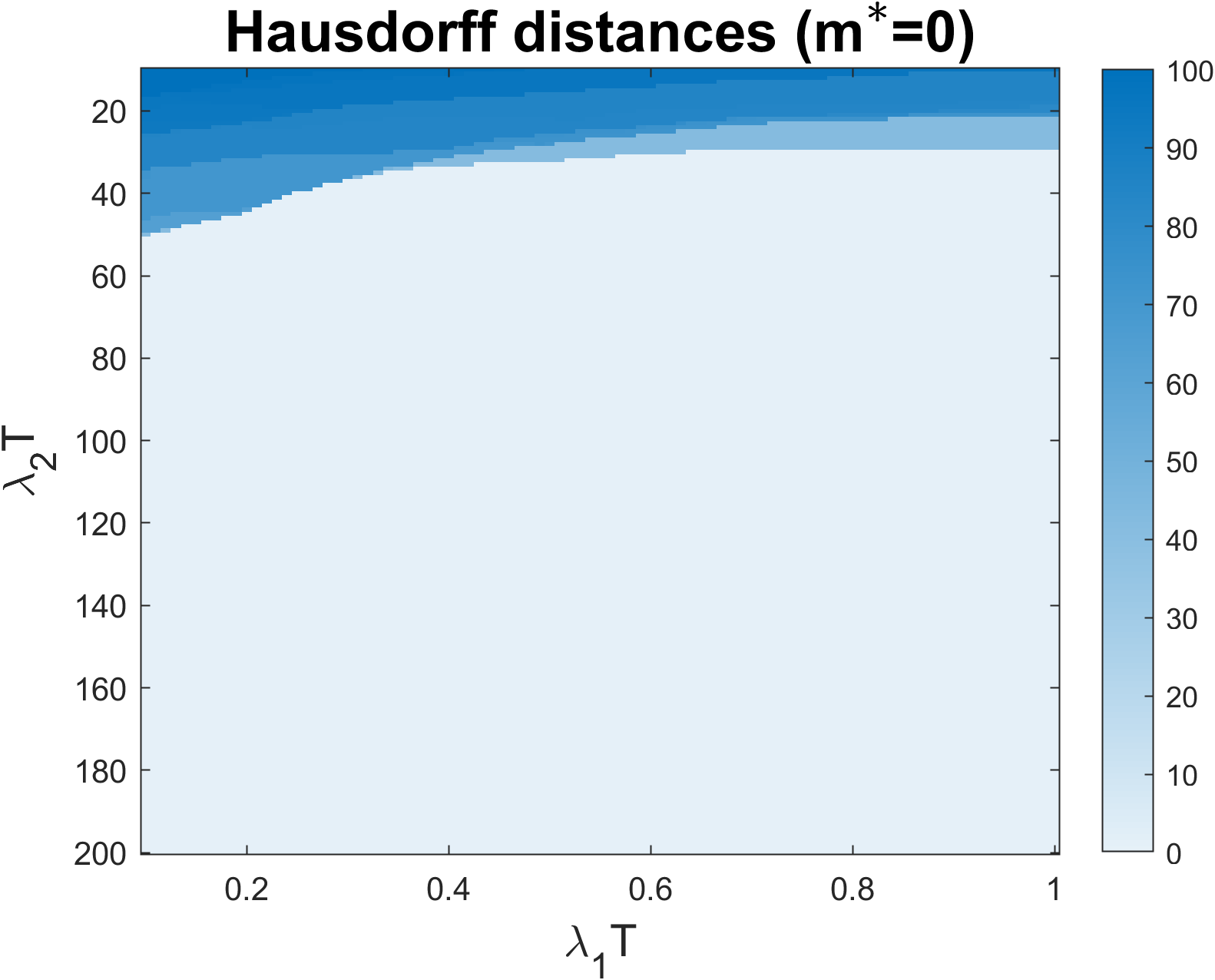}}\hfill
    \subfloat{\includegraphics[width=0.33\textwidth]{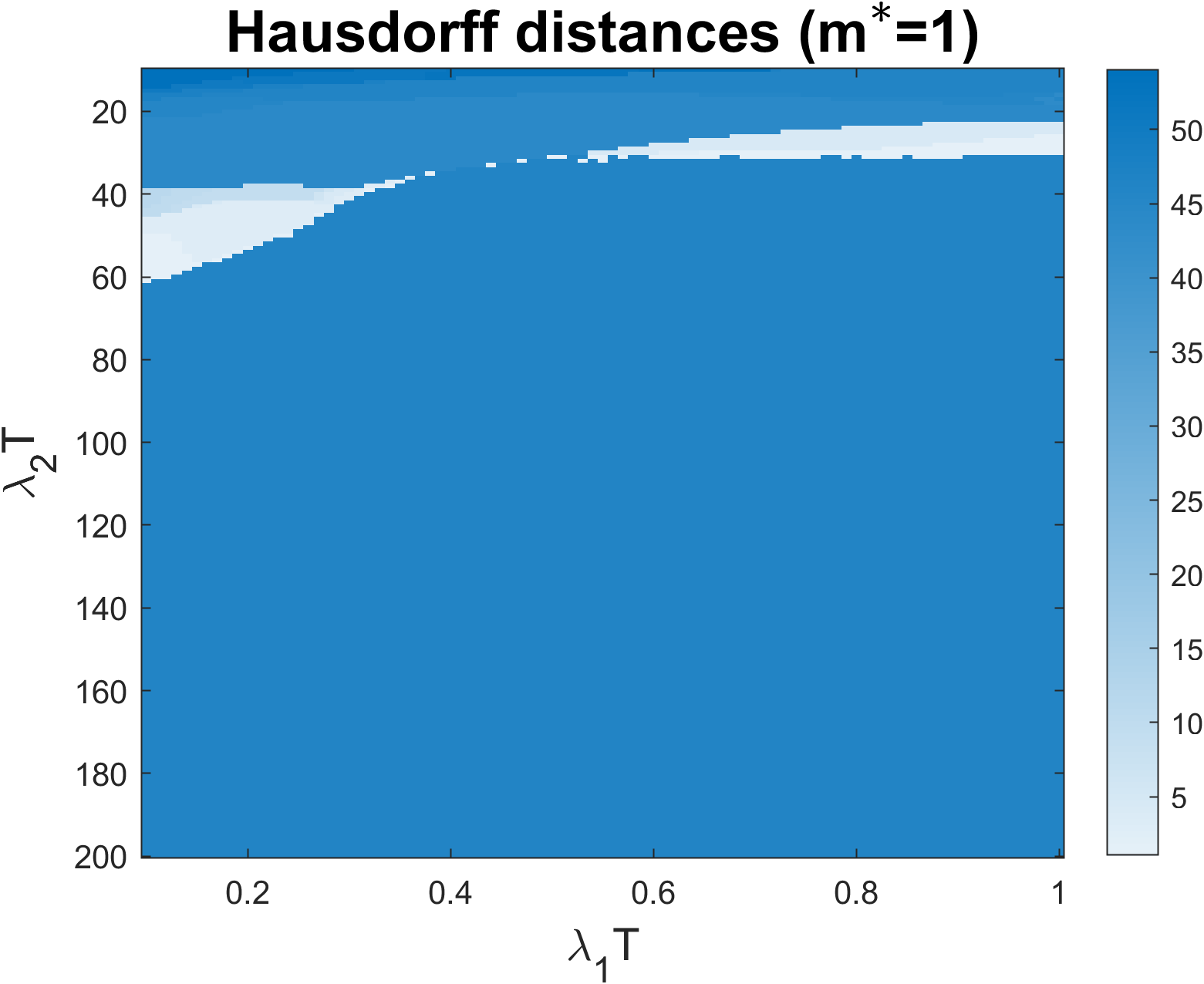}}\hfill
    \subfloat{\includegraphics[width=0.33\textwidth]{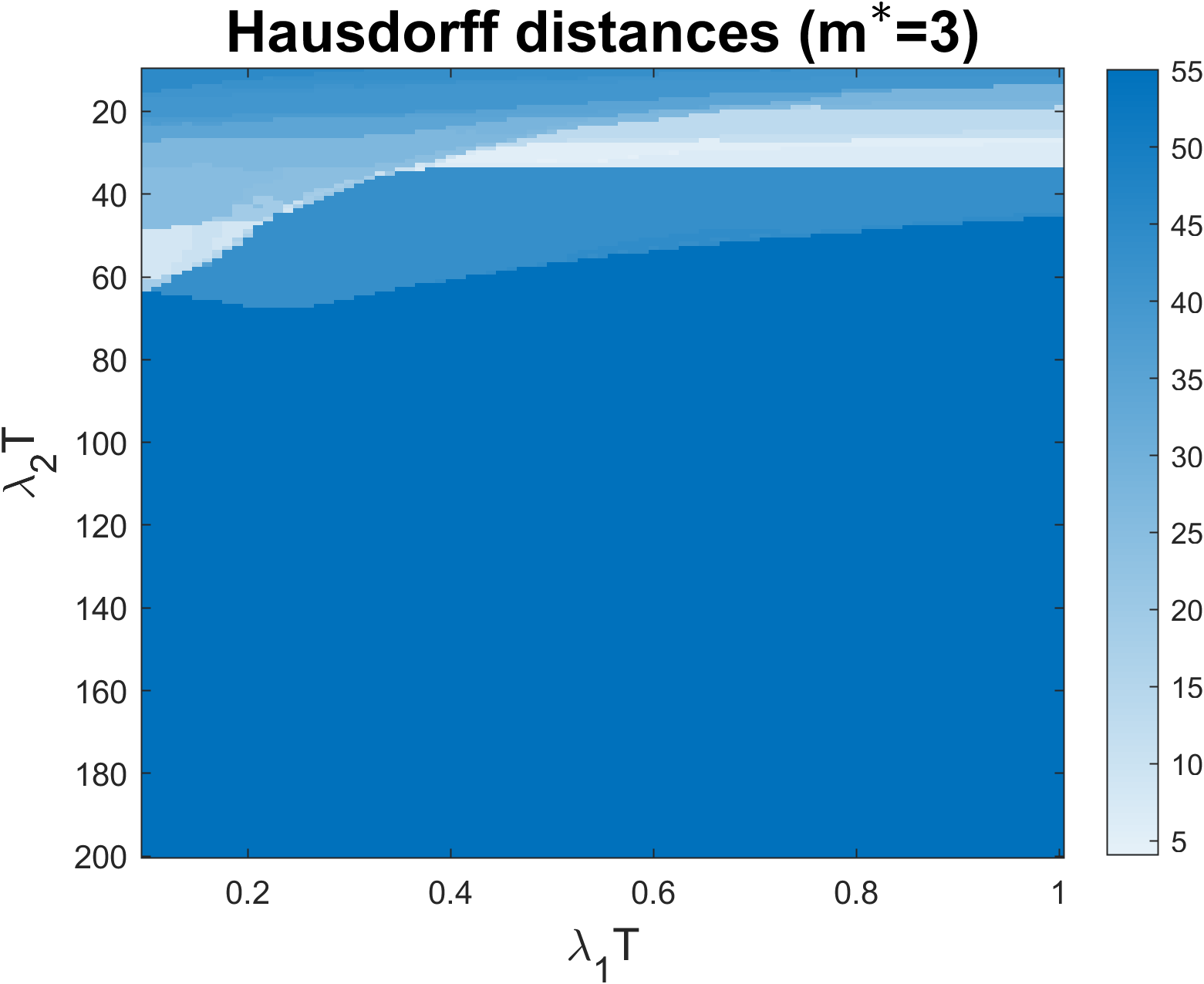}}\\

    \subfloat{\includegraphics[width=0.33\textwidth]{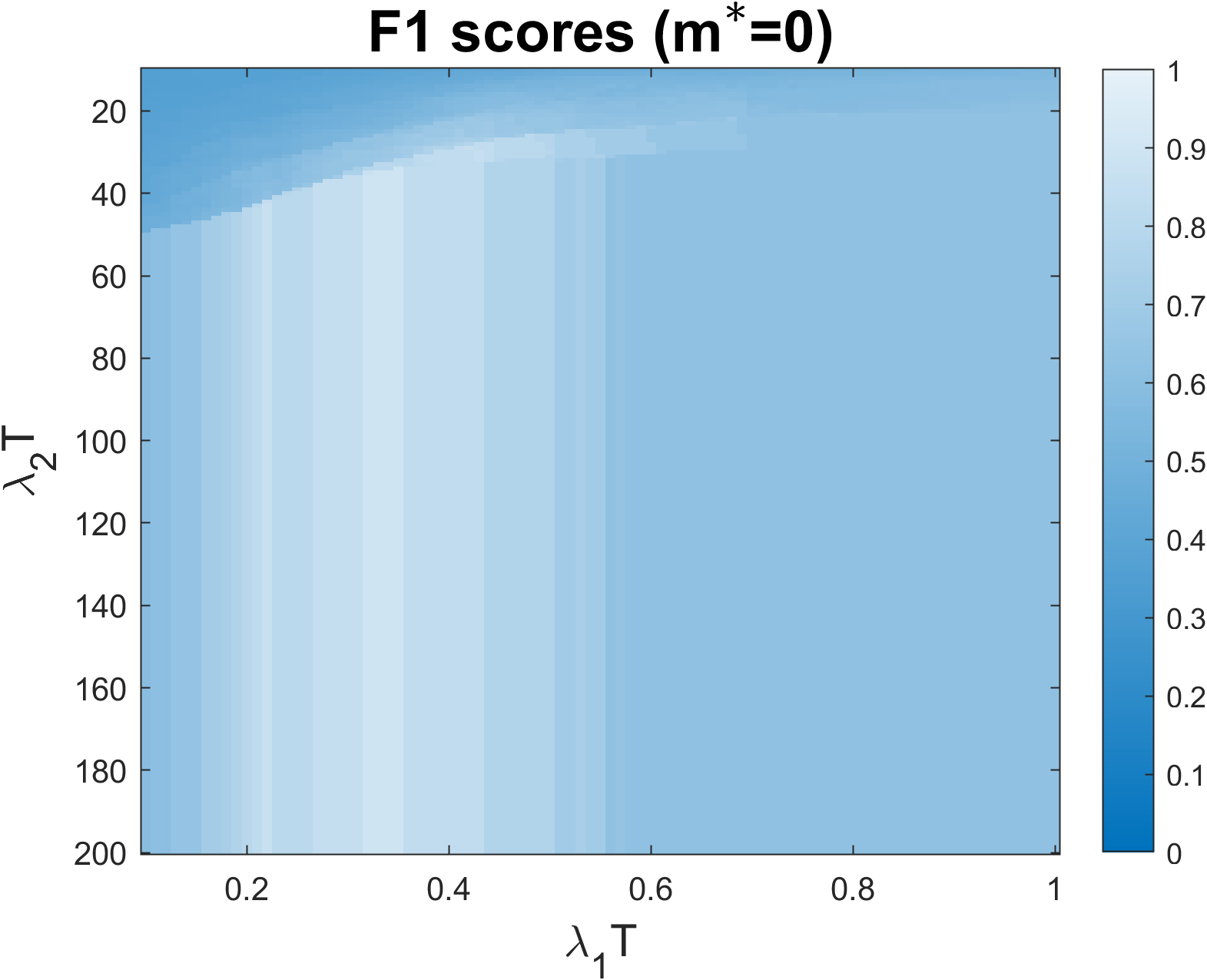}}\hfill
    \subfloat{\includegraphics[width=0.33\textwidth]{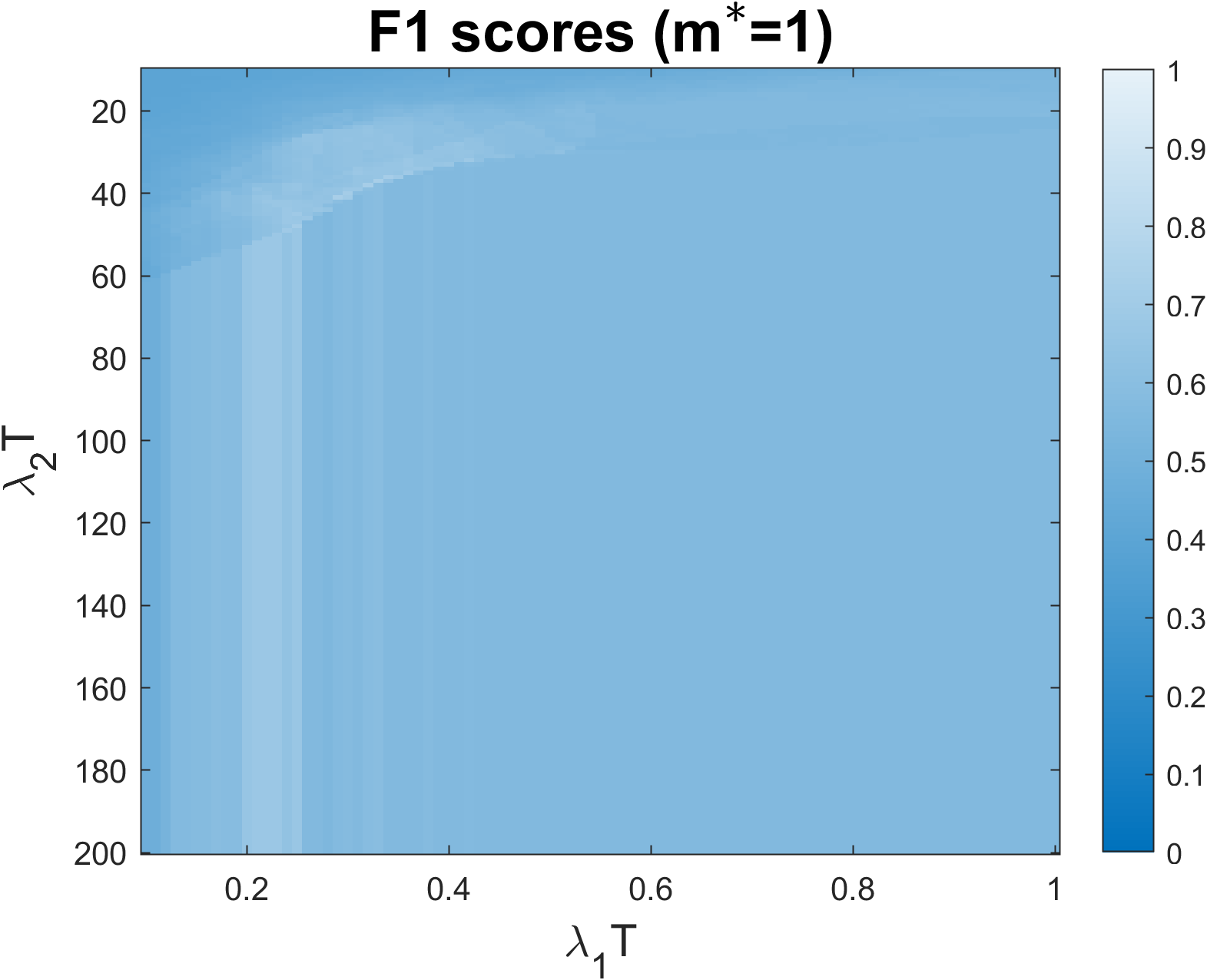}}\hfill
    \subfloat{\includegraphics[width=0.33\textwidth]{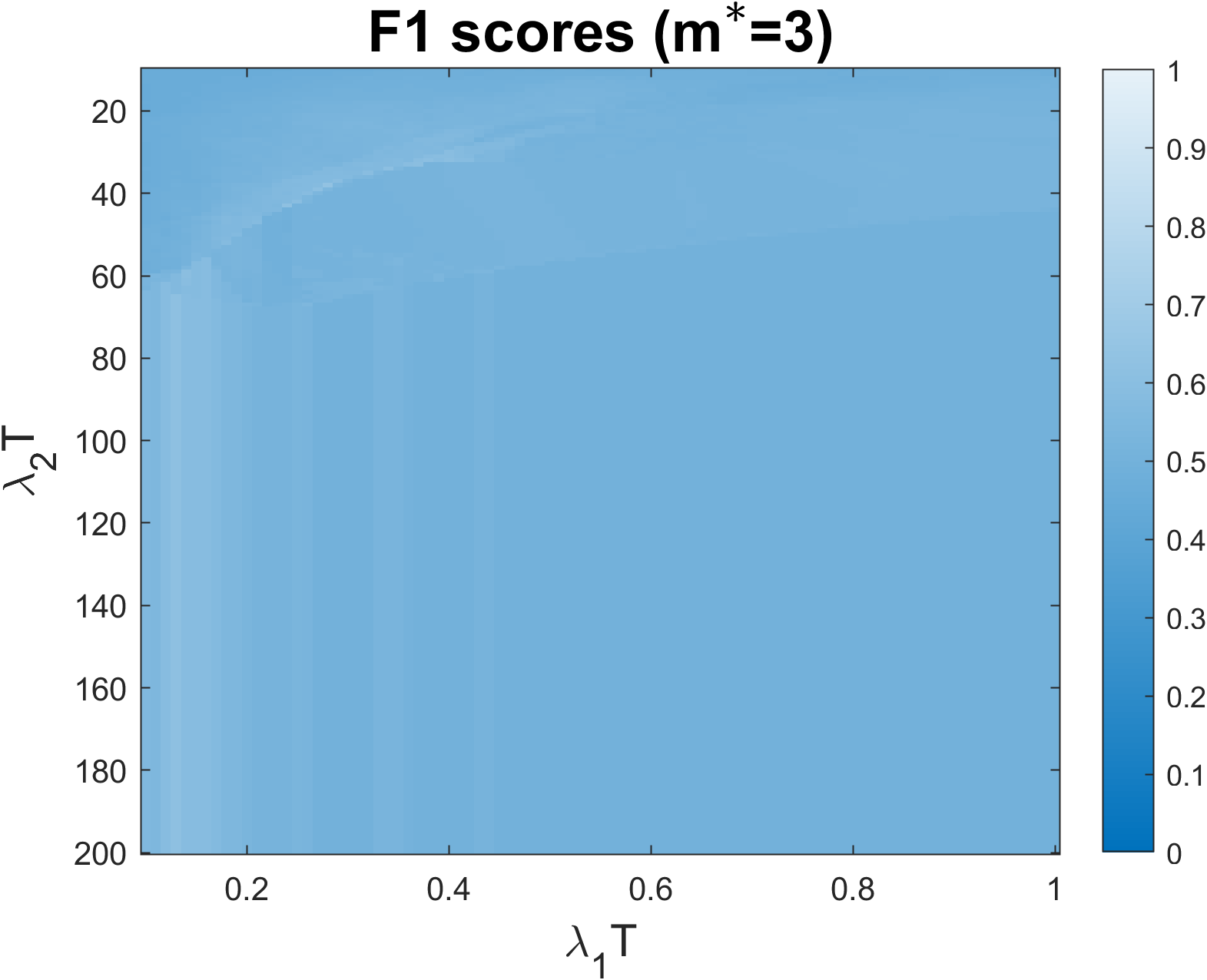}}
    \label{fig:sen-ana}
\end{figure}

The figures suggest consistent patterns across the surfaces of all four metrics.
Specifically, there are distinct boundaries splitting the metric surfaces into two primary regions: an upper and a lower region.
Each of these regions further subdivides into multiple subregions that exhibit similar characteristics across all four metrics.
The lower regions of the BIC, lossval, and F1 score surfaces are characterized by numerous vertical bars, indicating areas of potentially optimal parameter combinations.
In contrast, the Hausdorff distance surface displays a constant lower region.

An interesting observation is that the BIC-type criterion tends to favor smaller values of \( \lambda_1 \), whereas the lossval criterion leans towards slightly larger \( \lambda_1 \) values.
Furthermore, although both criteria struggle to identify the optimal Hausdorff distance (except when \( m^{\ast} = 0 \)), when comparing the tuning parameters selected by the BIC-type criterion to those by the lossval criterion, the lossval criterion exhibits a slight advantage; its optimal region (the white area in the subfigures) is larger, especially as $m^{\ast}$ increases.
For example, when \( m^{\ast} = 3 \), the white region extends to the upper right corner, which is preferable to the lower region.
This suggests that the lossval criterion may be more effective in identifying the optimal parameters.
These advantages of the lossval criterion over the BIC-type criterion are further corroborated by the results presented in the previous subsection.

The figures also shed light on the reasons behind the BIC-type criterion's poor performance in the context of GFDtL.
Specifically, the BIC-type criterion shows a preference for larger $ \lambda_2 $ values, which corresponds to fewer \rev{change points} in the estimation.
This preference can be directly attributed to the definition of BIC (cf. (\ref{eq:def-BIC})).
In particular, when there are \rev{change points}, $ K\log (T) $ is at least $ \log (T)p(p - 1) $, which typically dominates the loss value term in the BIC formula.
Consequently, the BIC-type criterion tends to favor estimators with fewer \rev{change points}, leading to suboptimal performance in detecting the true number of \rev{change points}, especially in scenarios with a higher number of actual \rev{change points}.

Furthermore, the F1 score surfaces provide additional insights into the model's performance across different parameter combinations.
The gradual transition from darker to lighter colors as $ \lambda_1 $ increases (for fixed $ \lambda_2 $) suggests that the model's ability to correctly identify true positives improves with larger $ \lambda_1 $ values, up to a certain point.
This observation aligns with the lossval criterion's preference for slightly larger $ \lambda_1 $ values compared to the BIC-type criterion.

\subsection{Sensitivity Analysis with Respect to $\lambda_3$}

\rev{We conduct a sensitivity analysis to evaluate the influence of the parameter $\lambda_3$ on the performance of the GFDtL algorithm.
Recall from Proposition~7 that, as a technical parameter to ensure the existence of optimal solution to the modified problem (5) and the equivalence between it and the original problem (2), $\lambda_3$ should be chosen sufficiently large.
The experiments are conducted on datasets simulated according to \textbf{Setting (ii)} in Section~5 of the main text.
We fix $\lambda_1T = 0.5$ and $\lambda_2T = 50$, and vary $\lambda_3$ from 1 to 101 with a step size of 10.
Four experimental configurations are considered: (i) $T = 100$, $p = 10$, $m^\ast = 1$; (ii) $T = 150$, $p = 10$, $m^\ast = 1$; (iii) $T = 100$, $p = 20$, $m^\ast = 1$; and (iv) $T = 150$, $p = 10$, $m^\ast = 3$.
For each configuration, we perform 5 independent experiments and report the mean and standard deviation of four performance metrics: Hausdorff distance (HD), F1 score, accuracy, and estimation error.}

\rev{The results are displayed in Figure \ref{fig:lamb3-sensitivity}.
For all four configurations, the performance metrics remain remarkably stable as $\lambda_3$ varies within the tested range, except for \( \lambda_3 = 1 \).
These observations indicate that once $\lambda_3$ is chosen to be sufficiently large (e.g., $\lambda_3 = 10$ in our synthetic experiments), its specific value does not significantly impact the algorithm's performance.
This robustness is consistent with the theoretical result in Proposition~7, which guarantees that when (2) has an optimal solution, then for $\lambda_3$ larger than a certain threshold, any optimal solution of (5) is optimal for (2).}

\rev{In conclusion, the choice of $\lambda_3$ does not have a substantial influence on the performance of the GFDtL procedure, provided it is set to a sufficiently large value.
This finding supports our recommendation in Section \ref{sec:implementation} to use $\lambda_3 = 10$ for synthetic datasets and $\lambda_3 = 50$ for real datasets.}

\begin{figure}[htbp]
    \centering
    \caption{\rev{Sensitivity analysis of $\lambda_3$: Performance metrics across different configurations.}}
    \subfloat{\includegraphics[width=0.24\textwidth]{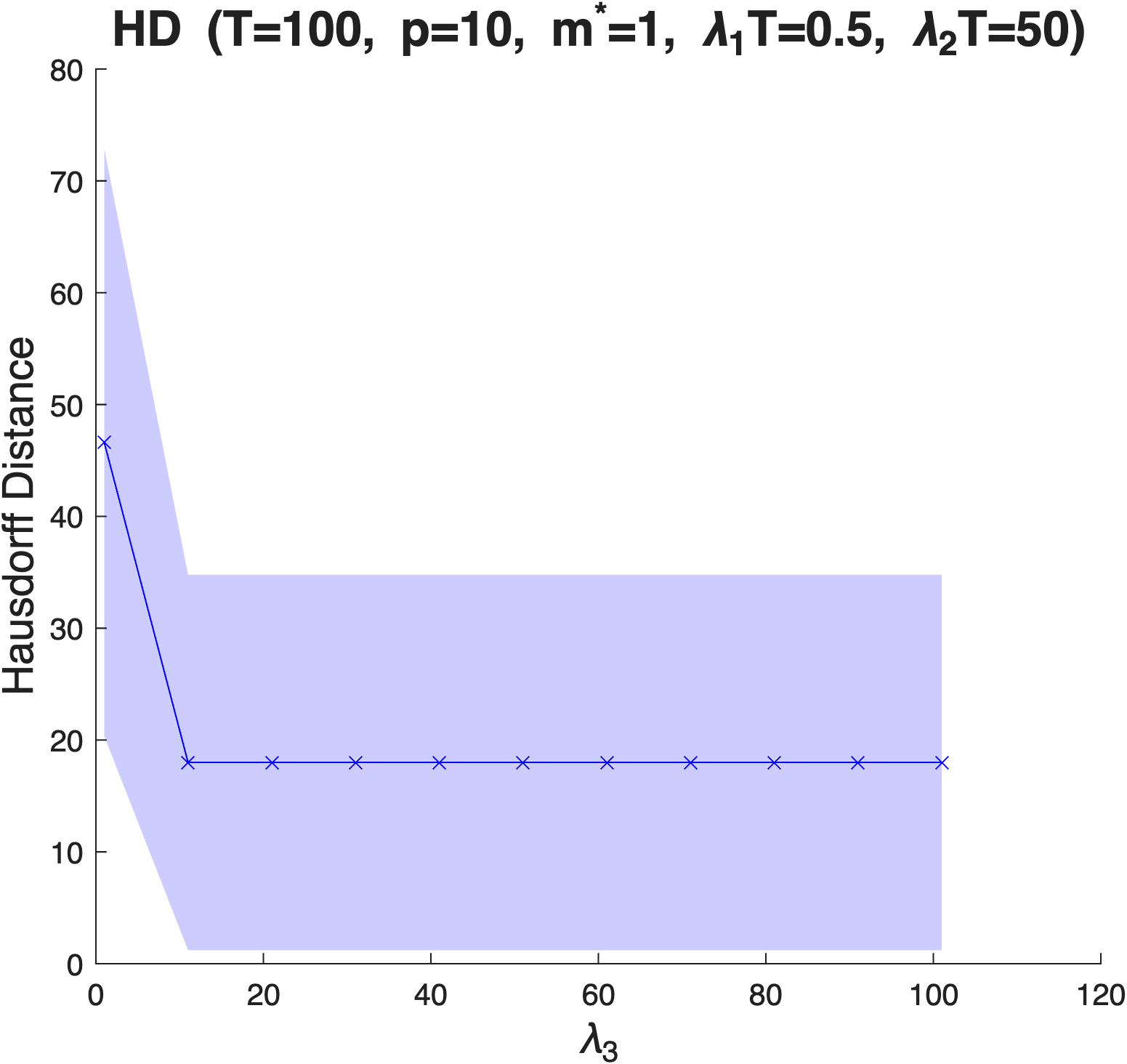}}
    \subfloat{\includegraphics[width=0.24\textwidth]{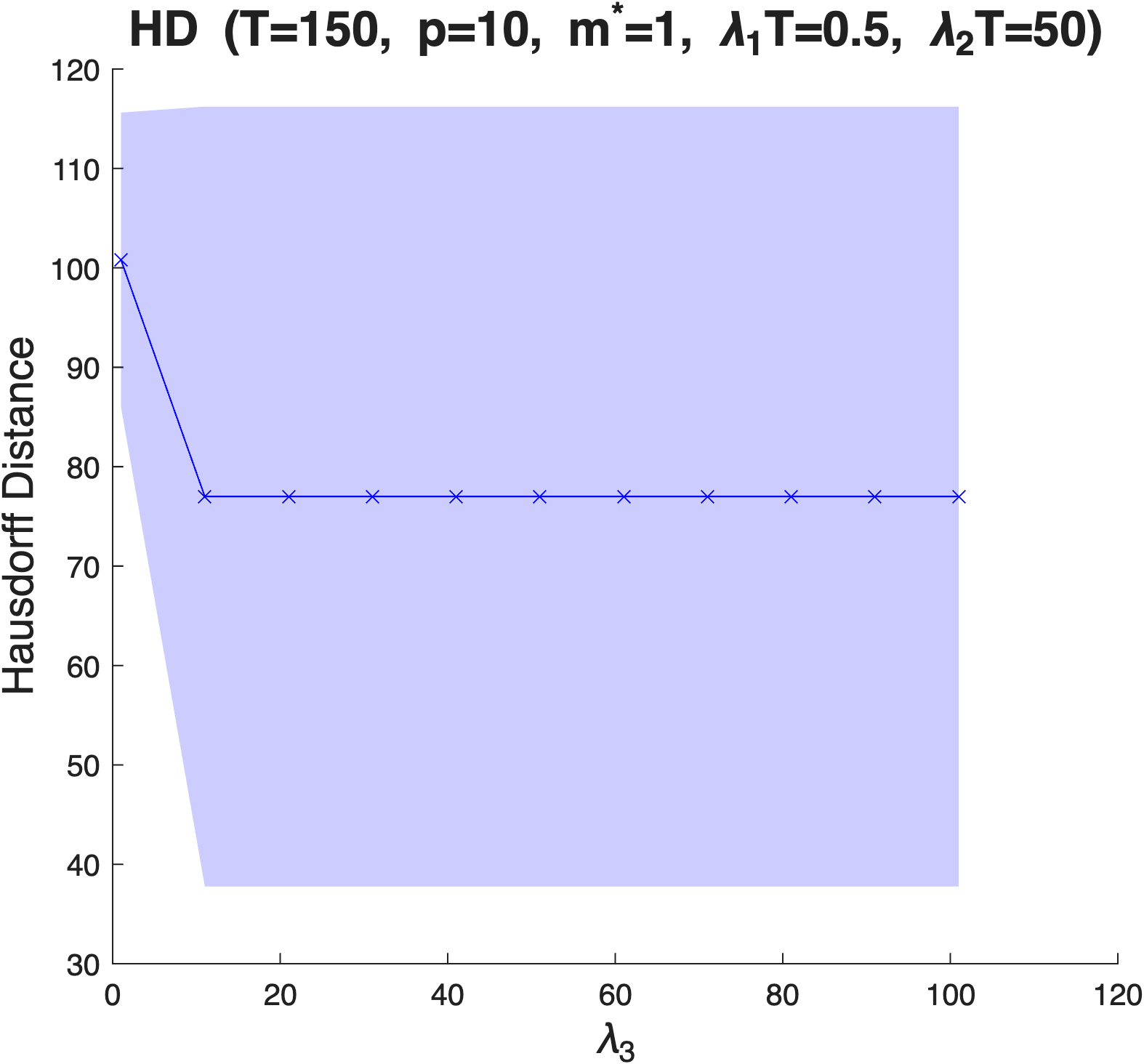}}
    \subfloat{\includegraphics[width=0.24\textwidth]{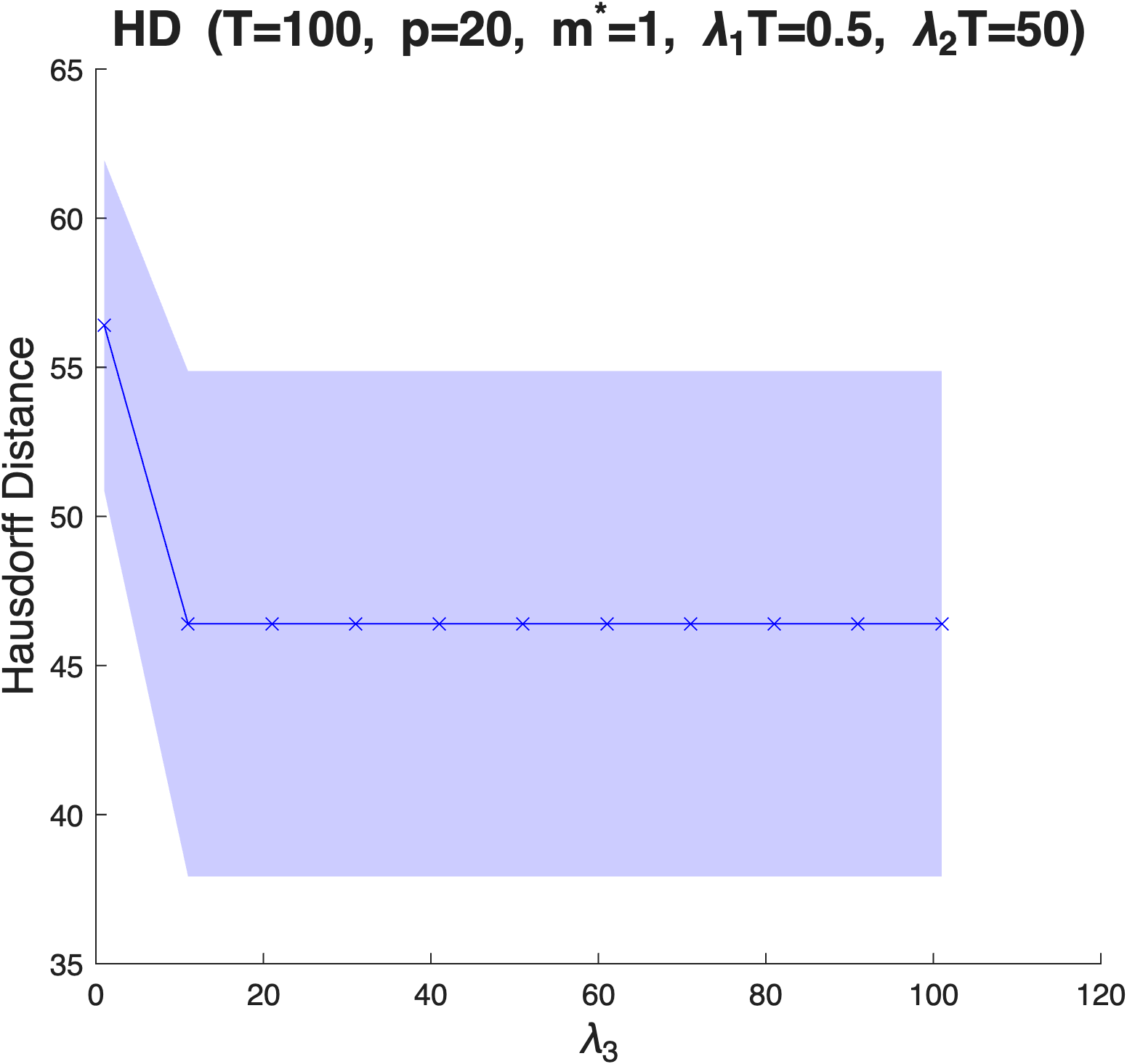}}
    \subfloat{\includegraphics[width=0.24\textwidth]{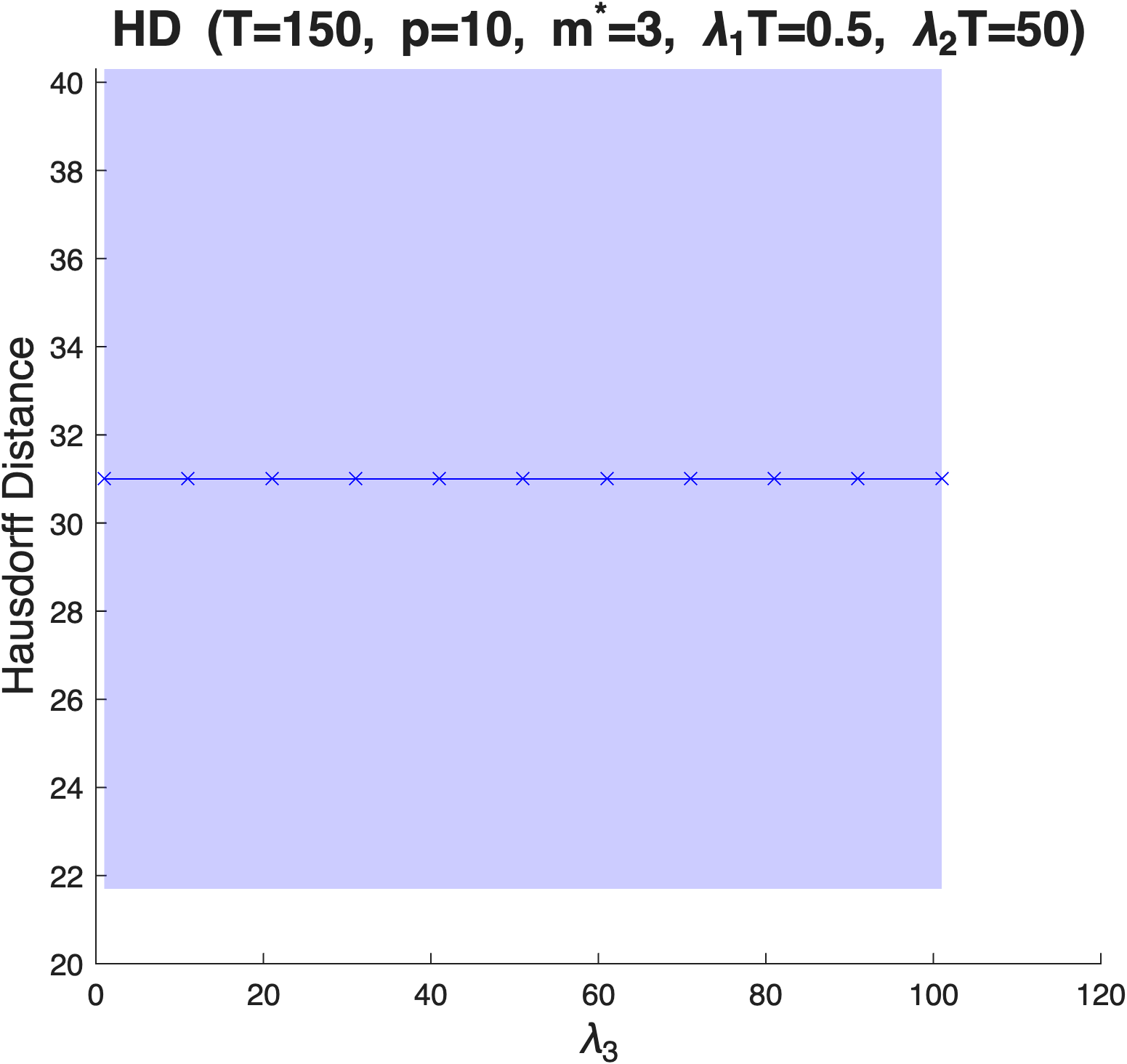}}\\
    \subfloat{\includegraphics[width=0.24\textwidth]{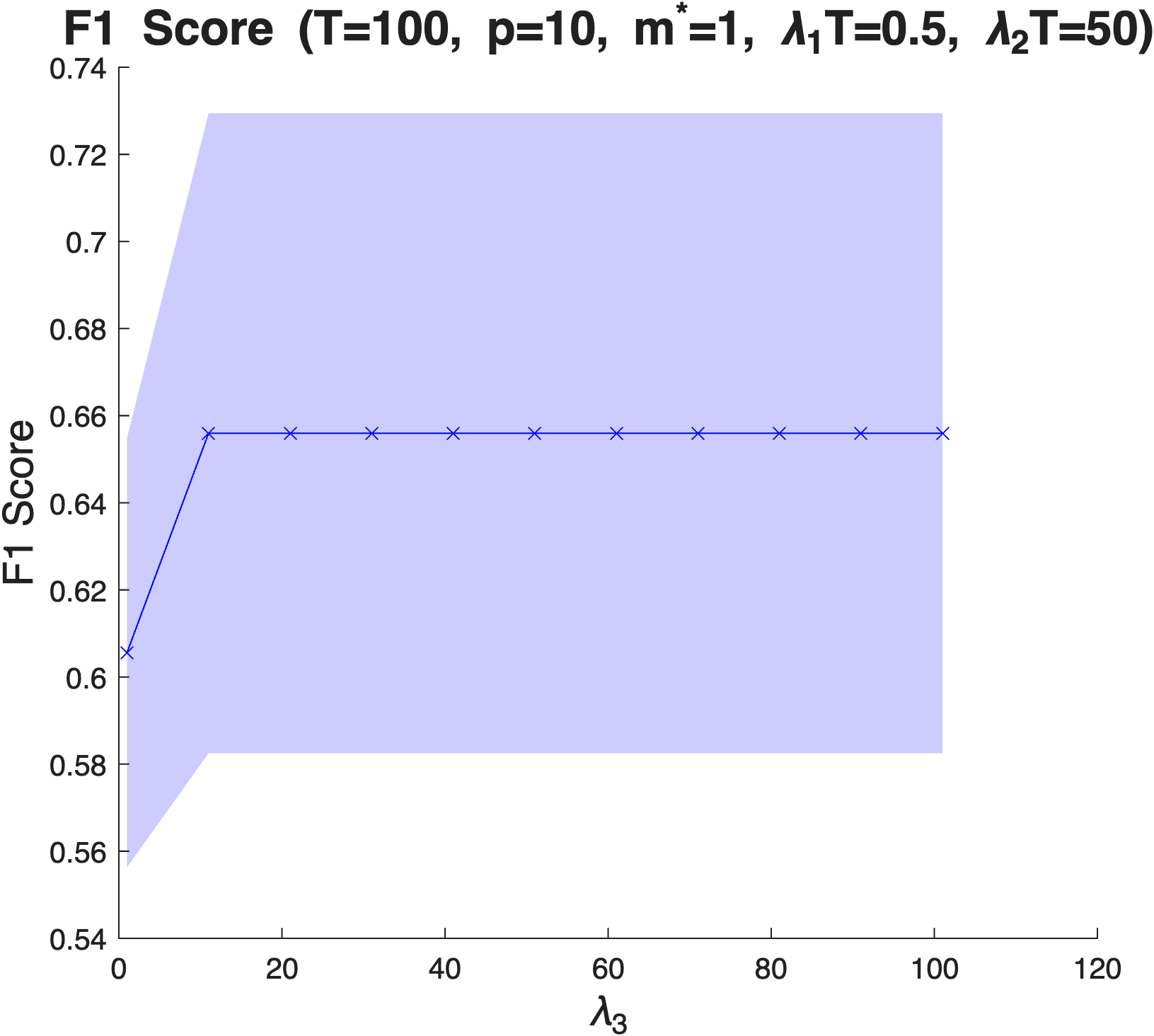}}
    \subfloat{\includegraphics[width=0.24\textwidth]{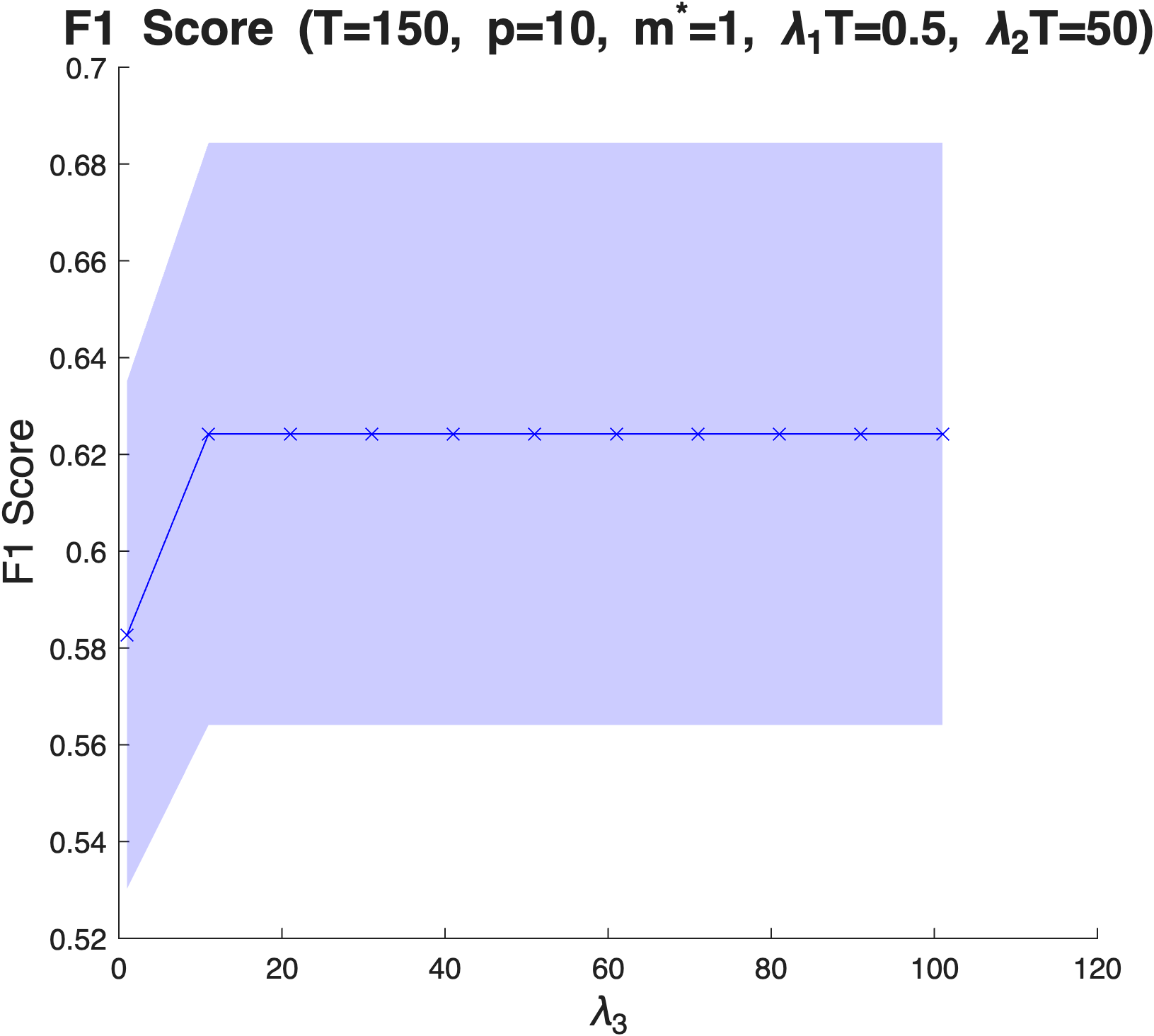}}
    \subfloat{\includegraphics[width=0.24\textwidth]{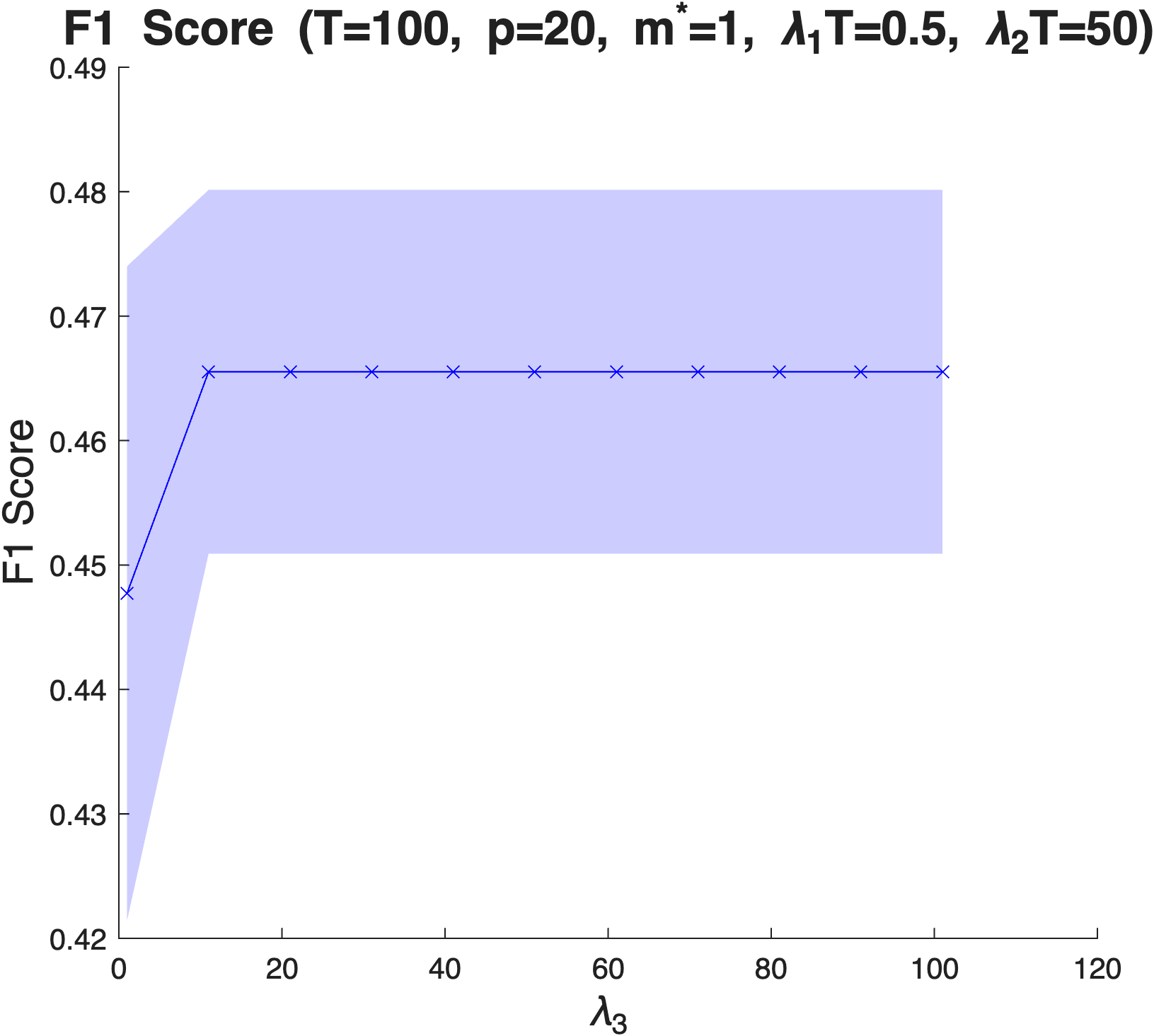}}
    \subfloat{\includegraphics[width=0.24\textwidth]{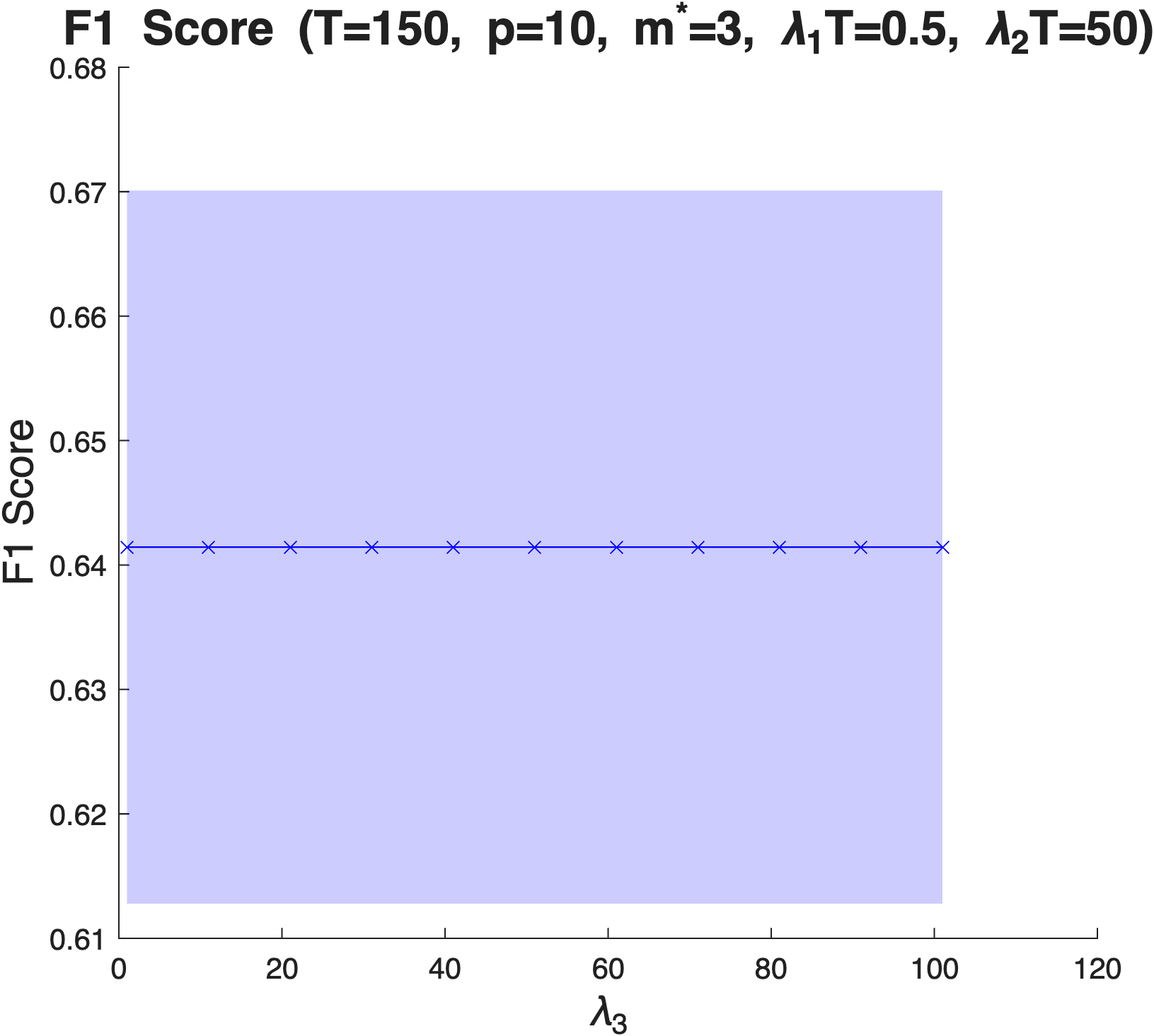}}\\
    \subfloat{\includegraphics[width=0.24\textwidth]{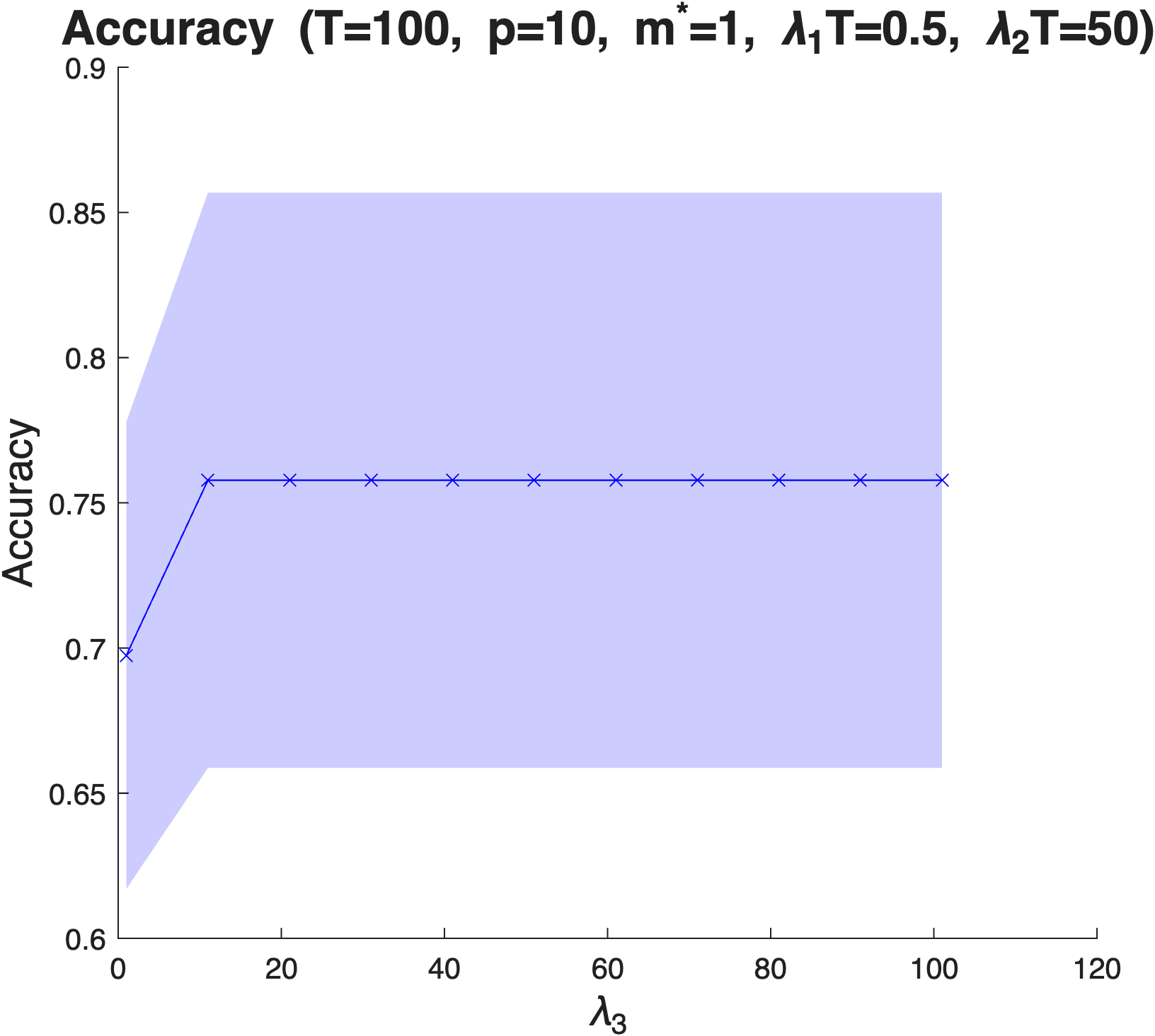}}
    \subfloat{\includegraphics[width=0.24\textwidth]{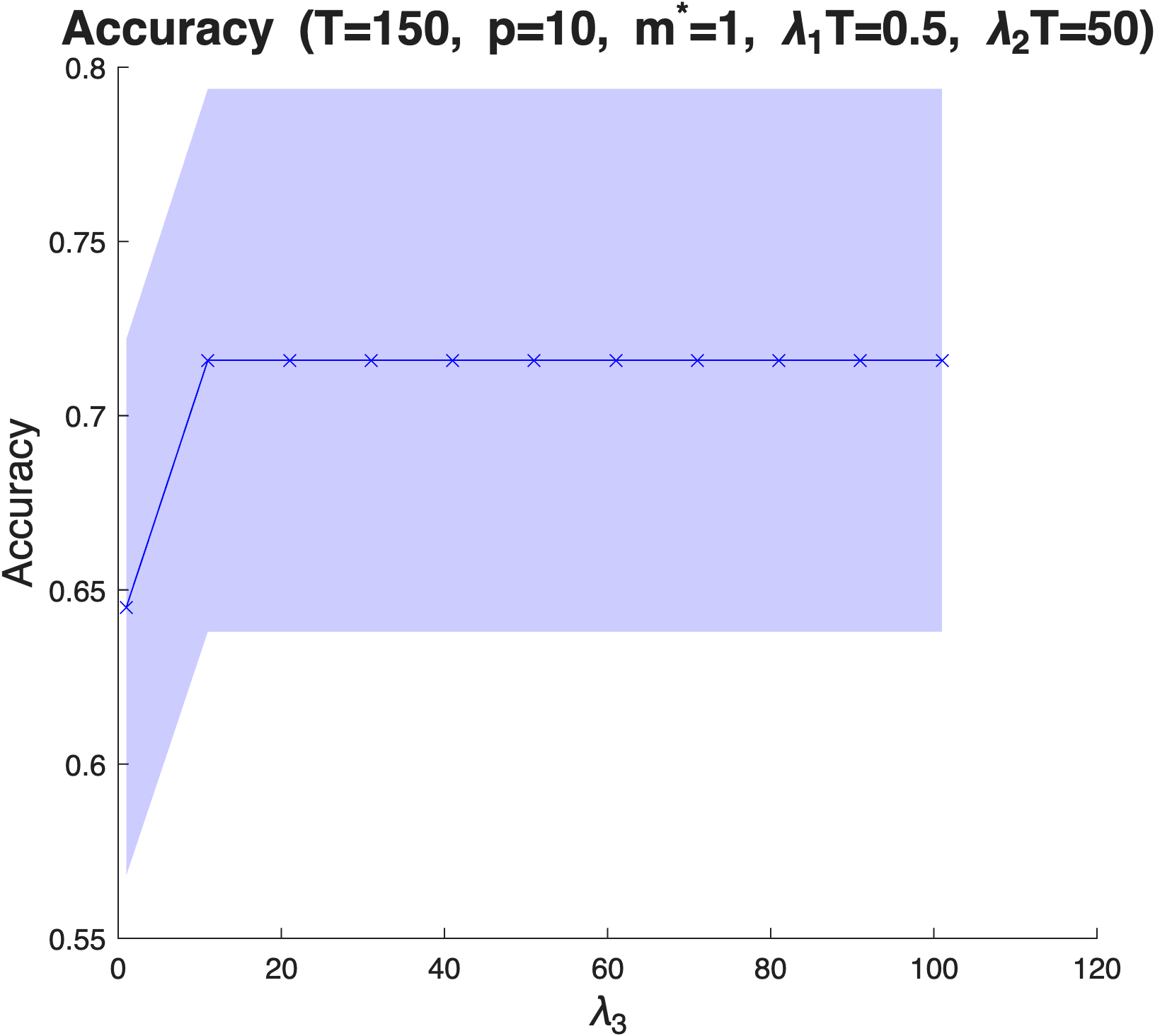}}
    \subfloat{\includegraphics[width=0.24\textwidth]{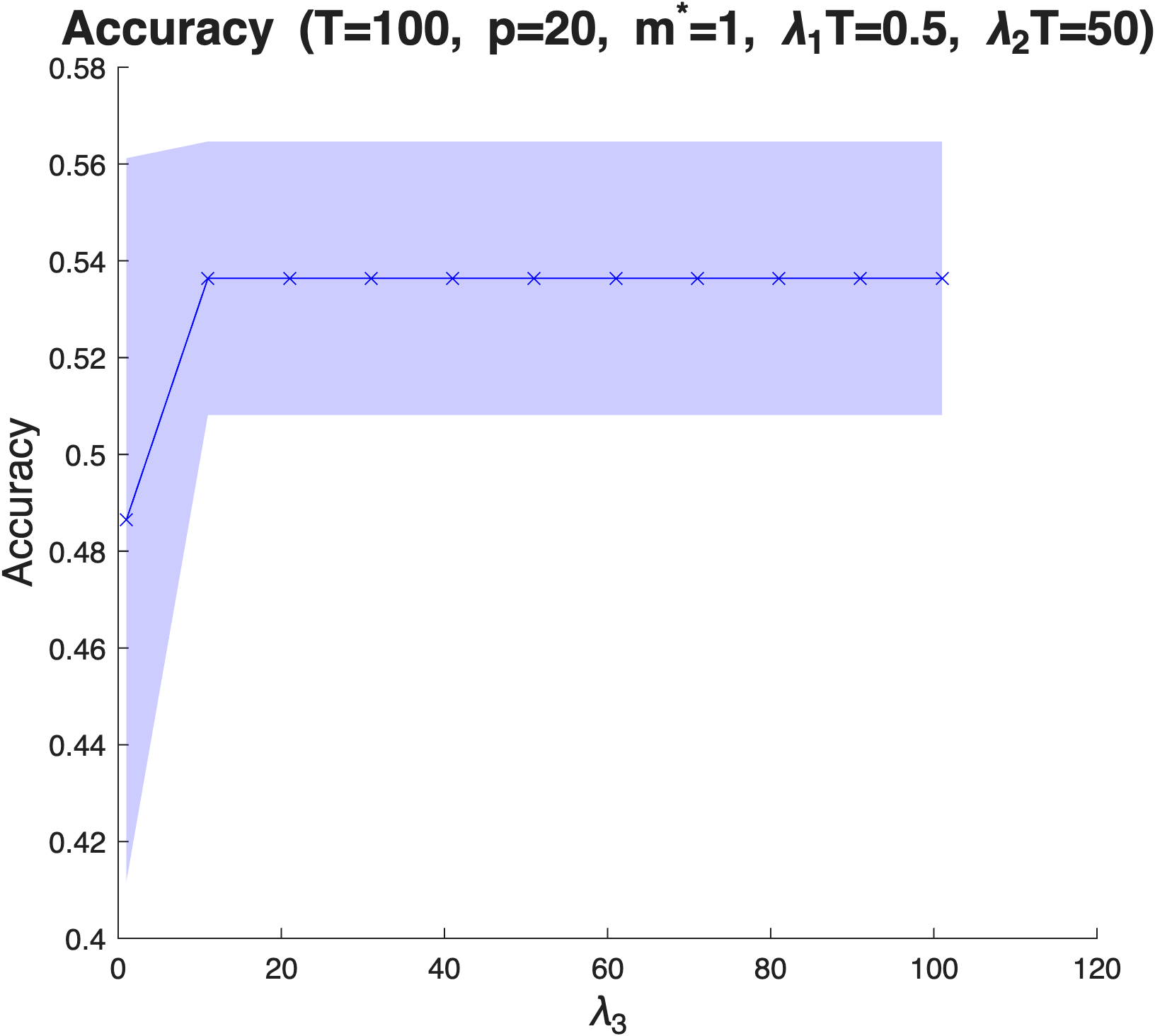}}
    \subfloat{\includegraphics[width=0.24\textwidth]{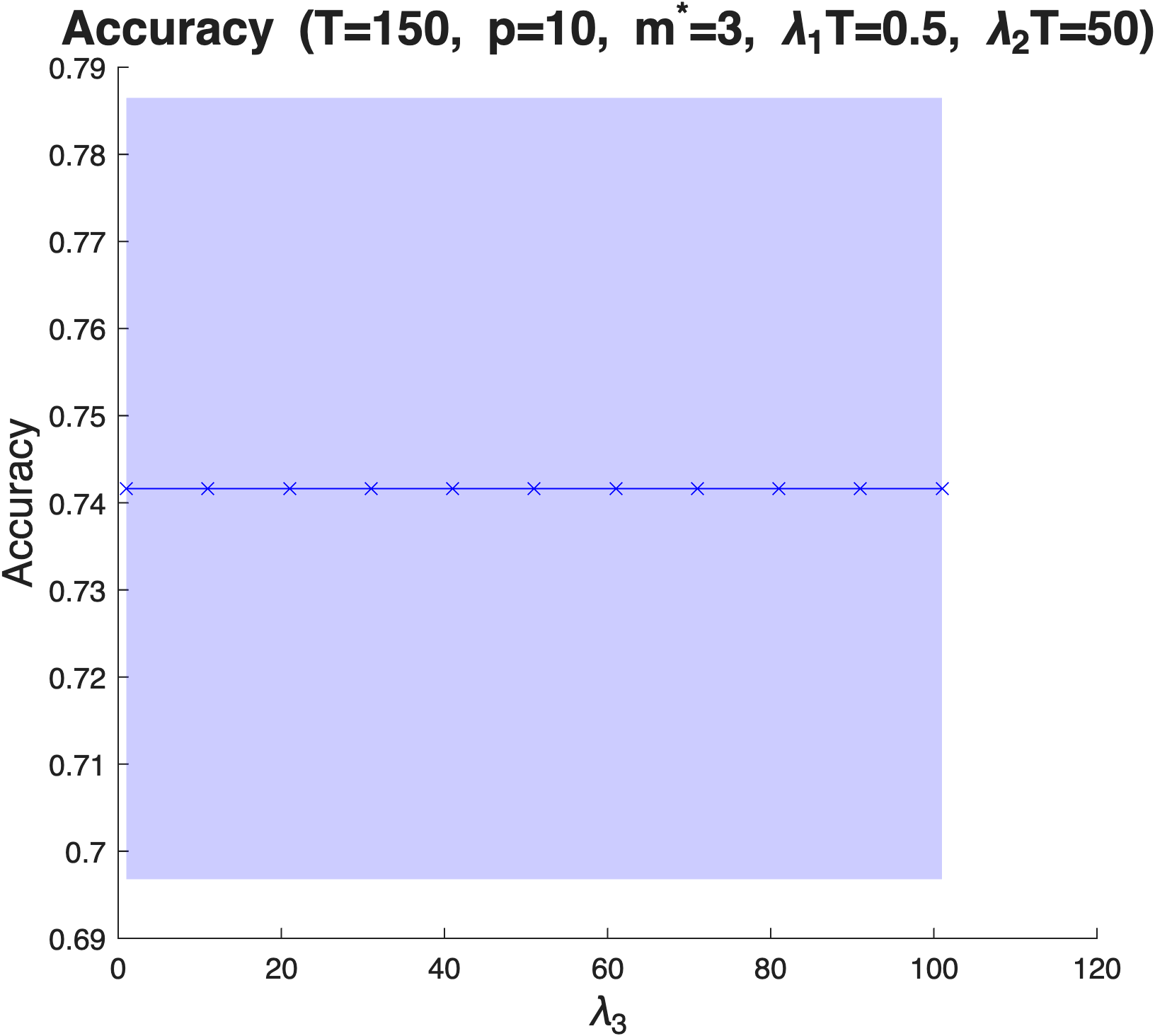}}\\
    \subfloat{\includegraphics[width=0.24\textwidth]{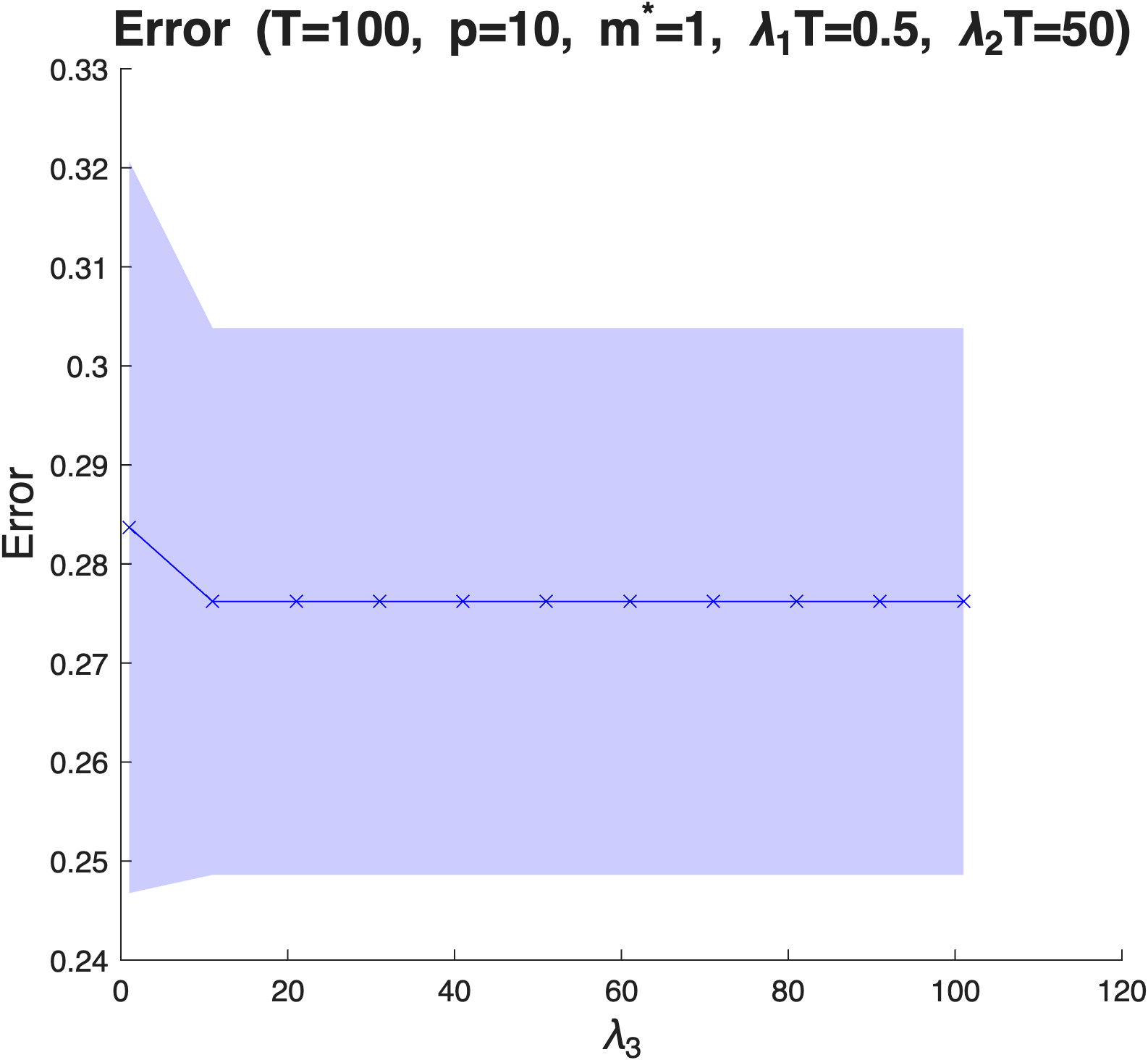}}
    \subfloat{\includegraphics[width=0.24\textwidth]{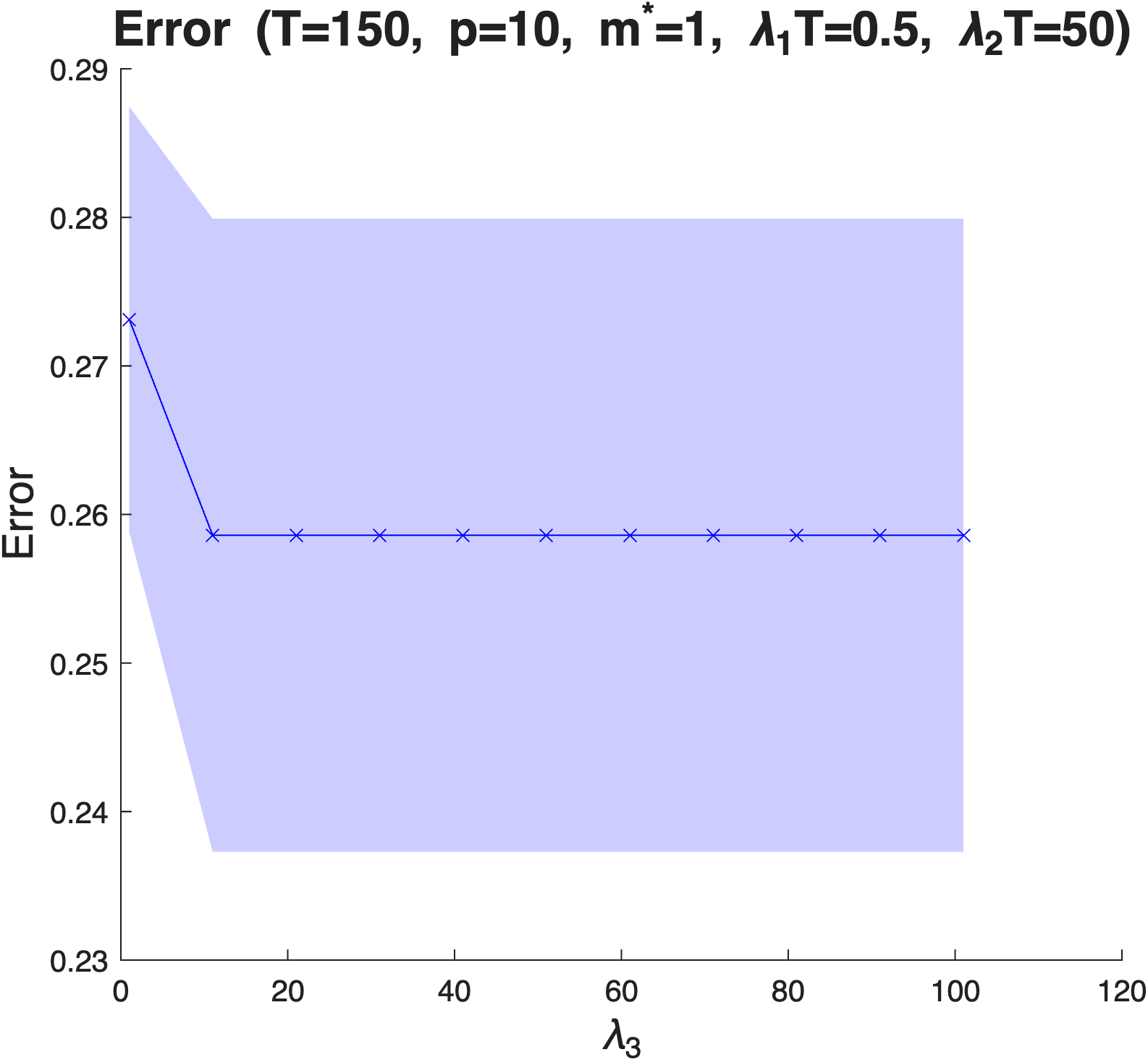}}
    \subfloat{\includegraphics[width=0.24\textwidth]{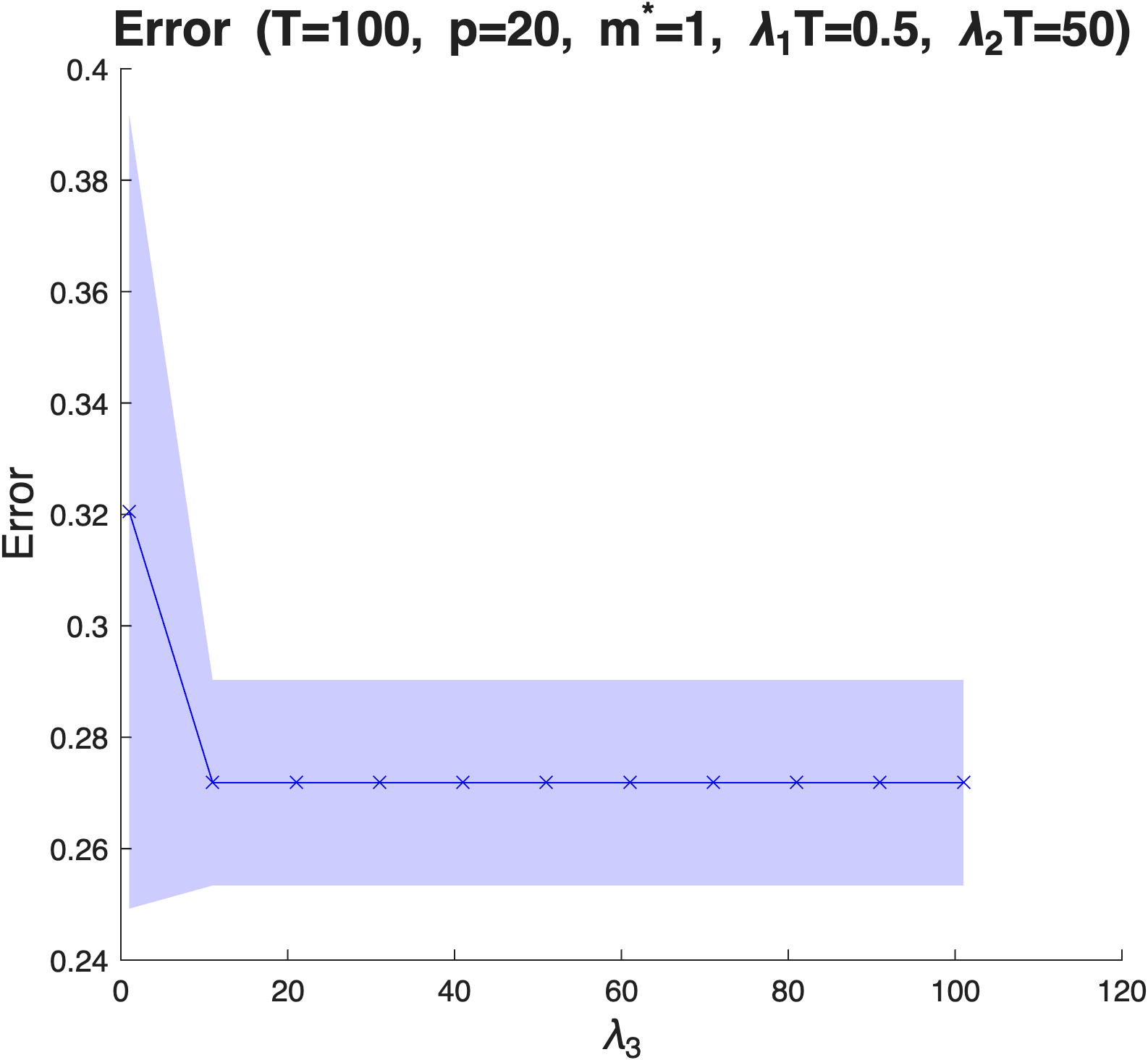}}
    \subfloat{\includegraphics[width=0.24\textwidth]{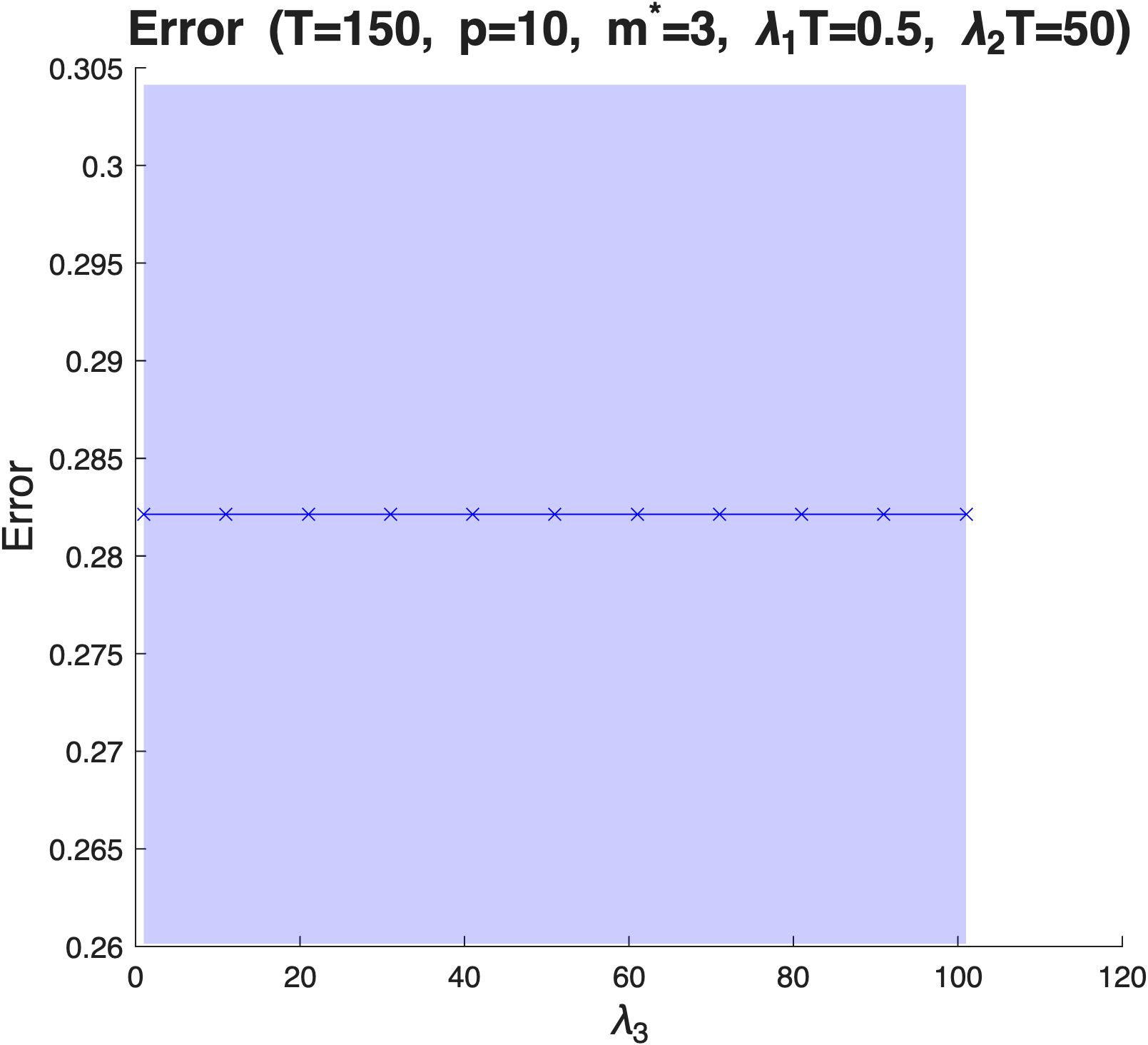}}
    \label{fig:lamb3-sensitivity}
\end{figure}

\subsection{Sensitivity Analysis with Respect to $\epsilon_{\mathrm{tol}}$}

\rev{We further investigate the sensitivity of the algorithm's performance to the choice of $\epsilon_{\mathrm{tol}}$, which is the threshold parameter used in the change point detection procedure (see Section \ref{sec:implementation}).
Specifically, a change point is detected at time $t$ if $\|\widehat{\Theta}_{t+1} - \widehat{\Theta}_t\|_F \geq \epsilon_{\mathrm{tol}}$.
The experiments are conducted on datasets simulated according to \textbf{Setting (ii)} in Section~5 of the main text.
We fix $\lambda_1T = 0.5$, $\lambda_2T = 50$, $\lambda_3 = 10$, and $\epsilon = 0.01$, and vary $\epsilon_{\mathrm{tol}}$ from $10^{-7}$ to $10^{-4}$ with a step size of $5 \times 10^{-7}$.
The same four experimental configurations as in the previous subsection are considered.
For each configuration, we perform 5 independent experiments.
Note that for a given dataset, the estimated precision matrices $\{\widehat{\Theta}_t\}_{t=1}^T$ remain the same across different values of $\epsilon_{\mathrm{tol}}$; only the change point detection results vary.
Therefore, we focus on the Hausdorff distance as the primary performance metric.}

\rev{The results are displayed in Figure \ref{fig:eps_tol-HD}.
We observe that the Hausdorff distance remains stable across all tested $\epsilon_{\mathrm{tol}}$ values.
This indicates that the choice of $\epsilon_{\mathrm{tol}}$ does not have a significant influence on the change point detection performance.
In our implementation, we use $\epsilon_{\mathrm{tol}} = 10^{-6}$ as the default value.}

\begin{figure}[htbp]
    \centering
    \caption{\rev{Sensitivity analysis of $\epsilon_{\mathrm{tol}}$: Hausdorff distance across different configurations.}}
    \subfloat{\includegraphics[width=0.25\textwidth]{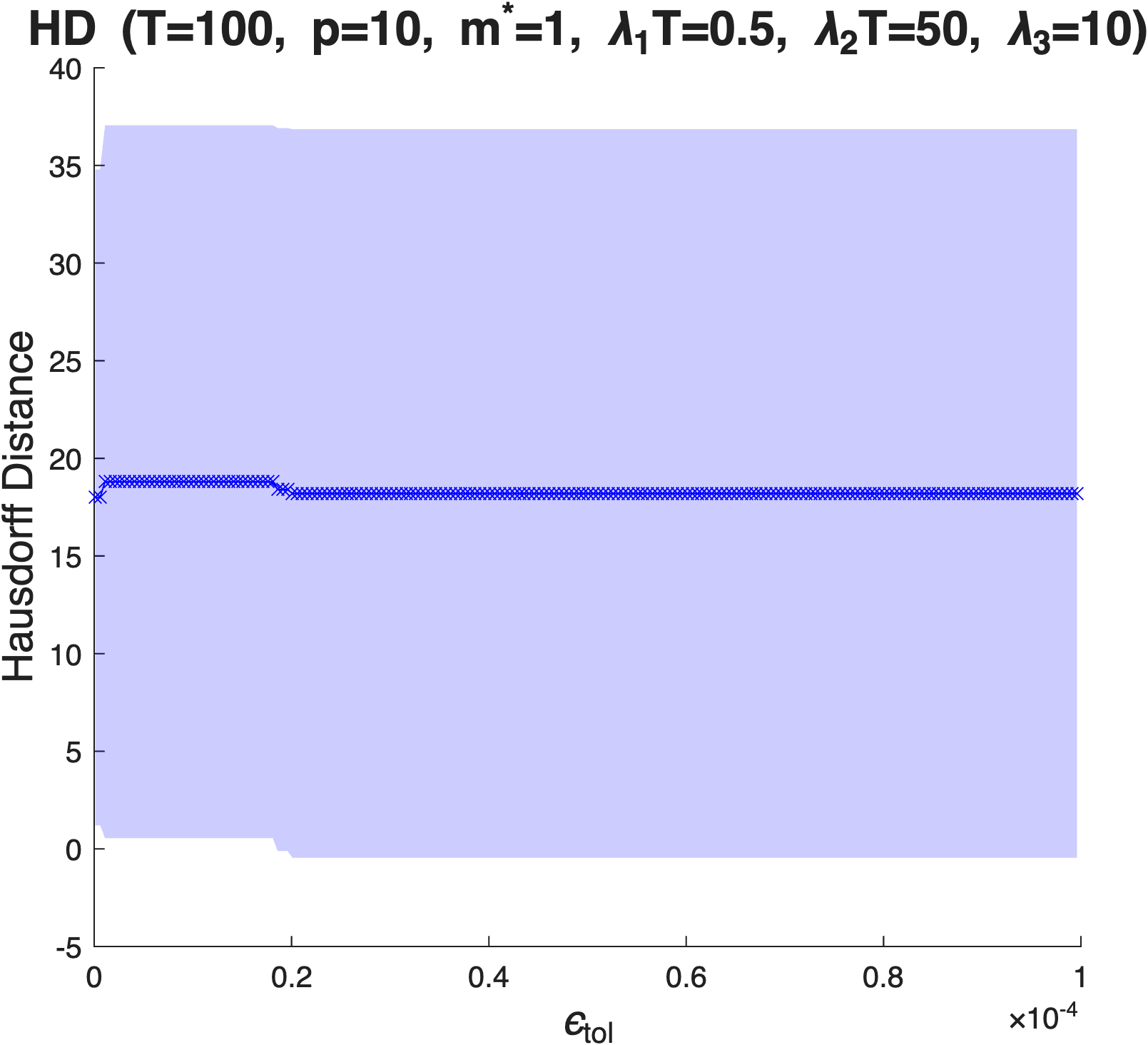}}
    \subfloat{\includegraphics[width=0.25\textwidth]{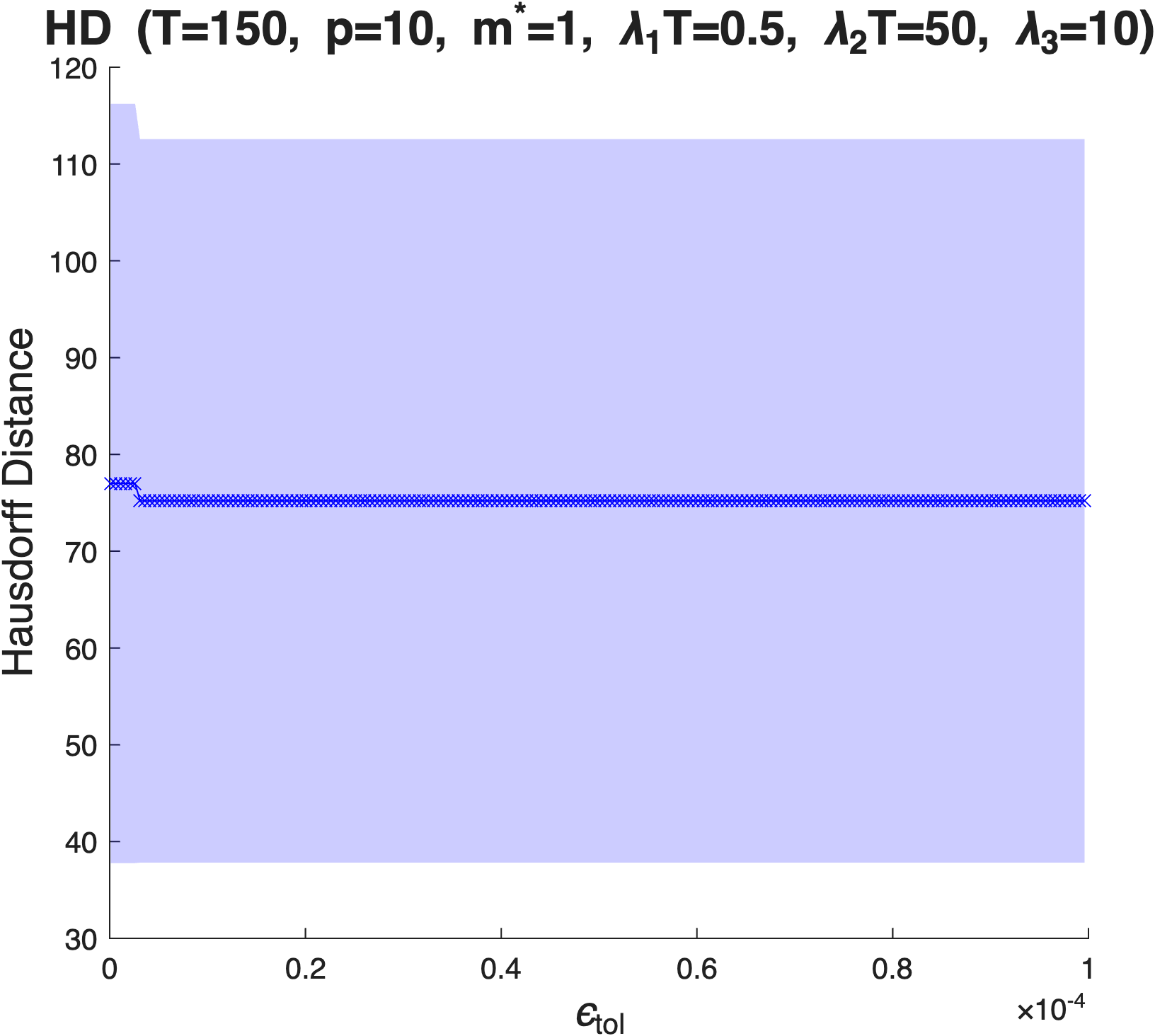}}
    \subfloat{\includegraphics[width=0.25\textwidth]{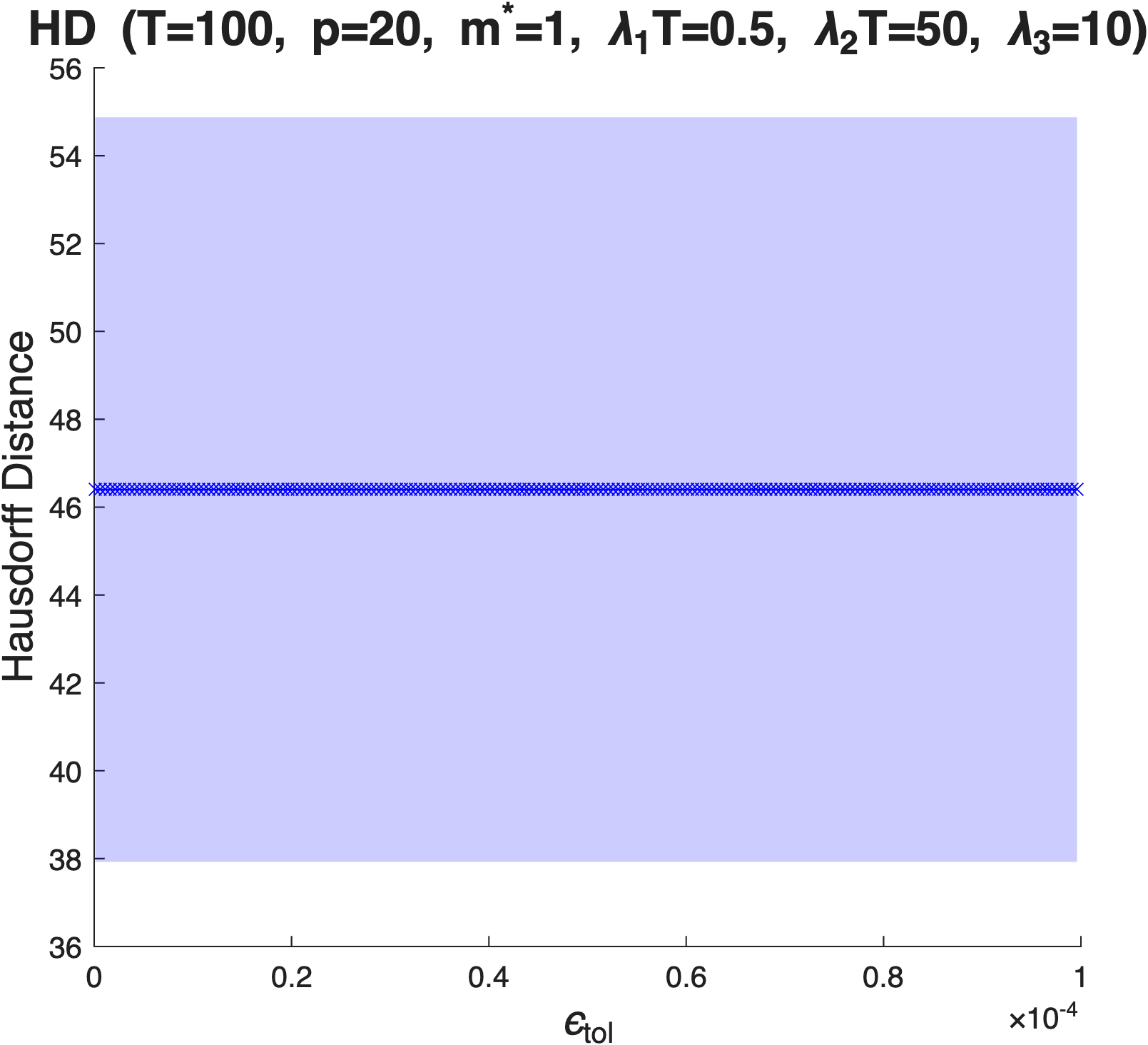}}
    \subfloat{\includegraphics[width=0.25\textwidth]{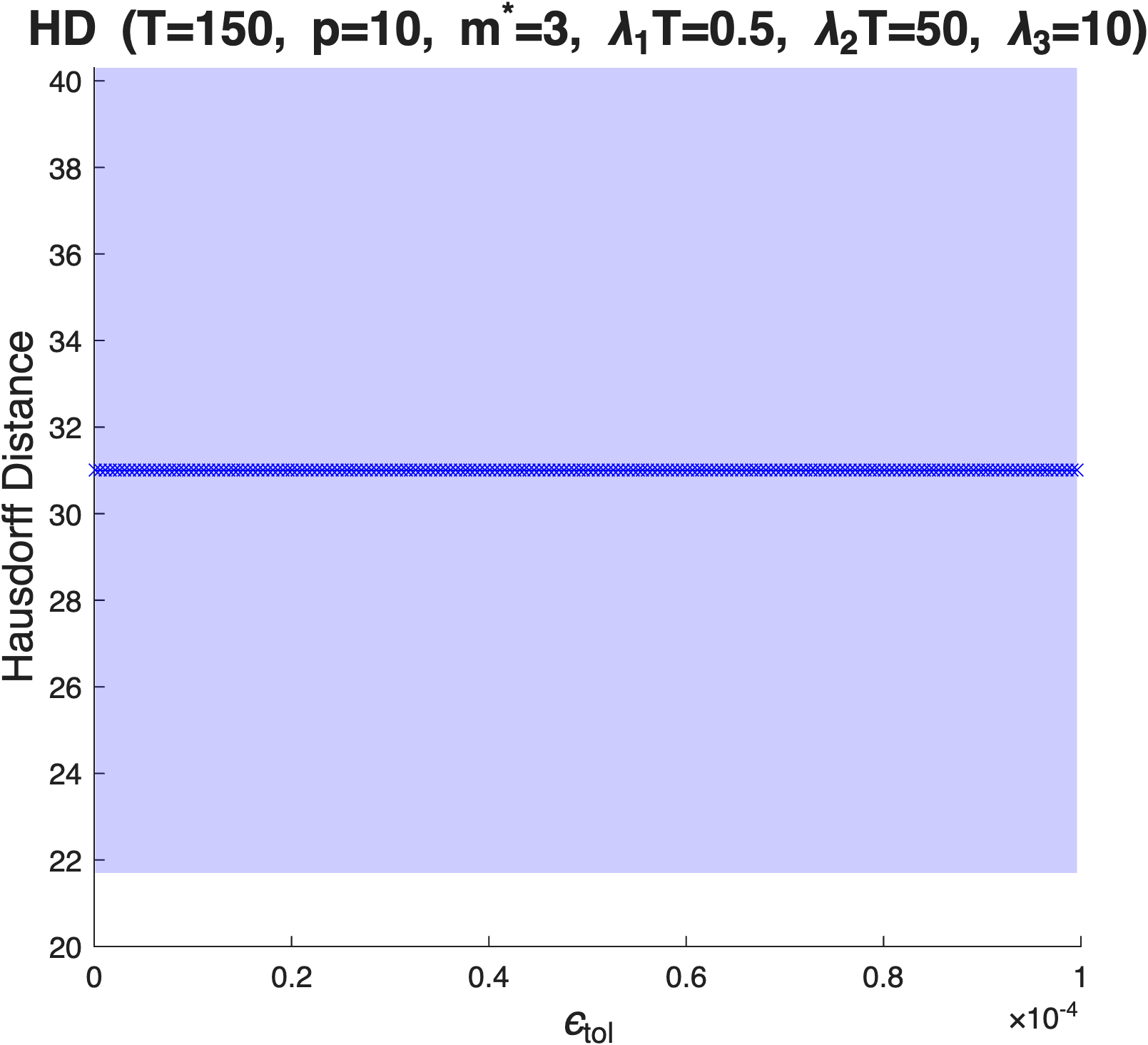}}
    \label{fig:eps_tol-HD}
\end{figure}

\section{Empirical Computational Complexity Analysis}\label{sec:time}

The computational complexity of Algorithm~1 is influenced by several factors, including the sample size \( T \), the dimension \( p \), the number of \rev{change points} \( m^{\ast} \), \( \lambda_3 \), and the pair \( (\lambda_1, \lambda_2) \).
To empirically analyze this complexity, we conduct a series of experiments based on \textbf{Setting (i)} in Section~5 varying each factor individually.
Specifically, for each factor, we perform 5 experiments with other factors held constant, and plot the averaged computation time to visualize the impact of each factor on the algorithm's complexity.
Here, we use $ \beta = 21 $.
These experiments are conducted on a desktop with Intel(R) Core(TM) i9-10900@2.8GHz (10 cores and 20 threads) and 64GB of RAM.

As displayed in Figure \ref{fig:com-time}, the computation time is approximately linear in \( T \), quadratic in \( p \), and not significantly affected by \( m^{\ast} \) and \( \lambda_3 \).
The impact of \( (\lambda_1, \lambda_2) \) is presented in Table \ref{tab:com-time-l1-l2}, where one can observe that the computation times for \( (0.1, 10) \) and \( (0.2, 10) \) are notably shorter.
This is because the algorithm terminates early as \eqref{eq:stop-cri-2} is satisfied within the first few iterations.
Additionally, the computation times for \( (0.1, 20) \), \( (0.2, 20) \), \( (0.3, 10) \), \( (0.4, 10) \), and \( (0.5, 10) \) are significantly higher, as these are marginal cases that are typically more challenging to solve.
Apart from these marginal cases, for large \( \lambda_1 \) and \( \lambda_2 \), Algorithm~1 requires nearly the same amount of time to converge.

\begin{figure}[ht]
    \centering
    \caption{Empirical computational complexity analysis.}
    \subfloat{\includegraphics[width=0.35\textwidth]{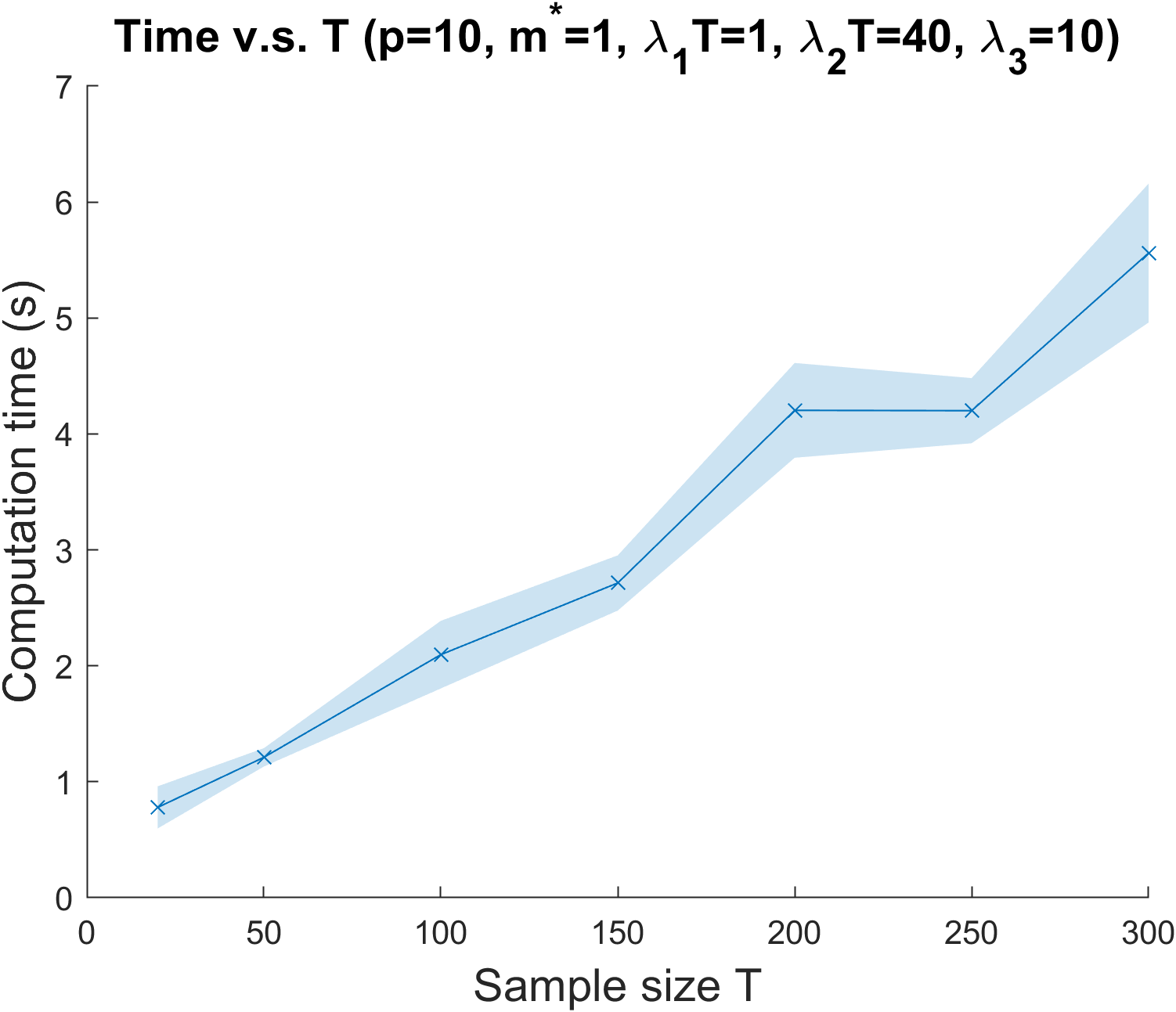}}\qquad\quad
    \subfloat{\includegraphics[width=0.35\textwidth]{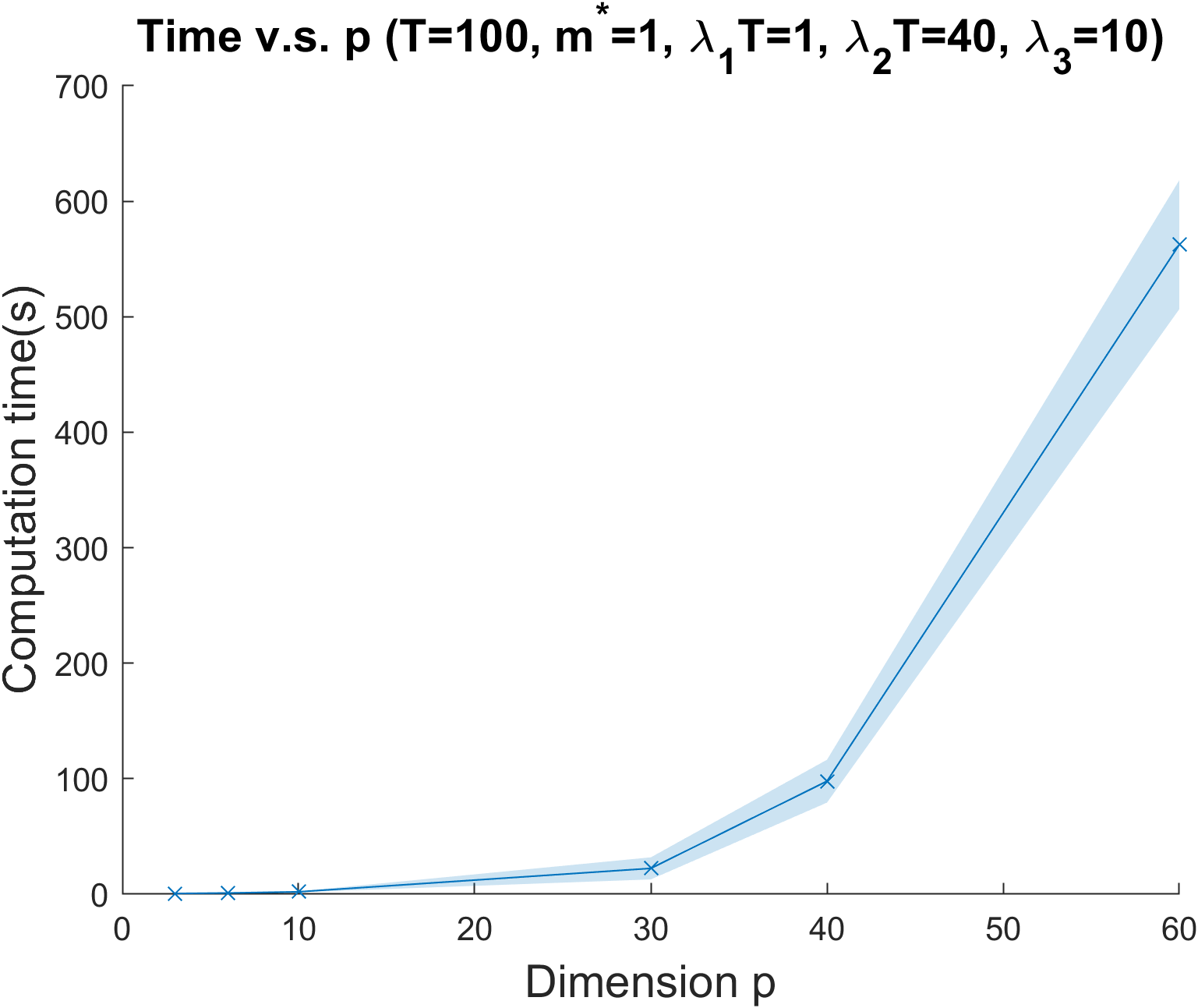}}\\

    \subfloat{\includegraphics[width=0.35\textwidth]{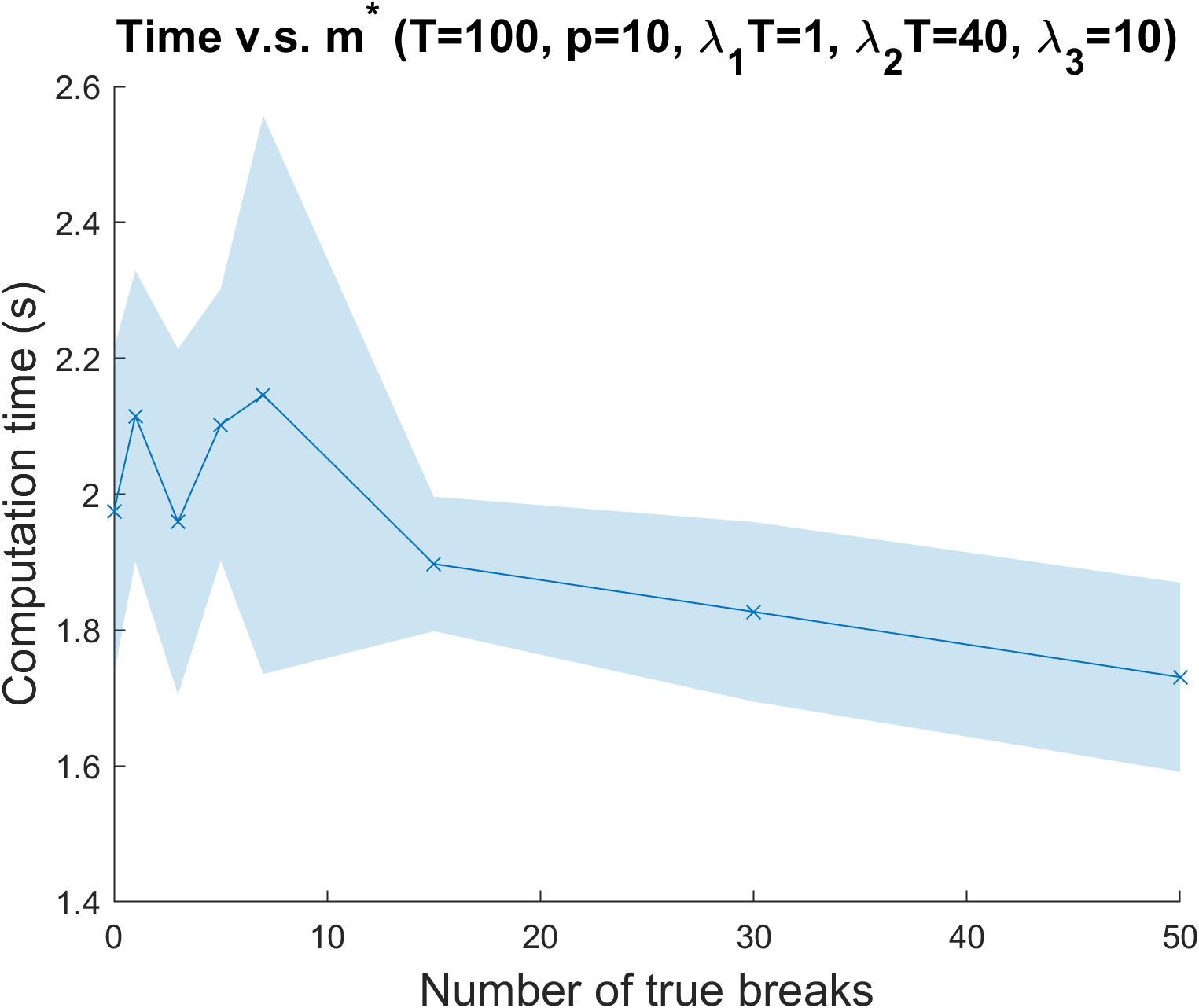}}\qquad\quad
    \subfloat{\includegraphics[width=0.35\textwidth]{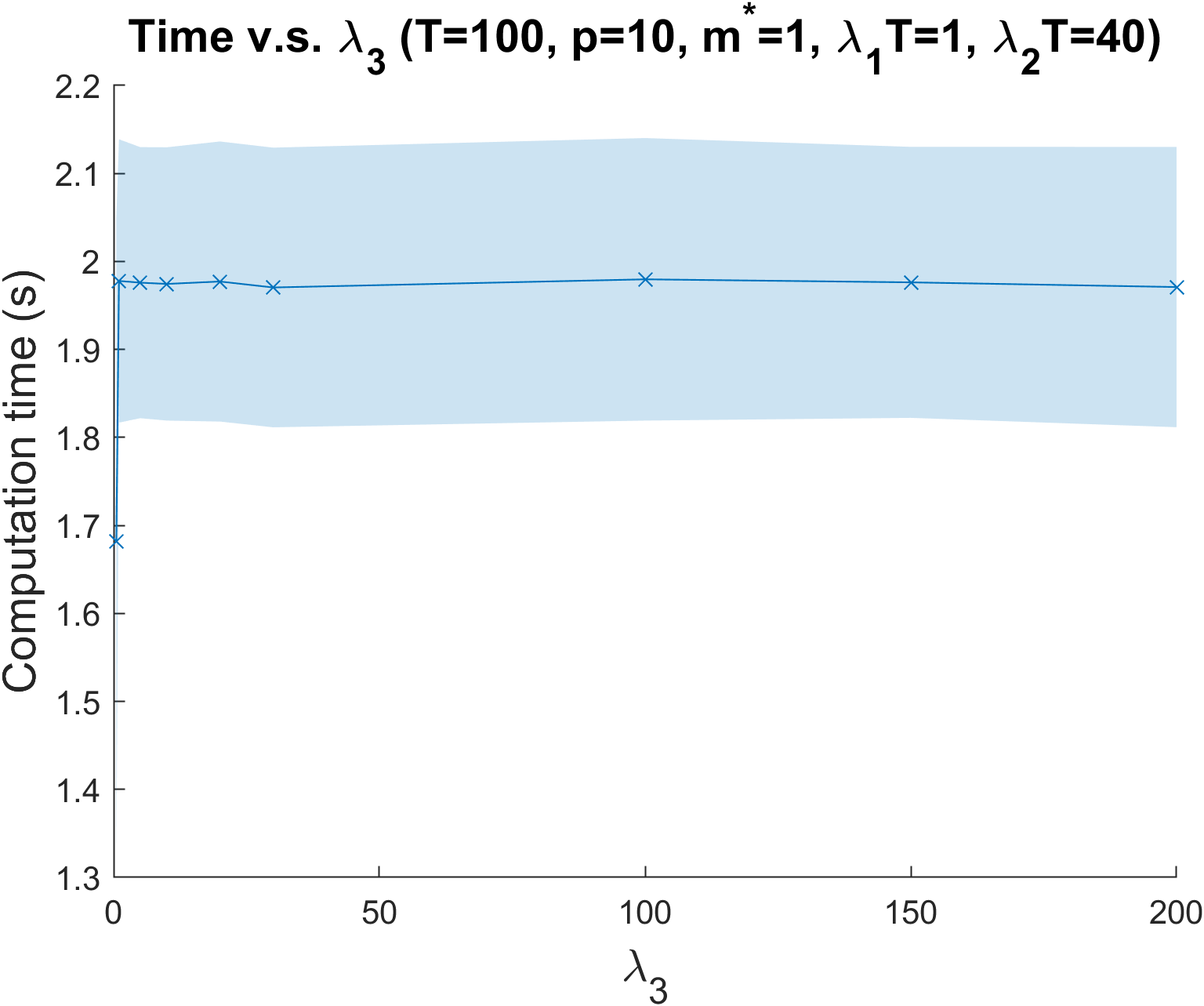}}
    \label{fig:com-time}
\end{figure}
\begin{table}[htb]
    \centering
    \caption{Computation time (s) v.s. $\lambda_1$ and $\lambda_2$ ($T=100, p=10, m^{\ast}=1, \lambda_3=10$)}
    \resizebox{\textwidth}{!}{%
    \begin{tabular}{c|cccccccccccccccccccc}
      \toprule
      $\lambda_1T \backslash \lambda_2T$ & 10    & 20    & 30    & 40   & 50   & 60   & 70   & 80   & 90   & 100  & 110  & 120  & 130  & 140  & 150  & 160  & 170  & 180  & 190  & 200  \\
      \midrule
      0.1                              & 3.52  & 33.71 & 12.47 & 6.71 & 5.17 & 4.17 & 5.58 & 5.58 & 5.83 & 5.73 & 5.58 & 5.51 & 5.52 & 5.51 & 5.53 & 5.44 & 5.52 & 5.50 & 5.57 & 5.47 \\
      0.2                              & 6.60  & 14.00 & 6.15  & 4.28 & 3.68 & 3.98 & 4.13 & 4.19 & 4.13 & 4.15 & 4.16 & 4.15 & 4.14 & 4.14 & 4.15 & 4.17 & 4.15 & 4.15 & 4.15 & 4.15 \\
      0.3                              & 27.25 & 7.13  & 3.69  & 2.98 & 2.67 & 2.91 & 2.90 & 2.91 & 2.91 & 2.91 & 2.91 & 2.91 & 2.91 & 2.92 & 2.91 & 2.91 & 2.92 & 2.91 & 2.93 & 2.91 \\
      0.4                              & 36.05 & 5.11  & 2.75  & 2.40 & 2.16 & 2.16 & 2.16 & 2.16 & 2.16 & 2.16 & 2.16 & 2.22 & 2.19 & 2.21 & 2.18 & 2.20 & 2.21 & 2.21 & 2.21 & 2.23 \\
      0.5                              & 23.22 & 4.65  & 2.68  & 2.28 & 2.11 & 2.12 & 2.12 & 2.11 & 2.12 & 2.12 & 2.18 & 2.17 & 2.17 & 2.21 & 2.17 & 2.18 & 2.18 & 2.17 & 2.17 & 2.17 \\
      0.6                              & 14.14 & 4.64  & 2.81  & 2.30 & 2.30 & 2.22 & 2.23 & 2.17 & 2.20 & 2.21 & 2.24 & 2.20 & 2.24 & 2.15 & 2.20 & 2.14 & 2.17 & 2.17 & 2.19 & 2.18 \\
      0.7                              & 10.38 & 4.78  & 2.87  & 2.19 & 2.26 & 2.30 & 2.15 & 2.18 & 2.16 & 2.23 & 2.17 & 2.19 & 2.21 & 2.21 & 2.21 & 2.15 & 2.18 & 2.18 & 2.19 & 2.19 \\
      0.8                              & 9.17  & 4.72  & 2.88  & 2.46 & 2.22 & 2.41 & 2.59 & 2.48 & 2.20 & 2.15 & 2.14 & 2.15 & 2.18 & 2.23 & 2.24 & 2.43 & 2.27 & 2.35 & 2.20 & 2.26 \\
      0.9                              & 8.73  & 4.80  & 3.03  & 2.45 & 2.18 & 2.18 & 2.21 & 2.15 & 2.15 & 2.13 & 2.16 & 2.21 & 2.21 & 2.14 & 2.15 & 2.13 & 2.20 & 2.17 & 2.13 & 2.13 \\
      1.0                              & 8.06  & 4.60  & 2.84  & 2.13 & 2.14 & 2.13 & 2.15 & 2.14 & 2.13 & 2.13 & 2.13 & 2.14 & 2.14 & 2.14 & 2.15 & 2.16 & 2.14 & 2.14 & 2.13 & 2.14 \\
      \bottomrule
    \end{tabular}%
    }
    \label{tab:com-time-l1-l2}
\end{table}

\section{Real Data Experiment}\label{sec:real_data}

In this section, the relevance of our method is compared with GFGL \rev{and TBFL} through a portfolio allocation experiment based on real financial data. The same computer was employed as in Section~5. We consider hereafter the stochastic process $(X_t)$ in $\Rb^p$ of log-stock returns, where $X_{t,j} = 100\times\log(P_{t,j}/P_{t-1,j}), 1 \leq j \leq p$ with $P_{t,j}$ the stock price of the $j$-th index at time $t$. The portfolio allocation will be performed with $20$ stocks data from the S\&P 500, which are representative of different economic sectors\footnote{The data can be found on \url{https://finance.yahoo.com} or \url{https://macrobond.com}. They are also available at \url{https://github.com/linyopt/GFDtL}.}: Alphabet, Amazon, American Airlines, Apple, Berkshire Hathaway, Boeing, Chevron, Equity Residential, ExxonMobil, Ford, General Electric, Goldman Sachs, Jacobs Engineering Group, JPMorgan, Lockheed Martin, Pfizer, Procter \& Gamble, United Parcel Service, Verizon, Walmart. The sample period is November 11, 2019 – March 27, 2020, corresponding to $T=100$ observations.

We analyse the economic performances obtained by GFDtL\rev{,} GFGL\rev{, and TBFL} through the GMVP investment problem. Following \cite{engle2006}, the latter problem at time $t$, in the absence of short-sales constraints, is defined as
\begin{equation*}
\min_{w_t} \; w^{\top}_t \; H_t \; w_t, \;\;\text{s.t.} \;\;\mathbf{1}^{\top}_{p \times 1} w_t = 1,
\end{equation*}
where $w_t$ is the vector of portfolio weights for time $t$ chosen at time $t-1$, $H_t$ is the $p \times p$ conditional (on the past) covariance matrix of $X_t$. The explicit solution is given by $\omega_t = H^{-1}_t \mathbf{1}_{p \times 1}/\mathbf{1}^{\top}_{p \times 1}H^{-1}_t\mathbf{1}_{p \times 1}$. Now it is natural to replace $H_t$ by an estimator $\widehat{H}_t$, yielding $\widehat{\omega}_t = \widehat{H}^{-1}_t \mathbf{1}_{p \times 1}/\mathbf{1}^{\top}_{p \times 1}\widehat{H}^{-1}_t\mathbf{1}_{p \times 1}$. As a function depending on $H_t$ only, the GMVP performance essentially depends on the precise measurement of the latter or, equivalently, the precision matrix. In our setting, we set $\widehat{H}^{-1}_t = \widehat{\Theta}_{t-1}$, estimated by the GFDtL\rev{,} GFGL\rev{, and TBFL} procedures, where $(\lambda^\ast_1,\lambda^\ast_2)$ are selected by Method c described in Section \ref{sec:tuning-para} \rev{for GFDtL and GFGL, and default parameters are used for TBFL}. We also consider the equally weighted portfolio, which will be denoted by $1/p$. The following performance metrics (annualized) will be reported: \textbf{AVG} as the average of portfolio returns computed as $\widehat{\omega}^\top_t X_t$, multiplied by $252$; \textbf{SD} as the standard deviation of portfolio returns, multiplied by $\sqrt{252}$; \textbf{IR} as the information ratio computed as $\textbf{AVG}/\textbf{SD}$. \\
The key performance measure is \textbf{SD}. The GMVP problem essentially aims to minimize the variance rather than to maximize the expected return. Hence, as emphasized in \cite{engle2019}, Section 6.2, high \textbf{AVG} and \textbf{IR} are desirable but should be considered of secondary importance compared with the quality of the measure of a covariance matrix estimator. We also report the number of \rev{change points} $nb$ detected by the procedure.

Both GFGL and GFDtL estimate the following \rev{change points}: February 25, 2020; March 9, 2020; March 17, 2020. These \rev{change points} are in line with the aftershock of the covid outbreak: the S\&P 500 index entered a downward trend period from February 20, 2020, and the S\&P 500 index reached a minimum value on March 23, 2020, which precedes the rally. \rev{TBFL detects a single \rev{change point} on January 3, 2020, which precedes the covid outbreak period and does not capture the market volatility changes during the crisis.} Despite the presence of the covid shock, our proposed GFDtL procedure provides the lowest \textbf{SD} and clearly outperforms both GFGL \rev{and TBFL}. The BIC-based selection results in relevant estimations of the \rev{change points}.

\begin{table}[h]
\centering
\caption{Annualized GMVP performance metrics.}\label{Table_metrics_econ_performance}
\scalebox{0.9}{\begin{tabular}{c|cccc}\hline\hline
& \multicolumn{4}{c}{S\&P 500 data}\\
  & \textbf{AVG} & \textbf{SD} & \textbf{IR} & $nb$\\
\hline

GFDtL & -50.29 & \textbf{35.76} &  -1.41 & 3 \\

GFGL &  -101.44 & 50.00 &  -2.03 & 3 \\

\rev{TBFL} & \rev{-52.03} & \rev{38.74} & \rev{-1.34} & \rev{1} \\

$1/p$ &  -77.77 & 45.69 &  -1.70  & - \\

\hline\hline
\end{tabular}}
{\flushleft Note: The lowest \textbf{SD} figure is in bold face.}
\end{table}

\section{Concluding Remarks}

\rev{We proposed a filtering method for change point detection in the precision matrix. We showed that the estimated change points are sufficiently close to the true change points and established the consistency of the estimated precision matrix in each regime. We also proposed an ADMM procedure to solve the optimization problem with change points. \\
Alternative filtering techniques could be considered, such as non-convex regularization with fusion, which would require different implementation strategies. Moreover, it would be worthwhile to integrate jumps occurring in both the mean and the second-order moments. We leave these topics for future research.}

\newpage

\noindent\textbf{Acknowledgements}

\mds

Benjamin Poignard was supported by JSPS KAKENHI (25K16617) and RIKEN-AIP.
Ting Kei Pong was supported partly by the Hong Kong Research Grants Council PolyU 15300221.
Akiko Takeda was supported by JSPS KAKENHI (23H03351).

\begin{appendices}

\makeatletter
\newcommand{\section@cntformat}{Appendix \thesection\ }
\makeatother

\section{An Example of Unsolvable (2)}\label{sec:opt-examples}

We present a simple example to illustrate that for some specific dataset $ \mathcal{X}_T $, even with $ \sum_{t=1}^T X_t X_t ^{\top} \succ 0 $, there may still exist some $ \lambda_1 $ and $ \lambda_2 $ such that (3) of the main text is infeasible, meaning that its optimal value is $-\infty$. Since (2) and (3)  of the main text have the same optimal value, this means (2) does not have solutions in this case.
\begin{example}\label{eg:opt-example}
Consider $ \mathcal{X}_2 = (X_1, X_2) $ where $ X_1 = (1, 0)^{\top} $ and $ X_2 = (0, 1)^{\top} $. Then
$$
X_1 X_1^{\top} =
\begin{bmatrix}
  1 & 0 \\
  0 & 0
\end{bmatrix}, \quad X_2 X_2^{\top} =
\begin{bmatrix}
  0 & 0 \\
  0 & 1
\end{bmatrix},
$$
and $ X_1X_1^{\top} \!+\! X_2X_2^{\top} \!=\! I_2 \!\succ\! 0 $.
The positive semi-definite constraint in (3) can be written as
$$
\begin{aligned}
  & Z_1 - I_2 +
    \begin{bmatrix}
      W_{11, 1} & W_{12, 1}/2 \\
      W_{12, 1}/2 & 0
    \end{bmatrix} -
    \begin{bmatrix}
      0 & Y_{12, 1} \\
      Y_{12, 1} & 0
    \end{bmatrix} \succeq 0, \\
  & -Z_1 - I_2 +
    \begin{bmatrix}
      0 & W_{21, 2}/2 \\
      W_{21, 2}/2 & W_{22, 2}
    \end{bmatrix} -
    \begin{bmatrix}
      0 & Y_{12, 2} \\
      Y_{12, 2} & 0
    \end{bmatrix} \succeq 0.
\end{aligned}
$$
After some simple calculations, we have
$$
\begin{aligned}
  & \begin{bmatrix}
    Z_{11, 1} + W_{11, 1} - 1 & Z_{12, 1} + W_{12, 1}/2 - Y_{12, 1} \\
    Z_{21, 1} + W_{12, 1}/2 - Y_{12, 1} & Z_{22, 1} - 1
  \end{bmatrix} \succeq 0, \\
  & \begin{bmatrix}
    Z_{11, 1} + 1 & Z_{12, 1} - W_{21, 2}/2 + Y_{12, 2} \\
    Z_{21, 1} - W_{21, 2}/2 + Y_{12, 2} & Z_{22, 1} - W_{22, 2} + 1
    \end{bmatrix} \preceq 0.
\end{aligned}
$$
Recall that the diagonal elements of a positive semi-definite matrix must be non-negative, then we can observe that the above condition requires that $ Z_{11, 1} \leq -1 $ and $ Z_{22, 1} \geq 1 $.
These two conditions imply that for any $ \lambda_1 > 0 $ and any $ \lambda_2 < 1/\sqrt{2} $, the dual problem (3) is infeasible and hence the primal problem (2) does not have solutions.
\end{example}

\section{Preliminary Results}\label{supp:preliminary}

\begin{lemma}\label{emp_var_cov_bound}
Under Assumption~1, Assumption~2 and Assumption~3-(i), if $p^2\log(pT)=o(T\delta_T)$ with $\delta_T \rightarrow 0$ as $T \rightarrow \infty$, then:
\begin{itemize}
    \item[(i)] $\underset{r-s \geq T\delta_T}{\underset{1 \leq s<r\leq T+1}{\sup}} \lambda_{\max}\Big(\frac{1}{r-s}\overset{r-1}{\underset{t=s}{\sum}}X_tX^\top_t\Big) \leq \overline{\mu}+ o_p(1)$.
    \item[(ii)] $\underline{\mu} +o_p(1) \leq \underset{r-s \geq T\delta_T}{\underset{1 \leq s<r\leq T+1}{\inf}} \lambda_{\min}\Big(\frac{1}{r-s}\overset{r-1}{\underset{t=s}{\sum}}X_tX^\top_t\Big)$.
\end{itemize}
\end{lemma}
\begin{proof} Let us prove Point (i). Recall that $\Sigma^\ast_{j}$ is the true variance-covariance of $X_t$, with $t \in \Bc^\ast_j$. Now take any $s \leq t \leq r-1$, with $s,r \in \{1,\ldots,T\}$ such that $r-s\geq T\delta_T$. Then, by the triangle inequality, under Assumption~2:
  $$
  \begin{aligned}
    \|\frac{1}{r-s}\overset{r-1}{\underset{t=s}{\sum}}X_tX^\top_t\|_s \leq & \:\: \|\frac{1}{r-s}\overset{r-1}{\underset{t=s}{\sum}}\text{Var}(X_t)\|_s + \|\frac{1}{r-s}\overset{r-1}{\underset{t=s}{\sum}}\big(X_tX^\top_t-\text{Var}(X_t)\big)\|_s \\
    \leq & \:\: \overline{\mu}+p\|\frac{1}{r-s}\overset{r-1}{\underset{t=s}{\sum}}\big(X_tX^\top_t-\text{Var}(X_t)\big)\|_{\max}.
  \end{aligned}
  $$
We show $\max_{\underset{r-s \geq T\delta_T}{1 \leq s<r\leq T+1}}\|\frac{1}{r-s}\overset{r-1}{\underset{t=s}{\sum}}\big(X_tX^\top_t-\text{Var}(X_t)\big)\|_s = o_p(1)$. By Assumption~1, \rev{we can apply Lemma A.2 of \cite{Qian2016} on $\zeta_{kl,t}\coloneqq  X_{k,t}X_{l,t}$. Note that the latter result follows from Theorem 1 of \cite{Merlevede2011}, which} states the mixing condition $\alpha(t) \leq \exp(-c_1t^{\gamma_1})$ for some $c_1,\gamma_1>0$; then for $c_{\rev{\alpha}}=1,\gamma_1=1$, we may take $\rho = \exp(-2c_1)$, which allows us to apply \rev{Lemma A.2 of \cite{Qian2016}}.
Thus, $\forall 1 \leq k,l \leq p$, there exist constants $C_1,C_2$ such that, \rev{for any $C>0$ large enough, with $r>s$:}
\rev{\begin{eqnarray*}
\lefteqn{\Pb\Big(\!\big|\frac{1}{\rev{\sqrt{r\!-\!s}}}\overset{r\!-\!1}{\underset{t=s}{\sum}}\big(X_{k,t}X^\top_{l,t}\!-\!\text{Var}(X_t)_{kl}\big)\big|\!\geq\!C\sqrt{\log(pT)}\!\Big)}\\
&\leq & (T\!+\!1) \exp\!\Big(\!-\,\frac{(\rev{C(r\!-\!s)\log(pT)})^{\frac{\gamma}{\rev{2}(1+\gamma)}}}{C_1}\!\Big) \!+\! \exp\!\Big(\rev{\!\!-\,\frac{C(r\!-\!s)\log(pT)}{C_2}}\!\Big).
\end{eqnarray*}}
Applying the previous inequality, by Bonferroni's inequality:
$$
\begin{aligned}
  & \Pb\Big(\underset{r-s \geq T\delta_T}{\underset{1 \leq s<r\leq T+1}{\sup}}\|\frac{1}{\sqrt{r-s}}\overset{r-1}{\underset{t=s}{\sum}}\big(X_tX^\top_t-\text{Var}(X_t)\big)\|_{\max}\geq C\sqrt{\log(pT)}\Big) \\
  \leq & \:\: T^2 \underset{r-s \geq T\delta_T}{\underset{1 \leq s<r\leq T+1}{\sup}} \Pb\Big(\underset{1 \leq k,l \leq p}{\max}\,|\frac{1}{\sqrt{r-s}}\overset{r-1}{\underset{t=s}{\sum}}\big(X_tX^\top_t-\text{Var}(X_t)\big)_{kl}|\geq C\sqrt{\log(pT)}\Big) \\
  \leq & \:\: T^2 \underset{r-s \geq T\delta_T}{\underset{1 \leq s<r\leq T+1}{\sup}} p^2 \Big\{ (T+1) \, \exp\Big(-\frac{(C(r-s)\log(pT)^{\frac{\gamma}{2(1+\gamma)}}}{C_1}\Big) \\
  & \:\:\quad + \exp\Big(-\frac{C(r-s)\log(pT)}{C_2}\Big)\Big\} \\
  \leq & \:\: \exp\Big(-\frac{(CT\delta_T\log(pT)^{\frac{\gamma}{2(1+\gamma)}}}{C_1}+4\log(pT)\Big) + \exp\Big(-\frac{CT\delta_T\log(pT)}{C_2}+2\log(pT)\Big),
\end{aligned}
$$
which goes to $0$ as $T \rightarrow \infty$ by Assumption~3-(i) implying $(T\delta_T\log(pT)^{(\gamma/(2(1+\gamma))} \propto \log(pT)$. Since $\|A\|_s \leq p \|A\|_{\max}$ for $A \in \Rb^{p\times p}$,
\begin{equation*}
p\max_{\underset{r-s \geq T\delta_T}{1 \leq s<r\leq T+1}}\|\frac{1}{r-s}\overset{r-1}{\underset{t=s}{\sum}}\big(X_tX^\top_t-\text{Var}(X_t)\big)\|_{\max} = O_p(p\sqrt{\frac{\log(pT)}{T\delta_T}}) = o_p(1),
\end{equation*}
under $(T\delta_T)^{-1}p^2\log(pT) \rightarrow 0$. The proof of Point (ii) is similar and thus omitted. The $\underline{\mu}$ term follows from Assumption~2.
\end{proof}
The next Lemma will be useful to bound the first order derivative w.r.t. $\Theta$ of the non-penalized D-trace loss function.
\begin{lemma}\label{bound_gradient}
Suppose Assumption~1, Assumption~2 and Assumption~3-(i) are satisfied. For a sequence $\delta_T \rightarrow 0$ as $T\rightarrow\infty$, then
\begin{equation*}
\normalfont\underset{r-s \geq T\delta_T}{\underset{1 \leq s<r\leq T+1}{\sup}} \|\frac{1}{\sqrt{r-s}}\overset{r-1}{\underset{t=s}{\sum}}\big(X_tX^\top_t-\text{Var}(X_t)\big)\|_{\max} = O_p(\sqrt{\log(pT)}).
\end{equation*}
\end{lemma}
\begin{proof}
The result follows the same steps as in the proof of Lemma \ref{emp_var_cov_bound}.
\end{proof}

\begin{lemma}\label{optimality_cond}
Consider problem (1). Define $\Gamma_t = \Theta_{t}-\Theta_{t-1}$, $ t \geq 2$ and $\Gamma_1 = \Theta_1$. The GFDtL estimator $\{\widehat\Theta_t\}^T_{t=1}$ satisfies the conditions
\begin{equation*}
\forall t \in \{1,\ldots,T\}, \; \frac{1}{T}\overset{T}{\underset{r=t}{\sum}}\Big(\frac{1}{2}\widehat{\Theta}_rX_rX^\top_r + \frac{1}{2}X_rX^\top_r \widehat{\Theta}_r-I_p\Big) + \lambda_1\overset{T}{\underset{r=t}{\sum}} \widehat{E}_{1r} + \lambda_2\widehat{E}_{2t}=\mathbf{0}_{p \times p},
\end{equation*}
where $\widehat{E}_{1t}, \widehat{E}_{2t}$ are the sub-gradient matrices defined by
\begin{equation*}
\forall u\neq v, \; \widehat{E}_{uv,1t} = \begin{cases}
\normalfont\text{sgn}(\overset{t}{\underset{s=1}{\sum}}\widehat{\Gamma}_{uv,s})& \text{if} \; \overset{t}{\underset{s=1}{\sum}}\widehat{\Gamma}_{uv,s} \neq 0,\\
\in [-1,1] & \text{otherwise},
\end{cases}
\end{equation*}
and $\widehat{E}_{21}=\mathbf{0}_{p \times p}$ and for $t=2,\ldots,T$, $\widehat{E}_{2t}$ satisfies
\begin{equation*}
 \widehat{E}_{2t} =\frac{\widehat{\Gamma}_{t}}{\|\widehat{\Gamma}_{t}\|_F}\;\; \text{if} \;\; \widehat{\Gamma}_{t} \neq \mathbf{0}_{p \times p},\;\; \text{and} \;\;
\|\widehat{E}_{2t}\|_F \leq 1 \;\; \text{if} \;\; \|\widehat{\Gamma}_{t}\|_F=0.
\end{equation*}
\end{lemma}
\begin{proof}
Defining $\Gamma_t = \Theta_{t}-\Theta_{t-1}$, $ t \geq 2$ and $\Gamma_1 = \Theta_1$, the problem stated in (1) can be recast as a minimization of the function
\begin{equation*}
\Gb_{\lambda_1,\lambda_2}(\{\Theta_t\}_{t=1}^T,\!\mathcal{X}_T)\!=\!\frac{1}{T}\overset{T}{\underset{t=1}{\sum}}\text{tr}(\frac{1}{2}\Big(\overset{t}{\underset{s=1}{\sum}}\Gamma_s\Big)^2 X_tX^\top_t) - \text{tr}(\overset{t}{\underset{s=1}{\sum}}\Gamma_s) + \lambda_1\overset{T}{\underset{t=1}{\sum}}\underset{u\neq v}{\sum}|\overset{t}{\underset{s=1}{\sum}}\Gamma_{uv,s}|+\lambda_2\overset{T}{\underset{t=2}{\sum}} \|\Gamma_t\|_F.
\end{equation*}
Invoking subdifferential calculus, a necessary and sufficient condition for $(\widehat\Gamma_t)_{1\leq t \leq T}$ to minimize $\Gb_{\lambda_1,\lambda_2}(\cdot,\mathcal{X}_T)$ is that for all $t = 1,\ldots,T$, $\mathbf{0}_{p \times p} \in \Rb^{p \times p}$ belongs to the subdifferential of $\Gb_{\lambda_1,\lambda_2}(\cdot,\mathcal{X}_T)$ with respect to $(\Gamma_t)_{1\leq t \leq T}$ at $(\widehat\Gamma_t)_{1\leq t \leq T}$, that is
\begin{equation*}
\frac{1}{T}\overset{T}{\underset{r=t}{\sum}}\Big(\frac{1}{2}\Big(\overset{r}{\underset{s=1}{\sum}}\widehat\Gamma_s\Big)X_rX^\top_r +\frac{1}{2} X_rX^\top_r \Big(\overset{r}{\underset{s=1}{\sum}}\widehat\Gamma_s\Big)-I_p\Big) + \lambda_1\overset{T}{\underset{r=t}{\sum}} \widehat{E}_{1r} + \lambda_2 \widehat{E}_{2t}=\mathbf{0}_{p \times p},
\end{equation*}
with the subgradient matrices defined as: $\widehat{E}_{21}=\mathbf{0}_{p \times p}$ and
\begin{equation*}
 \widehat{E}_{2t} =\frac{\widehat{\Gamma}_{t}}{\|\widehat{\Gamma}_{t}\|_F}\; \text{if} \; \widehat{\Gamma}_{t} \neq \mathbf{0}_{p \times p},\;\; \text{and} \;\;
\|\widehat{E}_{2t}\|_F \leq 1 \; \text{if} \; \|\widehat{\Gamma}_{t}\|_F=0,\;\; \text{and} \
\end{equation*}
\begin{equation*}
\forall u\neq v, \; \widehat{E}_{uv,1t} = \begin{cases}
\normalfont\text{sgn}(\overset{t}{\underset{s=1}{\sum}}\widehat{\Gamma}_{uv,s})& \text{if} \; \overset{t}{\underset{s=1}{\sum}}\widehat{\Gamma}_{uv,s} \neq 0,\\
\in [-1,1] & \text{otherwise}.
\end{cases}
\end{equation*}
Now if $t=\widehat{T}_j$ for $j \in \{1,\ldots,\widehat{m}\}$ is one of the estimated \rev{change dates}, then $\widehat{\Gamma}_t \neq \mathbf{0}_{p \times p}$, and
\begin{equation*}
\frac{1}{T}\overset{T}{\underset{r=\widehat{T}_j}{\sum}}\Big(\frac{1}{2}\widehat{\Theta}_rX_rX^\top_r + \frac{1}{2}X_rX^\top_r \widehat{\Theta}_r-I_p\Big) + \lambda_1\overset{T}{\underset{r=\widehat{T}_j}{\sum}} \widehat{E}_{1r} + \lambda_2\frac{\widehat{\Gamma}_{\widehat{T}_j}}{\|\widehat{\Gamma}_{\widehat{T}_j}\|_F}=\mathbf{0}_{p \times p},
\end{equation*}
since the \rev{change points} cannot occur at $t=1$ and $\overset{r}{\underset{s=1}{\sum}}\Gamma_s= \widehat{\Theta}_r$. When $t=1$, then the first order condition with respect to $\Gamma_t$ yields
\begin{equation*}
\frac{1}{T}\overset{T}{\underset{r=1}{\sum}}\Big(\frac{1}{2}\Big(\overset{r}{\underset{s=1}{\sum}}\widehat\Gamma_s\Big)X_rX^\top_r + \frac{1}{2}X_rX^\top_r \Big(\overset{r}{\underset{s=1}{\sum}}\widehat\Gamma_s\Big)-I_p\Big) + \lambda_1\overset{T}{\underset{r=1}{\sum}} \widehat{E}_{1r}=\mathbf{0}_{p \times p}.
\end{equation*}
\end{proof}

\begin{lemma}\label{lemma:strict-feasibility}
	Let $ \mathcal{X}_T = (X_1, \dots,X_T) $ be a given set of $ p $-dimensional vectors with $ \sum_{t=1}^T X_t X_t ^{\top} \succ 0 $.
	For an arbitrary fixed $ \lambda_1 > 0 $, let $ \mathcal{C}_{\lambda_1} $ be defined as
	$$
        \resizebox{\textwidth}{!}{%
		$\begin{aligned}
          \mathcal{C}_{\lambda_1} \coloneqq  \Big\{ \left\{ \{W_t\}_{t=1}^T, \{Y_{t, \mathrm{off}}\}_{t=1}^T, \{Z_t\}_{t=1}^{T-1} \right\} \,:\, & Z_0 = Z_T = \mathbf{0}_{p \times p}; Z_t \!-\! Z_{t - 1} \!-\! I_p \!+\! \Sym\left(X_tX_t^{\top})^{\frac{1}{2}}W_t\right) \!-\! Y_{t, \mathrm{off}} \succeq 0 \,\,\,\, \forall t; \\
                                                    & | Y_{uv, t, \mathrm{off}} | \leq \lambda_1T  \,\,\,\, \forall t, u, v; W_t \in \Rb^{p\times p}, Y_{t, \mathrm{off}} \in \mathcal{S}_{\mathrm{off}}^p, Z_t \in \mathcal{S}^p \:\: \forall t \Big\},
		\end{aligned}$}
	$$
    where $ Y_{uv, t, \mathrm{off}} $ is the $ (u, v) $-th element of $ Y_{t, \mathrm{off}} $.
	Then $ \mathcal{C}_{\lambda_1} $ has a Slater point.
\end{lemma}

\begin{proof}
Since $ \sum_{t=1}^T X_t X_t ^{\top} \succ 0 $, there exist a large $ c > 0 $ and a small $ \gamma > 0 $ such that
	$$
		T \gamma I_p \prec c \sum_{t=1}^T X_t X_t ^{\top} - T I_p = \sum_{t=1}^T (X_t X_t ^{\top})^{\frac{1}{2}} \overline{W}_t - \sum_{t=1}^T \overline{Y}_{t, \mathrm{off}} - T I_p,
	$$
 where $ \overline{W}_t \coloneqq  c (X_t X_t ^{\top})^{\frac{1}{2}} $ and $ \overline{Y}_{t, \mathrm{off}} \coloneqq  \mathbf{0}_{p \times p} $ for all $ t $.
	We can see from the above display that
	\begin{equation}
		\label{eq:pd-T-1}
        \begin{aligned}
              cX_T X_T ^{\top} - I_p \succ \:\: & (T - 1)I_p - c\sum_{t=1}^{T-1} X_t X_t ^{\top} + (T - 1)\gamma I_p + \gamma I_p \\
                                     \succ \:\: & (T - 1)I_p - c\sum_{t=1}^{T-1} X_t X_t ^{\top} + (T - 1)\gamma I_p.
        \end{aligned}
	\end{equation}

	To find $ \{\overline{Z}_t\}_{t=1}^{T-1} $ and $ \overline{Z}_0 = \overline{Z}_T = \mathbf{0}_{p \times p} $ such that $ \overline{Z}_t - \overline{Z}_{t - 1} - I_p + \Sym\left(X_t X_t^{\top})^{\frac{1}{2}} \overline{W}_t\right) - \overline{Y}_t \succ 0 $ for all $ t $, we need
	\begin{equation}
		\label{eq:requirements-dual-stri-feas}
		\left\{
		\begin{aligned}
			 & \overline{Z}_{T-1} \prec cX_T X_T^{\top} - I_p,                                                                   \\
			 & \overline{Z}_t \prec \overline{Z}_{t+1} + cX_{t+1} X_{t+1}^{\top} - I_p \,\,\,\, \forall t = 1, 2, \dots , T - 2, \\
			 & \overline{Z}_1 \succ  I_p - cX_1 X_1^{\top}.
		\end{aligned}
		\right.
	\end{equation}
	We claim that the following choice of \( \{\overline{Z}_t\}_{t=1}^{T - 1} \) defined recursively (starting from $\overline{Z}_0 = \mathbf{0}_{p \times p}$) satisfies \eqref{eq:requirements-dual-stri-feas}:
	\[
		\overline{Z}_t = \overline{Z}_{t-1} + I_p - cX_t X_t ^{\top}  + \gamma I_p \,\,\,\,\forall t = 1, \dots , T - 1.
	\]
	Indeed, it is routine to check that the second and third lines of \eqref{eq:requirements-dual-stri-feas} are satisfied.
	Then, using $\overline{Z}_0 = \mathbf{0}_{p \times p}$ and the above display recursively, we have
	\[ \overline{Z}_{T - 1} = (T - 1)I_p - c\sum_{t=1}^{T-1} X_t X_t ^{\top} + (T - 1)\gamma I_p. \]
	Then by \eqref{eq:pd-T-1}, \( \overline{Z}_{T - 1} \prec cX_T X_T^{\top} - I_p \). Hence $ \left\{ \{\overline{W}_t\}_{t=1}^T, \{\overline{Y}_{t, \mathrm{off}}\}_{t=}^T, \{\overline{Z}_t\}_{t=1}^{T-1} \right\} $ is a Slater point of $ \mathcal{C}_{\lambda_1} $.
\end{proof}

\section{Proofs}\label{proofs}

In this section, we report the proofs of the results stated in the main text.
\subsection{Proof of Theorem~1}

\textbf{\emph{Proof of point (i).}}\\
The proof builds upon the works of \cite{Harchaoui2010}, Proposition 3, \cite{Qian2016}, Theorem 3.1 and \cite{Gibberd2021}, Theorem 1. We define:
\begin{equation*}
    A_{T,j} = \big\{|\widehat{T}_j-T^\ast_j| \geq T\delta_T\big\}, \;\; C_T = \big\{\underset{1\leq j \leq m^\ast}{\max}|\widehat{T}_j-T^\ast_j|<\mathcal{I}_{\min}/2\big\}.
\end{equation*}
By union bound, $\Pb(\underset{1 \leq j \leq m^\ast}{\max}|\widehat{T}_j-T^\ast_j|\geq T \delta_T) \leq \sum^{m^\ast}_{j=1}\Pb(A_{T,j})$, $m^\ast<\infty$. So we aim to show:
\begin{equation*}
(a) \; \sum^{m^\ast}_{j=1}\Pb(A_{T,j} \cap C_T) \rightarrow 0, \;\; (b) \; \sum^{m^\ast}_{j=1}\Pb(A_{T,j} \cap C^c_T) \rightarrow 0,
\end{equation*}
with $C^c_T$ the complement of $C_T$.

\noindent\emph{Proof of (a).} We show:
\begin{equation*}
\sum^{m^\ast}_{j=1}\Pb(A^+_{T,j} \cap C_T) \rightarrow 0 \; \text{and} \; \sum^{m^\ast}_{j=1}\Pb(A^-_{T,j} \cap C_T) \rightarrow 0,
\end{equation*}
where $A^+_{T,j}=\{T^\ast_j-\widehat{T}_j \geq T \delta_T\}, A^-_{T,j}=\{\widehat{T}_j-T^\ast_j \geq T \delta_T\}$.
We prove $\sum^{m^\ast}_{j=1}\Pb(A^+_{T,j} \cap C_T) \rightarrow 0$ as the other case follows in the same spirit. In light of $C_T$:
\begin{equation}\label{c_T_def}
\forall j \in \{1,\ldots,m^\ast\}, \; T^\ast_{j-1} < \widehat{T}_j < T^\ast_{j+1}.
\end{equation}
By Lemma \ref{optimality_cond}, with $t = T^\ast_j$ and $t=\widehat{T}_j$, let $\Lambda(\Sigma) = \frac{1}{2}(\Sigma \otimes I_p + I_p \otimes \Sigma)$, in $\text{vec}(\cdot)$ form:
\begin{eqnarray*}
\frac{1}{T}\overset{T}{\underset{r=\widehat{T}_j}{\sum}}\Big[\Lambda(X_rX^\top_r)\text{vec}(\Theta^\ast_r+\widehat{\Theta}_r-\Theta^\ast_r)-\text{vec}(I_p)\Big] + \text{vec}\big(\lambda_1\overset{T}{\underset{r=\widehat{T}_j}{\sum}} \widehat{E}_{1r}+ \lambda_2 \frac{\widehat{\Gamma}_{\widehat{T}_j}}{\|\widehat{\Gamma}_{\widehat{T}_j}\|_F}\big)= \mathbf{0}_{p^2 \times 1},
\end{eqnarray*}
and
\begin{equation*}
\|\frac{1}{T}\!\overset{T}{\underset{r=T^\ast_j}{\sum}}\Big[\!\Lambda(X_rX^\top_r)\text{vec}(\Theta^\ast_r)\!-\!\text{vec}(I_p)\!\Big] \!+\! \frac{1}{T}\!\overset{T}{\underset{r=T^\ast_j}{\sum}} \Lambda(X_rX^\top_r)\text{vec}(\widehat{\Theta}_r\!-\!\Theta^\ast_r) \!+\! \lambda_1\!\text{vec}\big(\!\overset{T}{\underset{r=T^\ast_j}{\sum}}\!\widehat{E}_{1r}\big)\|_2\!\leq\!\lambda_2.
\end{equation*}
Therefore, under $T^\ast_j>\widehat{T}_j$, taking the differences, by the triangle inequality, we obtain:
\begin{eqnarray*}
2\lambda_2\!\geq\! \|\!\frac{1}{T}\!\overset{T^\ast_j-1}{\underset{r=\widehat{T}_j}{\sum}}\!\Big[\!\Lambda(X_rX^\top_r)\text{vec}(\Theta^\ast_r)\!-\!\text{vec}(I_p)\!\Big]  \!+\! \frac{1}{T}\overset{T^\ast_j-1}{\underset{r=\widehat{T}_j}{\sum}}\! \Lambda(X_rX^\top_r)\text{vec}(\widehat{\Theta}_r\!-\!\Theta^\ast_r)
\!+\! \text{vec}\big(\!\lambda_1\!\overset{T^\ast_j-1}{\underset{r=\widehat{T}_{j}}{\sum}} \widehat{E}_{1r}\big)\|_2.
\end{eqnarray*}
Each component of $\lambda_1\sum^{T^\ast_j-1}_{r=\widehat{T}_{j}}\widehat{E}_{1r}$ is bounded by $\pm \lambda_1(T^\ast_j-\widehat{T}_j)$. We deduce by the triangle inequality:
\begin{eqnarray}
\lefteqn{2 \lambda_2 + \lambda_1\sqrt{p(p-1)}(T^\ast_j-\widehat{T}_j) }\nonumber\\
& \geq & \|\frac{1}{T}\overset{T^\ast_j-1}{\underset{r=\widehat{T}_j}{\sum}}\Big[\Lambda(X_rX^\top_r)\text{vec}(\Theta^\ast_r)-\text{vec}(I_p)\Big] + \frac{1}{T}\overset{T^\ast_j-1}{\underset{r=\widehat{T}_j}{\sum}} \Lambda(X_rX^\top_r)\text{vec}(\widehat{\Theta}_{r}-\Theta^\ast_r)\|_2\nonumber\\
& = & \|\frac{1}{T}\overset{T^\ast_j-1}{\underset{r=\widehat{T}_j}{\sum}}\Big[\Lambda(X_rX^\top_r)\text{vec}(\Omega^\ast_j)-\text{vec}(I_p)\Big] + \frac{1}{T}\overset{T^\ast_j-1}{\underset{r=\widehat{T}_j}{\sum}} \Lambda(X_rX^\top_r)\text{vec}(\widehat{\Omega}_{j+1}-\Omega^\ast_j)\|_2\nonumber\\
& \geq & \|\frac{1}{T}\overset{T^\ast_j-1}{\underset{r=\widehat{T}_j}{\sum}} \Lambda(X_rX^\top_r)\text{vec}(\Omega^\ast_{j+1}-\Omega^\ast_j)\|_2 - \|\frac{1}{T}\overset{T^\ast_j-1}{\underset{r=\widehat{T}_j}{\sum}} \Lambda(X_rX^\top_r)\text{vec}(\widehat{\Omega}_{j+1}-\Omega^\ast_{j+1})\|_2 \nonumber\\
& & - \|\frac{1}{T}\overset{T^\ast_j-1}{\underset{r=\widehat{T}_j}{\sum}}\Big[\frac{1}{2}\Omega^\ast_j X_rX^\top_r+ \frac{1}{2}X_rX^\top_r\Omega^\ast_j-I_p\Big]\|_F\coloneqq R_{Tj,1}+R_{Tj,2}+R_{Tj,3}, \label{bound_optimality}
\end{eqnarray}
where the first equality holds since $\widehat{\Theta}_r = \widehat{\Omega}_{j+1}$ and $\Theta^\ast_r=\Omega^\ast_j$ for $r \in [\widehat{T}_j,T^\ast_j-1]$ by (\ref{c_T_def}).
Let the event:
\begin{equation*}
\overline{R}_{Tj} = \{2 \lambda_2 + \lambda_1\sqrt{p(p-1)}(T^\ast_j-\widehat{T}_j) \geq \frac{1}{3}R_{Tj,1}\}\cup \{R_{Tj,2}\geq \frac{1}{3}R_{Tj,1}\}\cup \{R_{Tj,3}\geq \frac{1}{3}R_{Tj,1}\}.
\end{equation*}
Since inequality (\ref{bound_optimality}) holds with probability one, then $\Pb(\overline{R}_{Tj})=1$. Therefore, we have:
$$
\begin{aligned}
  \Pb(A^+_{T,j}\cap C_T) \leq & \:\: \Pb(A^+_{T,j}\cap C_T\cap \{2 \lambda_2 + \lambda_1\sqrt{p(p-1)}(T^\ast_j-\widehat{T}_j) \geq \frac{1}{3}R_{Tj,1}\}) \\
                        & \:\quad + \Pb(A^+_{T,j}\cap C_T\cap \{R_{Tj,2}\geq \frac{1}{3}R_{Tj,1}\}) + \Pb(A^+_{T,j}\cap C_T\cap \{R_{Tj,3}\geq \frac{1}{3}R_{Tj,1}\}) \\
  \eqqcolon & \:\: AC_{j,1}+AC_{j,2}+AC_{j,3}.
\end{aligned}
$$
Let us first bound $\sum^{m^\ast}_{j=1}AC_{j,1}$. Using Section XIV.7 of \cite{pease1965} on the Kronecker sum, we have $\lambda_{\min}(\frac{1}{T^\ast_j-\widehat{T}_j}\sum^{T^\ast_j-1}_{r=\widehat{T}_j} \Lambda(X_rX^\top_r)= \lambda_{\min}(\frac{1}{T^\ast_j-\widehat{T}_j}\sum^{T^\ast_j-1}_{r=\widehat{T}_j}  X_rX^\top_r)$, for $1 \leq j \leq m^\ast$:
$$
	\begin{aligned}
		     & AC_{j,1} \leq \:\: \Pb(A^+_{T,j}\cap \{2 \lambda_2 + \lambda_1\sqrt{p(p-1)}(T^\ast_j-\widehat{T}_j) \geq \frac{1}{3}R_{Tj,1}\})                                                                                                                                                                                                                                              \\
		\leq & \:\: \resizebox{0.95\textwidth}{!}{$\displaystyle\Pb\Big(\|\frac{1}{T^\ast_j-\widehat{T}_j}\overset{T^\ast_j-1}{\underset{r=\widehat{T}_j}{\sum}} \Lambda(X_rX^\top_r)\text{vec}(\Omega^\ast_{j+1}-\Omega^\ast_j)\|_2 \leq \frac{3T}{T^\ast_j-\widehat{T}_j} \big[2 \lambda_2 + \lambda_1\sqrt{p(p-1)}(T^\ast_j-\widehat{T}_j)\big],T^\ast_j-\widehat{T}_j \geq T\delta_T\Big)$} \\
		\leq & \:\: \Pb\Big(\gamma^{\min}_{1,T,j}\|\Omega^\ast_{j+1}-\Omega^\ast_j\|_F \leq \frac{6T\lambda_2}{T^\ast_j-\widehat{T}_j} + 3T\lambda_1\sqrt{p(p-1)},T^\ast_j-\widehat{T}_j \geq T\delta_T\Big)                                                                                                                                                                                \\
		\leq & \:\: \Pb\Big(\gamma^{\min}_{1,T,j} \leq \frac{6\lambda_2}{\eta_{\min}\delta_T} + \frac{T\lambda_1\sqrt{p(p-1)}}{\eta_{\min}},T^\ast_j-\widehat{T}_j \geq T\delta_T\Big),
	\end{aligned}
$$
with $\gamma^{\min}_{1,T,j}=\lambda_{\min}(\frac{1}{T^\ast_j-\widehat{T}_j}\sum^{T^\ast_j-1}_{r=\widehat{T}_j} X_rX^\top_r)\geq \underline{\mu}/2>0$ with probability tending to one by Lemma \ref{emp_var_cov_bound}, and $\eta_{\min}=\underset{1 \leq j \leq m^\ast}{\min}\|\Omega^\ast_{j+1}-\Omega^\ast_j\|_F$. By $\lambda_2/(\eta_{\min}\delta_T)\rightarrow 0$, $T\lambda_1\sqrt{p(p-1)}/\eta_{\min}\rightarrow 0$ in Assumption~3-(iii), we deduce $\sum^{m^\ast}_{j=1}AC_{j,1} \rightarrow 0$. We now bound $\sum^{m^\ast}_{j=1}AC_{j,2}$. \rev{First, note that $\lambda_{\max}(\frac{1}{T^\ast_j-\widehat{T}_j}\sum^{T^\ast_j-1}_{r=\widehat{T}_j} \Lambda(X_rX^\top_r)= \lambda_{\max}(\frac{1}{T^\ast_j-\widehat{T}_j}\sum^{T^\ast_j-1}_{r=\widehat{T}_j}  X_rX^\top_r)$.}
For any $j=1,\ldots,m^\ast$:
$$
\begin{aligned}
  AC_{j, 2} = & \resizebox{0.91\textwidth}{!}{$\displaystyle\Pb\Big(\!A^+_{T,j} \!\cap\! C_T \!\cap\! \Big\{\!\|\!\frac{1}{T^\ast_j\!-\!\widehat{T}_j}\!\overset{T^\ast_j\!-\!1}{\underset{r\!=\!\widehat{T}_j}{\sum}}\! \Lambda(X_r\!X^\top_r\!)\text{vec}(\widehat{\Omega}_{j+1}\!-\!\Omega^\ast_{j+1})\!\|_2 \!\geq\! \frac{1}{3}\!\|\!\frac{1}{T^\ast_j\!-\!\widehat{T}_j}\!\overset{T^\ast_j\!-\!1}{\underset{r\!=\!\widehat{T}_j}{\sum}}\! \Lambda(X_r\!X^\top_r)\!\text{vec}(\Omega^\ast_{j+1}\!-\!\Omega^\ast_j)\!\|_2\!\Big\}\!\Big)$} \\
  \leq & \Pb\Big(A^+_{T,j}\cap C_T \cap \Big\{\gamma^{\max}_{T,j} \|\widehat{\Omega}_{j+1}-\Omega^\ast_{j+1}\|_F \geq \frac{1}{3}\gamma^{\min}_{1,T,j}\|\Omega^\ast_{j+1}-\Omega^\ast_j\|_F\}\Big),
\end{aligned}
$$
with $\gamma^{\max}_{T,j}=\lambda_{\max}(\frac{1}{T^\ast_j-\widehat{T}_j}\sum^{T^\ast_j-1}_{r=\widehat{T}_j} X_rX^\top_r) \leq 2 \overline{\mu}$ with probability tending to one by Lemma \ref{emp_var_cov_bound}. We now need to evaluate the bound for $\|\widehat{\Omega}_{j+1}-\Omega^\ast_{j+1}\|_F$. To do so, we rely on the KKT conditions of Lemma \ref{optimality_cond}. Note that with probability tending to one, $\lambda_{\max}(\frac{1}{T^\ast_j-\widehat{T}_j}\sum^{T^\ast_j-1}_{r=\widehat{T}_j} X_rX^\top_r) \leq 2 \overline{\mu}$. We have $\widehat{\Theta}_t = \widehat{\Omega}_{j+1}$ when $t \in [T^\ast_j,(T^\ast_j+T^\ast_{j+1})/2-1]$ as $\widehat{T}_j<T^\ast_j$ given $A^+_{T,j}$ and $\widehat{T}_{j+1}>(T^\ast_j+T^\ast_{j+1})/2$ given $C_T$. Therefore, by Lemma \ref{optimality_cond} with $l = (T^\ast_j+T^\ast_{j+1})/2$ and $l=T^\ast_j$, following the steps to obtain inequality (\ref{bound_optimality}), we get
\begin{eqnarray*}
\lefteqn{2 \lambda_2 + \lambda_1\sqrt{p(p-1)}[(T^\ast_j+T^\ast_{j+1})/2-T^\ast_j]}\\
& \geq & \resizebox{0.91\textwidth}{!}{$\displaystyle\|\frac{1}{T}\overset{(T^\ast_j+T^\ast_{j+1})/2-1}{\underset{r=T^\ast_j}{\sum}} \Lambda(X_rX^\top_r)\text{vec}(\widehat{\Omega}_{j+1}-\Omega^\ast_{j+1})\|_2 - \|\frac{1}{T}\sum^{(T^\ast_j+T^\ast_{j+1})/2-1}_{r=T^\ast_j}\Big[\frac{1}{2}\Omega^\ast_{j} X_rX^\top_r+ \frac{1}{2}X_rX^\top_r\Omega^\ast_{j}-I_p\Big]\|_F$}.
\end{eqnarray*}
Therefore, denoting by $\gamma^{\min}_{2,T,j} = \lambda_{\min}\Big(\frac{1}{T^\ast_{j+1}-T^\ast_j}\sum^{(T^\ast_j+T^\ast_{j+1})/2-1}_{r=T^\ast_j}X_rX^\top_r\Big)$, conditional on $C_T$:
$$
	\begin{aligned}
		\|\widehat{\Omega}_{j+1}-\Omega^\ast_{j+1}\|_F \leq & \:\: (\gamma^{\min}_{2,T,j})^{-1}\Big(\frac{2 T\lambda_2 + T\lambda_1\sqrt{p(p-1)}[(T^\ast_j+T^\ast_{j+1})/2-T^\ast_j]}{T^\ast_{j+1}-T^\ast_{j}}                                                     \\
		                                                    & \:\: \quad + \|\frac{2}{T^\ast_{j+1}-T^\ast_{j}}\sum^{(T^\ast_j+T^\ast_{j+1})/2-1}_{r=T^\ast_j}\Big[\frac{1}{2}\Omega^\ast_{j} X_rX^\top_r+ \frac{1}{2}X_rX^\top_r\Omega^\ast_{j}-I_p\Big]\|_F\Big).
	\end{aligned}
$$
By Lemma \ref{emp_var_cov_bound}, $\gamma^{\min}_{2,T,j}\geq \underline{\mu}/2>0$ with probability tending to one. We deduce
\begin{eqnarray}
\lefteqn{\sum^{m^\ast}_{j=1} \Pb\Big(\Big\{\|\widehat{\Omega}_{j+1}-\Omega^\ast_{j+1}\|_F \geq (\gamma^{\max}_{T,j})^{-1}\gamma^{\min}_{1,T,j}\|\Omega^\ast_{j+1}-\Omega^\ast_j\|_F/3\Big\}\cap C_T\Big)}\nonumber\\
& \leq & \resizebox{0.92\linewidth}{!}{$\displaystyle\sum^{m^\ast}_{j=1} \Pb\Big(\frac{2 T\lambda_2 + T\lambda_1\sqrt{p(p-1)}[(T^\ast_j+T^\ast_{j+1})/2-T^\ast_j]}{T^\ast_{j+1}-T^\ast_{j}} \geq (\gamma^{\max}_{T,j})^{-1}\gamma^{\min}_{1,T,j}\gamma^{\min}_{2,T,j}\|\Omega^\ast_{j+1}-\Omega^\ast_j\|_F/6\Big)$}\nonumber\\
& & \resizebox{0.92\linewidth}{!}{$\displaystyle+  \sum^{m^\ast}_{j=1} \Pb\Big(\|\frac{2}{T^\ast_{j+1}-T^\ast_{j}}\sum^{(T^\ast_j+T^\ast_{j+1})/2-1}_{r=T^\ast_j}\Big[\frac{1}{2}\Omega^\ast_j X_rX^\top_r+ \frac{1}{2}X_rX^\top_r\Omega^\ast_j-I_p\Big]\|_F \geq  (\gamma^{\max}_{T,j})^{-1}\gamma^{\min}_{1,T,j}\gamma^{\min}_{2,T,j}\|\Omega^\ast_{j+1}-\Omega^\ast_j\|_F/6\Big)$}\nonumber\\
& \leq &  \sum^{m^\ast}_{j=1} \Pb\Big(\frac{2 T\lambda_2}{\mathcal{I}_{\min}\eta_{\min}} + \frac{ T\lambda_1\sqrt{p(p-1)}}{\eta_{\min}} \geq (\gamma^{\max}_{T,j})^{-1}\gamma^{\min}_{1,T,j}\gamma^{\min}_{2,T,j}/6\Big)\label{control_sum}\\
& & \resizebox{0.92\linewidth}{!}{$\displaystyle+ \sum^{m^\ast}_{j=1} \Pb\Big(\|\frac{1}{(T^\ast_{j+1}-T^\ast_{j})/2}\sum^{(T^\ast_j+T^\ast_{j+1})/2-1}_{r=T^\ast_j}\Big[\frac{1}{2}\Omega^\ast_{j} X_rX^\top_r+ \frac{1}{2}X_rX^\top_r\Omega^\ast_{j}-I_p\Big]\|_F \geq  \overline{\mu}^{-1}\underline{\mu}^2\eta_{\min}/48\Big)$}. \nonumber
\end{eqnarray}
The first term tends to zero since $T\lambda_2/(\mathcal{I}_{\min}\eta_{\min}) \rightarrow 0$ and $T\lambda_1p/\eta_{\min}\rightarrow 0$ by Assumption~3-(ii) and (iii). As for the second term, using $\|AB\|_F \leq \|A\|_s\|B\|_F$, note that
\begin{eqnarray*}
\lefteqn{\|\frac{1}{(T^\ast_{j+1}-T^\ast_{j})/2}\sum^{(T^\ast_j+T^\ast_{j+1})/2-1}_{r=T^\ast_j}\Big[\frac{1}{2}\Omega^\ast_{j} X_rX^\top_r+ \frac{1}{2}X_rX^\top_r\Omega^\ast_{j}-I_p\Big]\|_F}\\
& = & \|\frac{1}{(T^\ast_{j+1}-T^\ast_{j})/2}\sum^{(T^\ast_j+T^\ast_{j+1})/2-1}_{r=T^\ast_j}\Big[\frac{1}{2}\Omega^\ast_{j}\Big(X_rX^\top_r-\Sigma^\ast_{j}\Big)+\frac{1}{2}\Big(X_rX^\top_r-\Sigma^\ast_{j}\Big)\Omega^\ast_{j}\Big]\|_F\\
& \leq & s^\ast_{\max}\,p\,\|\frac{1}{(T^\ast_{j+1}-T^\ast_{j})/2}\sum^{(T^\ast_j+T^\ast_{j+1})/2-1}_{r=T^\ast_j} \big(X_rX^\top_r-\Sigma^\ast_{j}\big)\|_{\max}.
\end{eqnarray*}
Therefore, for $C>0$ finite, applying Lemma \ref{bound_gradient}, we deduce that for any $j$:
\begin{equation*}
\Pb\Big(\|\frac{1}{(T^\ast_{j+1}-T^\ast_{j})/2}\sum^{(T^\ast_j+T^\ast_{j+1})/2-1}_{r=T^\ast_j}\Big[\frac{1}{2}\Omega^\ast_{j} X_rX^\top_r+ \frac{1}{2}X_rX^\top_r\Omega^\ast_{j}-I_p\Big]\|_F \geq  \overline{\mu}^{-1}\underline{\mu}^2\eta_{\min}/48\Big) \rightarrow 0,
\end{equation*}
since $(\eta_{\min}\mathcal{I}^{1/2}_{\min})^{-1}s^\ast_{\max}\,p\sqrt{\log(pT)}\rightarrow 0$. Hence, $\sum^{m^\ast}_{j=1}AC_{j,2} \rightarrow 0$. Let us now consider $\sum^{m^\ast}_{j=1}AC_{j,3}$. Applying the same reasoning to show the convergence of the second summation on the right-hand side of (\ref{control_sum}), we get
\begin{equation*}
\|\frac{1}{T^\ast_{j}-\widehat{T}_{j}}\sum^{T^\ast_j-1}_{r=\widehat{T}_j}\Big[\frac{1}{2}\Omega^\ast_{j} X_rX^\top_r+ \frac{1}{2}X_rX^\top_r\Omega^\ast_{j}-I_p\Big]\|_F = O_p(p\,s^\ast_{\max}\sqrt{\frac{\log(pT)}{T\delta_T}}) = o_p(\eta_{\min}),
\end{equation*}
when $T^\ast_j-\widehat{T}_j\geq T\delta_T$,
and
\begin{eqnarray*}
\lefteqn{\sum^{m^\ast}_{j=1}AC_{j,3} \leq \Pb(A^+_{T,j}\cap \{R_{Tj,3}\geq \frac{1}{3}R_{Tj,1}\})} \\
& \leq & \resizebox{0.92\linewidth}{!}{$\displaystyle \sum^{m^\ast}_{j=1}\!\Pb\Big(A^+_{T,j} \!\cap\! \Big\{\|\frac{1}{T^\ast_{j}\!-\!\widehat{T}_{j}}\sum^{T^\ast_j-1}_{r=\widehat{T}_j}\!\Big[\frac{1}{2}\Omega^\ast_{j} X_rX^\top_r\!+\! \frac{1}{2}X_rX^\top_r\Omega^\ast_{j}\!-\!I_p\!\Big]\|_F\!\geq\! \frac{1}{3}\|\frac{1}{T^\ast_j-\widehat{T}_j}\overset{T^\ast_j-1}{\underset{r=\widehat{T}_j}{\sum}}\! \Lambda(X_rX^\top_r)\text{vec}(\Omega^\ast_{j+1}\!-\!\Omega^\ast_j)\|_2\Big\}\Big)$}\\
& \leq &  \sum^{m^\ast}_{j=1} \Pb\Big(A^+_{T,j} \cap \Big\{\|\frac{1}{T^\ast_{j}-\widehat{T}_{j}}\sum^{T^\ast_j-1}_{r=\widehat{T}_j} \Big[\frac{1}{2}\Omega^\ast_{j} X_rX^\top_r+ \frac{1}{2}X_rX^\top_r\Omega^\ast_{j}-I_p\Big]\|_F\geq \frac{1}{3}\eta_{\min}\gamma^{\min}_{1,T,j}\Big),
\end{eqnarray*}
since $\gamma^{\min}_{1,T,j} \geq \underline{\mu}/2>0$ with probability tending to one, and $T\delta_T \leq T^\ast_j-\widehat{T}_j$, then under $(\sqrt{T\delta_T}\eta_{\min})^{-1}s^\ast_{\max}\,p\sqrt{\log(pT)} \rightarrow 0$, we deduce $\sum^{m^\ast}_{j=1}AC_{j,3} \rightarrow 0$. Consequently, we proved $\sum^{m^\ast}_{j=1}\Pb(A_{T,j} \cap C_T) \rightarrow 0$.

\noindent\emph{Proof of (b).} We prove (b) by showing $\sum^{m^\ast}_{j=1}\Pb(A^+_{T,j} \cap C^c_T) \rightarrow 0$ and $\sum^{m^\ast}_{j=1}\Pb(A^-_{T,j} \cap C^c_T) \rightarrow 0$. As in the proof of (a), we simply show $\sum^{m^\ast}_{j=1}\Pb(A^+_{T,j} \cap C^c_T) \rightarrow 0$. To do so, we define:
\begin{equation*}
\begin{array}{llll}
D^{(l)}_T\coloneqq \big\{\exists j \in \{1,\ldots,m^\ast\}, \widehat{T}_j \leq T^\ast_{j-1}\big\} \cap C^c_T,&&\\
D^{(m)}_T\coloneqq \big\{\forall j \in \{1,\ldots,m^\ast\}, T^\ast_{j-1}<\widehat{T}_j < T^\ast_{j+1}\big\} \cap C^c_T,&&\\
D^{(r)}_T\coloneqq \big\{\exists j \in \{1,\ldots,m^\ast\}, \widehat{T}_j \geq T^\ast_{j+1}\big\} \cap C^c_T,&&
\end{array}
\end{equation*}
where $C^c_T = \{\underset{1 \leq j \leq m^\ast}{\max}|\widehat{T}_j-T^\ast_j|\geq \mathcal{I}_{\min}/2\}$. Then, we have:
\begin{equation*}
\sum^{m^\ast}_{j=1}\Pb(A^+_{T,j} \cap C^c_T) = \sum^{m^\ast}_{j=1}\Big[\Pb(A^+_{T,j} \cap D^{(l)}_T)+\Pb(A^+_{T,j} \cap D^{(m)}_T) +\Pb(A^+_{T,j} \cap D^{(r)}_T)\Big].
\end{equation*}
We first bound $\sum^{m^\ast}_{j=1}\Pb(A^+_{T,j} \cap D^{(m)}_T)$. For any $j$:
\begin{eqnarray*}
\lefteqn{\Pb(A^+_{T,j} \cap D^{(m)}_T)}\\
& \leq & \Pb(A^+_{t,j} \cap \big\{\widehat{T}_{j+1}-T^\ast_j \geq \frac{1}{2}\mathcal{I}_{\min}\big\}\cap D^{(m)}_T) + \Pb(A^+_{t,j} \cap \big\{\widehat{T}_{j+1}-T^\ast_j < \frac{1}{2}\mathcal{I}_{\min}\big\}\cap D^{(m)}_T)\\
& \leq & \Pb(A^+_{t,j} \cap \big\{\widehat{T}_{j+1}-T^\ast_j \geq \frac{1}{2}\mathcal{I}_{\min}\big\}\cap D^{(m)}_T) + \Pb(A^+_{t,j} \cap \big\{T^\ast_{j+1}-\widehat{T}_{j+1} \geq \frac{1}{2}\mathcal{I}_{\min}\big\}\cap D^{(m)}_T),
\end{eqnarray*}
since $0 \leq \widehat{T}_{j+1}-T^\ast_j \leq \mathcal{I}_{\min}/2$ implies $T^\ast_{j+1}-\widehat{T}_{j+1}=(T^\ast_{j+1}-T^\ast_j)-(\widehat{T}_{j+1}-T^\ast_j)\geq \mathcal{I}_{\min}-\mathcal{I}_{\min}/2 = \mathcal{I}_{\min}/2$. Moreover, since
\begin{equation*}
\resizebox{\linewidth}{!}{$\displaystyle\Big\{A^+_{t,j} \cap \big\{T^\ast_{j+1}-\widehat{T}_{j+1} \geq \frac{1}{2}\mathcal{I}_{\min}\big\}\cap D^{(m)}_T\Big\} \subset \overset{m^\ast-1}{\underset{k=j+1}{\cup}}\Big[ \big\{T^\ast_k-\widehat{T}_k\geq \mathcal{I}_{\min}/2\big\} \cap \big\{\widehat{T}_{k+1}-T^\ast_k \geq \mathcal{I}_{\min}/2\big\}\cap D^{(m)}_T\Big],$}
\end{equation*}
we deduce:
\begin{eqnarray}\label{control_Dm}
\lefteqn{\sum^{m^\ast}_{j=1}\Pb(A^+_{T,j} \cap D^{(m)}_T)  \leq \sum^{m^\ast}_{j=1} \Pb(A^+_{t,j} \cap \big\{\widehat{T}_{j+1}-T^\ast_j \geq \frac{1}{2}\mathcal{I}_{\min}\big\}\cap D^{(m)}_T)  } \notag\\
&  & + \sum^{m^\ast}_{j=1} \sum^{m^\ast-1}_{k=j+1}  \Pb\Big(\big\{T^\ast_k-\widehat{T}_k\geq \mathcal{I}_{\min}/2\big\} \cap \big\{\widehat{T}_{k+1}-T^\ast_k \geq \mathcal{I}_{\min}/2\big\}\cap D^{(m)}_T\Big).
\end{eqnarray}
Let us treat the first term. By Lemma \ref{optimality_cond} with $t=\widehat{T}_j$ and $t=T^\ast_j$, we obtain:
\begin{equation*}
\frac{1}{T}\overset{T}{\underset{r=\widehat{T}_j}{\sum}}\Big[\Lambda(X_rX^\top_r)\text{vec}(\Theta^\ast_r+\widehat{\Theta}_r-\Theta^\ast_r)-\text{vec}(I_p)\Big] + \lambda_1\text{vec}\big(\overset{T}{\underset{r=\widehat{T}_j}{\sum}} \widehat{E}_{1r}\big)= \lambda_2 \text{vec}\big(\frac{\widehat{\Gamma}_{\widehat{T}_j}}{\|\widehat{\Gamma}_{\widehat{T}_j}\|_F}\big),
\end{equation*}
and
\begin{equation*}
\|\frac{1}{T}\overset{T}{\underset{r=T^\ast_j}{\sum}}\Big[\Lambda(X_rX^\top_r)\text{vec}(\Theta^\ast_r+\widehat{\Theta}_r-\Theta^\ast_r)-\text{vec}(I_p)\Big] + \lambda_1\text{vec}\big(\overset{T}{\underset{r=\widehat{T}_j}{\sum}} \widehat{E}_{1r}\big)\|_2 \leq 2\lambda_2.
\end{equation*}
We deduce
\begin{eqnarray*}
\lefteqn{\gamma^{\min}_{1,T,j}\|\widehat{\Omega}_{j+1}-\Omega^\ast_j\|_F - \|\frac{1}{T^\ast_j-\widehat{T}_j}\overset{T^\ast_j-1}{\underset{r=\widehat{T}_j}{\sum}}\Big[\frac{1}{2}\Omega^\ast_{j} X_rX^\top_r+ \frac{1}{2}X_rX^\top_r\Omega^\ast_{j}-I_p\Big]\|_F }\\
& \leq & \frac{1}{T^\ast_j-\widehat{T}_j} \|\overset{T^\ast_j-1}{\underset{r=\widehat{T}_j}{\sum}}\Lambda(X_rX^\top_r)\text{vec}(\widehat{\Omega}_{j+1}-\Omega^\ast_{j})+\overset{T^\ast_j-1}{\underset{r=\widehat{T}_j}{\sum}}\Big[\Lambda(X_rX^\top_r)\text{vec}(\Omega^\ast_j)-\text{vec}(I_p)\Big]\|_2\\
& \leq & \frac{2\lambda_2T}{T^\ast_j-\widehat{T}_j}+\lambda_1T\sqrt{p(p-1)}.
\end{eqnarray*}
As a consequence:
\begin{equation}\label{bound_param_1}
\resizebox{\linewidth}{!}{$\displaystyle\|\widehat{\Omega}_{j+1}-\Omega^\ast_j\|_F \leq (\gamma^{\min}_{1,T,j})^{-1}\Big[\frac{2\lambda_2T}{T^\ast_j-\widehat{T}_j}+\lambda_1T\sqrt{p(p-1)} + \|\frac{1}{T^\ast_j-\widehat{T}_j}\overset{T^\ast_j-1}{\underset{r=\widehat{T}_j}{\sum}}\Big[\frac{1}{2}\Omega^\ast_{j} X_rX^\top_r+ \frac{1}{2}X_rX^\top_r\Omega^\ast_{j}-I_p\Big]\|_F\Big].$}
\end{equation}
In the same vein, applying Lemma \ref{optimality_cond} with $t = \widehat{T}_{j+1}$ and $t=T^\ast_j$, we obtain:
\begin{eqnarray}
\lefteqn{\|\widehat{\Omega}_{j+1}-\Omega^\ast_{j+1}\|_F}\label{bound_param_2}\\
  &\leq & \resizebox{0.92\linewidth}{!}{$\displaystyle(\gamma^{\min}_{3,T,j})^{-1}\Big[\frac{2\lambda_2T}{\widehat{T}_{j+1}-T^\ast_j}+\lambda_1T\sqrt{p(p-1)} + \|\frac{1}{\widehat{T}_{j+1}-T^\ast_j}\overset{\widehat{T}_{j+1}-1}{\underset{r=T^\ast_j}{\sum}}\Big[\frac{1}{2}\Omega^\ast_{j} X_rX^\top_r+ \frac{1}{2}X_rX^\top_r\Omega^\ast_{j}-I_p\Big]\|_F\Big],$}\nonumber
\end{eqnarray}
where $\gamma^{\min}_{3,T,j}=\lambda_{\min}(\frac{1}{\widehat{T}_{j+1}-T^\ast_j}\overset{\widehat{T}_{j+1}-1}{\underset{r=T^\ast_j}{\sum}}X_rX^\top_r) \geq \underline{\mu}/2$ with probability one.  Let the event:
\begin{eqnarray}
\lefteqn{E_{T,j} \coloneqq \Big\{\|\Omega^\ast_{j+1}-\Omega^\ast_j\|_F \leq (\gamma^{\min}_{1,T,j})^{-1}\Big[\frac{2\lambda_2T}{T^\ast_j-\widehat{T}_j}+\lambda_1T\sqrt{p(p-1)}\Big]}\label{E_event}\\
& & \resizebox{0.92\linewidth}{!}{$\displaystyle+(\gamma^{\min}_{3,T,j})^{-1}\Big[\frac{2\lambda_2T}{\widehat{T}_{j+1}-T^\ast_j}+\lambda_1T\sqrt{p(p-1)}\Big] + (\gamma^{\min}_{1,T,j})^{-1}\|\frac{1}{T^\ast_j-\widehat{T}_j}\overset{T^\ast_j-1}{\underset{r=\widehat{T}_j}{\sum}}\Big[\frac{1}{2}\Omega^\ast_{j} X_rX^\top_r+ \frac{1}{2}X_rX^\top_r\Omega^\ast_{j}-I_p\Big]\|_F$} \nonumber\\
& & + (\gamma^{\min}_{3,T,j})^{-1}\|\frac{1}{\widehat{T}_{j+1}-T^\ast_j}\overset{\widehat{T}_{j+1}-1}{\underset{r=T^\ast_j}{\sum}}\Big[\frac{1}{2}\Omega^\ast_{j} X_rX^\top_r+ \frac{1}{2}X_rX^\top_r\Omega^\ast_{j}-I_p\Big]\|_F\Big\}. \nonumber
\end{eqnarray}
Therefore, by the triangle inequality, (\ref{bound_param_1}) and (\ref{bound_param_2}) imply that the event $E_{T,j}$ holds with probability one. Hence:
\begin{eqnarray}
\lefteqn{\sum^{m^\ast}_{j=1} \Pb(A^+_{t,j} \cap \big\{\widehat{T}_{j+1}-T^\ast_j \geq \frac{1}{2}\mathcal{I}_{\min}\big\}\cap D^{(m)}_T)}\nonumber\\
& = & \sum^{m^\ast}_{j=1} \Pb(E_{T,j}\cap A^+_{t,j} \cap \big\{\widehat{T}_{j+1}-T^\ast_j \geq \frac{1}{2}\mathcal{I}_{\min}\big\}\cap D^{(m)}_T)\nonumber\\
& \leq & \sum^{m^\ast}_{j=1} \Pb(E_{T,j}\cap \big\{T^\ast_j - \widehat{T}_j> T \delta_T\big\} \cap  \big\{\widehat{T}_{j+1}-T^\ast_j \geq \frac{1}{2}\mathcal{I}_{\min}\big\}) \nonumber\\
& \leq & \resizebox{0.92\linewidth}{!}{$\displaystyle\sum^{m^\ast}_{j=1} \Pb(\gamma^{\min}_{1,T,j})^{-1}\Big[\frac{2\lambda_2}{\delta_T}+\lambda_1T\sqrt{p(p-1)}\Big] + (\gamma^{\min}_{3,T,j})^{-1}\Big[\frac{4\lambda_2T}{\mathcal{I}_{\min}}+\lambda_1T\sqrt{p(p-1)}\Big] \geq \|\Omega^\ast_{j+1}-\Omega^\ast_j\|_F/3)$} \nonumber\\
& & \resizebox{0.92\linewidth}{!}{$\displaystyle+ \sum^{m^\ast}_{j=1} \Pb(\Big\{(\gamma^{\min}_{1,T,j})^{-1}\|\frac{1}{T^\ast_j-\widehat{T}_j}\overset{T^\ast_j-1}{\underset{r=\widehat{T}_j}{\sum}}\Big[\frac{1}{2}\Omega^\ast_{j} X_rX^\top_r+ \frac{1}{2}X_rX^\top_r\Omega^\ast_{j}-I_p\Big]\|_F \geq \|\Omega^\ast_{j+1}-\Omega^\ast_j\|_F/3\Big\}\cap \big\{T^\ast_j-\widehat{T}_j > T \delta_T\big\})$} \nonumber\\
& & \resizebox{0.92\linewidth}{!}{$\displaystyle+ \sum^{m^\ast}_{j=1} \Pb(\Big\{(\gamma^{\min}_{3,T,j})^{-1}\|\frac{1}{\widehat{T}_{j+1}-T^\ast_j}\overset{\widehat{T}_{j+1}-1}{\underset{r=T^\ast_j}{\sum}}\Big[\frac{1}{2}\Omega^\ast_{j} X_rX^\top_r+ \frac{1}{2}X_rX^\top_r\Omega^\ast_{j}-I_p\Big]\|_F \geq \|\Omega^\ast_{j+1}-\Omega^\ast_j\|_F/3\Big\}$} \nonumber\\
& & \qquad \quad \:\: \cap \big\{\widehat{T}_{j+1}-T^\ast_j \geq \mathcal{I}_{\min}/2\big\}). \label{bound_A+}
\end{eqnarray}
The first term in (\ref{bound_A+}) tends to zero under $\lambda_2/(\eta_{\min}\delta_T)\rightarrow 0$, $\lambda_2T/(\mathcal{I}_{\min}\eta_{\min})\rightarrow 0$, and $\lambda_1Tp/\eta_{\min}\rightarrow 0$. Moreover, note that
\begin{equation*}
\|\frac{1}{T^\ast_j-\widehat{T}_j}\overset{T^\ast_j-1}{\underset{r=\widehat{T}_j}{\sum}}\Big[\frac{1}{2}\Omega^\ast_{j} X_rX^\top_r+ \frac{1}{2}X_rX^\top_r\Omega^\ast_{j}-I_p\Big]\|_F = O_p(s^\ast_{\max}\,p\sqrt{\frac{\log(pT)}{T\delta_T}}) = o_p(\eta_{\min}),
\end{equation*}
and
\begin{equation*}
\|\frac{1}{\widehat{T}_{j+1}-T^\ast_j}\overset{\widehat{T}_{j+1}-1}{\underset{r=T^\ast_j}{\sum}}\Big[\frac{1}{2}\Omega^\ast_{j} X_rX^\top_r+ \frac{1}{2}X_rX^\top_r\Omega^\ast_{j}-I_p\Big]\|_F = O_p(s^\ast_{\max}\,p\sqrt{\frac{\log(pT)}{\mathcal{I}_{\min}}}) = o_p(\eta_{\min}),
\end{equation*}
under Assumption~3-(ii)-(iii). In the same manner, we can show that the second term in (\ref{control_Dm}) tends to zero.

We now consider $\sum^{m^\ast}_{j=1}\Pb(A^+_{T,j} \cap D^{(l)}_T)$. The probability of the event $A^+_{T,j} \cap D^{(l)}_T$ is upper bounded by:
\begin{equation*}
\Pb(D^{(l)}_T)\leq \sum^{m^\ast}_{j=1}2^{j-1}\Pb(\max(l \in \{1,\ldots,m^\ast\}:\widehat{T}_l \leq T^\ast_{l-1})=j).
\end{equation*}
Now $\max(l \in \{1,\ldots,m^\ast\}:\widehat{T}_l \leq T^\ast_{l-1})=j$ implies $\widehat{T}_j\leq T^\ast_{j-1}$ and $\widehat{T}_{l+1}>T^\ast_l$ for any $j \leq l \leq m^\ast$ and:
\begin{equation*}
\big\{\max(l \in \{1,\ldots,m^\ast\}:\widehat{T}_l \leq T^\ast_{l-1})=j\big\}\subset \overset{m^\ast-1}{\underset{k=j}{\cup}} \Big(\big\{T^\ast_k-\widehat{T}_k\geq \mathcal{I}_{\min}/2\big\}\cap\big\{\widehat{T}_{k+1}-T^\ast_k\geq \mathcal{I}_{\min}/2\big\}\Big).
\end{equation*}
Therefore:
\begin{eqnarray}
\lefteqn{\sum^{m^\ast}_{j=1}\Pb(A^+_{T,j} \cap D^{(l)}_T)}\label{control_Dl}\\
  & \leq & \resizebox{0.92\linewidth}{!}{$\displaystyle m^\ast \overset{m^\ast-1}{\underset{j=1}{\sum}}2^{j-1}\overset{m^\ast-1}{\underset{k=j}{\sum}} \Pb\Big(\big\{T^\ast_k-\widehat{T}_k\geq \mathcal{I}_{\min}/2\big\}\cap\big\{\widehat{T}_{k+1}-T^\ast_k\geq \mathcal{I}_{\min}/2\big\}\Big) + m^\ast2^{m^\ast-1}\Pb(T^\ast_{m^\ast}-\widehat{T}_{m^\ast}\geq \mathcal{I}_{\min}/2).$}\nonumber
\end{eqnarray}
First, we consider the second term of the right-hand side of (\ref{control_Dl}). Let $j=m^\ast$ in (\ref{E_event}), then $E_{T,m^\ast}$ holds with probability one. Therefore:
\begin{eqnarray*}
\lefteqn{m^\ast2^{m^\ast-1}\Pb(T^\ast_{m^\ast}-\widehat{T}_{m^\ast}\geq \mathcal{I}_{\min}/2)=m^\ast2^{m^\ast-1}\Pb(E_{T,m^\ast}\cap\big\{T^\ast_{m^\ast}-\widehat{T}_{m^\ast}\geq \mathcal{I}_{\min}/2\big\})}\\
  & \leq & \resizebox{0.92\linewidth}{!}{$\displaystyle m^\ast2^{m^\ast-1}\Pb(\gamma^{\min}_{1,T,m^\ast})^{-1}\Big[\frac{2\lambda_2}{\delta_T}+\lambda_1T\sqrt{p(p-1)}\Big] + (\gamma^{\min}_{3,T,m^\ast})^{-1}\Big[\frac{4\lambda_2T}{\mathcal{I}_{\min}}+\lambda_1T\sqrt{p(p-1)}\Big]\geq \|\Omega^\ast_{m^\ast+1}-\Omega^\ast_{m^\ast}\|_F/3)$} \\
  & & \resizebox{0.92\linewidth}{!}{$\displaystyle+ m^\ast2^{m^\ast-1} \Pb(\gamma^{\min}_{1,T,m^\ast})^{-1}\|\frac{1}{T^\ast_{m^\ast}\!-\!\widehat{T}_{m^\ast}}\overset{T^\ast_{m^\ast}-1}{\underset{r=\widehat{T}_{m^\ast}}{\sum}}\!\Big[\frac{1}{2}\Omega^\ast_{m^\ast} X_rX^\top_r\!+\! \frac{1}{2}X_rX^\top_r\Omega^\ast_{m^\ast}\!-\!I_p\Big]\|_F \!\geq\! \|\Omega^\ast_{m^\ast+1}\!-\!\Omega^\ast_{m^\ast}\|_F/3,T^\ast_{m^\ast}\!-\!\widehat{T}_{m^\ast} \!\geq\! \mathcal{I}_{\min}/2)$} \\
  & & \resizebox{0.92\linewidth}{!}{$\displaystyle+ m^\ast2^{m^\ast-1}\Pb(\gamma^{\min}_{3,T,m^\ast})^{-1}\|\frac{1}{T-T^\ast_{m^\ast}}\overset{T}{\underset{r=T^\ast_{m^\ast}}{\sum}}\Big[\frac{1}{2}\Omega^\ast_{m^\ast} X_rX^\top_r+ \frac{1}{2}X_rX^\top_r\Omega^\ast_{m^\ast}-I_p\Big]\|_F \geq \|\Omega^\ast_{m^\ast+1}-\Omega^\ast_{m^\ast}\|_F/3).$}
\end{eqnarray*}
Since $m^\ast2^{m^\ast-1} = O(T\log(T)$, then $\log(m^\ast2^{m^\ast-1}) = O(\log(T^{1+\eps/2})$. So under the conditions $(\sqrt{T\delta_T}\eta_{\min})^{-1}s^\ast_{\max}\,p\sqrt{\log(pT)} \rightarrow 0, (\mathcal{I}^{1/2}_{\min}\eta_{\min})^{-1}s^\ast_{\max}\,p\sqrt{\log(pT)} \rightarrow 0$, the right-hand side of the previous inequality converges to zero. As for the first term of (\ref{control_Dl}), applying $j=k$ in (\ref{E_event}):
\begin{eqnarray*}
\lefteqn{m^\ast \overset{m^\ast-1}{\underset{j=1}{\sum}}2^{j-1}\overset{m^\ast-1}{\underset{k=j}{\sum}} \Pb\Big(\big\{T^\ast_k-\widehat{T}_k\geq \mathcal{I}_{\min}/2\big\}\cap\big\{\widehat{T}_{k+1}-T^\ast_k\geq \mathcal{I}_{\min}/2\big\}\Big)}\\
& \leq & m^\ast 2^{m^\ast-1}\overset{m^\ast-1}{\underset{k=1}{\sum}} \Pb\Big(E_{T,k}\cap\big\{T^\ast_k-\widehat{T}_k\geq \mathcal{I}_{\min}/2\big\}\cap\big\{\widehat{T}_{k+1}-T^\ast_k\geq \mathcal{I}_{\min}/2\big\}\Big)\\
& \leq & \resizebox{0.92\linewidth}{!}{$\displaystyle m^\ast 2^{m^\ast-1}\overset{m^\ast-1}{\underset{k=1}{\sum}} \Big\{\Pb(\gamma^{\min}_{1,T,k})^{-1}\Big[\frac{2\lambda_2}{\delta_T}+\lambda_1T\sqrt{p(p-1)}\Big] + (\gamma^{\min}_{3,T,k})^{-1}\Big[\frac{4\lambda_2T}{\mathcal{I}_{\min}}+\lambda_1T\sqrt{p(p-1)}\Big]\geq \|\Omega^\ast_{k+1}-\Omega^\ast_{k}\|_F/3)$} \\
& & \resizebox{0.92\linewidth}{!}{$\displaystyle+ \Pb(\gamma^{\min}_{1,T,k})^{-1}\|\frac{1}{T^\ast_{k}-\widehat{T}_{k}}\overset{T^\ast_{k}-1}{\underset{r=\widehat{T}_{k}}{\sum}}\Big[\frac{1}{2}\Omega^\ast_{k} X_rX^\top_r+ \frac{1}{2}X_rX^\top_r\Omega^\ast_{k}-I_p\Big]\|_F \geq \|\Omega^\ast_{k+1}-\Omega^\ast_{k}\|_F/3,T^\ast_{k}-\widehat{T}_{k} \geq \mathcal{I}_{\min}/2)$} \\
&& \resizebox{0.92\linewidth}{!}{$\displaystyle+ \Pb(\gamma^{\min}_{3,T,k})^{-1}\|\frac{1}{\widehat{T}_{k+1}-T^\ast_{k}}\overset{\widehat{T}_{k+1}-1}{\underset{r=T^\ast_{k}}{\sum}}\Big[\frac{1}{2}\Omega^\ast_{k} X_rX^\top_r+ \frac{1}{2}X_rX^\top_r\Omega^\ast_{k}-I_p\Big]\|_F \geq \|\Omega^\ast_{k+1}-\Omega^\ast_{k}\|_F/3,\widehat{T}_{k+1}-T^\ast_{k}\geq \mathcal{I}_{\min}/2)\Big\}.$}
\end{eqnarray*}
The right-hand side of the last inequality converges to zero under the same conditions. Finally, we can prove that $\sum^{m^\ast}_{j=1}\Pb(A^+_{T,j} \cap D^{(r)}_T) \rightarrow 0$.

\vspace*{0.3cm}

\noindent\textbf{\emph{Proof of point (ii).}}\\
\noindent By point (i) and under Assumption~3-(ii), for any $j=1,\ldots,m^\ast$, $|\widehat{T}_j-T^\ast_j|=O_p(T\delta_T)$, which is $|\widehat{T}_j-T^\ast_j| = o_p(\mathcal{I}_{\min})$ under Assumption~3-(ii). Hence, $(T^\ast_{j-1}+T^\ast_j)/2 < \widehat{T}_j < T^\ast_j$ or $T^\ast_j \leq \widehat{T}_j < (T^\ast_j + T^\ast_{j+1})/2$ is satisfied for any $j$. Set $l = 1,\ldots,m^\ast$ and assume $(T^\ast_{l-1}+T^\ast_l)/2 < \widehat{T}_l < T^\ast_l$ and consider two cases: (ii-a) $(T^\ast_{l}+T^\ast_{l+1})/2 < \widehat{T}_{l+1} < T^\ast_{l+1}$ and (ii-b) $T^\ast_{l+1} \leq \widehat{T}_{l+1}$. In case (ii-a), by Lemma \ref{optimality_cond} with change-points $t = \widehat{T}_l$ and $t = \widehat{T}_{l+1}$:
\begin{eqnarray*}
\lefteqn{2\lambda_2\geq \|\frac{1}{T}\overset{T}{\underset{r = \widehat{T}_{l}}{\sum}} \Big[\Lambda(X_rX^\top_r)\text{vec}(\widehat{\Theta}_r)-\text{vec}(I_p)\Big]}\\
& & -\frac{1}{T}\overset{T}{\underset{r = \widehat{T}_{l+1}}{\sum}} \Big[\Lambda(X_rX^\top_r)\text{vec}(\widehat{\Theta}_r)-\text{vec}(I_p)\Big] + \text{vec}\big(\lambda_1\overset{T}{\underset{r=\widehat{T}_l}{\sum}} \widehat{E}_{1r}-\lambda_1\overset{T}{\underset{r=\widehat{T}_{l+1}}{\sum}} \widehat{E}_{1r}\big)\|_2\\
& = & \|\frac{1}{T}\overset{\widehat{T}_{l+1}-1}{\underset{r = \widehat{T}_l}{\sum}} \Big[\Lambda(X_rX^\top_r)\text{vec}(\widehat{\Theta}_r)-\text{vec}(I_p)\Big] + \text{vec}\big(\lambda_1\overset{\widehat{T}_{l+1}-1}{\underset{r=\widehat{T}_l}{\sum}} \widehat{E}_{1r}\big)\|_2.
\end{eqnarray*}
Therefore, we deduce
{\small{\begin{eqnarray*}
\lefteqn{2\lambda_2+\lambda_1\sqrt{p(p-1)}(\widehat{T}_{l+1}-\widehat{T}_{l})}\\
&\geq & \|\frac{1}{T}\overset{T^\ast_{l}-1}{\underset{r = \widehat{T}_{l}}{\sum}} \Big[\Lambda(X_rX^\top_r)\text{vec}(\widehat{\Theta}_r)-\text{vec}(I_p)\Big] + \frac{1}{T}\overset{\widehat{T}_{l+1}-1}{\underset{r = T^\ast_{l}}{\sum}} \Big[\Lambda(X_rX^\top_r)\text{vec}(\widehat{\Theta}_r)-\text{vec}(I_p)\Big]\Big]\|_2\\
& = & \resizebox{0.92\linewidth}{!}{$\displaystyle\|\frac{1}{T}\overset{T^\ast_{l}-1}{\underset{r = \widehat{T}_{l}}{\sum}} \Big[\Lambda(X_rX^\top_r)\text{vec}(\widehat{\Omega}_{l+1}-\Omega^\ast_l+\Omega^\ast_l)-\text{vec}(I_p)\Big] + \frac{1}{T}\overset{\widehat{T}_{l+1}-1}{\underset{r = T^\ast_{l}}{\sum}} \Big[\Lambda(X_rX^\top_r)\text{vec}(\widehat{\Omega}_{l+1}-\Omega^\ast_{l+1}+\Omega^\ast_{l+1})-\text{vec}(I_p)\Big]\Big]\|_2$} \\
& \geq & \resizebox{0.92\linewidth}{!}{$\displaystyle\|\frac{1}{T}\overset{\widehat{T}_{l+1}-1}{\underset{r = T^\ast_{l}}{\sum}} \Big[\Lambda(X_rX^\top_r)\text{vec}(\widehat{\Omega}_{l+1}-\Omega^\ast_{l+1}+\Omega^\ast_{l+1})-\text{vec}(I_p)\Big]\Big]\|_2 -  \|\frac{1}{T}\overset{T^\ast_{l}-1}{\underset{r = \widehat{T}_{l}}{\sum}} \Big[\Lambda(X_rX^\top_r)\text{vec}(\widehat{\Omega}_{l+1}-\Omega^\ast_l+\Omega^\ast_l)-\text{vec}(I_p)\Big]\|_2$} \\
& \geq & \frac{\widehat{T}_{l+1}-T^\ast_l}{T}\Big\{\|\frac{1}{\widehat{T}_{l+1}-T^\ast_l}\overset{\widehat{T}_{l+1}-1}{\underset{r = T^\ast_{l}}{\sum}} \Lambda(X_rX^\top_r)\text{vec}(\widehat{\Omega}_{l+1}-\Omega^\ast_{l+1})\|_2 \\
& & - \|\frac{1}{\widehat{T}_{l+1}-T^\ast_l}\overset{\widehat{T}_{l+1}-1}{\underset{r = T^\ast_{l}}{\sum}}\Big[\frac{1}{2}\Omega^\ast_{l+1}X_rX^\top_r+\frac{1}{2}X_rX^\top_r\Omega^\ast_{l+1}-I_p\Big]\|_F\Big\} \\
& & - \frac{T^\ast_l-\widehat{T}_l}{T}\|\frac{1}{T^\ast_l-\widehat{T}_l}\overset{T^\ast_{l}-1}{\underset{r = \widehat{T}_{l}}{\sum}} \Big[\Lambda(X_rX^\top_r)\text{vec}(\widehat{\Omega}_{l+1}-\Omega^\ast_l+\Omega^\ast_l)-\text{vec}(I_p)\Big]\|_2.
\end{eqnarray*}}}
Therefore, since $\|Ax\|_2 \geq \lambda_{\min}(A)\|x\|_2$ for $A\succeq 0$, using part (i) of Theorem~1 and the bound on $\|\widehat{\Omega}_{l+1}-\Omega^\ast_l\|_F$, we obtain
\begin{eqnarray*}
\lefteqn{2\lambda_2+\lambda_1\sqrt{p(p-1)}(\widehat{T}_{l+1}-\widehat{T}_{l})}\\
  & \geq & \resizebox{0.92\linewidth}{!}{$\displaystyle\frac{\widehat{T}_{l+1}-T^\ast_l}{T} \Big\{\gamma^{\min}_{T,l}\|\widehat{\Omega}_{l+1}-\Omega^\ast_{l+1}\|_F-\|\frac{1}{\widehat{T}_{l+1}-T^\ast_l}\overset{\widehat{T}_{l+1}-1}{\underset{r = T^\ast_{l}}{\sum}}\Big[\frac{1}{2}\Omega^\ast_{l+1}X_rX^\top_r+\frac{1}{2}X_rX^\top_r\Omega^\ast_{l+1}-I_p\Big]\|_F\Big\}$} \\
& & - O_p(\frac{T^\ast_l-\widehat{T}_l}{T}),
\end{eqnarray*}
where $\gamma^{\min}_{T,l} = \lambda_{\min}(\frac{1}{\widehat{T}_{l+1}-T^\ast_l}\overset{\widehat{T}_{l+1}-1}{\underset{r=T^\ast_l}{\sum}}X_rX^\top_r) \geq \underline{\mu}/2$ with probability tending to one. We deduce
\begin{eqnarray*}
\lefteqn{2\lambda_2+\lambda_1\sqrt{p(p-1)}(\widehat{T}_{l+1}-\widehat{T}_{l})}\\
&\geq & \frac{\widehat{T}_{l+1}-T^\ast_l}{T} \Big\{\gamma^{\min}_{T,l}\|\widehat{\Omega}_{l+1}-\Omega^\ast_{l+1}\|_F-O_p( s^\ast_{\max}\,p\sqrt{\frac{\log(pT)}{I^\ast_{l+1}}})\Big\}- O_p(\frac{T^\ast_l-\widehat{T}_l}{T}).
\end{eqnarray*}
As a consequence, it can be deduced that
\begin{equation}\label{bound_prob_regime}
\|\widehat{\Omega}_{l+1}-\Omega^\ast_{l+1}\|_F=O_p(\frac{\lambda_2T}{I^\ast_{l+1}}+\lambda_1Tp(1+\frac{T\delta_T}{I^\ast_{l+1}})+\frac{T\delta_T}{I^\ast_{l+1}} + s^\ast_{\max}\,p\sqrt{\frac{\log(pT)}{I^\ast_{l+1}}}).
\end{equation}
In case (ii-b), by Lemma \ref{optimality_cond}, with change-points $t = \widehat{T}_l$ and $t = \widehat{T}_{l+1}$, we have
\begin{eqnarray*}
\lefteqn{2\lambda_2 + \lambda_1\sqrt{p(p-1)}(\widehat{T}_{l+1}-\widehat{T}_l) \geq \|\frac{1}{T}\overset{\widehat{T}_{l+1}-1}{\underset{r = \widehat{T}_l}{\sum}} \Big[\Lambda(X_rX^\top_r)\text{vec}(\widehat{\Theta}_r)-\text{vec}(I_p)\Big]\|_2}\\
& = & \|\frac{1}{T}\overset{T^\ast_{l}-1}{\underset{r = \widehat{T}_l}{\sum}} \Big[\Lambda(X_rX^\top_r)\text{vec}(\widehat{\Theta}_r)-\text{vec}(I_p)\Big] + \frac{1}{T}\overset{T^\ast_{l+1}-1}{\underset{r = T^\ast_l}{\sum}} \Big[\Lambda(X_rX^\top_r)\text{vec}(\widehat{\Theta}_r)-\text{vec}(I_p)\Big]\\
& & + \frac{1}{T}\overset{\widehat{T}_{l+1}-1}{\underset{r = T^\ast_{l+1}}{\sum}} \Big[\Lambda(X_rX^\top_r)\text{vec}(\widehat{\Theta}_r)-\text{vec}(I_p)\Big]\|_2 \\
& = & \resizebox{0.92\linewidth}{!}{$\displaystyle\|\frac{1}{T}\overset{T^\ast_{l}-1}{\underset{r = \widehat{T}_l}{\sum}} \Big[\Lambda(X_rX^\top_r)\text{vec}(\widehat{\Omega}_{l+1}-\Omega^\ast_l+\Omega^\ast_l)-\text{vec}(I_p)\Big] + \frac{1}{T}\overset{T^\ast_{l+1}-1}{\underset{r = T^\ast_l}{\sum}} \Big[\Lambda(X_rX^\top_r)\text{vec}(\widehat{\Omega}_{l+1}-\Omega^\ast_{l+1}+\Omega^\ast_{l+1})-\text{vec}(I_p)\Big]$} \\
& & + \frac{1}{T}\overset{\widehat{T}_{l+1}-1}{\underset{r = T^\ast_{l+1}}{\sum}} \Big[\Lambda(X_rX^\top_r)\text{vec}(\widehat{\Omega}_{l+1}-\Omega^\ast_{l+2}+\Omega^\ast_{l+2})-\text{vec}(I_p)\Big]\|_2\\
& \geq & \resizebox{0.92\linewidth}{!}{$\displaystyle\|\frac{1}{T}\overset{T^\ast_{l+1}-1}{\underset{r = T^\ast_l}{\sum}} \Big[\Lambda(X_rX^\top_r)\text{vec}(\widehat{\Omega}_{l+1}-\Omega^\ast_{l+1}+\Omega^\ast_{l+1})-\text{vec}(I_p)\Big]\|_2 - \|\frac{1}{T}\overset{T^\ast_{l}-1}{\underset{r = \widehat{T}_l}{\sum}} \Big[\Lambda(X_rX^\top_r)\text{vec}(\widehat{\Omega}_{l+1}-\Omega^\ast_l+\Omega^\ast_l)-\text{vec}(I_p)\Big]\|_2$} \\
& & - \|\frac{1}{T}\overset{\widehat{T}_{l+1}-1}{\underset{r = T^\ast_{l+1}}{\sum}} \Big[\Lambda(X_rX^\top_r)\text{vec}(\widehat{\Omega}_{l+1}-\Omega^\ast_{l+2}+\Omega^\ast_{l+2})-\text{vec}(I_p)\Big]\|_2.
\end{eqnarray*}
With $\overline{\gamma}^{\min}_{T,l} = \lambda_{\min}(\frac{1}{I_{\min}}\overset{T^\ast_{l+1}-1}{\underset{r=T^\ast_l}{\sum}}X_rX^\top_r) \geq \underline{\mu}/2$ with probability tending to one, we deduce
\begin{eqnarray*}
\lefteqn{2\lambda_2+\lambda_1\sqrt{p(p-1)}(\widehat{T}_{l+1}-\widehat{T}_{l})}\\
  &\geq & \resizebox{0.92\linewidth}{!}{$\displaystyle\frac{I^\ast_{l+1}}{T} \Big\{\overline{\gamma}^{\min}_{T,l}\|\widehat{\Omega}_{l+1}-\Omega^\ast_{l+1}\|_F-O_p( s^\ast_{\max}\,p\sqrt{\frac{\log(pT)}{I^\ast_{l+1}}})\Big\} - O_p(\frac{T^\ast_l-\widehat{T}_l}{T})-O_p(\frac{\widehat{T}_{l+1}-T^\ast_{l+1}}{T}).$}
\end{eqnarray*}
Hence, (\ref{bound_prob_regime}) holds. Using similar arguments, we can show that the latter is satisfied when $T^\ast_l \leq \widehat{T}_l < (T^\ast_l+T^\ast_{l+1})/2$.

\subsection{Proof of Theorem~3}

Using the result of Theorem~1,  we aim to show that:
\begin{equation}\label{proba_bound}
\Pb\big(\{h(\widehat{\Tc}_{\widehat{m}},\Tc^\ast_{m^\ast})> T\delta_T\}\cap \{m^\ast < \widehat{m} \leq m_{\max}\}\big) \rightarrow 0 \;\; \text{as} \;\; T \rightarrow \infty.
\end{equation}
To so, we define:
\begin{equation*}
\begin{array}{llll}
L_{m,k,1} = \big\{\forall 1 \leq l \leq m, |\widehat{T}_l - T^\ast_k|>T\delta_T \; \text{and} \; \widehat{T}_l < T^\ast_k\big\},&&\\
L_{m,k,2} = \big\{\forall 1 \leq l \leq m, |\widehat{T}_l - T^\ast_k|>T\delta_T \; \text{and} \; \widehat{T}_l > T^\ast_k\big\},&&\\
L_{m,k,3} = \big\{\exists 1 \leq l \leq m-1, |\widehat{T}_l - T^\ast_k|>T\delta_T, |\widehat{T}_{l+1} - T^\ast_k|>T\delta_T \; \text{and} \; \widehat{T}_l < T^\ast_k < \widehat{T}_{l+1}\big\}.&&
\end{array}
\end{equation*}
The probability (\ref{proba_bound}) can be bounded as:
\begin{eqnarray*}
\lefteqn{\Pb\big(\{h(\widehat{\Tc}_{\widehat{m}},\Tc^\ast_{m^\ast})> T\delta_T\}\cap \{m^\ast < \widehat{m} \leq m_{\max}\}\big) \leq \overset{m_{\max}}{\underset{m=m^\ast+1}{\sum}} \Pb\big(h(\widehat{\Tc}_{\widehat{m}},\Tc^\ast_{m^\ast})> T\delta_T\big)}\\
  & \leq & \resizebox{0.92\linewidth}{!}{$\displaystyle\overset{m_{\max}}{\underset{m=m^\ast+1}{\sum}} \overset{m^\ast}{\underset{k=1}{\sum}}\Pb\big(\forall l \in \{1,\ldots,m\}, |\widehat{T}_l-T^\ast_k|>T\delta_T\big) = \overset{m_{\max}}{\underset{m=m^\ast+1}{\sum}} \overset{m^\ast}{\underset{k=1}{\sum}} \Big[\Pb\big(L_{m,k,1}\big)+\Pb\big(L_{m,k,2}\big)+\Pb\big(L_{m,k,3}\big)\Big].$}
\end{eqnarray*}
We first focus on $\sum^{m_{\max}}_{m=m^\ast+1} \sum^{m^\ast}_{k=1}\Pb\big(L_{m,k,1}\big)$, which can be expressed as:
\begin{equation*}
\Pb\big(L_{m,k,1}\big) = \Pb\big(L_{m,k,1}\cap\{\widehat{T}_m>T^\ast_{k-1}\}\big)+\Pb\big(L_{m,k,1}\cap\{\widehat{T}_m\leq T^\ast_{k-1}\}\big).
\end{equation*}
By Lemma \ref{optimality_cond} with change-points $t = \widehat{T}_m$ and $t = T^\ast_k$, given the case $T^\ast_k\geq \widehat{T}_m > T^\ast_{k-1}$:
\begin{eqnarray*}
\frac{1}{T}\overset{T}{\underset{r=\widehat{T}_m}{\sum}}\Big[\Lambda(X_rX^\top_r)\text{vec}(\widehat{\Theta}_r)-\text{vec}(I_p)\Big] + \text{vec}\big(\lambda_1\overset{T}{\underset{r=\widehat{T}_m}{\sum}} \widehat{E}_{1r}+ \lambda_2 \frac{\widehat{\Gamma}_{\widehat{T}_m}}{\|\widehat{\Gamma}_{\widehat{T}_m}\|_F}\big)= \mathbf{0}_{p^2 \times 1},
\end{eqnarray*}
and
\begin{eqnarray*}
\|\frac{1}{T}\overset{T}{\underset{r=T^\ast_k}{\sum}}\Big[\Lambda(X_rX^\top_r)\text{vec}(\widehat{\Theta}_r)-\text{vec}(I_p)\Big] + \text{vec}\big(\lambda_1\overset{T}{\underset{r=T^\ast_k}{\sum}} \widehat{E}_{1r}\big) \|_2 \leq  \lambda_2.
\end{eqnarray*}
Therefore, taking the differences, we get:
\begin{eqnarray*}
\lefteqn{2 \lambda_2 + \lambda_1\sqrt{p(p-1)}(T^\ast_k - \widehat{T}_m) \geq \|\frac{1}{T}\overset{T^\ast_k-1}{\underset{r=\widehat{T}_m}{\sum}}\Big[\Lambda(X_rX^\top_r)\text{vec}(\widehat{\Theta}_r)-\text{vec}(I_p)\Big] \|_2}\\
& \geq & \|\frac{1}{T}\overset{T^\ast_k-1}{\underset{r=\widehat{T}_m}{\sum}}\Lambda(X_rX^\top_r)\text{vec}(\widehat{\Omega}_{m+1} - \Omega^\ast_{k+1}) + \frac{1}{T}\overset{T^\ast_k-1}{\underset{r=\widehat{T}_m}{\sum}}\Lambda(X_rX^\top_r)\text{vec}(\Omega^\ast_{k+1} - \Omega^\ast_k) \\
& & + \frac{1}{T}\overset{T^\ast_k-1}{\underset{r=\widehat{T}_m}{\sum}}\Big[\Lambda(X_rX^\top_r)\text{vec}(\Omega^\ast_k)-\text{vec}(I_p)\Big] \|_2.
\end{eqnarray*}
Therefore, the event $\Bc_T$ defined as
\begin{eqnarray*}
\lefteqn{\Bc_T \coloneqq  \Big\{\|\Omega^\ast_{k+1}-\Omega^\ast_k\|_F \leq (\gamma^{\min}_{4,T,m,k})^{-1}\Big[\frac{2\lambda_2T}{T^\ast_k-\widehat{T}_m}+\lambda_1T\sqrt{p(p-1)}}\\
  && \resizebox{0.92\linewidth}{!}{$\displaystyle+\|\frac{1}{T^\ast_k-\widehat{T}_m}\overset{T^\ast_k-1}{\underset{r=\widehat{T}_m}{\sum}}\Lambda(X_rX^\top_r)\text{vec}(\widehat{\Omega}_{m+1} - \Omega^\ast_{k+1})\|_2+\|\frac{1}{T^\ast_k-\widehat{T}_m}\overset{T^\ast_k-1}{\underset{r=\widehat{T}_m}{\sum}}\Big[\Lambda(X_rX^\top_r)\text{vec}(\Omega^\ast_k)-\text{vec}(I_p)\Big] \|_2\Big] \Big\},$}
\end{eqnarray*}
where $\gamma^{\min}_{4,T,m,k}=\lambda_{\min}(\frac{1}{T^\ast_k-\widehat{T}_{m}}\sum^{T^\ast_k-1}_{r=\widehat{T}_{m}}X_rX^\top_r) \geq \underline{\mu}/2$ with probability tending to one, holds with probability one. Hence, we deduce
\begin{eqnarray*}
\resizebox{\linewidth}{!}{$\displaystyle\overset{m_{\max}}{\underset{m=m^\ast+1}{\sum}} \overset{m^\ast}{\underset{k=1}{\sum}} \Pb\big(L_{m,k,1} \cap \{\widehat{T}_m>T^\ast_{k-1}\}\big) = \overset{m_{\max}}{\underset{m=m^\ast+1}{\sum}} \overset{m^\ast}{\underset{k=1}{\sum}} \Pb\big(\Bc_T \cap L_{m,k,1} \cap \{\widehat{T}_m>T^\ast_{k-1}\}\big) \leq M_{1,1}+M_{1,2}+M_{1,3},$}
\end{eqnarray*}
with
\begin{eqnarray*}
M_{1,1} \coloneqq \overset{m_{\max}}{\underset{m=m^\ast+1}{\sum}} \overset{m^\ast}{\underset{k=1}{\sum}}\Pb\big(\|\Omega^\ast_{k+1}-\Omega^\ast_k\|_F/3\leq(\gamma^{\min}_{4,T,m,k})^{-1}\Big[2\lambda_2\delta^{-1}_T+\lambda_1T\sqrt{p(p-1)}\Big]\big),
\end{eqnarray*}
$$
\resizebox{\linewidth}{!}{$\displaystyle M_{1,2} \!\coloneqq\! \overset{m_{\max}}{\underset{m=m^\ast+1}{\sum}} \overset{m^\ast}{\underset{k=1}{\sum}} \Pb\big(T^\ast_k\!-\!\widehat{T}_m\!>\!T\delta_T,\|\Omega^\ast_{k+1}\!-\!\Omega^\ast_k\|_F/3 \!\leq\! (\gamma^{\min}_{4,T,m,k})^{-1} \|\frac{1}{T^\ast_k-\widehat{T}_m}\overset{T^\ast_k-1}{\underset{r=\widehat{T}_m}{\sum}}\!\Lambda(X_rX^\top_r)\text{vec}(\widehat{\Omega}_{m+1} \!-\! \Omega^\ast_{k+1})\|_2\big),$}
$$
$$
\resizebox{\linewidth}{!}{$\displaystyle M_{1,3} \!\coloneqq\! \overset{m_{\max}}{\underset{m=m^\ast+1}{\sum}} \overset{m^\ast}{\underset{k=1}{\sum}} \Pb\big(T^\ast_k\!-\!\widehat{T}_m\!>\!T\delta_T,\|\Omega^\ast_{k+1}\!-\!\Omega^\ast_k\|_F/3 \!\leq\! (\gamma^{\min}_{4,T,m,k})^{-1} \|\frac{1}{T^\ast_k-\widehat{T}_m}\overset{T^\ast_k-1}{\underset{r=\widehat{T}_m}{\sum}}\!\Big[\Lambda(X_rX^\top_r)\text{vec}(\Omega^\ast_k)\!-\!\text{vec}(I_p)\Big] \|_2\big).$}
$$
In the same vein as in the analysis of (\ref{bound_A+}), we can show that $M_{1,1}, M_{1,3} \rightarrow 0$ as $T \rightarrow \infty$. $M_{1,2}$ requires more arguments. By Lemma (\ref{optimality_cond}), with change-points $t = T^\ast_k$ and $t = T^\ast_{k+1}$:
\begin{eqnarray*}
\|\frac{1}{T}\overset{T}{\underset{r=T^\ast_k}{\sum}}\Big[\Lambda(X_rX^\top_r)\text{vec}(\widehat{\Omega}_{m+1})-\text{vec}(I_p)\Big] + \text{vec}\big(\lambda_1\overset{T}{\underset{r=T^\ast_k}{\sum}} \widehat{E}_{1r}\big) \|_2 \leq  \lambda_2,
\end{eqnarray*}
and
\begin{eqnarray*}
\|\frac{1}{T}\overset{T}{\underset{r=T^\ast_{k+1}}{\sum}}\Big[\Lambda(X_rX^\top_r)\text{vec}(\widehat{\Omega}_{m+1})-\text{vec}(I_p)\Big] + \text{vec}\big(\lambda_1\overset{T}{\underset{r=T^\ast_{k+1}}{\sum}} \widehat{E}_{1r}\big) \|_2 \leq  \lambda_2.
\end{eqnarray*}
Therefore
\begin{eqnarray*}
\lefteqn{2\lambda_2 + \lambda_1\sqrt{p(p-1)}(T^\ast_{k+1}-T^\ast_k) }\\
& \geq & \|\frac{1}{T}\overset{T^\ast_{k+1}-1}{\underset{r=T^\ast_{k}}{\sum}}\Lambda(X_rX^\top_r)\text{vec}(\widehat{\Omega}_{m+1} - \Omega^\ast_{k+1}) +\frac{1}{T}\overset{T^\ast_{k+1}-1}{\underset{r=T^\ast_k}{\sum}}\Big[\Lambda(X_rX^\top_r)\text{vec}(\Omega^\ast_{k+1})-\text{vec}(I_p)\Big]\|_2,
\end{eqnarray*}
which implies
$$
\resizebox{\linewidth}{!}{$\displaystyle\|\widehat{\Omega}_{m+1}-\Omega^\ast_{k+1}\|_F \leq (\gamma^{\min}_{5,T,k})^{-1} \Big[\frac{2\lambda_2T}{T^\ast_{k+1}-T^\ast_{k}}+\lambda_1T\sqrt{p(p-1)}+\|\frac{1}{T^\ast_{k+1}-T^\ast_{k}}\overset{T^\ast_{k+1}-1}{\underset{r=T^\ast_k}{\sum}}\Big[\Lambda(X_rX^\top_r)\text{vec}(\Omega^\ast_{k+1})-\text{vec}(I_p)\Big]\|_2\Big],$}
$$
with $\gamma^{\min}_{5,T,k}=\lambda_{\min}(\frac{1}{T^\ast_{k+1}-T^\ast_{k}}\sum^{T^\ast_{k+1}-1}_{r=T^\ast_{k}}X_rX^\top_r) \geq \underline{\mu}/2$ with probability tending to one. We deduce
\begin{eqnarray}
\lefteqn{M_{1,2} \leq \overset{m_{\max}}{\underset{m=m^\ast+1}{\sum}} \overset{m^\ast}{\underset{k=1}{\sum}} \Pb\big(\|\Omega^\ast_{k+1}-\Omega^\ast_k\|_F/3 \leq (\gamma^{\min}_{4,T,m,k})^{-1} \gamma^{\max}_{2,T,m,k}\|\widehat{\Omega}_{m+1}-\Omega^\ast_{k+1}\|_F\big) } \label{proba_M2}\\
  & \leq & \resizebox{0.92\linewidth}{!}{$\displaystyle\overset{m_{\max}}{\underset{m=m^\ast+1}{\sum}} \overset{m^\ast}{\underset{k=1}{\sum}}\Big[ \Pb\big( \gamma^{\min}_{4,T,m,k} (\gamma^{\max}_{2,T,m,k})^{-1}\|\Omega^\ast_{k+1}-\Omega^\ast_k\|_F/6 \leq (\gamma^{\min}_{5,T,k})^{-1}\Big[\frac{2\lambda_2T}{\mathcal{I}_{\min}}+\lambda_1T\sqrt{p(p-1)}\Big]\big)$} \nonumber\\
& & \resizebox{0.92\linewidth}{!}{$\displaystyle+ \Pb\big( \gamma^{\min}_{4,T,m,k} (\gamma^{\max}_{2,T,m,k})^{-1}\|\Omega^\ast_{k+1}\!-\!\Omega^\ast_k\|_F/6 \!\leq\! (\gamma^{\min}_{5,T,k})^{-1} \|\frac{1}{T^\ast_{k+1}\!-\!T^\ast_{k}}\overset{T^\ast_{k+1}-1}{\underset{r=T^\ast_k}{\sum}}\!\Big[\Lambda(X_rX^\top_r)\text{vec}(\Omega^\ast_{k+1})\!-\!\text{vec}(I_p)\Big]\|_2 \big)\Big],$} \nonumber
\end{eqnarray}
where $\gamma^{\max}_{2,T,m,k} = \lambda_{\max}(\frac{1}{T^\ast_{k}-\widehat{T}_{m}}\sum^{T^\ast_{k}-1}_{r=\widehat{T}_{m}}X_rX^\top_r) \leq 2 \overline{\mu}$ with probability tending to one. The first term in the second inequality of (\ref{proba_M2}) tends to zero under the conditions $\lambda_2T/(\mathcal{I}_{\min}\eta_{\min}) \rightarrow 0$ and $\lambda_1Tp / \eta_{\min} \rightarrow 0$. And under $(\eta_{\min}\mathcal{I}^{1/2}_{\min})^{-1}s^\ast_{\max}\,p\sqrt{\log(pT)}\rightarrow 0$, the second term tends to zero. Therefore, we conclude $\sum^{m_{\max}}_{m=m^\ast+1} \sum^{m^\ast}_{k=1} \Pb\big(L_{m,k,1} \cap \{\widehat{T}_m>T^\ast_{k-1}\}\big) \rightarrow 0$ as $T \rightarrow \infty$. Based on similar arguments, we can show $\sum^{m_{\max}}_{m=m^\ast+1} \sum^{m^\ast}_{k=1} \Pb\big(L_{m,k,1} \cap \{\widehat{T}_m \leq T^\ast_{k-1}\}\big) \rightarrow 0$ as $T \rightarrow \infty$. Therefore, $\sum^{m_{\max}}_{m=m^\ast+1} \sum^{m^\ast}_{k=1}\Pb\big(L_{m,k,1}\big) \rightarrow 0$ as $T \rightarrow \infty$. Similarly, it can be proved that $\sum^{m_{\max}}_{m=m^\ast+1} \sum^{m^\ast}_{k=1}\Pb\big(L_{m,k,2}\big) \rightarrow 0$ as $T \rightarrow \infty$. \\
We now consider $\sum^{m_{\max}}_{m=m^\ast+1} \sum^{m^\ast}_{k=1}\Pb\big(L_{m,k,3}\big)$. Define
\begin{equation*}
\begin{array}{llll}
\resizebox{\linewidth}{!}{$\displaystyle L^{(1)}_{m,k,3} \coloneqq L_{m,k,3}\cap \{T^\ast_{k-1}<\widehat{T}_l < \widehat{T}_{l+1}<T^\ast_{k+1}\}, L^{(2)}_{m,k,3} \coloneqq L_{m,k,3}\cap \{T^\ast_{k-1}<\widehat{T}_l <T^\ast_{k+1}, \widehat{T}_{l+1}\geq T^\ast_{k+1}\},$} &&\\
\resizebox{\linewidth}{!}{$\displaystyle L^{(3)}_{m,k,3} \coloneqq L_{m,k,3}\cap \{\widehat{T}_l \leq T^\ast_{k-1},T^\ast_{k-1} <\widehat{T}_{l+1}<T^\ast_{k+1}\}, L^{(4)}_{m,k,3} \coloneqq L_{m,k,3}\cap \{\widehat{T}_l \leq T^\ast_{k-1},T^\ast_{k+1} <\widehat{T}_{l+1}\}.$} &&
\end{array}
\end{equation*}
First, we consider $L^{(1)}_{m,k,3}$. By Lemma (\ref{optimality_cond}), for the change-points $t = T^\ast_k$ and $t = \widehat{T}_l$, we obtain
\begin{eqnarray}
\lefteqn{2\lambda_2 + \lambda_1\sqrt{p(p-1)}(T^\ast_k-\widehat{T}_l)}\label{changepoint_1}\\
& \geq & \|\frac{1}{T}\overset{T^\ast_{k}-1}{\underset{r=\widehat{T}_{l}}{\sum}}\Lambda(X_rX^\top_r)\text{vec}(\widehat{\Omega}_{l+1} - \Omega^\ast_{k}) +\frac{1}{T}\overset{T^\ast_{k}-1}{\underset{r=\widehat{T}_{l}}{\sum}}\Big[\Lambda(X_rX^\top_r)\text{vec}(\Omega^\ast_{k})-\text{vec}(I_p)\Big]\|_2,\nonumber
\end{eqnarray}
and for the change-points $t = T^\ast_k$ and $t = \widehat{T}_{l+1}$, we get
\begin{eqnarray}
\lefteqn{2\lambda_2 + \lambda_1\sqrt{p(p-1)}(\widehat{T}_{l+1}-T^\ast_k)}\label{changepoint_2}\\
& \geq & \|\frac{1}{T}\overset{\widehat{T}_{l+1}-1}{\underset{r=T^\ast_k}{\sum}}\Lambda(X_rX^\top_r)\text{vec}(\widehat{\Omega}_{l+1} - \Omega^\ast_{k+1}) +\frac{1}{T}\overset{\widehat{T}_{l+1}-1}{\underset{r=T^\ast_{k}}{\sum}}\Big[\Lambda(X_rX^\top_r)\text{vec}(\Omega^\ast_{k+1})-\text{vec}(I_p)\Big]\|_2.\nonumber
\end{eqnarray}
Moreover, by the triangle inequality, we have
\begin{eqnarray*}
\lefteqn{\|\Omega^\ast_{k+1}-\Omega^\ast_k\|_F \leq \|\widehat{\Omega}_{l+1}-\Omega^\ast_k\|_F + \|\widehat{\Omega}_{l+1}-\Omega^\ast_{k+1}\|_F}\\
  & \leq & \resizebox{0.92\linewidth}{!}{$\displaystyle(\gamma^{\min}_{6,T,k,l})^{-1}\Big[ \frac{2\lambda_2T}{T^\ast_k-\widehat{T}_l} + \lambda_1T\sqrt{p(p-1)} + \|\frac{1}{T^\ast_k-\widehat{T}_l}\overset{T^\ast_{k}-1}{\underset{r=\widehat{T}_{l}}{\sum}}\Big[\Lambda(X_rX^\top_r)\text{vec}(\Omega^\ast_{k})-\text{vec}(I_p)\Big]\|_2 \Big]$} \\
& & \resizebox{0.92\linewidth}{!}{$\displaystyle+ (\gamma^{\min}_{7,T,k,l})^{-1}\Big[ \frac{2\lambda_2T}{\widehat{T}_{l+1}-T^\ast_k} + \lambda_1T\sqrt{p(p-1)} + \|\frac{1}{\widehat{T}_{l+1}-T^\ast_k}\overset{\widehat{T}_{l+1}-1}{\underset{r=T^\ast_k}{\sum}}\Big[\Lambda(X_rX^\top_r)\text{vec}(\Omega^\ast_{k})-\text{vec}(I_p)\Big]\|_2 \Big],$}
\end{eqnarray*}
with $\gamma^{\min}_{6,T,k,l}=\lambda_{\min}(\frac{1}{T^\ast_k-\widehat{T}_l}\sum^{T^\ast_k-1}_{r=\widehat{T}_l}X_rX^\top_r) \geq \underline{\mu}/2,\gamma^{\min}_{7,T,k,l}=\lambda_{\min}(\frac{1}{\widehat{T}_{l+1}-T^\ast_k}\sum^{\widehat{T}_{l+1}-1}_{r=T^\ast_k}X_rX^\top_r) \geq \underline{\mu}/2$ with probability tending to one. So $\sum^{m_{\max}}_{m=m^\ast+1} \sum^{m^\ast}_{k=1}\Pb\big(L^{(1)}_{m,k,3}\big)$ is upper bounded as follows
\begin{eqnarray*}
\lefteqn{\sum^{m_{\max}}_{m=m^\ast+1} \sum^{m^\ast}_{k=1}\Pb\big(L^{(1)}_{m,k,3}\big)}\\
  & \leq & \resizebox{0.92\linewidth}{!}{$\displaystyle\sum^{m_{\max}}_{m=m^\ast+1} \sum^{m^\ast}_{k=1}\Pb\big(\|\Omega^\ast_{k+1}-\Omega^\ast_{k}\|_F/3 \leq \big(\gamma^{\min}_{6,T,k,l})^{-1}+(\gamma^{\min}_{7,T,k,l})^{-1}\big)\big[2\lambda_2\delta^{-1}_T+\lambda_1T\sqrt{p(p-1)}\big]\big)$}
  \\
& & \resizebox{0.92\linewidth}{!}{$\displaystyle+ \sum^{m_{\max}}_{m=m^\ast+1} \sum^{m^\ast}_{k=1}\Pb\big(\|\Omega^\ast_{k+1}-\Omega^\ast_{k}\|_F/3 \leq (\gamma^{\min}_{6,T,k,l})^{-1} \|\frac{1}{T^\ast_k-\widehat{T}_l}\overset{T^\ast_{k}-1}{\underset{r=\widehat{T}_{l}}{\sum}}\Big[\Lambda(X_rX^\top_r)\text{vec}(\Omega^\ast_{k})-\text{vec}(I_p)\Big]\|_2,T^\ast_{k}-\widehat{T}_l\geq T \delta_T\big)$} \\
& & \resizebox{0.92\linewidth}{!}{$\displaystyle+ \sum^{m_{\max}}_{m=m^\ast+1} \sum^{m^\ast}_{k=1}\Pb\big(\|\Omega^\ast_{k+1}-\Omega^\ast_{k}\|_F/3 \leq (\gamma^{\min}_{7,T,k,l})^{-1}\|\frac{1}{\widehat{T}_{l+1}-T^\ast_k}\overset{\widehat{T}_{l+1}-1}{\underset{r=T^\ast_k}{\sum}}\Big[\Lambda(X_rX^\top_r)\text{vec}(\Omega^\ast_{k})-\text{vec}(I_p)\Big]\|_2, \widehat{T}_{l+1}-T^\ast_{k}\geq T \delta_T\big),$}
\end{eqnarray*}
which tends to zero in the spirit as in (\ref{bound_A+}). For $L^{(2)}_{m,k,3}$, by Lemma \ref{optimality_cond} with change-points $t = T^\ast_k$ and $t = \widehat{T}_l$ to obtain (\ref{changepoint_1}) and with change-points $t=T^\ast_{k}$, $t = T^\ast_{k+1}$, we get
\begin{eqnarray}
\lefteqn{2\lambda_2 + \lambda_1\sqrt{p(p-1)}(T^\ast_{k+1}-T^\ast_{k})}\label{changepoint_3}\\
& \geq & \|\frac{1}{T}\overset{T^\ast_{k+1}-1}{\underset{r=T^\ast_{k}}{\sum}}\Lambda(X_rX^\top_r)\text{vec}(\widehat{\Omega}_{l+1} - \Omega^\ast_{k+1}) +\frac{1}{T}\overset{T^\ast_{k+1}-1}{\underset{r=T^\ast_{k}}{\sum}}\Big[\Lambda(X_rX^\top_r)\text{vec}(\Omega^\ast_{k+1})-\text{vec}(I_p)\Big]\|_2.\nonumber
\end{eqnarray}
By the triangle inequality, we have
\begin{eqnarray*}
\lefteqn{\|\Omega^\ast_{k+1}-\Omega^\ast_{k}\|_F\leq \|\widehat{\Omega}_{l+1}-\Omega^\ast_{k}\|_F+\|\widehat{\Omega}_{l+1}-\Omega^\ast_{k+1}\|_F}\\
  & \leq & \resizebox{0.92\linewidth}{!}{$\displaystyle(\gamma^{\min}_{6,T,k,l})^{-1}\Big[ \frac{2\lambda_2T}{T^\ast_k-\widehat{T}_l} + \lambda_1T\sqrt{p(p-1)} + \|\frac{1}{T^\ast_k-\widehat{T}_l}\overset{T^\ast_{k}-1}{\underset{r=\widehat{T}_{l}}{\sum}}\Big[\Lambda(X_rX^\top_r)\text{vec}(\Omega^\ast_{k})-\text{vec}(I_p)\Big]\|_2 \Big]$} \\
  & & \resizebox{0.92\linewidth}{!}{$\displaystyle+ (\gamma^{\min}_{8,T,k})^{-1}\Big[ \frac{2\lambda_2T}{T^\ast_{k+1}-T^\ast_k} + \lambda_1T\sqrt{p(p-1)} + \|\frac{1}{T^\ast_{k+1}-T^\ast_k}\overset{T^\ast_{k+1}-1}{\underset{r=T^\ast_k}{\sum}}\Big[\Lambda(X_rX^\top_r)\text{vec}(\Omega^\ast_{k+1})-\text{vec}(I_p)\Big]\|_2 \Big],$}
\end{eqnarray*}
with $\gamma^{\min}_{8,T,k}=\lambda_{\min}(\frac{1}{T^\ast_{k+1}-T^\ast_k}\sum^{T^\ast_{k+1}-1}_{r=T^\ast_k}X_rX^\top_r) \geq \underline{\mu}/2$ with probability tending to one. Therefore, we obtain
\begin{eqnarray*}
\lefteqn{\sum^{m_{\max}}_{m=m^\ast+1} \sum^{m^\ast}_{k=1}\Pb\big(L^{(2)}_{m,k,3}\big)}\\
  & \leq & \resizebox{0.92\linewidth}{!}{$\displaystyle\sum^{m_{\max}}_{m=m^\ast+1} \sum^{m^\ast}_{k=1}\Pb\big(\|\Omega^\ast_{k+1}-\Omega^\ast_{k}\|_F/3 \leq (\gamma^{\min}_{6,T,k,l})^{-1}\big[2\lambda_2\delta^{-1}_T+\lambda_1T\sqrt{p(p-1)}\big]+(\gamma^{\min}_{8,T,k,l})^{-1} \big[\frac{2\lambda_2T}{\mathcal{I}_{\min}}+\lambda_1T\sqrt{p(p-1)}\big]\big)$}
  \\
  & & \resizebox{0.92\linewidth}{!}{$\displaystyle+ \sum^{m_{\max}}_{m=m^\ast+1} \sum^{m^\ast}_{k=1}\Pb\big(\|\Omega^\ast_{k+1}-\Omega^\ast_{k}\|_F/3 \leq (\gamma^{\min}_{6,T,k,l})^{-1} \|\frac{1}{T^\ast_k-\widehat{T}_l}\overset{T^\ast_{k}-1}{\underset{r=\widehat{T}_{l}}{\sum}}\Big[\Lambda(X_rX^\top_r)\text{vec}(\Omega^\ast_{k})-\text{vec}(I_p)\Big]\|_2,T^\ast_{k}-\widehat{T}_l\geq T \delta_T\big)$} \\
  & & \resizebox{0.92\linewidth}{!}{$\displaystyle+ \sum^{m_{\max}}_{m=m^\ast+1} \sum^{m^\ast}_{k=1}\Pb\big(\|\Omega^\ast_{k+1}-\Omega^\ast_{k}\|_F/3 \leq (\gamma^{\min}_{8,T,k,l})^{-1}\|\frac{1}{T^\ast_{k+1}-T^\ast_k}\overset{T^\ast_{k+1}-1}{\underset{r=T^\ast_k}{\sum}}\Big[\Lambda(X_rX^\top_r)\text{vec}(\Omega^\ast_{k})-\text{vec}(I_p)\Big]\|_2\big),$}
\end{eqnarray*}
which will tend to zero based on similar arguments as in the convergence of (\ref{bound_A+}). For $L^{(3)}_{m,k,3}$, by Lemma \ref{optimality_cond} with change-points $t = T^\ast_{k-1}$ and $t = T^\ast_k$, we have
\begin{eqnarray}
\lefteqn{2\lambda_2 + \lambda_1\sqrt{p(p-1)}(T^\ast_k-T^\ast_{k-1})}\label{changepoint_4}\\
& \geq & \|\frac{1}{T}\overset{T^\ast_{k}-1}{\underset{r=T^\ast_{k-1}}{\sum}}\Lambda(X_rX^\top_r)\text{vec}(\widehat{\Omega}_{l+1} - \Omega^\ast_{k}) +\frac{1}{T}\overset{T^\ast_{k}-1}{\underset{r=T^\ast_{k-1}}{\sum}}\Big[\Lambda(X_rX^\top_r)\text{vec}(\Omega^\ast_{k})-\text{vec}(I_p)\Big]\|_2,\nonumber
\end{eqnarray}
and with change-points $t=T^\ast_{k}$, $t = \widehat{T}_{l+1}$, we get
\begin{eqnarray}
\lefteqn{2\lambda_2 + \lambda_1\sqrt{p(p-1)}(\widehat{T}_{l+1}-T^\ast_k)}\label{changepoint_5}\\
& \geq & \|\frac{1}{T}\overset{\widehat{T}_{l+1}-1}{\underset{r=T^\ast_{k}}{\sum}}\Lambda(X_rX^\top_r)\text{vec}(\widehat{\Omega}_{l+1} - \Omega^\ast_{k+1}) +\frac{1}{T}\overset{\widehat{T}_{l+1}-1}{\underset{r=T^\ast_{k}}{\sum}}\Big[\Lambda(X_rX^\top_r)\text{vec}(\Omega^\ast_{k+1})-\text{vec}(I_p)\Big]\|_2.\nonumber
\end{eqnarray}
By the triangle inequality, we deduce
\begin{eqnarray*}
\lefteqn{\|\Omega^\ast_{k+1}-\Omega^\ast_{k}\|_F\leq \|\widehat{\Omega}_{l+1}-\Omega^\ast_{k}\|_F+\|\widehat{\Omega}_{l+1}-\Omega^\ast_{k+1}\|_F}\\
  & \leq & \resizebox{0.92\linewidth}{!}{$\displaystyle(\gamma^{\min}_{9,T,k})^{-1}\Big[ \frac{2\lambda_2T}{T^\ast_k-T^\ast_{k-1}} + \lambda_1T\sqrt{p(p-1)} + \|\frac{1}{T^\ast_k-T^\ast_{k-1}}\overset{T^\ast_{k}-1}{\underset{r=T^\ast_{k-1}}{\sum}}\Big[\Lambda(X_rX^\top_r)\text{vec}(\Omega^\ast_{k})-\text{vec}(I_p)\Big]\|_2 \Big]$} \\
  & & \resizebox{0.92\linewidth}{!}{$\displaystyle+ (\gamma^{\min}_{10,T,k})^{-1}\Big[ \frac{2\lambda_2T}{\widehat{T}_{l+1}-T^\ast_k} + \lambda_1T\sqrt{p(p-1)} + \|\frac{1}{\widehat{T}_{l+1}-T^\ast_k}\overset{\widehat{T}_{l+1}-1}{\underset{r=T^\ast_k}{\sum}}\Big[\Lambda(X_rX^\top_r)\text{vec}(\Omega^\ast_{k+1})-\text{vec}(I_p)\Big]\|_2 \Big],$}
\end{eqnarray*}
with $\gamma^{\min}_{9,T,k}=\lambda_{\min}(\frac{1}{T^\ast_{k}-T^\ast_{k-1}}\sum^{T^\ast_{k}-1}_{r=T^\ast_{k-1}}X_rX^\top_r) \geq \underline{\mu}/2$, $\gamma^{\min}_{10,T,k,l}=\lambda_{\min}(\frac{1}{\widehat{T}_{l+1}-T^\ast_k}\sum^{\widehat{T}_{l+1}-1}_{r=T^\ast_{k}}X_rX^\top_r) \geq \underline{\mu}/2$  with probability tending to one. We deduce
\begin{eqnarray*}
\lefteqn{\sum^{m_{\max}}_{m=m^\ast+1} \sum^{m^\ast}_{k=1}\Pb\big(L^{(3)}_{m,k,3}\big)}\\
  & \leq & \resizebox{0.92\linewidth}{!}{$\displaystyle\sum^{m_{\max}}_{m=m^\ast+1} \sum^{m^\ast}_{k=1}\Pb\big(\|\Omega^\ast_{k+1}-\Omega^\ast_{k}\|_F/3 \leq (\gamma^{\min}_{9,T,k})^{-1}\big[\frac{2\lambda_2T}{\mathcal{I}_{\min}}+\lambda_1T\sqrt{p(p-1)}\big]+(\gamma^{\min}_{10,T,k,l})^{-1} \big[2\lambda_2\delta^{-1}_T+\lambda_1T\sqrt{p(p-1)}\big]\big)$}
  \\
  & & \resizebox{0.92\linewidth}{!}{$\displaystyle+ \sum^{m_{\max}}_{m=m^\ast+1} \sum^{m^\ast}_{k=1}\Pb\big(\|\Omega^\ast_{k+1}-\Omega^\ast_{k}\|_F/3 \leq (\gamma^{\min}_{9,T,k})^{-1} \|\frac{1}{T^\ast_k-T^\ast_{k-1}}\overset{T^\ast_{k}-1}{\underset{r=T^\ast_{k-1}}{\sum}}\Big[\Lambda(X_rX^\top_r)\text{vec}(\Omega^\ast_{k})-\text{vec}(I_p)\Big]\|_2\big)$} \\
  & & \resizebox{0.92\linewidth}{!}{$\displaystyle+ \sum^{m_{\max}}_{m=m^\ast+1} \sum^{m^\ast}_{k=1}\Pb\big(\|\Omega^\ast_{k+1}-\Omega^\ast_{k}\|_F/3 \leq (\gamma^{\min}_{10,T,k,l})^{-1}\|\frac{1}{\widehat{T}_{l+1}-T^\ast_{k+1}}\overset{\widehat{T}_{l+1}-1}{\underset{r=T^\ast_k}{\sum}}\Big[\Lambda(X_rX^\top_r)\text{vec}(\Omega^\ast_{k+1})-\text{vec}(I_p)\Big]\|_2,$} \\
& & \qquad \qquad \qquad \widehat{T}_{l+1}-T^\ast_k\geq T\delta_T\big),
\end{eqnarray*}
which tends to zero based on the same arguments as in the convergence of (\ref{bound_A+}). Finally, to analyze $L^{(4)}_{m,k,3}$, applying Lemma \ref{optimality_cond} with $t = T^\ast_{k-1},t = T^\ast_k$ to obtain (\ref{changepoint_4}) and with $t = T^\ast_{k},t = T^\ast_{k+1}$ to obtain (\ref{changepoint_3}). By the triangle inequality, we have
\begin{eqnarray*}
\lefteqn{\|\Omega^\ast_{k+1}-\Omega^\ast_{k}\|_F\leq \|\widehat{\Omega}_{l+1}-\Omega^\ast_{k}\|_F+\|\widehat{\Omega}_{l+1}-\Omega^\ast_{k+1}\|_F}\\
  & \leq & \resizebox{0.92\linewidth}{!}{$\displaystyle(\gamma^{\min}_{9,T,k})^{-1}\Big[ \frac{2\lambda_2T}{T^\ast_k-T^\ast_{k-1}} + \lambda_1T\sqrt{p(p-1)} + \|\frac{1}{T^\ast_k-T^\ast_{k-1}}\overset{T^\ast_{k}-1}{\underset{r=T^\ast_{k-1}}{\sum}}\Big[\Lambda(X_rX^\top_r)\text{vec}(\Omega^\ast_{k})-\text{vec}(I_p)\Big]\|_2 \Big]$} \\
  & & \resizebox{0.92\linewidth}{!}{$\displaystyle+ (\gamma^{\min}_{8,T,k})^{-1}\Big[ \frac{2\lambda_2T}{T^\ast_{k+1}-T^\ast_k} + \lambda_1T\sqrt{p(p-1)} + \|\frac{1}{T^\ast_{k+1}-T^\ast_k}\overset{T^\ast_{k+1}-1}{\underset{r=T^\ast_k}{\sum}}\Big[\Lambda(X_rX^\top_r)\text{vec}(\Omega^\ast_{k+1})-\text{vec}(I_p)\Big]\|_2 \Big],$}
\end{eqnarray*}
we deduce
\begin{eqnarray*}
\lefteqn{\sum^{m_{\max}}_{m=m^\ast+1} \sum^{m^\ast}_{k=1}\Pb\big(L^{(4)}_{m,k,3}\big)}\\
& \leq & \sum^{m_{\max}}_{m=m^\ast+1} \sum^{m^\ast}_{k=1}\Pb\big(\|\Omega^\ast_{k+1}-\Omega^\ast_{k}\|_F/3 \leq \big(\gamma^{\min}_{9,T,k})^{-1}+(\gamma^{\min}_{8,T,k})^{-1}\big)\big[\frac{2\lambda_2T}{\mathcal{I}_{\min}}+\lambda_1T\sqrt{p(p-1)}\big]\big)\\
& & \resizebox{0.92\linewidth}{!}{$\displaystyle+ \sum^{m_{\max}}_{m=m^\ast+1} \sum^{m^\ast}_{k=1}\Pb\big(\|\Omega^\ast_{k+1}-\Omega^\ast_{k}\|_F/3 \leq (\gamma^{\min}_{9,T,k})^{-1} \|\frac{1}{T^\ast_k-T^\ast_{k-1}}\overset{T^\ast_{k}-1}{\underset{r=T^\ast_{k-1}}{\sum}}\Big[\Lambda(X_rX^\top_r)\text{vec}(\Omega^\ast_{k})-\text{vec}(I_p)\Big]\|_2\big)$} \\
& & \resizebox{0.92\linewidth}{!}{$\displaystyle+ \sum^{m_{\max}}_{m=m^\ast+1} \sum^{m^\ast}_{k=1}\Pb\big(\|\Omega^\ast_{k+1}-\Omega^\ast_{k}\|_F/3 \leq (\gamma^{\min}_{8,T,k,l})^{-1}\|\frac{1}{T^\ast_{k+1}-T^\ast_k}\overset{T^\ast_{k+1}-1}{\underset{r=T^\ast_k}{\sum}}\Big[\Lambda(X_rX^\top_r)\text{vec}(\Omega^\ast_{k})-\text{vec}(I_p)\Big]\|_2\big),$}
\end{eqnarray*}
which also tends to zero, as in the proof of  the convergence to zero of (\ref{bound_A+}). We conclude that $\Pb\big(\{h(\widehat{\Tc}_{\widehat{m}},\Tc^\ast_{m^\ast})> T\delta_T\}\cap \{m^\ast < \widehat{m} \leq m_{\max}\}\big) \rightarrow 0$ as $T \rightarrow \infty$.

\subsection{Proof of Proposition~4}
\noindent\textbf{\emph{Proof of point~(i).}}\\
\noindent We can rewrite (2) as the following constrained optimization problem:
\begin{equation}
	\everymath{\displaystyle}
	\label{eq:ori-D-tr-con-prob}
	\begin{aligned}
		\min_{\mathbf{X}} \quad & \sum_{t=1}^T \left[ \frac{1}{2}\tr(U_t^{\top}U_t) - \tr(\Theta_t) + \delta_{\cdot \succeq \epsilon I_p}(\Theta_t) \right] + \lambda_1T\sum_{t=1}^T\| \Upsilon_{t, \mathrm{off}} \|_{1, \text{off}} + \lambda_2T\sum_{t=1}^{T-1}\| D_t \|_F \\
		\text{s.t.} \quad       & U_t = (X_t X_t ^{\top})^{\frac{1}{2}} \Theta_t, \, \Upsilon_{t, \mathrm{off}} = \Theta_{t, \mathrm{off}}\,\,\,\, \forall t = 1, \dots , T;                                                                                                                     \\
		                        & D_t = \Theta_{t + 1} - \Theta_t \,\,\,\, \forall t = 1, \dots , T-1,
	\end{aligned}
\end{equation}
where we write \( \mathbf{X} = \left\{ \{\Theta_t\}_{t=1}^T, \{U_t\}_{t=1}^T, \{\Upsilon_{t, \mathrm{off}}\}_{t=1}^T, \{D_t\}_{t=1}^{T-1} \right\} \) for short; $\delta_{\cdot \succeq \epsilon I_p}(\cdot)$ is the indicator function of the set $\{S\,:\, S\succeq \epsilon I_p\}$; $ \Upsilon_{t, \mathrm{off}} $ is a matrix whose diagonal elements are 0; $ \Theta_{t, \mathrm{off}} $ is the copy of $ \Theta_t $ with the diagonal elements set to 0.

Denote the dual variables by $ \mathbf{Y} = \{\{W_t\}_{t=1}^T, \{Y_{t, \mathrm{off}}\}_{t=1}^T, \{Z_t\}_{t=1}^{T-1}\} $ for simplicity, where $ W_t \in \Rb^{p\times p} $, $ Y_{t, \mathrm{off}} \in \mathcal{S}_{\mathrm{off}}^p $, and $ Z_t \in \mathcal{S}^p $ for all $ t $.
The Lagrangian function of \eqref{eq:ori-D-tr-con-prob} is
\begin{align*}
  L(\mathbf{X}; \mathbf{Y}) = & \sum_{t=1}^T \left[ \frac{1}{2}\tr(U_t^{\top}U_t) - \tr(\Theta_t) + \delta_{\cdot \succeq \epsilon I_p}(\Theta_t) \right] + \lambda_1T\sum_{t=1}^T \| \Upsilon_{t, \mathrm{off}} \|_{1, \text{off}}                                                                    \\
                              & + \lambda_2T\sum_{t=1}^{T-1} \| D_t \|_F - \sum_{t=1}^T \left\langle W_t, U_t - (X_t X_t^{\top}) ^{\frac{1}{2}} \Theta_t \right\rangle                                                                                                                                 \\
                              & - \sum_{t=1}^T \left\langle Y_{t, \mathrm{off}}, \Theta_{t, \mathrm{off}} - \Upsilon_{t, \mathrm{off}} \right\rangle - \sum_{t=1}^{T-1} \left\langle Z_t, \Theta_{t + 1} - \Theta_t - D_t \right\rangle                                                                \\
  =                           & \sum_{t=1}^T \left[ \frac{1}{2}\tr(U_t^{\top}U_t) - \left\langle W_t, U_t \right\rangle \right] + \sum_{t=1}^T \left[ \lambda_1T \| \Upsilon_{t, \mathrm{off}} \|_{1, \text{off}} + \left\langle Y_{t, \mathrm{off}}, \Upsilon_{t, \mathrm{off}} \right\rangle \right] \\
                              & + \sum_{t=1}^{T - 1}\left[  \lambda_2T \| D_t \|_F + \left\langle Z_t, D_t \right\rangle \right]                                                                                                                                                                       \\
                              & + \sum_{t=1}^T\left[ \left\langle Z_t - Z_{t-1} -I_p + \Sym\left((X_t X_t ^{\top})^{\frac{1}{2}}W_t\right) - Y_{t, \mathrm{off}}, \Theta_t \right\rangle + \delta_{\cdot \succeq \epsilon I_p}(\Theta_t) \right],
\end{align*}
where for convenience, we set \( Z_0 = Z_T = \mathbf{0}_{p \times p} \); we further note that $ \left\langle Y_{t, \mathrm{off}}, \Theta_{t, \mathrm{off}} \right\rangle = \left\langle Y_{t, \mathrm{off}}, \Theta_t \right\rangle $.
Now, let \( \zeta_t = Z_t - Z_{t-1} -I_p + \Sym\left((X_t X_t ^{\top})^{\frac{1}{2}}W_t\right) - Y_{t, \mathrm{off}} \), we have
\begin{align*}
  \min_{U_t}\,\, \frac{1}{2}\tr(U_t ^{\top}U_t) - \left\langle W_t, U_t \right\rangle                                            = & -\frac{1}{2} \tr(W_t^{\top}W_t); \\
  \min_{\Upsilon_{t, \mathrm{off}}}\,\, \lambda_1T\| \Upsilon_{t, \mathrm{off}} \|_{1, \text{off}} + \left\langle Y_{t, \mathrm{off}}, \Upsilon_{t, \mathrm{off}} \right\rangle =               &
  \begin{cases}
    0       & \text{if } | Y_{t,uv} | \leq \lambda_1T \,\,\,\, \forall u \neq v, \\
    -\infty & \text{otherwise};
  \end{cases}                                                                                     \\
  \min_{D_t} \,\, \lambda_2T \| D_t \|_F + \left\langle Z_t, D_t \right\rangle =                                                &
  \begin{cases}
    0       & \text{if } \| Z_t \|_F \leq \lambda_2T, \\
    -\infty & \text{otherwise};
  \end{cases}                                                                                                                \\
  \min_{\Theta_t} \,\, \left\langle \zeta_t, \Theta_t \right\rangle + \delta_{\cdot \succeq \epsilon I_p}(\Theta_t) =                &
  \begin{cases}
    \epsilon \tr(\zeta_t) & \text{if } \zeta_t \succeq 0, \\
    -\infty               & \text{otherwise}.
  \end{cases}
\end{align*}
Therefore, we have the dual problem as in (3). Finally, the equality of the optimal values follow from \cite[Theorem 31.1]{R70} upon noting that there exists $\mathbf{X}$ with $\Theta_t \succ \epsilon I_p$ for all $t$ satisfying the equality constraints in \eqref{eq:ori-D-tr-con-prob}.

\noindent\textbf{\emph{Proof of point~(ii).}}\\
\noindent Since $ \sum_{t=1}^T X_t X_t ^{\top} \succ 0 $, by Lemma~\ref{lemma:strict-feasibility}, the set $\mathcal{C}_{\lambda_1}$ has a Slater point $ \left\{ \{\overline{W}_t\}_{t=1}^T, \{\overline{Y}_{t, \mathrm{off}}\}_{t=1}^T, \{\overline{Z}_t\}_{t=1}^{T-1} \right\} $.
Then the result follows directly with $ \overline{\lambda}_2 = 1 + \max \{\|\overline{Z}_t\|_F\} / T $.
The existence of solutions to the primal problem comes from the strong duality thanks to the strict feasibility of the dual problem; see, for example \cite[Theorem 31.1]{R70}.

\subsection{Proof of Proposition~6}
\noindent\textbf{\emph{Proof of point~(i).}}\\
\noindent We first rewrite (5) as the following constrained optimization problem:
\begin{equation}
	\everymath{\displaystyle}
	\begin{split}
		\min_{\mathbf{X}} \,\, & \sum_{t=1}^T \! \left[\! \frac{1}{2}\tr(U_t^{\top}U_t) \!-\! \tr(\Theta_t) \!+\! \delta_{\cdot \succeq \epsilon I_p}(\Theta_t) \! \right] \!+\! \lambda_1T \!\sum_{t=1}^T\| \Upsilon_{t, \mathrm{off}} \|_{1, \text{off}} \!+\! \lambda_2T\sum_{t=1}^{T-1}\mathcal{R}(\| D_t \|_F; \lambda_3) \\
		\text{s.t.} \,\,       & U_t = (X_t X_t ^{\top})^{\frac{1}{2}} \Theta_t, \, \Upsilon_{t, \mathrm{off}} = \Theta_{t, \mathrm{off}}, \,\,\,\, \forall t = 1, \dots , T;                                                                                                                                           \\
		                        & D_t = \Theta_{t + 1} - \Theta_t \,\,\,\, \forall t = 1, \dots , T - 1,
	\end{split}
	\raisetag{25pt}\label{eq:D-tr-con-prob}
\end{equation}
where we write \( \mathbf{X} = \left\{ \{\Theta_t\}_{t=1}^T, \{U_t\}_{t=1}^T, \{\Upsilon_{t, \mathrm{off}}\}_{t=1}^T, \{D_t\}_{t=1}^{T-1} \right\} \) for short; $ \Upsilon_{t, \mathrm{off}} $ is a matrix whose diagonal elements are 0; $ \Theta_{t, \mathrm{off}} $ is the copy of $ \Theta_t $ with the diagonal elements set to 0.

Denote the dual variables by $ \mathbf{Y} = \{\{W_t\}_{t=1}^T, \{Y_{t, \mathrm{off}}\}_{t=1}^T, \{Z_t\}_{t=1}^{T-1}\} $ for simplicity, where $ W_t \in \Rb^{p\times p} $, $ Y_{t, \mathrm{off}} \in \mathcal{S}_{\mathrm{off}}^p $, and $ Z_t \in \mathcal{S}^p $ for all $ t $.
The Lagrangian function of \eqref{eq:D-tr-con-prob} is
\begin{align*}
  L(\mathbf{X}; \mathbf{Y}) = & \sum_{t=1}^T \left[ \frac{1}{2}\tr(U_t^{\top}U_t) - \tr(\Theta_t) + \delta_{\cdot \succeq \epsilon I_p}(\Theta_t) \right] + \lambda_1T\sum_{t=1}^T \| \Upsilon_{t, \mathrm{off}} \|_{1, \text{off}}                                                                      \\
                              & + \lambda_2T\sum_{t=1}^{T-1} \mathcal{R}(\| D_t \|_F; \lambda_3) - \sum_{t=1}^T \left\langle W_t, U_t - (X_t X_t^{\top}) ^{\frac{1}{2}} \Theta_t \right\rangle                                                                                                        \\
                              & - \sum_{t=1}^T \left\langle Y_{t, \mathrm{off}}, \Theta_{t, \mathrm{off}} - \Upsilon_{t, \mathrm{off}} \right\rangle - \sum_{t=1}^{T-1} \left\langle Z_t, \Theta_{t + 1} - \Theta_t - D_t \right\rangle                                                                   \\
  =                           & \sum_{t=1}^T \left[ \frac{1}{2}\tr(U_t^{\top}U_t) - \left\langle W_t, U_t \right\rangle \right] + \sum_{t=1}^T \left[ \lambda_1T \| \Upsilon_{t, \mathrm{off}} \|_{1, \text{off}} + \left\langle Y_{t, \mathrm{off}}, \Upsilon_{t, \mathrm{off}} \right\rangle \right] \\
                              & + \sum_{t=1}^{T - 1}\left[  \lambda_2T \mathcal{R}(\| D_t \|_F; \lambda_3) + \left\langle Z_t, D_t \right\rangle \right]                                                                                                                                               \\
                              & + \sum_{t=1}^T\left[ \left\langle Z_t - Z_{t-1} -I_p + \Sym\left((X_t X_t ^{\top})^{\frac{1}{2}}W_t\right) - Y_{t, \mathrm{off}}, \Theta_t \right\rangle + \delta_{\cdot \succeq \epsilon I_p}(\Theta_t) \right],
\end{align*}
where for convenience, we set \( Z_0 = Z_T = \mathbf{0}_{p \times p} \); we also note that $ \left\langle Y_{t, \mathrm{off}}, \Theta_{t, \mathrm{off}} \right\rangle = \left\langle Y_{t, \mathrm{off}}, \Theta_t \right\rangle $.
Now, let \( \zeta_t = Z_t - Z_{t-1} -I_p + \Sym\left((X_t X_t ^{\top})^{\frac{1}{2}}W_t\right) - Y_{t, \mathrm{off}} \), we have
\begin{equation}
	\label{eq:Lagrangian-min-modified-I}
	\begin{aligned}
		\min_{U_t}\,\, \frac{1}{2}\tr(U_t ^{\top}U_t) - \left\langle W_t, U_t \right\rangle                                            = & -\frac{1}{2} \tr(W_t^{\top}W_t); \\
		\min_{\Upsilon_{t, \mathrm{off}}}\,\, \lambda_1T\| \Upsilon_{t, \mathrm{off}} \|_{1, \text{off}} + \left\langle Y_{t, \mathrm{off}}, \Upsilon_{t, \mathrm{off}} \right\rangle =               &
		\begin{cases}
			0       & \text{if } | Y_{t,uv} | \leq \lambda_1T \,\,\,\, \forall u \neq v, \\
			-\infty & \text{otherwise};
		\end{cases}                                                                                     \\
		\min_{\Theta_t} \,\, \left\langle \zeta_t, \Theta_t \right\rangle + \delta_{\cdot \succeq \epsilon I_p}(\Theta_t) =                &
		\begin{cases}
			\epsilon \tr(\zeta_t) & \text{if } \zeta_t \succeq 0, \\
			-\infty               & \text{otherwise}.
		\end{cases}
	\end{aligned}
\end{equation}

For the problem $ \min_{D_t} \lambda_2T \mathcal{R}(\| D_t \|_F; \lambda_3) + \left\langle Z_t, D_t \right\rangle $ for each $ t $, it holds that
\[
  \begin{aligned}
    \quad & \min_{D_t} \,\, \lambda_2T \mathcal{R}(\| D_t \|_F; \lambda_3) + \left\langle Z_t, D_t \right\rangle \\
    = \quad & \min \Big\{
    \underbrace{\min_{\| D_t \|_F \leq \lambda_3} \lambda_2 T \| D_t \|_F + \left\langle Z_t, D_t \right\rangle}_{\circnum{1}}, \underbrace{\min_{\| D_t \|_F \ge \lambda_3} \lambda_2 T (\| D_t \|_F^2 - \lambda_3^2 + \lambda_3) + \left\langle Z_t, D_t \right\rangle}_{\circnum{2}}
          \Big\}.
  \end{aligned}
\]

For \circnum{1}, we have
\begin{equation}
	\label{eq:formula-Dt<=lamb3}
	\begin{aligned}
		  & \min_{\| D_t \|_F \leq \lambda_3} \lambda_2 T \| D_t \|_F + \left\langle Z_t, D_t \right\rangle = \min_{\substack{\| D_t \|_F \leq \lambda_3; \\ D_t = \alpha Z_t, \alpha \ge 0}} \lambda_2 T \| D_t \|_F - \| Z_t \|_F \| D_t \|_F \\
		= & \min_{0 \leq r \leq \lambda_3} (\lambda_2 T - \| Z_t \|_F) r = -\left( \| Z_t \|_F - \lambda_2T \right)_+ \lambda_3,
	\end{aligned}
\end{equation}
where the first equality follows from the (equality case in) Cauchy-Schwarz inequality; \( (\cdot)_+ = \max \{\cdot , 0\} \).

For \circnum{2}, one can see that
\begin{align}
	\label{eq:formula-Dt>lamb3}
		                        & \min_{\| D_t \|_F \ge \lambda_3} \lambda_2T(\| D_t \|_F^2 - \lambda_3^2 + \lambda_3) + \left\langle Z_t, D_t \right\rangle                                                  \nonumber\\
		\overset{(\text{a})}{=} & \min_{r \ge \lambda_3 } \lambda_2T ( r^2 - \lambda_3^2 + \lambda_3 ) - \| Z_t \|_F r                                                                                        \nonumber\\
		=                       & \min_{r \ge \lambda_3} \lambda_2T \left( \left( r - \frac{\| Z_t \|_F}{2\lambda_2T} \right)^2 - \frac{\| Z_t \|_F^2}{4\lambda_2^2T^2} - \lambda_3^2 + \lambda_3 \right) \nonumber\\
		\overset{(\text{b})}{=} & \lambda_2T \left( \left(\lambda_3 - \frac{\| Z_t \|_F}{2\lambda_2T}\right)_+^2 - \frac{\| Z_t \|_F^2}{4\lambda_2^2T^2} - \lambda_3^2 + \lambda_3 \right),
\end{align}
where (a) comes from the (equality case in) Cauchy-Schwarz inequality, and (b) is true since the minimum is attained at
$r = \max \left\{ \frac{\| Z_t \|_F}{2\lambda_2T} , \lambda_3 \right\}$. One can see this by first locating the vertex of the quadratic objective.

Using \eqref{eq:formula-Dt<=lamb3} and \eqref{eq:formula-Dt>lamb3}, we have
\[
	\begin{aligned}
		  & \min_{D_t} \,\, \lambda_2T \mathcal{R}(\| D_t \|_F; \lambda_3) + \left\langle Z_t, D_t \right\rangle                                                                                                                                            \\
		= & \min \left\{ -\left( \| Z_t \|_F - \lambda_2T \right)_+ \lambda_3, \lambda_2T \left( \left( \lambda_3 - \frac{\| Z_t \|_F}{2\lambda_2T}\right)_+ ^2 - \frac{\| Z_t \|_F^2}{4\lambda_2^2T^2} - \lambda_3^2 + \lambda_3 \right) \right\}.
	\end{aligned}
\]
The above display is exactly the definition of \( \mathcal{G}(\| Z_t \|_F; \lambda_3) \).
Using this and \eqref{eq:Lagrangian-min-modified-I}, we can conclude that the dual problem of (5) is (6). Finally, the equality of optimal values follows from \cite[Theorem 31.1]{R70} upon noting that there exists $\mathbf{X}$ with $\Theta_t \succ \epsilon I_p$ for all $t$ satisfying the equality constraints in (5).

\noindent\textbf{\emph{Proof of point~(ii).}}\\
\noindent Since $ \sum_{t=1}^T X_t X_t ^{\top} \succ 0 $, by Lemma \ref{lemma:strict-feasibility}, the set $\mathcal{C}_{\lambda_1}$ and hence  the dual problem has a Slater point $ \left\{ \{\overline{W}_t\}_{t=1}^T, \{\overline{Y}_{t, \mathrm{off}}\}_{t=1}^T, \{\overline{Z}_t\}_{t=1}^{T-1} \right\} $.
The existence of solutions to the primal problem comes from strong duality thanks to the strict feasibility of the dual problem; see, for example \cite[Theorem 31.1]{R70}.

\subsection{Proof of Proposition~7}
\noindent For simplicity, for a given pair of fixed $\lambda_1$ and $\lambda_2$, we denote the objective functions of (2) and (5) with $ \lambda_3 $ by $ F $ and $ G_{\lambda_3} $, respectively.
From the definition of $ \mathcal{G}(\cdot; \lambda_3) $ in (4), we know that
\begin{equation}
  \label{eq:F<=G}
  F\left( \{\Theta_t\}_{t=1}^T \right) \leq G_{\lambda_3}\left(  \{\Theta_t\}_{t=1}^T  \right) \text{ for any } \{\Theta_t\}_{t=1}^T.
\end{equation}

\noindent\textbf{\emph{Proof of point~(i).}}\\
\noindent Suppose that $ \lambda_1, \lambda_2 $ are such that (2) has solutions.
Let $ \{\Theta_t^{*}\}_{t=1}^T $ be an arbitrary solution to (2) and define
$$
	\overline{\lambda}_3 = \max \left\{\max_{t=1, \dots ,T-1} \{\| \Theta_{t+1}^{*} - \Theta_t^{*} \|_F\}, 0.5 \right\}.
$$
Then it holds that
\begin{equation}
  \label{eq:F=G}
  F\left(\{\Theta_t^{*}\}_{t=1}^T\right) = G_{\lambda_3}\left(\{\Theta_t^{*}\}_{t=1}^T\right) \text{ for any } \lambda_3 \geq \overline{\lambda}_3.
\end{equation}

Fix any $ \lambda_3 \geq \overline{\lambda}_3 $. Suppose that $ \{\widehat{\Theta}_t\}_{t=1}^T $ is a solution to (5) with this $ \lambda_3$, then we have
$$
F\left( \{\widehat{\Theta}_t\}_{t=1}^T \right) \overset{(\text{a})}{\leq} G_{\lambda_3}\left( \{\widehat{\Theta}_t\}_{t=1}^T \right) \overset{(\text{b})}{\leq} G_{\lambda_3}\left( \{\Theta_t^{*}\}_{t=1}^T \right) \overset{(\text{c})}{=} F\left( \{\Theta_t^{*}\}_{t=1}^T \right),
$$
where (a) comes from \eqref{eq:F<=G}; (b) holds thanks to the assumption that $ \{\widehat{\Theta}_t\}_{t=1}^T $ is a solution to (5) with $ \lambda_3 $; (c) is true because of \eqref{eq:F=G}.
Therefore, $ \{\widehat{\Theta}_t\}_{t=1}^T $ is also a solution to (2).

\noindent\textbf{\emph{Proof of point~(ii).}}\\
\noindent
It suffices to show that, for arbitrary fixed $\lambda_1, \lambda_2$, if there exists $\lambda_3 \geq 0.5$ such that there exists a solution $\{\Theta_t^{*}\}_{t=1}^T$ to (5) with $\lambda_3$ that satisfies
\begin{equation}
  \label{eq:cond-diff<l3}
  \max_{t=1, \dots, T-1} \| \Theta_{t+1}^{*} - \Theta_t^{*} \|_F < \lambda_3,
\end{equation}
then (2) has solutions. To this end, we notice from \eqref{eq:cond-diff<l3} and the definition of $\mathcal{R}$ in (4) that
\begin{equation}
  \label{eq:same-subdiff}
  \partial F(\{\Theta_t^{*}\}_{t=1}^T) = \partial G_{\lambda_3}(\{\Theta_t^{*}\}_{t=1}^T),
\end{equation}
where $\partial F(\{\Theta_t^{*}\}_{t=1}^T)$ and $\partial G_{\lambda_3}(\{\Theta_t^{*}\}_{t=1}^T)$ are the subdifferentials of $F$ and $G_{\lambda_3}$ at $\{\Theta_t^{*}\}_{t=1}^T$, respectively.

Given that $\{\Theta_t^{*}\}_{t=1}^T$ is a solution to (5), the optimality condition implies
$$
0 \in \partial G_{\lambda_3}(\{\Theta_t^{*}\}_{t=1}^T).
$$
This, along with \eqref{eq:same-subdiff}, shows that $\{\Theta_t^{*}\}_{t=1}^T$ is also a solution to (2).

\end{appendices}
\vskip 0.2in

\bibliography{biblio}

\end{document}